\documentclass[a4paper,11pt]{amsart}
\usepackage{quiver}
\usepackage{pdflscape}

\usepackage{amsthm,amsmath,amssymb,stmaryrd,enumerate,euscript,mathrsfs,a4,array, microtype, stmaryrd, url}
\usepackage{multirow}
\usepackage{arydshln}
\usepackage{pictexwd,dcpic,graphicx,cite}
\usepackage[export]{adjustbox}
\usepackage{subcaption}
\usepackage[latin1]{inputenc}
\usepackage{indentfirst}
\usepackage{calligra}
\usepackage{graphicx}
\usepackage{amssymb}
\usepackage{fancyhdr}
\usepackage{amsmath}
\usepackage{latexsym}
\usepackage{amsthm}
\usepackage[all]{xy} 
\usepackage{amscd}
\usepackage{hyperref}
\usepackage{tikz-cd}
\usepackage{tikz}
\usetikzlibrary{nfold} 
\usepackage{cleveref}
\usepackage[a4paper,top=3cm,bottom=3cm,inner=3cm,outer=3cm]{geometry}
\usepackage{tikz}
\usetikzlibrary{positioning,decorations.markings,arrows.meta,calc,fit,quotes,cd,math,arrows,backgrounds,shapes.geometric}
\usepackage{amsmath,mathtools}
\usepackage{xparse}

\makeatletter
\newenvironment{smallermatrix}[1][c]
{\null\,\vcenter\bgroup
  \Let@\restore@math@cr\default@tag
  \baselineskip0pt \lineskip0.4pt \lineskiplimit0pt
  \ialign\bgroup\if#1l\else\hfil\fi$\m@th\scriptstyle##$\if#1r\else\hfil\fi&&\thickspace\hfil
  $\m@th\scriptstyle##$\hfil\crcr
}{%
  \crcr\egroup\egroup\,%
}
\makeatother

\NewDocumentCommand{\ts}{O{c} e{^?_}}{
  \begin{smallermatrix}[#1]
  \mathstrut\IfValueT{#2}{#2} \\
  \mathstrut\IfValueT{#3}{#3} \\
  \mathstrut\IfValueT{#4}{#4}
  \end{smallermatrix}%
}

\makeatletter
\def\hlinewd#1{%
\noalign{\ifnum0=`}\fi\hrule \@height #1 \futurelet
\reserved@a\@xhline}
\makeatother

\graphicspath{ {images/} }

\theoremstyle{plain}
\newtheorem{theorem}{Theorem}[section]
\newtheorem*{theorem*}{Theorem}
\newtheorem{prop}[theorem]{Proposition}

\newtheorem{lemma}[theorem]{Lemma}

\theoremstyle{definition} 
\newtheorem{defn}[theorem]{Definition}

\theoremstyle{definition}
\newtheorem{rmk}[theorem]{Remark}
\newtheorem{eg}[theorem]{Example}

\newtheorem*{rmk*}{Remark}
\newtheorem*{rmks*}{Remarks}
\newtheorem*{notation*}{Notation}
\newtheorem*{convention*}{Convention}

\newcommand{\co}{\mathit{co}}

\newcommand{\cat}{\mathcal} 

\newcommand{\catA}{\cat{A}}
\newcommand{\catB}{\cat{B}}
\newcommand{\catC}{\cat{C}}
\newcommand{\catD}{\cat{D}}
\newcommand{\catE}{\cat{E}}
\newcommand{\ca}{\cat{A}}
\newcommand{\cb}{\cat{B}}
\newcommand{\cc}{\cat{C}}
\newcommand{\cd}{\cat{D}}
\newcommand{\ce}{\cat{E}}

\newcommand{\cv}{\cat{V}}
\newcommand{\cx}{\cat{X}}
\newcommand{\cy}{\cat{Y}}
\newcommand{\cj}{\cat{J}}

\newcommand{\mul}{\mathbb}

\newcommand{\mc}{\mul{C}}
\newcommand{\md}{\mul{D}}

\newcommand{\ck}{\mathcal{K}}
\newcommand{\cl}{\mathcal{L}}

\newcommand{\tcat}{\mathbf}

\newcommand{\Gray}{\tcat{2}\text{-}\tcat{Cat}}
\newcommand{\GCat}{\tcat{G}\text{-}\tcat{Cat}}
\newcommand{\GCatp}{\tcat{G}\text{-}\tcat{Cat}_p}
\newcommand{\GCatm}{\mathbb{L}\tcat{ax}}

\newcommand{\GCatmp}{\mathbb{P}\tcat{sd}}

\newcommand{\multab}{\mathbb{A}\tcat{b}}
\newcommand{\atwo}{{\mathbf 2}}
\newcommand{\dcat}[1]{\bat{\mathbb #1}}
\newcommand{\Surj}{\mathbb{S}\tcat{urj}}
\newcommand{\bat}[1]{\mathbf{#1}}
\newcommand{\Sq}[1]{\mathbb S\bat{q}(#1)}

\newcommand{\kmulti}{\mathbf{Mult}_k}

\newcommand{\skkmulti}{\mathbf{SkMult}_k}

\newcommand{\Cat}{\tcat{Cat}}
\newcommand{\Set}{\mathbf{Set}}

\newcommand{\lax}{\mathbf{Lax}}
\newcommand{\psd}{\mathbf{Psd}}
\newcommand{\adje}{\mathbf{AE}}
\newcommand{\oplax}{\mathbf{Oplax}}

\makeatletter
\protected\def\xvcenter{%
  \hbox\bgroup$\everyvbox{\everyvbox{}\aftergroup\m@th\aftergroup$\aftergroup\egroup}%
  \vcenter
}

\DeclareRobustCommand{\midscript}[1]{
  \mathchoice{\mid@script\scriptstyle{#1}}
    {\mid@script\scriptstyle{#1}}
    {\mid@script\scriptscriptstyle{#1}}
    {\mid@script\scriptscriptstyle{#1}}
}
\newcommand{\mid@script}[2]{
  \vcenter{\hbox{$\m@th#1#2$}}
}

\makeatletter

\makeindex
\title{A skew approach to enrichment for Gray-categories}
\author{John Bourke}
\address{Department of Mathematics and Statistics, Masaryk University, Kotl\'a\v rsk\'a 2, Brno 61137, Czech Republic}
\email{bourkej@math.muni.cz}

\author{Gabriele Lobbia}
\address{Department of Mathematics and Statistics, Masaryk University, Kotl\'a\v rsk\'a 2, Brno 61137, Czech Republic}
\email{lobbia@math.muni.cz}
\thanks{The first-named author acknowledges the support of the Grant agency of the Czech Republic under the grant 22-02964S}

\begin{document}

\begin{abstract}
It is well known that the category of Gray-categories does not admit a monoidal biclosed structure that models weak higher-dimensional transformations.  In this paper, the first of a series on the topic, we describe several skew monoidal closed structures on the category of Gray-categories, one of which captures higher lax transformations, and another which models higher pseudo-transformations.  
\end{abstract}

\maketitle


\section{Introduction}

The idea of approaching higher-dimensional categories via iterated enrichment is an appealingly simple one.  Indeed, if $\cv$ is a symmetric monoidal category, then so is the category $\cv$-$\Cat$ of small $\cv$-enriched categories.  Now taking $0$-$\Cat = \Set$ to be the cartesian closed category of sets and defining $(n+1)\textnormal{-}\Cat := (n\textnormal{-}\Cat)\textnormal{-}\Cat$ yields $1$-categories, $2$-categories and onwards to $n$-categories for all finite $n$.  Moreover, the categories $n$-$\Cat$ are very well behaved --- in addition to be cartesian closed, they are locally finitely presentable.

The downside is that the $n$-categories obtained via this process are the strict ones.  In fact, whilst every weak $2$-category, or bicategory, is equivalent to a strict $2$-category, not every weak $3$-category is equivalent to a strict one.  In particular, beyond dimension $2$, strict $n$-categories are insufficiently general.  This rules out iterative enrichment in \emph{cartesian closed categories} as a framework for (weak) higher-dimensional categories. 


There is an elegant solution to this problem in dimension $2$, which allows us to move one step further.  Whilst $\Gray$ is cartesian closed, it also admits a number of non-cartesian monoidal biclosed structures, including the lax and pseudo variants of the Gray-tensor product.   The latter of these is symmetric monoidal closed, and we will refer to as the Gray-tensor product $\otimes_p$.  It has a number of good properties.

\begin{enumerate}
\item Given 2-categories $\ca$ and $\cb$ the morphisms of the internal hom $\psd(\ca,\cb)$ are \emph{pseudonatural transformations}, the most important transformations in 2-category theory.
\item The Gray tensor product equips the model category $\Gray$ with the structure of a \emph{monoidal model category} \cite{Lack2002A-quillen}; in particular, it equips the homotopy category of $\Gray$ with the structure of a symmetric monoidal closed category.
\item Each weak $3$-category is equivalent to a \emph{Gray-category} \cite{Gordon1995Coherence}: a category enriched in $(\Gray,\otimes_p)$.
\end{enumerate}

By virtue of (3), it suffices to work in a Gray-category rather than a weak $3$-category-- see, for example, \cite{Day1997Monoidal}.  Gray-categories are more manageable than weak $3$-categories, and differ primarily from strict ones only in that the \emph{middle four interchange} does not hold on the nose, but rather up to coherent isomorphism.  For these reasons, Gray-categories are often called \emph{semistrict} $3$-categories.


Given this, it is natural to ask whether there exists an analogous symmetric monoidal closed structure on the category $\GCat$ of Gray-categories satisfying similar good properties to the above ones.  However, as argued by Crans in the introduction to \cite{Crans1999A-tensor}, it is impossible that there exists even a monoidal biclosed structure capturing weak transformations.  

The obstruction is as follows.  Suppose that $[\ca,\cb]$ is the internal hom for a monoidal biclosed structure, and suppose that its 1-cells $\eta\colon F \to G$ are \emph{weak transformations}.  In particular, letting $\mathbf 2 = (0 \to 1)$, the Gray-category $[\mathbf 2,\cb]$ would have arrows as its objects, and some kind of lax or pseudo-commuting squares as its $1$-cells, with pasting of these squares as composition.  

Now $\eta\colon F \to G \in [\ca,\cb]$ can be viewed as a Gray-functor $\mathbf 2 \to [\ca,\cb]$ and so, by biclosedness, as a Gray-functor $\eta\colon \ca \to [\mathbf 2 ,\cb]$.  This would send $x$ to the component $\eta_{x}\colon Fx \to Gx$ and $f\colon x \to y$ to a lax square $\eta_f$ as in the diagram below left; but now functoriality in $\ca$ would force the equation 
\begin{equation}
\label{eq:comp}
\begin{gathered}
\begin{tikzcd}[ampersand replacement=\&]
	Fx \& Gx \\
	Fy \& Gy \\
	Fz \& Gz
	\arrow[""{name=0, anchor=center, inner sep=0}, "{\eta_x}", from=1-1, to=1-2]
	\arrow[""{name=1, anchor=center, inner sep=0}, "{\eta_y}"{description}, from=2-1, to=2-2]
	\arrow["Gf", from=1-2, to=2-2]
	\arrow["Gg", from=2-2, to=3-2]
	\arrow["Ff"', from=1-1, to=2-1]
	\arrow["Fg"', from=2-1, to=3-1]
	\arrow[""{name=2, anchor=center, inner sep=0}, "{\eta_z}"', from=3-1, to=3-2]
	\arrow["{\eta_f}", shorten <=9pt, shorten >=9pt, Rightarrow, from=0, to=1]
	\arrow["{\eta_g}", shorten <=9pt, shorten >=9pt, Rightarrow, from=1, to=2]
\end{tikzcd} \hspace{0.5cm}=\hspace{0.5cm}
\begin{tikzcd}[ampersand replacement=\&]
	Fx \& Gx \\
	Fz \& Gz
	\arrow[""{name=0, anchor=center, inner sep=0}, "{\eta_x}", from=1-1, to=1-2]
	\arrow[""{name=1, anchor=center, inner sep=0}, "{\eta_z}"', from=2-1, to=2-2]
	\arrow["{F(gf)}"', from=1-1, to=2-1]
	\arrow["{G(gf)}", from=1-2, to=2-2]
	\arrow["{\eta_{gf}}", shorten <=9pt, shorten >=9pt, Rightarrow, from=0, to=1]
\end{tikzcd}
\end{gathered}
\end{equation}
to hold.  But this equation, unnaturally strict at the $3$-dimensional level of Gray-categories, is in fact \emph{impossible} --- the failure of the middle-4-interchange in Gray-categories prevents such transformations from being composable.   

A careful treatment of the above obstruction is given in \cite{Bourke2015A-cocategorical}, which also shows that it is impossible to have a monoidal biclosed structure on $\GCat$ interacting well with Lack's model structure on $\GCat$ \cite{Lack2011A-quillen}.  These negative results indicate that if we wish to construct a homotopically well-behaved tensor product of Gray-categories, encoding an element of weakness appropriate to semistrict $4$-categories, then we should move outside the usual settings for enriched category theory --- namely, symmetric monoidal closed and monoidal biclosed categories.  But what sort of enriching structure might we look for?


Looking again at the above obstruction, we see that the problem stems from biclosedness, which forces weak transformations to be Gray-functors $\ca \to [\mathbf 2,\cb]$, and so forces Equation~\eqref{eq:comp} to hold strictly.  On the other hand, it would be more natural if this equation held only up to an invertible $3$-cell $\eta_{g,f}$ 

\begin{equation}
\label{eq:comp2}
\begin{gathered}
\begin{tikzcd}[ampersand replacement=\&]
         Fx \& Gx \\
	Fy \& Gy \\
	Fz \& Gz
        \arrow[""{name=0, anchor=center, inner sep=0}, "{\eta_x}", from=1-1, to=1-2]
	\arrow[""{name=1, anchor=center, inner sep=0}, "{\eta_z}"{description}, from=2-1, to=2-2]
	\arrow["Gf", from=1-2, to=2-2]
	\arrow["Gg", from=2-2, to=3-2]
	\arrow["Ff"', from=1-1, to=2-1]
	\arrow["Fg"', from=2-1, to=3-1]
	\arrow[""{name=2, anchor=center, inner sep=0}, "{\eta_z}"', from=3-1, to=3-2]
	\arrow["{\eta_f}", shorten <=9pt, shorten >=9pt, Rightarrow, from=0, to=1]
	\arrow["{\eta_g}", shorten <=9pt, shorten >=9pt, Rightarrow, from=1, to=2]
\end{tikzcd} \hspace{0.5cm}\overset{\eta_{g,f}}{\Rrightarrow}\hspace{0.5cm}
\begin{tikzcd}[ampersand replacement=\&]
	Fx \& Gx \\
	Fz \& Gz.
	\arrow[""{name=0, anchor=center, inner sep=0}, "{\eta_x}", from=1-1, to=1-2]
	\arrow[""{name=1, anchor=center, inner sep=0}, "{\eta_z}"', from=2-1, to=2-2]
	\arrow["{F(gf)}"', from=1-1, to=2-1]
	\arrow["{G(gf)}", from=1-2, to=2-2]
	\arrow["{\eta_{gf}}", shorten <=9pt, shorten >=9pt, Rightarrow, from=0, to=1]
\end{tikzcd}
\end{gathered}
\end{equation}

and such a $3$-cell could be encoded by some sort of \emph{weak map} $\ca \rightsquigarrow [\mathbf 2,\cb]$.   Given this, we might try abandoning strict Gray-functors and work with a category $\GCat_p$ of Gray-categories and weak maps.  However, this leads to new problems.  Firstly, categories of weak maps are typically poorly behaved, lacking general limits and colimits.  Consequently, even if we had a suitable internal hom $[\ca,\cb]$ on such a category, it is unlikely that we would be able to find a representing tensor product for it.  

Therefore what we ultimately want is
\begin{itemize}
\item a structure on $\GCat$ encoding a notion of weak map $\ca \rightsquigarrow \cb$, in addition to a tensor product $\ca \otimes \cb$ and internal hom $[\ca,\cb]$, all interacting nicely.
\end{itemize}


Though not immediately apparent, the categorical structure that encodes these various features is that of a \emph{closed skew monoidal category}.  Skew monoidal categories were introduced by Szlachanyi in 2012 \cite{Szlachanyi2012Skew} in the study of bialgebroids over rings and have since found applications in diverse areas, including $2$-category theory \cite{Bourke2016Skew}, operad theory \cite{Lack2018Operadic}, dg-categories \cite{Shoikhet2018}, cartesian differential categories \cite{Garner2021Cartesian} and theoretical computer science \cite{Altenkirch}.  As with monoidal categories, they are categories equipped with a tensor product $\otimes\colon\cc^2 \to \cc$ and unit object $i$ but now the coherence constraints are non-invertible and have the form
$$\alpha\colon (a \otimes b) \otimes c \to a \otimes (b \otimes c),\, l\colon i \otimes a \to a,\, r\colon a \to a \otimes i.$$
 Whilst this choice seems a little arbitrary at first, it has a number of interesting consequences.  For instance, the endofunctor $i \otimes -$ and $- \otimes i$ underlie a comonad and monad respectively, related by a distributive law.  It follows that that each skew monoidal category encodes a notion of weak morphism $a \rightsquigarrow b$, defined as a morphism $i \otimes a \to b$ in the Kleisli category for the comonad.  
 
Similarly to the monoidal case, a skew monoidal category is closed if each $- \otimes a$ has a right adjoint $[a,-]$.  Much as in the classical context of Eilenberg-Kelly \cite{EilenbergS:cloc}, this yields a \emph{skew closed category} in the sense of Street \cite{Street2013Skew}, which involves a function
$$\cc(a,b) \to \cc(i,[a,b]).$$
Unlike in the classical setting, however, this is not necessarily invertible --- indeed, as emphasised in \cite{Bourke2016Skew}, the elements $i \to [a,b]$ are the weak morphisms $i \otimes a \to b$, rather than just the \emph{strict} morphisms $a \to b$ of $\cc$.

Here, we find the first hint of a connection between skew structures and $\GCat$.  In 2012, Gohla \cite{GohlaB:mapp} described a notion of pseudomap $F\colon \ca \rightsquigarrow \cb$ of Gray-categories and a so-called \emph{mapping space} $[\ca,\cb]$, here called $\lax(\ca,\cb)$.  
These pseudomaps $F\colon \ca \to \cb$ are rather strict --- as with normal pseudofunctors in $2$-category theory, the only weakening is an invertible coherence $2$-cell $F^2_{g,f}\colon F(g)F(f) \cong F(gf)$ parametrised by a composable pair of $1$-cells $(f\colon x \to y,g\colon y \to z)$.  In fact, this is just the coherence data that we need to introduce an invertible $3$-cell as in \eqref{eq:comp2}.

With regards the mapping-space $\lax(\ca,\cb)$, its objects are pseudomaps whilst the $1$-cells are \emph{lax transformations}, which involve data $(\eta_a,\eta_f,\eta_{g,f})$ as specified above, amongst further coherence data. We note that the lax transformations of \cite{GohlaB:mapp} are \emph{normal lax}, in the sense that they satisfy the normality condition that $\eta_f$ is an identity if $f$ is an identity.

Our first main goal is to prove the following result.

\newtheorem*{restate:laxskewmonoidal}{Theorem~\ref{thm:laxskewmonoidal}}
\begin{restate:laxskewmonoidal}
There is a closed skew monoidal structure $(\GCat, \otimes_l, 1)$ whose internal hom is $\lax(\ca,\cb)$.
\end{restate:laxskewmonoidal}

Our second goal is to prove an analogue of this result, involving a mapping space $\psd(\ca,\cb)$ of \emph{pseudo-transformations} introduced here.  In these, the $2$-cells $\eta_f$ form part of specified adjoint equivalences satisfying further coherence equations.  Our second main result is then:

\newtheorem*{restate:pseudoskewmonoidal}{Theorem~\ref{thm:pseudoskewmonoidal}}
\begin{restate:pseudoskewmonoidal}
There is a closed skew monoidal structure $(\GCat, \otimes_p, 1)$ whose internal hom is $\psd(\ca,\cb)$.
\end{restate:pseudoskewmonoidal}


Before diving into the paper, let us now give an outline of our approach to proving our main theorems.  Given that they involve highly complex structures --- see the appendix for the full axiomatics --- a key challenge is not just to prove them, but to prove them clearly and efficiently.  With this in mind, let us consider the various possible approaches to skew structures, and identify which will work best for our purpose.  There are several essentially interchangable approaches, as drawn in the following schematic diagram.

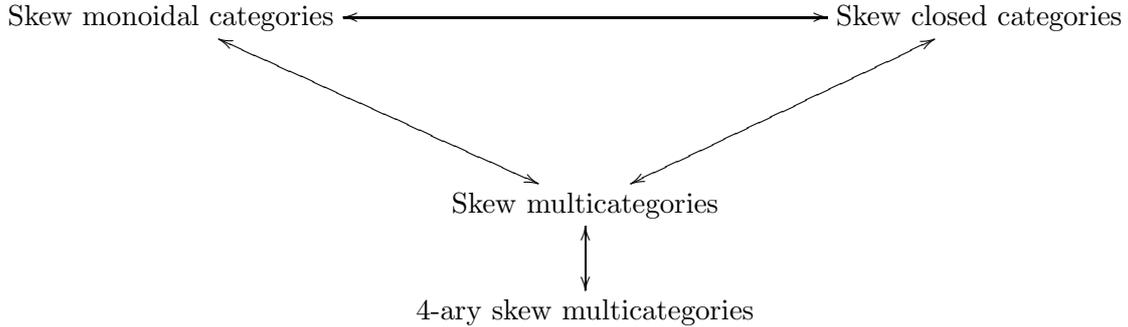
\begin{equation*}
\xymatrix{\textnormal{Skew monoidal categories} \ar@{<->}[rr] && \textnormal{Skew closed categories}\\\\
& \textnormal{Skew multicategories} \ar@{<->}[uul] \ar@{<->}[uur]\\
& \textnormal{$4$-ary skew multicategories} \ar@{<->}[u]}
\end{equation*}

The first approach would be to work with skew monoidal categories directly, constructing the skew monoidal structure on $\GCat$ and showing that it is closed afterwards.  Whilst the skew monoidal axioms are reasonably straightforward, the disadvantage is that tensor products are typically describable only via generators and relations, and not easy to work with directly.  

The second approach would be to use skew closed categories $(\cc,[-,-],i)$, which involve coherence constraints $L\colon[b,c] \to [[a,b],[a,c]]$, $i\colon[i,a] \to a$ and $j\colon i \to [a,a]$ \cite{Street2013Skew}.  Skew closed categories are said to be monoidal if each endofunctor $[a,-]$ has a left adjoint and the resulting \emph{monoidal skew closed categories} then amount to the same as closed skew monoidal categories.  The advantage of focusing on the skew closed structure is that the internal hom, as in our present example, can often be described directly.  The disadvantage is that the axiomatics, especially those involving $L$, are relatively unintuitive, and this can make constructing such a skew structure an ordeal.  Following an explanation of how this should go by the first-named-author, in 2016 Gohla \cite{GohlaB:incomplete} made a start in following this approach in a more restrictive setting.\begin{footnote}{Gohla considered a version where $[A,B]$ contains strict Gray-functors and a variant of pseudo-transformations involving invertible $2$-cells rather than adjoint equivalences.}\end{footnote}  He defined a $3$-graph map $L\colon[b,c] \to [[a,b],[a,c]]$ between Gray-categories but verified its functoriality only in dimension $1$.  Moreover, he did not verify naturality of $L$, nor describe the unit or unit maps, nor consider any of the axioms for a skew closed category.  


A third approach to skew structures was introduced by the first-named-author and Lack in 2018 \cite{BourkeLack:skew}, namely \emph{skew multicategories}.  Skew multicategories are generalised multicategories with two kinds of multimap $A_1,\ldots,A_n \to B$, one called ``tight" and the other ``loose".  Amongst these, the left-representable closed skew multicategories again amount to the same as closed skew monoidal categories.  
The advantage of the multicategorical approach is that one can define multimaps in a direct way and the axioms are relatively straightforward.  The disadvantage is that, unlike the other two concepts, which have finite definitions, here one has $n$-ary multimaps for all finite $n$, and infinitely many axioms to check (albeit, a finite axiom schema thereof).   This is potentially problematic for us, since in our complex example, we wish to do as few calculations as possible.

Fortunately, this is not a fundamental problem.  Indeed, since monoidal categories and representable multicategories are equivalent concepts \cite{Hermida2000Representable} and since the former only require considering tensor products of four objects at a time, it is natural to expect that representable multicategories can be described in a way that involves only multimaps of at most dimension $4$.  We call the resulting structures \emph{$4$-ary multicategories} and there is also a skew version.  Again, each left-representable closed $4$-ary skew multicategory gives rise to a closed skew monoidal category and we will employ this result in proving our main theorems.

Having settled on our approach, let us now outline the construction of the $4$-ary skew multicategory $\GCatm$ which is the technical heart of the paper.
\begin{itemize}
\item Objects will be Gray-categories $\ca$ and nullary maps $a\colon \diamond \to \ca$ are objects of $\ca$;
\item Tight and loose unary maps $\ca \to \cb$ are Gray-functors and Gohla's pseudomaps;
\item The \emph{definitions} of tight and loose binary maps $F\colon \ca,\cb \to \cc$ are chosen so that they bijectively correspond to tight and loose unary maps $\lambda_2^{-1}(F)\colon \ca \to \lax(\cb,\cc)$ into the mapping space;
\item Similarly, the $3$-ary and $4$-ary maps $F$ will be defined so that they correspond to $2$-ary and $3$-ary maps $\lambda_3^{-1} F$ and $\lambda_4^{-1} F$ into the mapping space;
\end{itemize}

Building on work of Gohla \cite{GohlaB:mapp}, we will show that his mapping space $\lax(\ca,\cb)$ extends to a bifunctor $\lax(-,-)$ defined on pseudomaps.  Transporting across the bijections $\lambda_i$, this enables us to define substitutions of unary maps with multimaps of all dimensions.\\
Of the remaining substitutions, the most important are the two substitutions of binary into binary.   The first of these, $G \circ_1 F$, is depicted below left.
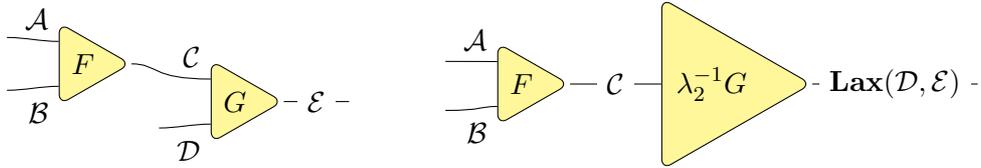
\begin{center}
\begin{tikzpicture}[triangle/.style = {fill=yellow!50, regular polygon, regular polygon sides=3,rounded corners}]

\path
	(6.5,1.5) node [triangle,draw,shape border rotate=-90,label=135:$\ca$,label=230:$\cb$,inner sep=2pt] (b') {$F$}
	(8.5,1) node [triangle,draw,shape border rotate=-90,label=135:$\cc$,label=230:$\cd$,inner sep=2pt] (c') {$G$};

\draw [-] (7.5,0.65) .. controls +(right:0.25cm) and +(left:0.25cm).. (c'.223);
\draw [] (b') .. controls (7.5,1.5) and +(left:0.8cm).. (c'.137);
\draw [] (c') to node[fill=white] {$\ce$} (10,1);
\draw [-] (5.5,1.15) .. controls +(right:0.25cm) and +(left:0.25cm).. (b'.223);
\draw [-] (5.5,1.85) .. controls +(right:0.25cm) and +(left:0.25cm).. (b'.137);
\end{tikzpicture}
\hspace{1cm}
\begin{tikzpicture}
[triangle/.style = {fill=yellow!50, regular polygon, regular polygon sides=3,rounded corners}]

\path
	(8.5,1) node [triangle,draw,shape border rotate=-90,label=135:$\ca$,label=230:$\cb$,inner sep=2pt] (c') {${F}$}
(11,1) node [triangle,draw,shape border rotate=-90,inner sep=0pt] (c'') {$\lambda_2^{-1} G$};

\draw [-] (7.5,0.65) .. controls +(right:0.5cm) and +(left:0.25cm).. (c'.223);
\draw [] (7.5,1.3) .. controls +(right:0.5cm) and +(left:0.25cm).. (c'.137);
\draw [] (c') to node[fill=white] {$\cc$} (c'');
\draw [-] (c'') to node[fill=white] {$\lax(\cd,\ce)$} (14.5,1);
\end{tikzpicture}
\end{center}

We define it by first transposing $G$ as above right, then substituting binary into unary, before transposing back and so defining $G \circ_1 F:=\lambda_3(\lambda_2^{-1} G \circ_1 F)$.  

Substitution $\circ_2$ in the second variable is depicted below left.  This case is crucial, since it encodes the same information as the associator $\alpha$ in a skew monoidal category or the map $L$ in a skew closed one.

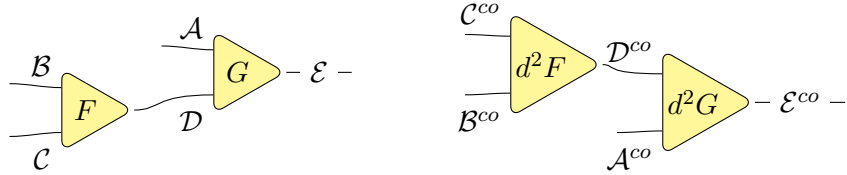
\begin{center}
\begin{tikzpicture}[triangle/.style = {fill=yellow!50, regular polygon, regular polygon sides=3,rounded corners}]

\path
        (6.5,0.5) node [triangle,draw,shape border rotate=-90,label=135:$\cb$,label=230:$\cc$,inner sep=2pt] (q) {$F$} 
	(8.5,1) node [triangle,draw,shape border rotate=-90,label=135:$\ca$,label=230:$\cd$,inner sep=2pt] (c') {$G$};

\draw [-] (q) .. controls (7.5,0.5) and +(left:0.8cm).. (c'.223);
\draw [] (7.5,1.35) .. controls +(right:0.25cm) and +(left:0.25cm).. (c'.137);
\draw [] (c') to node[fill=white] {$\ce$} (10,1);
\draw [-] (5.5,0.15) .. controls +(right:0.25cm) and +(left:0.25cm).. (q.223);
\draw [-] (5.5,0.85) .. controls +(right:0.25cm) and +(left:0.25cm).. (q.137);
\end{tikzpicture}
\hspace{1cm}
\begin{tikzpicture}[triangle/.style = {fill=yellow!50, regular polygon, regular polygon sides=3,rounded corners}]

\path
	(6.5,1.5) node [triangle,draw,shape border rotate=-90,label=135:$\cc^{co}$,label=230:$\cb^{co}$,inner sep=-1.5pt] (b') {$d^2F$}
	(8.5,1) node [triangle,draw,shape border rotate=-90,label=135:$\cd^{co}$,label=230:$\ca^{co}$,inner sep=-1.5pt] (c') {$d^2G$};

\draw [-] (7.5,0.6) .. controls +(right:0.25cm) and +(left:0.25cm).. (c'.223);
\draw [] (b') .. controls (7.5,1.5) and +(left:0.8cm).. (c'.137);
\draw [] (c') to node[fill=white] {$\ce^{co}$} (10.5,1);
\draw [-] (5.5,1.1) .. controls +(right:0.25cm) and +(left:0.25cm).. (b'.223);
\draw [-] (5.5,1.9) .. controls +(right:0.25cm) and +(left:0.25cm).. (b'.137);
\end{tikzpicture}
\end{center}
We obtain  $\circ_2$ using the \emph{duality} $(-)^{co}$ which reverses the orientation of $2$-cells in a Gray-category.  In particular, using this, we show that loose $2$-ary maps $F\colon \ca,\cb \to \cc$ correspond to loose $2$-ary maps $d^2 F\colon \cb^{co},\ca^{co} \to \cc^{co}$; this allows then to form the composite above right, substitute those using the already defined $\circ_1$, and then apply the inverse of the higher duality $d^3$ to obtain
$G \circ_2 F := d^3(d^2 G \circ_1 d^2 F)\colon\ca,\cb,\cc \to \ce$.  

Following this idea, the meat of our approach boils down to establishing the \emph{closedness} and \emph{duality bijections} $\lambda_i$ and $d^i$ for $i \leq 4$.  The remaining details, including the verification of the axioms are then completely straightforward.  

Having outlined the construction of the $4$-ary skew multicategory, let us now give a more detailed overview of the paper.  
In Section~\ref{sect:structures} we introduce skew monoidal categories and the various flavours of multicategory that we will need, before describing the construction of a skew monoidal closed category from a $4$-ary skew multicategory satisfying the appropriate properties.  

In Section~\ref{sect:mappingspace}, we give some background on Gray-categories, before discussing Gohla's pseudomaps and his lax mapping space $\lax(\ca,\cb)$, as well as an oplax variant $\oplax(\ca,\cb)$.  The technical heart of the paper is Section~\ref{sect:shortmult}, in which we construct the $4$-ary skew multicategory $\GCatm$.  Studying its properties in Section~\ref{sect:skewmonlax}, we are able to prove, in Theorem~\ref{thm:laxskewmonoidal}, our first main theorem --- namely, the existence of the skew monoidal closed structure $(\GCat,\otimes_l,1)$ with internal hom $\lax(\ca,\cb)$.

In Section~\ref{sect:pseudo} we establish analogous results to those above but replacing lax transformations by \emph{pseudo-transformations}.  In particular, in Theorem~\ref{thm:pseudoskewmonoidal} we describe a skew monoidal closed structure $(\GCat,\otimes_p,1)$ with internal hom $\psd(\ca,\cb)$ containing pseudo-transformations as $1$-cells.  

Finally in Section~\ref{sect:sharp}, we use the notion of \emph{sharp multimap} to show to refine the two aforementioned skew monoidal structures to \emph{sharp} variations $(\GCat,\otimes^{\sharp}_l,1)$ and $(\GCat,\otimes^{\sharp}_p,1)$ whose associated internal homs $\lax(\ca,\cb)_s$ and $\psd(\ca,\cb)_s$ are the full sub-Gray-categories of $\lax(\ca,\cb)$ and $\psd(\ca,\cb)$ containing the strict Gray-functors.
 
\subsection*{Further background and related work}

In the 1990's there began a significant stream of research on braided monoidal bicategories and semistrict braided monoidal 2-categories  \cite{Kapranov, Baez1996, Day1997Monoidal,Crans1998Generalized,Gurski2011}, and it was recognised (see e.g. \cite{Baez1996}) that the latter should be doubly degenerate \emph{semistrict $4$-categories}.  

A key player in this development was Crans, who began investigating the theory of semistrict $4$-categories.  In \cite{Crans1999A-tensor} he described a monoidal structure on $\GCat$, intended as an analogue of the lax Gray-tensor product, which encodes certain lax transformations satisfying the problematic equation \eqref{eq:comp}.  Later in \cite{Crans2000On} he discussed a second monoidal structure, a pseudo variant of the aforementioned one, and defined \emph{$4$-teisi} (his semistrict $4$-categories) as categories enriched in it.  However doubly degenerate $4$-teisi are stricter than semistrict braided monoidal 2-categories.  Moreover,  neither monoidal structure is closed (on either side) because of the problems concerning the equation \eqref{eq:comp} already discussed.  This has the undesirable consequence that $\GCat$ is not self-enriched with respect to either monoidal structure.

A more recent approach to semistrict $4$-categories is that of Bar and Vicary \cite{Bar2017} who define \emph{quasistrict 4-categories}, but no connection with enrichment has been made.  Let us also mention the work of Batanin, Cisinski and Weber \cite{Batanin} which describes a tight relationship between normalised globular operads and lax monoidal structures, although this has not led to a definition of semistrict $4$-category.  

\subsection*{Overview and future work} 

In the present paper, we move away from classical monoidal categories, as used by Crans \cite{Crans1999A-tensor}, to the skew monoidal categories of Szlachanyi \cite{Szlachanyi2012Skew}.  Our main goal is to establish closed skew monoidal categories as the natural setting for enrichment on the category $\GCat$ of Gray-categories, one capable of capturing weak higher dimensional transformations.  

 In a followup paper, we will focus primarily on the skew monoidal structures capturing pseudo-transformations, since these are the ones which are relevant to semistrict $4$-categories.  We will investigate their interaction with Lack's model structure on $\GCat$ \cite{Lack2011A-quillen}, showing in particular that the sharp skew monoidal structure $(\GCat,\otimes^{\sharp}_p,1)$ induces a symmetric monoidal closed structure on the homotopy category of Gray-categories.  Moreover, we will investigate categories enriched in these skew structures as semistrict 4-categories.
 
 \subsection*{Correction} 
 This Arxiv version corrects an error in the previous version, published in \cite{BourkeLobbia}.  Whilst the main results, Theorem~\ref{thm:laxskewmonoidal} and Theorem~\ref{thm:pseudoskewmonoidal}, on the skew monoidal structures $(\GCat,\otimes_l,1)$ and $(\GCat,\otimes_p,1)$ remain correct as stated, there was a mistake in the construction of the closed $4$-ary skew multicategory $\GCatmp$ in \Cref{sect:corrected} which was used in the construction of $(\GCat,\otimes_p,1)$.  
 
 The mistake was that the tight binary maps $\mathbf F \colon \ca, \cb \to \cc$ in $\GCatmp$ were missing the new condition \eqref{eq:adjointequivalence} about the equality of certain adjoint equivalences, which is required in order that $\mathbf F$ bijectively corresponds to a strict map $\ca \to \psd(\cb,\cc)$.  (The mistake was caused by the false claim that the forgetful Gray-functor $\psd(\cb,\cc) \to \lax(\cb,\cc)$ reflects identity 2-cells.)  
 
 In the updated \Cref{sect:corrected} we give the correct definitions of the tight multimaps and make the straightforward correction to the proof of Theorem~\ref{thm:summarypseudocase} that $\GCatmp$ is closed.  Following on from this, in Section~\ref{sect:pskewmonoidal}, we add the equation for the equality of adjoint equivalences to the presentation of $\ca \otimes_p \cb$, and in Example~\ref{eg:sharp}, we again add the equation for the equality of adjoint equivalences to the description of a sharp binary map.

\section{Skew monoidal and multicategorical structures}
\label{sect:structures}

In this section, we introduce the various categorical structures that we will use in the paper, beginning with skew monoidal categories.  Following this, we discuss multicategories and skew multicategories, giving \emph{$k$-ary} --- meaning, $k$-truncated --- versions of both.  Finally, we explain how $4$-ary skew multicategories satisfying certain properties give rise to closed skew monoidal categories.  

\subsection{Skew monoidal categories}

Skew monoidal categories were introduced by Szlach\'{a}nyi \cite{Szlachanyi2012Skew} in the study of bialgebroids over rings.  
A (left) skew monoidal category is a category $\cc$ together with a functor $\otimes\colon\cc\times \cc \to \cc$, a unit object $i \in \cc$, and natural families $\alpha \colon(a \otimes b) \otimes c \to a \otimes (b \otimes c)$, $l \colon i \otimes a \to a$ and $r \colon a \to a \otimes i$ satisfying five axioms described in \cite{Szlachanyi2012Skew}.  We will not need the axioms here, but remark that they neatly correspond to the five words 
\begin{center}
$abcd$\\
$iab$ $aib$ $abi$ \\
$ii$
\end{center}
of which the first corresponds to Mac Lane's pentagon.  

If the left unit map $l\colon i \otimes a \to a$ is invertible for each $a$, then $(\cc,\otimes,i)$ is said to be \emph{left normal}.  Dually, it is \emph{right normal} if $r\colon a \to a \otimes i$ is invertible for each $a$.
A monoidal category is precisely a skew monoidal category in which of the constraints $\alpha, l$ and $r$ are invertible.

Finally, a skew monoidal category $(\cc,\otimes,i)$ is said to be \emph{closed} if each endofunctor $- \otimes a$ has a right adjoint $[a,-]$.

\subsection{Classical and skew multicategories}

There are two equivalent approaches to multicategories.  In perhaps the better known of the two, one considers substitution in all variables simultaneously --- see, for instance, \cite{Hermida2000Representable} and \cite{LeinsterBook}.  
%
%
%
On the other hand, the first approach to multicategories, due to Lambek \cite{Lambek:multicategories}, focuses upon substitution in a single variable, as depicted below.

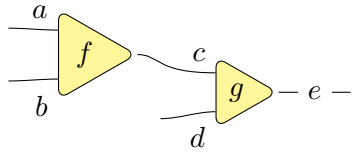
\begin{center}
\begin{tikzpicture}[triangle/.style = {fill=yellow!50, regular polygon, regular polygon sides=3,rounded corners}]

\path
	(6.5,1.5) node [triangle,draw,shape border rotate=-90,label=135:$a$,label=230:$b$,inner sep=2pt] (b') {$f$}
	(8.5,1) node [triangle,draw,shape border rotate=-90,label=135:$c$,label=230:$d$,inner sep=2pt] (c') {$g$};

\draw [-] (7.5,0.65) .. controls +(right:0.25cm) and +(left:0.25cm).. (c'.223);
\draw [] (b') .. controls (7.5,1.5) and +(left:0.8cm).. (c'.137);
\draw [] (c') to node[fill=white] {$e$} (10,1);
\draw [-] (5.5,1.15) .. controls +(right:0.25cm) and +(left:0.25cm).. (b'.223);
\draw [-] (5.5,1.85) .. controls +(right:0.25cm) and +(left:0.25cm).. (b'.137);
\end{tikzpicture}
\end{center}

It is this approach that we will follow. In addition, we will make a refinement by defining the notion of a \emph{$k$-ary multicategory}, for $k \in \mathbb N_{\geq 1} \cup \{\omega\}$, which has multimaps of arity $n$ only for natural numbers $n \leq k$.  In particular, $\omega$-multicategories are just multicategories.


\begin{defn}
Let $k \in \mathbb N_{\geq 1} \cup \{\omega\}$.  A $k$-ary \emph{multicategory} $\mc$ consists of 
\begin{itemize}
\item a collection of objects;
\item for each natural number $n \leq k$ and list $a_1,\ldots, a_n,b$ of objects a set $\mc_n(a_1,\ldots,a_n;b)$ of \emph{$n$-ary multimaps}; 
\item for each object $a$, an element $1_a \in \mc_1(a;a)$; 
\item for natural numbers $n,m \leq k$ with $n+m-1 \leq k$ and $i \in \{1,\ldots,n\}$ a substitution function
\begin{equation}\label{eq:substitution}
\circ_i\colon \mc_n(\overline{b};c) \times \mc_m(\overline{a};b_i) \to \mc_{n+m-1}(\overline{b}_{<i},\overline{a},\overline{b}_{>i};c)
\end{equation}
where $\overline{b}_{<i}$ and $\overline{b}_{>i}$ denote the sublists of $\overline{b}$ in indices less than and greater than $i$, respectively. 
\end{itemize}

For each $h \in \mc_n(b_1,\ldots,b;c)$ and $i \in \{1,\ldots,m\}$ with $m \leq k$, these are required to satisfy the unit equations 
\begin{equation}
h \circ_i 1_{b_i} = h \hspace{0.5cm} \textnormal{and} \hspace{0.5cm} 1_c \circ h = h
\end{equation}
and the associativity equations 
\begin{align}
h \circ_i(g\circ_j f) = (h\circ_i g)\circ_{j+i-1}f  \hspace{0.5cm} &\textnormal{for} \hspace{0.5cm} 1 \leq i \leq m, 1 \leq j \leq n   \label{eq:classic-ass-line}\\
(h\circ_i g)\circ_{n+j-1}f=(h \circ_j f) \circ_{i} g \hspace{0.5cm} &\textnormal{for} \hspace{0.5cm} 1 \leq i<j \leq m \label{eq:classic-ass-not-line} 
\end{align}
for composable $f,g,h$ of arity $l,n,m \leq k$ respectively whenever these are defined (that is, when the four composites involved in each of these two equations have arity $\leq k$.)\begin{footnote}{ We note that the equation \eqref{eq:classic-ass-not-line} was called a \emph{commutativity equation} by Lambek \cite{Lambek:multicategories} (see also \cite{Levy}).  Here we refer to both \eqref{eq:classic-ass-line} and  \eqref{eq:classic-ass-not-line} as associativity equations since they both become such when considering the corresponding multicategories in the many-variable-substitution sense of \cite{Hermida2000Representable, LeinsterBook}, and since we have no need to separate \eqref{eq:classic-ass-line} and  \eqref{eq:classic-ass-not-line} into two groups in what follows.}\end{footnote}
\end{defn}

The evident structure preserving morphisms between $k$-ary multicategories are called multifunctors and these form the morphisms of the category $\kmulti$ of $k$-ary multicategories.

%

The associativity equations are best understood using diagrams.  For instance, the well-definedness of the following diagram represents an instance of \eqref{eq:classic-ass-line} 
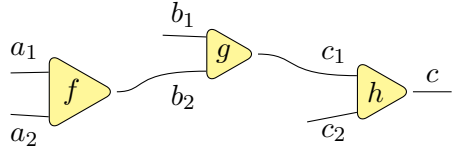
\begin{center}
\begin{tikzpicture}[triangle/.style = {fill=yellow!50, regular polygon, regular polygon sides=3,rounded corners}]
\path 
	(0,0.5) node [triangle,draw,shape border rotate=-90,inner sep=1pt,label=135:$a_1$,label=230:$a_2$] (b) {$f$} 
	(2,1) node [triangle,draw,shape border rotate=-90,inner sep=1pt,label=135:$b_1$,label=230:$b_2$] (a) {$g$}
	(4,0.5) node [triangle,draw,shape border rotate=-90,inner sep=1pt, label=135:$c_1$,label=230:$c_2$] (c) {$h$};

\draw [-] (-0.8,0.2) to (b.228);
\draw [-] (-0.8,0.75) to (b.139);
\draw [-] (3.1,0.1) to (c.228);
\draw [-] (a) .. controls +(right:1cm) and +(left:1cm).. (c.139);
\draw [-] (b) .. controls +(right:1cm) and +(left:1cm).. (a.225);
\draw [-] (c) to node [above] {$c$} (5,0.5);
\draw [-] (1.2,1.25) to (a.135);
\end{tikzpicture}  
\end{center}

whilst the following equality represents an instance of \eqref{eq:classic-ass-not-line}.
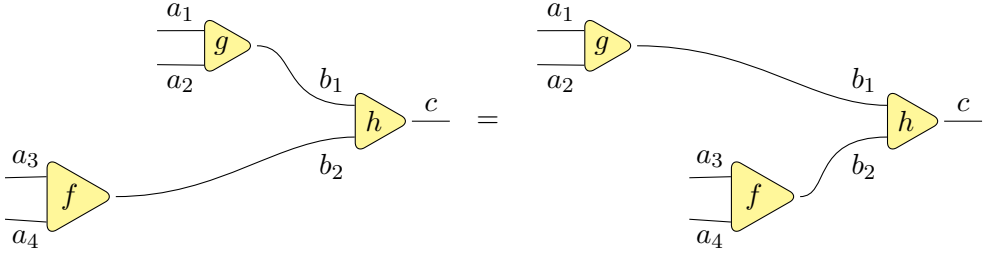
\begin{equation*}
\label{ax:fin-A}
\begin{gathered}
\begin{tikzpicture}[triangle/.style = {fill=yellow!50, regular polygon, regular polygon sides=3,rounded corners}]
\path (0,0) node [triangle,draw,shape border rotate=-90,inner sep=1.1pt] (a) {$f$}
	(2,2) node [triangle,draw,shape border rotate=-90,inner sep=1.1pt] (b) {$g$} 
	(4,1) node [triangle,draw,shape border rotate=-90,inner sep=1.1pt,label=135:$b_1$,label=230:$b_2$] (c) {$h$} 
	(9,0) node [triangle,draw,shape border rotate=-90,inner sep=1.1pt] (a') {$f$}
	(7,2) node [triangle,draw,shape border rotate=-90,inner sep=1.1pt] (b') {$g$} 
	(11,1) node [triangle,draw,shape border rotate=-90,inner sep=1.1pt, label=135:$b_1$,label=230:$b_2$] (c') {$h$};
	
\draw [-] (a) .. controls +(right:2cm) and +(left:1cm).. (c.220);
\draw [-] (b) .. controls +(right:1cm) and +(left:1cm).. (c.140);
\draw [-] (a') .. controls +(right:1cm) and +(left:1cm).. (c'.220);
\draw [-] (b') .. controls +(right:2cm) and +(left:1cm).. (c'.140);
	
\draw [-] (c) to node [above] {$c$} (5,1);
\draw [-] (1.15,2.2) to node [above] {$a_1$} (b.139);
\draw [-] (1.15,1.75) to node [below] {$a_2$} (b.228);
\draw [-] (-.85,0.25) to node [above] {$a_3$} (a.139);
\draw [-] (-.85,-0.3) to node [below] {$a_4$} (a.228);

\node () at (5.5,1) {$=$};

	\draw [-] (c') to node [above] {$c$} (12,1);
\draw [-] (6.15,2.2) to node [above] {$a_1$} (b'.139);
\draw [-] (6.15,1.75) to node [below] {$a_2$} (b'.228);
\draw [-] (8.15,0.25) to node [above] {$a_3$} (a'.139);
\draw [-] (8.15,-0.3) to node [below] {$a_4$} (a'.228);
\end{tikzpicture}
\end{gathered}
\end{equation*}

\begin{rmk}
Each $k$-ary multicategory has an underlying category $\mc_1$ with the same objects as $\mc$ and morphisms the unary ones.  In particular, one can think of a $k$-ary multicategory as a category equipped with extra structure.  

Substitution of unary and $n$-ary multimaps equips the sets $\mc_n(a_1,\ldots,a_n;b)$ of $n$-ary multimaps with the structure of a functor 
\begin{equation*}
\mc_n(-;-)\colon(\mc_1^{n})^{op} \times \mc_1 \to \Set.
\end{equation*} 
with respect to which the substitution maps $\circ_i\colon \mc_n(\overline{b};c) \times \mc_m(\overline{a};b_i) \to \mc_{n+m-1}(\overline{b}_{<i},\overline{a},\overline{b}_{>i};c)$
are natural in each variable.  

Observe that the category $\mc_1$ and functor $\mc_n(-;-)$ already encode all substitutions involving $n$-ary and unary morphisms, plus any associativity axioms involving two unary maps, as well as the unit axioms.  The remaining associativity axioms involving a unary map are captured by naturality of the substitutions.

When constructing the $4$-ary multicategory $\GCatm^l$ in the next section, we will begin with a category $\GCat_p$, then construct the functors $\GCatm^l_n(-;-)$ of $n$-ary maps and the substitutions $\circ_i$, verifying that these are natural as we construct them, in this way verifying all associativity equations involving unary maps as well as the unit equations.
\end{rmk}

Skew multicategories were introduced in \cite{BourkeLack:skew}.  In general, these involve two types of multimorphisms, tight and loose, together with a function viewing tight multimaps as loose ones and satisfying some axioms.  The concept simplifies significantly when these functions are subset inclusions and we will abuse the terminology by giving the definition under this restriction.  Again, we give a $k$-ary version, with $\omega$-ary skew multicategories being the usual skew multicategories.

\begin{defn}
Let $k \in \mathbb N_{\geq 1} \cup \{\omega\}$.  A $k$-ary skew multicategory $\mc$ consists of a $k$-ary multicategory $\mc^l$ as above, whose multimaps we call \emph{loose}, together with a subset $\mc^t_n(a_1,\ldots,a_n;b) \subseteq \mc^l_n(a_1,\ldots,a_n;b)$ for $1 \leq n \leq k$ of \emph{tight} $n$-ary multimaps.  
\item Furthermore, we require that:
\begin{itemize}
\item the unary identities in $\mc^l$ are tight;
\item the composite multimap $g \circ_1 f$ is tight whenever $g$ and $f$ are tight, and for $i \neq 1$, we require that $g \circ_i f$ is tight whenever $g$ is.
\end{itemize}
\end{defn}

A multifunctor $F\colon\mc \to \md$ between $k$-ary skew multicategories is a multifunctor $F\colon\mc^l \to \md^l$ between the underlying $k$-ary multicategories of loose multimaps which preserves tightness of multimaps.  These form the morphisms of the category $\skkmulti$ of $k$-ary skew multicategories.

In a $k$-ary skew multicategory $\mc$, the objects and tight morphisms form a subcategory $j\colon\mc^t_1 \to \mc^l_1$ of the category of loose unary morphisms.  Note that a skew multicategory in which all loose multimaps are tight is precisely a multicategory.


\begin{eg}
There is a multicategory $\multab$ of abelian groups, in which multimaps $F\colon A_1,\ldots,A_n \to B$ are arbitrary functions $F\colon A_1 \times \ldots \times A_n \to B$.  It becomes a skew multicategory when, for $n \geq 1$, we say $F$ is tight just when each $F(-,a_2,\ldots,a_n)\colon A_1 \to B$ is a homomorphism of abelian groups --- that is, it is a homomorphism in the first variable.
\end{eg}

\subsection{Closedness}

Let $k \geq 2$.  A $k$-ary multicategory $\mc$ is said to be \emph{closed}, if for all $b,c$, there exists an object $[b,c]$ and binary map 

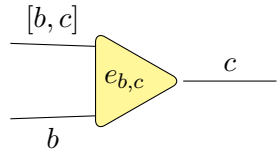
\begin{displaymath}
\begin{tikzpicture}[triangle/.style = {fill=yellow!50, regular polygon, regular polygon sides=3,rounded corners}]
\path (2,2) node [triangle,draw,shape border rotate=-90, inner sep=0pt] (m) {$e_{b,c}$};

\draw [-] (0.5,2.5) to node [above] {$[b,c]$} (m.130);
\draw [-] (0.5,1.5) to node [below] {$b$} (m.230);
\draw [-] (m) to node [above] {$c$} (4,2);
\end{tikzpicture}
\end{displaymath}

for which the induced functions
$e_{b,c}\circ_1-\colon\mc_n(\overline{a};[b,c]) \to \mc_{n+1}(\overline{a},b;c)$  are bijections for $0 \leq n < k$. It is said to be \emph{biclosed} if, moreover, there there exists an object $[b,c]'$ and binary map $e'_{b,c}\colon b,[b,c]'\to c$ for which the induced functions $e'_{b,c}\circ_2-\colon\mc_n(\overline{a};[b,c]) \to \mc_{n+1}(b,\overline{a};c)$  are bijections for $0 \leq n < k$.

For a $k$-ary skew multicategory $\mc$ to be \emph{closed}, we firstly require that its $k$-ary multicategory of loose multimaps $\mc^l$ is closed.  Furthermore, we require that $e_{b,c}\colon [b,c],b\to c$ is tight and that the induced function $e_{b,c}\circ_1-\colon\mc^t_n(\overline{a};[b,c]) \to \mc^t_{n+1}(\overline{a},b;c)$ is also a bijection for $1 \leq n < k$.  We will not consider biclosedness in the skew context since it does not feature in our examples.

\begin{eg}
The skew multicategory $\multab$ is closed --- the internal hom $[A,B]$ is the abelian group of \emph{functions} from $A$ to $B$ with addition and unit defined pointwise as in $B$.
\end{eg}

\subsection{Left representability}

Let $\mc$ be a $k$-ary skew multicategory where $k \geq 2$.  A \emph{tight binary map classifier} for $a$ and $b$ consists of a representation of $\mc^t_2(a,b;-)\colon\mc_1^t \to \Set$ -- in other words, a tight binary map 

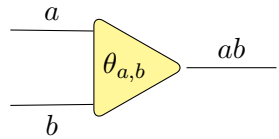
\begin{displaymath}
\begin{tikzpicture}[triangle/.style = {fill=yellow!50, regular polygon, regular polygon sides=3,rounded corners}]
\path (2,2) node [triangle,draw,shape border rotate=-90, inner sep=0pt] (m) {$\theta_{a,b}$};

\draw [-] (0.5,2.5) to node [above] {$a$} (m.130);
\draw [-] (0.5,1.5) to node [below] {$b$} (m.230);
\draw [-] (m) to node [above] {$ab$} (4,2);
\end{tikzpicture}
\end{displaymath}

for which the induced function $-\circ \theta_{a,b}\colon\mc^{t}_1(ab;c) \to \mc^{t}_2(a,b;c)$
is a bijection for all $c$. It is \emph{left universal} if, moreover, the induced function 
$$-\circ_1 \theta_{a,b}\colon\mc^{t}_n(ab,\overline{c};d) \to \mc^{t}_{n+1}(a,b,\overline{c};d)$$
is a bijection for all $n < k$ and $\overline{c}$ a tuple of the appropriate length.  

Similarly, a \emph\emph{nullary map classifier} is a representation of $\mc^{l}_2(-;-)\colon\mc^t_1 \to \Set$ --- thus, a certain nullary map 
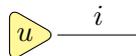
\begin{displaymath}
\begin{tikzpicture}[triangle/.style = {fill=yellow!50, regular polygon, regular polygon sides=3,rounded corners}]
\path (2,2) node [triangle,draw,shape border rotate=-90, inner sep=1pt] (m) {$u$};
\draw [-] (m) to node [above] {$i$} (3.5,2);
\end{tikzpicture}
\end{displaymath}

in $\mc$.  It is \emph{left universal} if, moreover, the induced function 
$$-\circ_1 u\colon \mc^{t}_{n+1}(i,\overline{c};d) \to \mc^{l}_{n}(\overline{c};d)$$
is a bijection for each $n <  k$ and $\overline{c}$ an $n$-tuple. 

A $k$-ary skew multicategory $\mc$ is said to be \emph{left representable} if it admits left universal nullary and tight binary map classifiers.  

\subsection{The closed left representable case}\label{sect:closedrep}
In the case that $\mc$ is closed, any nullary or tight binary map classifiers are \emph{automatically left universal}.  (Indeed, closedness allows us to remove the extra variable $\overline{x}$ defining left universality by shifting it from the inputs of the multimap to the output.)  Hence, a closed $k$-ary skew multicategory $\mc$ is left representable just when it admits nullary and tight binary map classifiers.  

Moreover, the internal homs assemble into an endofunctor $[b,-]\colon\mc^t_1 \to \mc^t_1$ and the natural isomorphisms $\mc^t_2(a,b;c) \cong \mc^t_1(a;[b,c])$ ensure that $\catC$ admits tight binary map classifiers just when each $[b,-]$ has a left adjoint.  Summing up, 
\begin{itemize}
\item a closed $k$-ary skew multicategory is left representable just when it admits a nullary map classifier and each endofunctor $[b,-]$ has a left adjoint. 
\end{itemize} 
We will use this result in verifying left representability in our main examples.  See Proposition 4.7 of \cite{BourkeLack:skew} for a detailed proof in the case of skew multicategories, and note that the proof therein applies equally to the $k$-ary setting.

\subsection{$4$-ary skew multicategories versus skew monoidal categories}\label{sect:From}

In Section 6.2 of \cite{BourkeLack:skew}, it was shown that left representable skew multicategories give rise to skew monoidal categories.  Here we describe a truncated version of this result: namely, each left representable $4$-ary skew multicategory $\mc$ gives rise to a skew monoidal structure on its underlying category $\mc^t_1$ of tight unary maps.  Let us now describe the associated skew monoidal structure --- for a detailed verification of the axioms, we refer the reader to Lemma 4.5.8 of the second-named author's Phd thesis \cite{LobbiaThesis}.

\begin{enumerate}
\item The tensor product of two objects $a,b \in \mc^t_1$ is the tight binary map classifier $ab$ whilst the unit $i$ is the nullary map classifier.
\item Left representability ensures that the composite tight multimap below left 
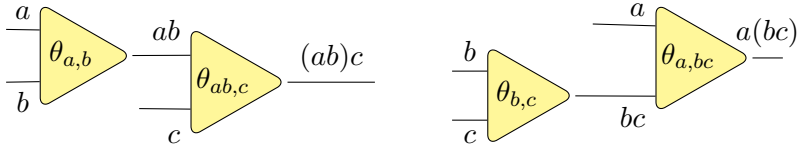
\begin{equation*}
\begin{tikzpicture}[triangle/.style = {fill=yellow!50, regular polygon, regular polygon sides=3,rounded corners}]

\path 
	(7,1.35) node [triangle,draw,shape border rotate=-90,inner sep=0pt] (b') {$\theta_{a,b}$} 
	(9,1) node [triangle,draw,shape border rotate=-90,inner sep=-1.5pt,label=135:$ab$,label=230:$c$] (c') {$\theta_{ab, c}$};

\draw [-] (6.15,1.7) to node [above] {$a$} (b'.140);
\draw [-] (6.15,1) to node [below] {$b$} (b'.220);
\draw [-] (b') to (c'.140);
\draw [-] (7.9,0.65) to node [below] {$$} (c'.220);
\draw [-] (c') to node [above] {$(ab)c$} (11,1);
\end{tikzpicture}
\hspace{1cm}
\begin{tikzpicture}[triangle/.style = {fill=yellow!50, regular polygon, regular polygon sides=3,rounded corners}]
\path (0.25,0.5) node [triangle,draw,shape border rotate=-90,inner sep=0pt] (a) {$\theta_{b,c}$} 
	(2.5,1) node [triangle,draw,shape border rotate=-90,inner sep=-1.5pt,label=135:$a$,label=230:$bc$] (c) {$\theta_{a,bc}$} 
;
	
\draw [-] (a) to (c.230);
\draw [-] (1.25,1.45) to (c.135);
	
\draw [-] (c) to node [above] {$a(bc)$} (3.75,1);
\draw [-] (-.6,0.825) to node [above] {$b$} (a.140);
\draw [-] (-.6,0.175) to node [below] {$c$} (a.220);

\end{tikzpicture}
\end{equation*}
is a \emph{tight $3$-ary map classifier}.  By its universal property, the tight ternary map above right induces a unique tight unary morphism $\alpha\colon (ab)c\to a(bc)$ which yields the rhs above under pre-composition with the lhs.  This is the associator.
\item 
Left universality also implies that 
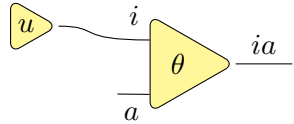
\begin{equation*}
\begin{gathered}
\begin{tikzpicture}[triangle/.style = {fill=yellow!50, regular polygon, regular polygon sides=3,rounded corners}]
\path
	(2,1.5) node [triangle,draw,shape border rotate=-90,inner sep=1pt] (b) {$u$} 
	(4,1) node [triangle,draw,shape border rotate=-90,label=135:$i$,label=230:$a$] (ab) {$\theta$};

\draw [-] (3.2,0.6) to (ab.227);
\draw [-] (b) .. controls +(right:1cm) and +(left:1cm).. (ab.140);
	
\draw [-] (ab) to node [above] {$ia$} (5.5,1);

\end{tikzpicture}
\end{gathered}
\end{equation*}
classifies loose unary maps with domain $a$; it follows that there exists a unique tight morphism $l\colon ia\to a$ which yields the identity on $a$ when precomposed with the above multimap.  This is the left unit map for the skew monoidal structure.  

Let us note that the fact that the aforementioned classification property of $ia$ amounts to the statement that $ia$ is the value of the left adjoint to the subcategory inclusion $j\colon \mc^t_1 \to \mc^l_1$ with unit depicted above, and with counit $l\colon ia \to a$.
\item 
Finally, the right unit map $r\colon a\to ai$ is defined as the composite
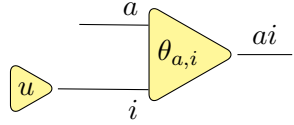
\begin{equation*}
\label{eq:univ-rho}
\begin{gathered}
\begin{tikzpicture}[triangle/.style = {fill=yellow!50, regular polygon, regular polygon sides=3,rounded corners}]
\path (0.5,0.55) node [triangle,draw,shape border rotate=-90,inner sep=1pt] (a) {$u$} 
	(2.5,1) node [triangle,draw,shape border rotate=-90,inner sep=0pt,label=135:$a$,label=230:$i$] (c) {$\theta_{a,i}$};

\draw [-] (a) to (c.230);
\draw [-] (1.2,1.4) to (c.135);
\draw [-] (c) to node [above] {$ai$} (4,1);

\end{tikzpicture}
\end{gathered}
\end{equation*}
which is tight since $\theta_{a,i}$ is.
\end{enumerate}
If $\mc$ is, moreover, closed, then the composite natural isomorphism $$\mc^t_1(ab;c) \cong \mc^t_2(a,b;c) \cong \mc^t_1(a;[b,c])$$ show that the skew monoidal structure on $\mc_1^t$ is closed too, again with internal hom $[b,c]$.

\begin{rmk}
It appears that there is much more to say about $k$-ary multicategories --- in particular, it seems that it is possible to complete each $k$-ary multicategory to a multicategory in a universal way (both in the ordinary and skew settings).  Given that in the present paper we only need to consider $k$-ary skew multicategories for $k \leq 4$, we leave this to future developments.
\end{rmk}

\section{Pseudomorphisms and mapping spaces for Gray-categories}
\label{sect:mappingspace}

We begin the present section by reminding the reader about Gray-categories.  Following Gohla \cite{GohlaB:mapp}, we then define pseudomaps of Gray-categories and the mapping space $\lax(\catA,\catB)$ of lax transformations, establishing functoriality of $\lax(-,-)$ on pseudomaps in Proposition~\ref{prop:asd}.  We also introduce the Gray-category $\oplax(\catA,\catB)$ of oplax transformations.  

\subsection{Gray-categories}
A Gray-category $\ca$ is a $(\Gray,\otimes_p,1)$-enriched category, where $\otimes_p$ denotes the Gray-tensor product, as described in Section 3.2 of \cite{Gurskibook} or Section 4 of \cite{Gordon1995Coherence}.   In particular, it amounts to:
\begin{enumerate}[(i)]
\item A category $\ca_1$ of $0$ and $1$-cells with composition of $f\colon x \to y$ and $g\colon y \to z$ denoted by $gf\colon x \to z$, and with identities denoted by $1_x\colon x \to x$.
\item\label{ax:strict-hom} For $0$-cells $x,y$, a $2$-category $\ca(x,y)$ whose objects are the $1$-cells of $\ca_1$, and whose $n$-cells are referred to as the $(n+1)$-cells of $\ca$.
\item Given $1$-cells $f\colon w \to x$ and $g\colon y \to z$, there are \emph{whiskering} $2$-functors $- \circ f\colon \ca(x,y) \to \ca(w,y)$ and $g \circ -\colon \ca(x,y) \to \ca(x,z)$ whose action on objects is composition in $\ca_1$.  These whiskering operations are functorial in $\ca_1$ and compatible, in the sense that $(g \circ -) \circ f = g \circ (- \circ f)$.
\item Given a pair of $2$-cells 
\[\begin{tikzcd}[ampersand replacement=\&]
	x \& y \& z
	\arrow[""{name=0, anchor=center, inner sep=0}, "{f_1}", curve={height=-12pt}, from=1-1, to=1-2]
	\arrow[""{name=1, anchor=center, inner sep=0}, "{f_2}", curve={height=-12pt}, from=1-2, to=1-3]
	\arrow[""{name=2, anchor=center, inner sep=0}, "{f_2'}"', curve={height=12pt}, from=1-2, to=1-3]
	\arrow[""{name=3, anchor=center, inner sep=0}, "{f_1'}"', curve={height=12pt}, from=1-1, to=1-2]
	\arrow["\psi"{xshift=0.1cm}, shorten <=3pt, shorten >=3pt, Rightarrow, from=1, to=2]
	\arrow["\phi"{xshift=0.1cm}, shorten <=3pt, shorten >=3pt, Rightarrow, from=0, to=3]
\end{tikzcd}\]
there is an invertible $3$-cell
\[\begin{tikzcd}
	{f_2f_1} & {f_2f'_1} \\
	{f'_2f_1} & {f'_2f'_1}
	\arrow[""{name=0, anchor=center, inner sep=0}, "f_2\circ\phi", from=1-1, to=1-2]
	\arrow["\psi\circ f_1'", from=1-2, to=2-2]
	\arrow["\psi\circ f_1"', from=1-1, to=2-1]
	\arrow[""{name=1, anchor=center, inner sep=0}, "f_2'\circ\phi"', from=2-1, to=2-2]
	\arrow["{\psi:\phi}"{xshift=0.1cm}, shorten <=9pt, shorten >=9pt, Rightarrow, from=0, to=1]
\end{tikzcd}\]	
called the interchange map.  The interchange maps satisfy several equations, for which we refer the reader to \cite[Definition~1.1]{GambLob:psm}. 
\end{enumerate}

In this section, for a Gray-category $\ca$, we use $x, y, z, \ldots$ to denote its 0-cells, 
$f_1 \colon x \to y$, $f_2 \colon x \to z$, \ldots for its 1-cells, $\phi \colon f \to f' ,\psi \colon g \to g' \, \ldots$ for 2-cells and $\Lambda \colon \alpha \to \alpha' \, , \Theta \colon \beta \to \beta' \, \ldots$ for 3-cells.  Compositions in the hom $2$-categories $\ca(x,y)$ will be denoted by 
$$-\cdot-\colon\ca(x,y)[g,h]\times\ca(x,y)[f,g]\to\ca(x,y)[f,h].$$
and we will use $\Lambda'\ast\Lambda$ for vertical composition of 3-cells in $\ca(x,y)$.  

We will rarely need any such notation, since composites in the hom-2-categories will almost always be given using diagrams.  Our diagram style follows that of Gurski \cite{Gurskibook}.  In diagrams, we will abuse notation for whiskering $2/3$-cells by a $1$-cell and only write $1\circ-$ and $-\circ1$, as in the diagram below. Finally, in some larger diagrams, where the context will make it clear, we will avoid labelling the interchange and only indicate it with a dashed double arrow again as in the diagram below. 

\[\begin{tikzcd}[ampersand replacement=\&]
	{f_2f_1} \& {f_2f'_1} \\
	{f'_2f_1} \& {f'_2f'_1}
	\arrow[""{name=0, anchor=center, inner sep=0}, "1\circ\phi", from=1-1, to=1-2]
	\arrow["\psi\circ1", from=1-2, to=2-2]
	\arrow["\psi\circ1"', from=1-1, to=2-1]
	\arrow[""{name=1, anchor=center, inner sep=0}, "1\circ\phi"', from=2-1, to=2-2]
	\arrow[shorten <=9pt, shorten >=9pt, Rightarrow, dashed, from=0, to=1]
\end{tikzcd}\]	

%
%

\subsection{Homomorphisms of Gray-categories}

Given Gray-categories $\ca$ and $\cb$, a Gray-functor $F\colon\ca \to \cb$ is a morphism of underlying $3$-graphs which preserves the Gray-category structure on the nose.  We often refer to these Gray-functors as \emph{strict maps}, and they form the morphisms of the category $\GCat$.  

In \cite{GohlaB:mapp}, Gohla defined the more general notion of a pseudo-Gray-functor $F\colon\ca \to \cb$ which we will call a \emph{pseudomap}. A pseudomap $F\colon\ca\to\cb$ consists of a $3$-graph map together with a cocycle $F^2$, that is,
	\begin{itemize}
	\item  for any pair of 1-cells $f_1\colon x\to y$ and $f_2\colon y\to z$ in $\ca$, an invertible 2-cell 
	$$F^2_{f_2,f_1}\colon Ff_2\circ Ff_1\to F(f_2\circ f_1)$$
	 in $\cb$ satisfying the nine axioms listed in \Cref{subsec:ax-pseudofunct}. 
	\end{itemize}  
We point out that pseudomaps preserve identity $1$-cells strictly and are very close to the notion of normal pseudofunctor in $2$-category theory.

Pseudomaps $F\colon\ca \to \cb$ and $G \colon \cb \to \cc$ can themselves be composed by composing the underlying functions on $3$-graphs and defining the invertible $2$-cell $(GF)^2_{f_2,f_1}$ to be the composite
\begin{equation*}
\xymatrix{
GFf_2\circ GFf_1 \ar[rr]^{G^2_{Ff_2,Ff_1}} && G(Ff_2\circ Ff_1) \ar[rr]^{G(F^2_{f_2,f_1})} && GF(f_2\circ f_1) .
}
\end{equation*}
In this way, we obtain a category $\GCat_p$ of Gray-categories and pseudomaps, which contains $\GCat$ as a subcategory via an identity on objects inclusion $j\colon\GCat \to \GCat_p$.

\begin{rmk}\label{rmk:comonad}
As explained in Section 2 of \cite{GohlaB:mapp}, the category $\GCat_p$ is, in fact, the Kleisli category for a comonad $Q$ on $\GCat$, with $j\colon\GCat \to \GCat_p$ the canonical inclusion to the Kleisli category.  We now describe the comonad $Q$.


Given a Gray-category $\ca$, consider its underlying category $\ca_1$, and let $\epsilon\colon FU\ca_1 \to \ca_1$ denote the counit of the adjunction $F \dashv U$ between categories and reflexive graphs.  Here $FU\ca_1$ has objects as in $\ca_1$ and morphisms are of two kinds: (1) identities $1 \colon x \to x \in \ca_1$ and (2) sequences of composable \emph{non-identity} morphisms $\overline{f} = [f_1,\ldots,f_n]:x \to y$ in $\ca_1$ where the domain of $f_1$ is $x$ and codomain of $f_n$ is $y$.  The functor $\epsilon \colon FU\ca_1 \to \ca_1$ sends composable sequences of morphisms to their composite.  The counit $p\colon Q\ca \to \ca$ of the comonad is the cartesian lifting of $\epsilon$ along the forgetful functor $(-)_0\colon \GCat \to \Cat$: in other words, $Q\ca$ has underlying category $FU\ca_1$ but has $Q\ca(x,y)[\overline{f},\overline{g}] = \ca(x,y)[\epsilon \overline{f}, \epsilon \overline{g}]$ --- in particular, $p\colon Q\ca \to \ca$ is fully faithful on 2-cells and 3-cells.  Note that since $p$ is identity on objects, full on $1$-cells and fully faithful both on $2$-cells and $3$-cells, it is a \emph{triequivalence} of Gray-categories.

In referring to $Q$, we will call it the \emph{pseudomap classifier comonad}.  Indeed, its universal property from the Kleisli adjunction is that there is a universal pseudomap $q \colon \ca \to Q\ca$ through which each pseudomap $F \colon \ca \to \cb$ factors uniquely as a strict map $F' \colon Q\ca \to \cb$.  The universal pseudomap $q \colon \ca \to Q\ca$ is described as follows: it is the identity on objects, sends identity $1$-cells to identity $1$-cells, and sends non-identities $f \colon x \to y$ to words $[f] \colon x \to y$.  It is also the identity on $2$-cells and $3$-cells.  The cocycle $q^2_{g,f}$ is given by $1_{gf} \colon [g] \circ [f] \to [gf]$.  We leave to the interested reader the construction of $F'$, who can also refer for further details to Section 2 of \cite{GohlaB:mapp}.
\end{rmk}


\subsection{Gohla's mapping space of lax transformations}

In \cite{GohlaB:mapp}, Gohla defined a ``mapping space" Gray-category $\lax(\ca,\cb)$ between Gray-categories $\ca$ and $\cb$. Its cells in increasing dimension are pseudomaps, lax transformations, modifications and perturbations.  We now spell out the data of the cells of each dimension, leaving the axioms to \Cref{app:ax-gohla-mapp}. 

\begin{itemize}
\item \emph{1-cells:} A \emph{lax transformation} $\alpha\colon F\to G$ consists of:
	\begin{itemize}
	\item For any object $x\in\ca$, a 1-cell $\alpha_x\colon Fx\to Gx$ in $\cb$.
	\item For any 1-cell $f\colon x\to y$ in $\ca$, a 2-cell  in $\cb$ of the form
\[\begin{tikzcd}
	{Fx} & {Gx} \\
	{Fy} & {Gy.}
	\arrow[""{name=0, anchor=center, inner sep=0}, "{\alpha_x}", from=1-1, to=1-2]
	\arrow[""{name=1, anchor=center, inner sep=0}, "{\alpha_Y}"', from=2-1, to=2-2]
	\arrow["Ff"', from=1-1, to=2-1]
	\arrow["{Gf}", from=1-2, to=2-2]
	\arrow["{\alpha_f}"{xshift=0.1cm}, shorten <=10pt, shorten >=10pt, Rightarrow, from=0, to=1]
\end{tikzcd}\]

	\item For any 2-cell $\phi\colon f\to f'$ in $\ca$, a 3-cell $\alpha_\phi$ in $\cb$ of the form
\[\begin{tikzcd}
	{Gf\alpha_x} & {\alpha_y Ff} \\
	{Gf'\alpha_x} & {\alpha_y Ff'.}
	\arrow[""{name=0, anchor=center, inner sep=0}, "{\alpha_f}", from=1-1, to=1-2]
	\arrow[""{name=1, anchor=center, inner sep=0}, "{\alpha_{f'}}"', from=2-1, to=2-2]
	\arrow["G\phi\circ1"', from=1-1, to=2-1]
	\arrow["{1\circ F\phi}", from=1-2, to=2-2]
	\arrow["{\alpha_\phi}"{xshift=0.1cm}, shorten <=10pt, shorten >=10pt, Rightarrow, from=0, to=1]
\end{tikzcd}\]
	\item For any pair of 1-cells $f_1\colon x\to y$ and $f_2\colon y\to z$ in $\ca$, an invertible 3-cell $\alpha^2_{f_2,f_1}$ in $\cb$ of the form
\begin{equation}\label{eq:laxnat1}
\begin{tikzcd}
	{Gf_2Gf_1\alpha_x} & {Gf_2\alpha_y Ff_1} & {\alpha_z Ff_2Ff_1} \\
	{G(f_2f_1) \alpha_x} && {\alpha_zF(f_2f_1).}
	\arrow["{1\circ\alpha_{f_1}}", from=1-1, to=1-2]
	\arrow["{G^2_{f_2,f_1}\circ1}"', from=1-1, to=2-1]
	\arrow["{\alpha_{f_2}\circ1}", from=1-2, to=1-3]
	\arrow["{1\circ F^2_{f_2,f_1}}", from=1-3, to=2-3]
	\arrow[""{name=0, anchor=center, inner sep=0}, "{\alpha_{f_2f_1}}"', from=2-1, to=2-3]
	\arrow["{\alpha^2_{f_2,f_1}}"{xshift=0.1cm}, shorten <=7pt, shorten >=7pt, Rightarrow, from=1-2, to=0]
\end{tikzcd}
\end{equation}
	\end{itemize}
This data has to satisfy the seven axioms listed in \Cref{subsec:ax-lax-transf}. 

\item \emph{2-cells:} A \emph{modification} $\Lambda\colon\alpha\to\beta$ consists of:
	\begin{itemize}
	\item For any object $x\in\ca$, a 2-cell $\Lambda_x\colon \alpha_x\to \beta_x$ in $\cb$.
	\item For any 1-cell $f\colon x\to y$ in $\ca$, a 3-cell $\Lambda_f$ in $\cb$ of the form
\[\begin{tikzcd}
	{Gf\alpha_x} & {Gf\beta_x} \\
	{\alpha_y Ff} & {\beta_y Ff.}
	\arrow["{\alpha_f}"', from=1-1, to=2-1]
	\arrow["{\beta_{f}}", from=1-2, to=2-2]
	\arrow[""{name=0, anchor=center, inner sep=0}, "{1\circ\Lambda_x}", from=1-1, to=1-2]
	\arrow[""{name=1, anchor=center, inner sep=0}, "{\Lambda_y\circ1}"', from=2-1, to=2-2]
	\arrow["{\Lambda_f}"{xshift=0.1cm}, shorten <=8pt, shorten >=8pt, Rightarrow, from=0, to=1]
\end{tikzcd}\]
	\end{itemize}
This data has to satisfy the three axioms listed in \Cref{subsec:ax-mod}.

\item \emph{3-cells:} A \emph{perturbation} $\theta\colon\Lambda\to\Gamma$ consists of, for any object $x\in\ca$, a 3-cell $\theta_x\colon \Lambda_x\to\Gamma_x$ in $\cb$ subject to the axiom in \Cref{subsec:ax-pert}.

\end{itemize}
The Gray-category structure on $\lax(\ca,\cb)$ was constructed in \cite{GohlaB:mapp} and is described explicitly in \Cref{app:lax-gray-cat-str}.  

It follows immediately from this description that there are evaluation Gray-functors $ev_x\colon \lax(\ca,\cb) \to \cb$ where $x$ is an object of $\ca$: here, $ev_x$ sends a pseudomap $F$ to its value $Fx$, a lax transformation $\alpha$ to its component $\alpha_x$, a modification $\Gamma$ to the $2$-cell $\Gamma_x$ and a perturbation to the $3$-cell $\theta_x$.  Since perturbations are merely indexed families of $3$-cells, it is immediate that the evaluation maps $ev_x\colon \lax(\ca,\cb) \to \cb$ are jointly faithful on $3$-cells.

\begin{rmk}\label{rmk:PathSpace1}
Whilst we have given a direct description of $\lax(\ca,\cb)$ above, we note that Gohla takes a more abstract approach beginning by constructing a \emph{path space} functor $P\colon\GCat \to \GCat$. He then employs iterated powers of $P$ to obtain the general mapping spaces $\lax(\ca,\cb)$.  We describe the construction in more detail in Remark~\ref{rmk:PathSpace2}.
\end{rmk}

\subsection{The mapping space bifunctor $\lax(-,-)$}

In Section 6 of \cite{GohlaB:mapp}, Gohla describes functoriality properties of $\lax(\ca,\cb)$ in $\ca$ and $\cb$, which we now review and then extend.

\begin{itemize}
\item Firstly, given a \emph{pseudomap} $F\colon\ca\to\ca'$ there is, by Theorem 6.2 of \cite{GohlaB:mapp}, a \emph{strict} map $\lax(F,\cb)\colon\lax(\ca',\cb)\to\lax(\ca,\cb)$.  
On $0$-cells this is given by pre-composition by $F$ and on higher cells it is given by left whiskering.  Full details are given in \Cref{app:lax.F.B}.
\item Secondly, given a \emph{strict map} $G \colon \cb \to \cb'$, there is, by Theorem 6.4 of  \cite{GohlaB:mapp}, a strict map $\lax(\ca,G)\colon\lax(\ca,\cb)\to\lax(\ca,\cb')$.  On $0$-cells this is given by post-composition by $G$ and on $1/2/3$-cells it is given by right whiskering.  Right whiskering by $G$ simply applies the strict map $G$ to the components of the relevant lax transformation/modification or perturbation.  
\end{itemize}
These operations yield functors $\lax(-,\cb)$ and $\lax(\ca,-)$ and using the description of $\lax(F,\cb)$ in \Cref{app:lax.F.B} together with the action of $\lax(\ca,G)$ by application, it is trivial to verify the compatibility $\lax(\ca,G)\lax(F,\cb) = \lax(F,\cb')\lax(\ca',G)$ so producing a bifunctor $\lax(-,-)\colon\GCatp^{op} \times \GCat \to \GCat$.  



For our purposes, we need to extend this to a bifunctor $\lax(-,-)\colon\GCatp^{op} \times \GCatp \to \GCatp$, which requires constructing a pseudomap $\lax(\ca,G)\colon\lax(\ca,\cb)\to\lax(\ca,\cb')$ for each pseudomap $G \colon \cb \to \cb'$.  Again, this is given by right whiskering, but this is now more subtle, involving conjugation by the invertible cocyle $G^2$.  We give full explicit details of these right whiskering operations in \Cref{app:lax.A.G}, for instance, describing the data of the whiskered lax transformation $G\alpha$.  

However, it would be rather tedious to give a direct verification using these formulae that $\lax(\ca,G)\colon\lax(\ca,\cb)\to\lax(\ca,\cb')$ is a pseudomap --- for instance, we would need to verify that $G\alpha$, so defined, satisfies the axioms for a lax transformation.  Therefore, we now provide a conceptual construction of the pseudomap $\lax(\ca,G)\colon\lax(\ca,\cb)\to\lax(\ca,\cb')$ which avoids any such calculations.  (This approach is guided by the theory of weak maps for algebraic weak factorisation systems \cite{Bourke2016Weak}, but we will keep our treatment essentially elementary, leaving a full treatment of the algebraic aspects as an option for the reader in the proof of Proposition~\ref{prop:asd}.)

To this end, let us define an \emph{algebraic surjection} to be a strict map $F \colon \ca \to \cb$ of Gray-categories which is surjective on $0$-cells, full on $1$-cells, fully faithful on both $2$-cells and $3$-cells and, moreover, equipped with a \emph{lifting function} $\phi$ giving:
\begin{itemize}
\item for each $x \in \cb$, an object $\phi x \in \ca$ such that $G( \phi x) = x$;
\item for each $x,y \in \ca$ and $1$-cell $f \colon Gx \to Gy \in \cb$, a $1$-cell $\phi(x,f,y) \colon x \to y$ such that $G \phi(x,f,y) = f$ and satisfying the normality condition $\phi(x,1_{Gx},x) = 1_x$.
\end{itemize}
We denote an algebraic surjection as above by $(F,\phi) \colon \ca \to \cb$ and sometimes abuse notation by writing $\phi(f) \colon x \to y$ instead of $\phi(x,f,y)$ when the context is clear.  

We remark that algebraic surjections are trivial fibrations for Lack's model structure on Gray-categories \cite{Lack2011A-quillen}, but have the stronger property of being faithful on $2$-cells.  

\begin{eg}\label{eg:surj}
The counit $p \colon Q\ca \to \ca$ is an algebraic surjection, when equipped with the structure $\phi x = x$ on objects and, on maps $f$, $\phi(x,f,y) = [f]$ if $f$ is non-identity and $\phi(x,f,y) = 1_x$ if it is the identity.  
\end{eg}

\begin{lemma}\label{lem:ss1}
Each algebraic surjection $(F,\phi) \colon \ca \to \cb$ has a canonical section $F_{\phi} \colon \cb \to \ca$ in $\GCat_p$.
\end{lemma}
\begin{proof}
On objects, we set $F_{\phi} x = \phi (x)$ and, at $f \colon x \to y$, we define $F_{\phi} f = \phi(f) \colon F_{\phi}x \to F_{\phi}y$ so that $F(F_{\phi}x) = x$ and $F(F_{\phi}f) = f$.  By the normality condition for $\phi$, $F_{\phi}$ preserves identity $1$-cells.  Since $F$ is fully faithful on $2$-cells and $3$-cells, there exists a unique way in which to complete $F_{\phi}$ to a morphism of $3$-graphs such that $FF_{\phi} = 1$.  In fact, there is a unique way to complete $F_{\phi}$ to a pseudomap such that $FF_{\phi} = 1$.  Indeed, given $1$-cells $f \colon x \to y$ and $g \colon y \to z$ we obtain the parallel $1$-cells $F_{\phi}g \circ F_{\phi}f, F_{\phi}(g \circ f) \colon F_{\phi}x \rightrightarrows F_{\phi}z$ in $\ca$ both of which are sent by $F$ to $g \circ f$; by $2$-fully faithfulness of $F$ on $2$-cells, there exists a unique invertible $2$-cell $(F_{\phi})^2_{g,f} \colon F_{\phi}g \circ F_{\phi}f \Rightarrow F_{\phi}(g \circ f)$ such that $F(F_{\phi})^2_{g,f} = 1$.

It remains to verify the pseudomap equations for $F_{\phi}$.  Having verified that $F_{\phi}$ preserves identity $1$-cells, all of the remaining equations concern equalities between parallel $2$-cells or $3$-cells in $\ca$.  Since $F$ is faithful on $2$-cells and $3$-cells, it suffices to show that each of these equations hold after application of $F$, whereupon they are sent to the corresponding equation for $FF_{\phi}$ to be a pseudomap.  And these hold since $FF_{\phi}$ is the identity on $\cb$.
\end{proof}

\begin{lemma}\label{lem:ss2}
Given an algebraic surjection $(F,\phi) \colon \ca \to \cb$, we obtain an algebraic surjection $(\lax(\cc,F),\phi^{\cc}) \colon \lax(\cc,\ca) \to \lax(\cc,\cb)$.
\end{lemma}
\begin{proof}
Given a pseudomap $G \colon \cc \to \cb$ we set $\phi^{\cc}(G) = F_{\phi} \circ G \colon \cc \to \cb \to \ca$, noting that $\lax(\cc,F)F_{\phi} \circ G = F \circ F_{\phi} \circ G = G$ as required.

At a lax transformation $\alpha \colon F G \to F H$, we must construct a lax transformation $\phi^{\cc}\alpha \colon G \to H$.  At $x \in \cc$, we define $\phi^{\cc}\alpha_x = \phi(\alpha_x) \colon Gx \to Hx$.  At $f \colon x \to y \in \ca$, we use $2$-fully faithfulness of $F$ to define $\phi^{\cc}\alpha_f \colon Gf \circ \phi^{\cc}\alpha_x \Rightarrow \phi^{\cc}\alpha_y \circ Ff$ as the unique $2$-cell such that $F \phi^{\cc}\alpha_f = \alpha_f$; similarly, we use $3$-fully faithfulness to define the invertible $3$-cells $\phi^{\cc}\alpha_{\phi}$ and $(\phi^{\cc}\alpha)^2_{f_2,f_1}$ as the unique $3$-cells with the appropriate domain and codomain such that $F(\phi^{\cc}\alpha_{\phi}) = \alpha_{\phi}$ and $F((\phi^{\cc}\alpha)^2_{f_2,f_1}) = \alpha^2_{f_2,f_1}$.  

All of the equations for $\phi^{\cc}\alpha$ involve equalities of parallel $2$-cells and $3$-cells; since $F$ is fully faithful on $2$-cells and $3$-cells, it suffices to show that these equations hold under application of $F$, under which, they become the corresponding equations for the lax transformation $\alpha$.  

We must also verify that if $\alpha$ is the identity, then so is $\phi^{\cc}\alpha$, but this follows from normality of $\phi$ and fully faithfulness of $F$ on $2$-cells and $3$-cells.

Since $F$ is fully faithful on $2$-cells and $3$-cells, and since modifications and perturbations are defined using only $2$-cells and $3$-cells, $\lax(\cc,F)$ is also fully faithful on $2$-cells and $3$-cells, completing the proof that it is an algebraic surjection.

\end{proof}

Now consider a pseudomap $G \colon \cb \to \cb'$.  By the universal property of the pseudomap classifier $Q\cb$, there exists a unique strict map $G' \colon Q\cb \to \cb'$ such that $G' \circ q =G$.  Now by Example~\ref{eg:surj}, $(p,\phi):Q\cb \to \cb$ is an algebraic surjection; hence by \Cref{lem:ss2} so is $(\lax(\cc,p),\phi^{\cc})$; therefore, by \Cref{lem:ss1}, we obtain a pseudomap $\lax(\cc,p)_{\phi^{\cc}} \colon \lax(\ca,\cb) \to \lax(\ca,Q\cb)$.  Therefore, we define $\lax(\ca,G)$ as the composite pseudomap
\begin{equation}\label{eq:compositepseudo}
\xymatrix{
\lax(\ca,\cb) \ar[rr]^{\lax(\cc,p)_{\phi^{\cc}}} && \lax(\ca,Q\cb) \ar[rr]^{\lax(\ca,G')} && \lax(\ca,\cb').
}
\end{equation}

Unpacking this definition yields the explicit description of $\lax(\ca,G)$ given in \Cref{app:lax.A.G}.  Note that if $G$ is strict, $\lax(\ca,G)$ is just the strict map given by direct application of $G$. With this in place, let us state the main result of this section.

\begin{prop}\label{prop:asd}
We obtain a functor $\lax(-,-)\colon\GCatp^{op} \times \GCatp \to \GCatp$ such that $\lax(F,\cb)$ is always strict and $\lax(\ca,G)$ is strict whenever $G$ is.
%
%
\end{prop}
\begin{proof}
It remains to verify that $\lax(\ca,-)$ preserves composition and satisfy the compatibility $\lax(\ca,G)\lax(F,\cb) = \lax(F,\cb')\lax(\ca',G)$.  There are two approaches.  

The first is elementary --- these equations follow from a few straightforward calculations using the explicit description of $\lax(\ca,G)$ in \Cref{app:lax.A.G}.  We leave this approach to the interested reader.

The second approach builds upon the conceptual definition of $\lax(\ca,G)$ using algebraic surjections, and makes serious use of the theory of algebraic weak factorisation systems (awfs), and in particular results from \cite{Bourke2016Accessible} and \cite{ Bourke2016Weak}.  For the reader with a good knowledge of awfs, we describe this proof now.  We will freely use notation and terminology from \cite{Bourke2016Accessible} and \cite{Bourke2016Weak}. 

To this end, the starting observation is that there is a small subcategory $\cj \hookrightarrow \GCat^{\atwo}$ of the category of morphisms such that algebraic surjections are morphisms equipped with a lifting operation against $\cj$ in the sense of \cite{Garner2011}.  Though the details are not needed, we note that $\cj$ contains as objects the generating cofibrations for Lack's model structure on $\GCat$ \cite{Lack2011A-quillen}, plus a further one encoding that algebraic surjections are faithful on $2$-cells.  Finally, in order to encode the normality condition, we add $id_1 \colon 1 \to 1$ to $\cj$ together with the unique morphism $j \to id_1$ where $j$ is the generating cofibration encoding fullness on $1$-cells.  

It follows that algebraic surjections form a concrete double category $\Surj \to \Sq{\GCat}$ in the sense of \cite{Bourke2016Accessible}, where $\Sq{\GCat}$ is the double category of commutative squares in $\GCat$ and $\Surj \to \Sq{\GCat}$ the forgetful double functor.  It is easy to describe it directly.  The objects and horizontal morphisms of $\Surj$ are those of $\GCat$. The vertical morphisms are the algebraic surjections, whilst the squares are commuting squares as below

$$\xy
(0,0)*+{\ca}="b0"; (15,0)*+{\cc}="c0"; (0,-15)*+{\cb}="d0"; (15,-15)*+{\cd}="e0";
{\ar ^{R} "b0";"c0"};
{\ar _{(F,\phi)} "b0";"d0"};
{\ar ^{(G,\theta)}"c0";"e0"};
{\ar_{S} "d0";"e0"};
\endxy$$

which commute with the lifting functions $\phi$ and $\theta$.  Given composable algebraic surjections $(F,\phi) \colon \ca \to \cb$ and $(G,\theta) \colon \cb \to \cc$, the composite algebraic surjection $GF$ has lifting function $\theta*\phi$ with formulae $\theta*\phi(x) = \phi(\theta(x))$ and $(\theta*\phi)(x,f,y) = \phi(x,\theta(Fx,f,Fy),y)$.  

Since $\GCat$ is locally presentable and $\cj$ is small, by Proposition 23 of \cite{Bourke2016Accessible}, Garner's algebraic small object argument generates an awfs $(L,R)$ with concrete double category of right maps $\Surj \to \Sq{\GCat}$.  Furthermore, since $\GCat$ has an initial object $\varnothing$, the awfs gives rise to a \emph{cofibrant replacement comonad} $Q$ on $\GCat$, whose value at $\ca \in \GCat$ is obtained by taking the $(L,R)$-factorisation $R! \circ L! \colon \varnothing \to Q\ca \to \ca$ of the unique map $! \colon \varnothing \to \ca$.  

In fact, this $Q$ is precisely the pseudomorphism classifier comonad of Remark~\ref{rmk:comonad}.  To see this requires exhibiting $p \colon Q\ca \to \ca$ as the free algebraic surjection on $! \colon \varnothing \to \ca$: that is, given an algebraic surjection $(G, \theta) \colon \cb \to \cc$ and morphism $F \colon \ca \to \cc$, we must show that there exists a unique morphism of algebraic surjections as depicted below
$$\xy
(0,0)*+{Q\ca}="b0"; (15,0)*+{\cb}="c0"; (0,-15)*+{\ca}="d0"; (15,-15)*+{\cc}="e0";
{\ar ^{K} "b0";"c0"};
{\ar _{(p,\phi)} "b0";"d0"};
{\ar ^{(G,\theta)}"c0";"e0"};
{\ar_{F} "d0";"e0"};
\endxy$$
where $(p,\phi)$ is as described in Example~\ref{eg:surj}.  This is straightforward.  Indeed, in order that $(K,F)$ commutes with the lifting functions $\phi$ and $\theta$ on objects, we are forced to define $Kx = \theta(Fx)$.  On morphisms, consider a generating non-identity $[f]\colon x \to y \in Q\ca$.  Since this is the lifting $\phi(f) \colon x \to y$, we are forced to define $K([f]) = \theta(Ff) \colon Kx \to Ky$.  Since $(Q\ca)_1$ is the free category on the underlying reflexive graph of $\ca_1$, this extends uniquely to a functor $K \colon (Q\ca)_1 \to \cb$ commuting with the lifting functions and making the square commute, insofar as it is defined.  Finally, the definition of $K$ on $2$-cells and $3$-cells is forced on us by the fact that $G$ is fully faithful on $2$-cells and $3$-cells, and the requirement that the square commute.

Having established that the pseudomorphism classifier $Q$ is the cofibrant replacement comonad for the awfs, we see that the Kleisli category $\GCat_Q = \GCat_p$ is the category of \emph{weak maps} for the awfs $(L,R)$.  The core of our proof concerns applying the universal property of the category of weak maps $\GCat_p$, established in Theorem 10 of \cite{Bourke2016Weak}, which asserts that the inclusion $j \colon \GCat \to \GCat_p$ freely adds a section to each algebraic surjection in a way that respects vertical composition and morphisms of algebraic surjections.  We describe this universal property precisely now.  The universal morphism which adds the sections is specified by a concrete double functor as depicted below

\begin{equation*}
\xymatrix{
\Surj \ar[d]_{} \ar[rr]^{} && \dcat{SplEpi}(\GCat_p) \ar[d]^{} \\
\Sq{\GCat} \ar[rr]_{\Sq{j}} && \Sq{\GCat_p}
}
\end{equation*}

where $\dcat{SplEpi}(\GCat_p) \to \Sq{\GCat}$ is the concrete double category of split epimorphisms in $\GCat_p$.  In fact, this concrete double functor assigns to a vertical morphism $(F,\phi)$ the split epimorphism $(F,F_{\phi})$ constructed in Lemma~\ref{lem:ss1}.  

The universal property is that for any category $\cd$, precomposition with the above concrete double functor induces a bijection between extensions of $F \colon \GCat \to \cd$ along $j$ as depicted on the left below

\begin{equation}
  \xymatrix{
    & \GCat_p \ar@{.>}[dr]^{} & & & \Surj \ar[d] \ar@{.>}[r]^-{} & \dcat{SplEpi}(\cd) \ar[d] \\ 
    \GCat \ar[ur]^{\text{\textnormal{j}}} \ar[rr]_{F} & & \cd & & \Sq{\cc}
    \ar[r]_-{\Sq{F}} & \Sq{\cd} 
  }
\end{equation}
and extensions of $\Sq{F}$ to a concrete double functor as above right.

To apply this universal property, observe that Lemma~\ref{lem:ss2} provides us with the data of a concrete double functor as in the left square below.

\begin{equation*}
\xymatrix{
\Surj \ar[d]_{} \ar[rr]^{} && \Surj \ar[d]^{} \ar[r] & \dcat{SplEpi}(\GCat_p) \ar[d] \\
\Sq{\GCat} \ar[rr]_{\Sq{\lax(\ca,-})} && \Sq{\GCat} \ar[r]_{\Sq{j}} & \Sq{\GCat_p}
}
\end{equation*}

To verify that this is a concrete double functor it remains to verify that the assignment of Lemma~\ref{lem:ss2} respects composition of and morphisms of algebraic surjections.  
To this end, given a morphism $(R,S) \colon (F,\phi) \to (G,\theta)$ of algebraic surjections we must show that
\begin{equation*}
\xymatrix{
\lax(\ca,\cb) \ar[d]_{(\lax(\ca,F),\phi^{\ca})} \ar[rr]^{\lax(\ca,R)} && \lax(\ca,\cd) \ar[d]^{(\lax(\ca,G),\theta^{\ca})} \\
\lax(\ca,\cc) \ar[rr]_{\lax(\ca,S)} && \lax(\ca,\ce)
}
\end{equation*} 
is a morphism of algebraic surjections.  At $Z \in \lax(\ca,\cb)$, this says that $RF_{\phi} Z = G_{\theta} S Z$.  In fact $RF_{\phi} = G_{\theta} S$ since, as described above, each morphism of algebraic surjections gives rise to a morphism of split epis in $\GCat_p$.  At $\alpha \colon FY \to FZ \in \lax(\ca,\cb)$, we must prove that $R\phi^{\ca}(\alpha) \colon RY \to RZ$ equals $\theta^{\ca}(S\alpha)$.  Now postcomposing both by $G$ yields $S\alpha$ and since $G$ is faithful on $2$-cells and $3$-cells,  it therefore suffices to show that both lax transformations have the same component at $x \in \ca$.  But the former has component $R(\phi(\alpha_x))$ whilst the latter has component $\theta(S\alpha_x)$ and these coincide since $(R,S)$ is a morphism of algebraic surjections.  The argument concerning preservation of composition of algebraic surjections is equally straightforward.

Therefore, by the universal property of the composite concrete double functor, we obtain a unique extension of $\lax(\ca,-)\colon\GCat \to \GCat \to \GCat_p$ along $j$ to $\lax(\ca,-)$, and following Theorem 10 of \cite{Bourke2016Weak}, this gives precisely the formula for $\lax(\ca,G)$ given in \eqref{eq:compositepseudo} above.  

Finally, consider $\lax(\cc,-), \lax(\cd,-) \colon \GCat \to \GCat$ and the natural transformation $\lax(G,-) \colon \lax(\cc,-) \to \lax(\cd,-)$ induced by a pseudomap $G \colon \cd \to \cc$.  We will show that given an algebraic surjection $(F,\phi) \colon \ca \to \cb$, the commutative square
\begin{equation*}
\xymatrix{
\lax(\cc,\ca) \ar[d]_{(\lax(\cc,F),\phi^{\cc})} \ar[rr]^{\lax(G,\ca)} && \lax(\cd,\ca) \ar[d]^{(\lax(\cd,F),\phi^{\cd})} \\
\lax(\cc,\cb) \ar[rr]_{\lax(G,\cb)} && \lax(\cd,\cb)
}
\end{equation*}
in $\GCat$ is a morphism of algebraic surjections, where the vertical morphisms have lifting functions as described in Lemma~\ref{lem:ss2}.  It then follows from the $2$-dimensional aspect of the universal property of $\GCat_p$, as described in Theorem 10 of \cite{Bourke2016Weak}, that $\lax(G,-) \colon \lax(\cc,-) \to \lax(\cd,-)$ is natural with respect to the extended functors $\lax(\cc,-), \lax(\cd,-) \colon \GCat_p \to \GCat_p$, which is precisely what is required to obtain a bifunctor.

At $X \in \lax(\cc,\cb)$ the lifting $\phi^{\cc}(X)$ is, by definition, the composite $F_{\phi} X \colon \cc \to \cb \to \ca$.  Likewise the lifting $\phi^{\cd}(XG) = F_{\phi}(XG)$.  As such, associativity $F_{\phi}(XG) = (F_{\phi}X)G$ in $\GCat_p$ ensures that the square commutes with lifting functions on objects.  At a lax transformation $\alpha \colon FX \to FY \in \lax(\cc,\cb)$, we have the lifting $\phi^{\cc}\alpha \colon X \to Y \in \lax(\cc,\ca)$.  We must prove that we have an equality of lax transformations $(\phi^{\cc}\alpha )G = \phi^{\cc}(\alpha G) :XG \to YG$.  Since both are sent by the strict morphism $\lax(\cd,F)$ to $\alpha G$, and since the algebraic surjection $\lax(\cd,F)$ is faithful both on $2$-cells and $3$-cells, it suffices to show that $(\phi^{\cc}\alpha)_{Gx} = \phi^{\cc}(\alpha G)_x$ at each $0$-cell $x$ of $\cc$.  By definition, both have value $\phi(\alpha_{Gx}) \colon XGx \to YGx$, as required.

\end{proof}

\subsection{Oplax transformations}

In order to define the Gray-category $\oplax(\ca,\cb)$ of oplax transformations, it will be convenient to use duality.  To this end, given a Gray-category $\ca$, we write $\ca^\co$ for the Gray-category whose $2$-cells are reversed, but whose $0,1$ and $3$-cells have the same orientation.  

There is an isomorphism of categories $(-)^{co}\colon\GCat_p \to \GCat_p$ sending $\ca$ to $\ca^\co$ and sending a pseudomap $F\colon\ca \to \cb$ to the pseudomap $F^\co\colon\ca^\co \to \cb^\co$ in which:
\begin{itemize}
\item $F^{co}$ has the same action on $0,1,2$ and $3$-cells.
\item At 1-cells $f\colon x\to y$ and $g\colon y\to z$ in $\ca^\co$, the invertible 2-cell $(F^\co)^2_{f_2,f_1}\colon F^\co f_2\circ F^\co f_1\to F^\co(f_2f_1)$ in $\cb^\co$ is defined as the 2-cell 
\begin{center}
$(F^2_{f_2,f_1})^{-1}\colon F(f_2f_1)=F^\co(f_2f_1)\to F^\co f_2\circ F^\co f_1=Ff_2\circ Ff_1$ in $\cb$. 
\end{center}
\end{itemize}

With this in place, let us define $\oplax(\ca,\cb)$.  The $0$-cells are pseudomaps $F\colon\ca \to \cb$ as before.

\begin{itemize}
\item \emph{1-cells:} An oplax transformation $\alpha\colon F\to G$ consists of:
\begin{itemize}
\item For any object $x\in\ca$, a 1-cell $\alpha_x\colon Fx\to Gx \in\cb$.

\item For any 1-cell $f\colon x\to y$ in $\ca$, a 2-cell  in $\cb$ 
\[\begin{tikzcd}
	{Fx} & {Gx} \\
	{Fy} & {Gy.}
	\arrow[""{name=0, anchor=center, inner sep=0}, "{\alpha_x}", from=1-1, to=1-2]
	\arrow[""{name=1, anchor=center, inner sep=0}, "{\alpha_y}"', from=2-1, to=2-2]
	\arrow["Ff"', from=1-1, to=2-1]
	\arrow["{Gf}", from=1-2, to=2-2]
	\arrow["{\alpha_f}"{xshift=-0.1cm}, shorten <=10pt, shorten >=10pt, Rightarrow, from=1, to=0]
\end{tikzcd}\]


\item For any 2-cell $\phi\colon f\to f'$ in $\ca$ a 3-cell in $\cb$ of the form
\[\begin{tikzcd}
	{\alpha_y Ff} && {\alpha_y Ff'} \\
	{Gf\alpha_x} && {Gf'\alpha_x.}
	\arrow[""{name=0, anchor=center, inner sep=0}, "{1\circ F\phi}", from=1-1, to=1-3]
	\arrow["{\alpha_{f'}}", from=1-3, to=2-3]
	\arrow["{\alpha_{f}}"', from=1-1, to=2-1]
	\arrow[""{name=1, anchor=center, inner sep=0}, "G\phi\circ1"', from=2-1, to=2-3]
	\arrow["{\alpha_\phi}"{xshift=0.1cm}, shorten <=8pt, shorten >=8pt, Rightarrow, from=0, to=1]
\end{tikzcd}\]

\item For 1-cells $f_1\colon x\to y$ and $f_2\colon y\to z$ in $\ca$, an invertible 3-cell $\alpha_2^{f_2,f_1}$ in $\cb$ as below left.

\[\begin{tikzcd}
	{\alpha_zFf_2Ff_1} && {\alpha_zF(f_2f_1)} \\
	{Gf_2\alpha_yFf_1} \\
	{Gf_2Gf_1\alpha_x} && {G(f_2f_1)\alpha_x}
	\arrow[""{name=0, anchor=center, inner sep=0}, "{1\circ F_{f_2,f_1}^2}", from=1-1, to=1-3]
	\arrow["{\alpha_{f_2}\circ1}"', from=1-1, to=2-1]
	\arrow[""{name=1, anchor=center, inner sep=0}, "{G_{f_2,f_1}^2\circ1}"', from=3-1, to=3-3]
	\arrow["{1\circ\alpha_{f_1}}"', from=2-1, to=3-1]
	\arrow["{\alpha_{f_2f_1}}", from=1-3, to=3-3]
	\arrow["{\alpha^{f_2,f_1}_2}"{xshift=0.1cm}, shorten <=25pt, shorten >=25pt, Rightarrow, from=0, to=1]
\end{tikzcd}
\hspace{0.5cm}
\begin{tikzcd}
	{\alpha_zF(f_2f_1)} && {\alpha_zFf_2Ff_1} \\
	&& {Gf_2\alpha_yFf_1} \\
	{G(f_2f_1)\alpha_x} && {Gf_2Gf_1\alpha_x}
	\arrow[""{name=0, anchor=center, inner sep=0}, "{1\circ (F_{f_2,f_1}^2)^{-1}}", from=1-1, to=1-3]
	\arrow["{\alpha_{f_2}\circ1}", from=1-3, to=2-3]
	\arrow[""{name=1, anchor=center, inner sep=0}, "{(G_{f_2,f_1}^2)^{-1}\circ1}"', from=3-1, to=3-3]
	\arrow["{1\circ\alpha_{f_1}}", from=2-3, to=3-3]
	\arrow["{\alpha_{f_2f_1}}"', from=1-1, to=3-1]
	\arrow["{({\alpha^{f_2,f_1}_2})_\sharp}"'{xshift=-0.1cm}, shorten <=25pt, shorten >=25pt, Rightarrow, from=0, to=1]
\end{tikzcd}\]

Rather than giving the axioms for an oplax transformation directly, we will obtain them via duality.  To this end, observe that pre- and post-composing the inverse $(\alpha^{f_2,f_1}_2)^{-1}$ of the $3$-cell above left by the invertible $2$-cells $\alpha_z\circ (F_{f_2,f_1}^2)^{-1}$ and $(G_{f_2,f_1}^2)^{-1}\circ \alpha_x$, we obtain an invertible $3$-cell as above right.

 Then $\alpha\colon F \to G$ is said to be an oplax transformation if $(\alpha_x,\alpha_f,\alpha_\phi,({\alpha^{f_2,f_1}_2})_\sharp)$ defines a lax transformation $\alpha^{+}\colon F^{co} \to G^{co} \in \lax(\ca^{co},\cb^{co})$.

\end{itemize}

\item \emph{2-cells:}
A modification $\Gamma\colon\alpha\to\beta$ consists of:
\begin{itemize}
\item For any object $x\in\ca$, a 2-cell $\Gamma_x\colon\alpha_x\to\beta_x$ in $\cb$.
\item For any 1-cell $f\colon x\to y$ in $\ca$, a 3-cell in $\cb$
\[\begin{tikzcd}
	{\alpha_yFf} & {Gf\alpha_x} \\
	{\beta_yFf} & {Gf\beta_x.}
	\arrow["{\Gamma_y\circ1}"', from=1-1, to=2-1]
	\arrow[""{name=0, anchor=center, inner sep=0}, "{\alpha_f}", from=1-1, to=1-2]
	\arrow[""{name=1, anchor=center, inner sep=0}, "{\beta_f}"', from=2-1, to=2-2]
	\arrow["{1\circ\Gamma_x}", from=1-2, to=2-2]
	\arrow["{\Gamma_f}"{xshift=0.1cm}, shorten <=8pt, shorten >=8pt, Rightarrow, from=0, to=1]
\end{tikzcd}\]
such that the data $(\Gamma_x, \Gamma_f)$ defines a modification $\Gamma^{+}\colon\beta^{+} \to \alpha^{+} \in \lax(\ca^{co},\cb^{co})$. 
\end{itemize}

\item \emph{3-cells:} A perturbation $\theta\colon\Gamma\to\Delta$ involves, for any object $x$ in $\ca$, a 3-cell $\theta_x\colon\Gamma_x\to\Delta_x \in \cb$ such that the data $\theta_x$ defines a perturbation $\theta^+\colon\Gamma^+ \to \Delta^+ \in \lax(\ca^{co},\cb^{co})$.
\end{itemize}

In this way, we obtain an isomorphism of $3$-graphs $(-)^{+}\colon\oplax(\ca,\cb) \cong \lax(\ca^{co},\cb^{co})^{co}$ acting as $(-)^{co}$ on $0$-cells and $(-)^{+}$ as described for higher cells.  By transport of structure, this induces a unique Gray-category structure on $\oplax(\ca,\cb)$ such that $(-)^+$ is an isomorphism of Gray-categories.  We write $(-)^{-}$ for the inverse.
Transporting again, there exists a unique structure of functor $\oplax(-,-)\colon\GCat_p^{op} \times \GCat_p \to \GCat_p$ such that the isomorphisms $(-)_{\ca,\cb}^{+}\colon\oplax(\ca,\cb) \cong \lax(\ca^{co},\cb^{co})^{co}$ are natural in $\ca$ and $\cb$.  At the same time, the inverse maps assemble into natural isomorphisms $(-)_{\ca,\cb}^{-}\colon\lax(\ca,\cb) \cong \oplax(\ca^{co},\cb^{co})^{co}$.  

It is straightforward to calculate the action of $(-)_{\ca,\cb}^{-}$ but we will not need an explicit description of it, except we note that given a lax transformation $\alpha\colon F \to G$, the invertible $3$-cell $\alpha^2_{f_2,f_1}$ of Diagram~\ref{eq:laxnat1} bijectively corresponds to a $3$-cell $({\alpha^2_{f_2,f_1}})^\flat$ as below 
\[\begin{tikzcd}[ampersand replacement=\&]
	{G(f_2f_1) \alpha_x} \&\& {\alpha_zF(f_2f_1)} \\
	{Gf_2Gf_1\alpha_x} \& {Gf_2\alpha_y Ff_1} \& {\alpha_z Ff_2Ff_1}
	\arrow["{1\circ\alpha_{f_1}}"', from=2-1, to=2-2]
	\arrow["{{(G^2_{f_2,f_1})}^{-1}\circ1}"', from=1-1, to=2-1]
	\arrow["{\alpha_{f_2}\circ1}"', from=2-2, to=2-3]
	\arrow["{1\circ {(F^2_{f_2,f_1})}^{-1}}", from=1-3, to=2-3]
	\arrow[""{name=0, anchor=center, inner sep=0}, "{\alpha_{f_2f_1}}", from=1-1, to=1-3]
	\arrow["{(\alpha^2_{f_2,f_1})\flat}"{xshift=0.1cm}, shorten <=5pt, shorten >=5pt, Rightarrow, from=0, to=2-2]
\end{tikzcd}\]
obtained by pre- and post-composing $(\alpha^2_{f_2,f_1})^{-1}$ by the invertible $2$-cells ${(G^2_{f_2,f_1})}^{-1}\circ \alpha_x$ and $\alpha_z \circ {(F^2_{f_2,f_1})}^{-1}$. 

\section{The $4$-ary skew multicategory of Gray-categories and lax multimaps}\label{sect:shortmult}

In this section, we construct a $4$-ary skew multicategory $\GCatm$ whose category of loose unary maps is $\GCat_p$, and whose category of tight unary maps is $\GCat$.  
Our approach is to build it from the bottom up, first building a $2$-ary skew multicategory, then extending it to a $3$-ary skew multicategory before completing it to a $4$-ary skew multicategory.


\begin{notation*}
Since we will work with multiple Gray-categories at a time, let us fix some notation. We will use $\ca,\cb,\cc...$ for Gray-categories, the table below shows how we will denote the various $n$-cells in these Gray-categories. 

\begin{center}
\begin{tabular}{ccccc}
Gray-category & 0-cells 
& 1-cells & 2-cells & 3-cells \\
$\ca$ & $a,a',a''...$ 
& $A,A',A_1...$ & $\alpha,\alpha',\alpha_1...$ & $\Lambda$ \\
$\cb$ & $b,b',b''...$ 
& $B,B',B_1...$ & $\beta,\beta',\beta_1...$ & $\Theta$ \\
$\cc$ & $c,c',c''...$ 
& $C,C',C_1...$ & $\gamma,\gamma',\gamma_1...$ & $\Gamma$ \\
$\cd$ & $d,d',d''...$ 
& $D,D',D_1...$ & $\delta,\delta',\delta_1...$ & $\Delta$ 
\end{tabular}
\end{center}
\end{notation*}

\subsection{Nullary maps}

The set $\GCatm^l_0(\diamond;\ca)$ of nullary maps is defined to be the set of objects of $\ca$.  Given $a\in\GCatm^l_0(\diamond;\ca)$ and a pseudomap $F\colon\ca \to \cb$ we define $F \circ a :=Fa$.  
This equips the set of nullary maps with the action of a functor $\GCatm^l_0(\diamond;-)\colon\GCat_p \to \Set$.

\subsection{Unary maps}

We set $\GCatm^l_1(-;-):= \GCat_p(-,-)$ --- the loose unary maps are the pseudomaps, with their substitution given by composition in $\GCat_p$.  The set $\GCatm^t_1(\ca;\cb)$ of tight unary maps consists of the Gray-functors. The identities are tight and composite $G \circ_1 F$ of two tight maps is tight, since $\GCat$ is a subcategory of $\GCat_p$.



Let us note that we have the following trivial result, relating nullary and unary maps, which is our first indication of closedness.

\begin{lemma}\label{lemma:nmaps}
There is a bijection $\lambda_1\colon\GCatm^l_0(\diamond;\lax(\ca,\cb)) \cong \GCatm^l_1(\ca,\cb)$, natural in each variable.
\end{lemma}
\begin{proof}
$\lambda_1$ is in fact the identity in each component.
\end{proof}

\subsection{Binary maps}\label{binary-maps}


A \emph{loose binary map} $F\colon\catA,\catB\rightarrow\catC$ consists of:
\begin{enumerate}[(1)]
\item at $a\in\catA$ a pseudomap $F^a\colon \catB\rightarrow\catC$, and at $b\in\catB$ a pseudomap $F_b\colon\catA\rightarrow\catC$ satisfying $F^a(b)=F_b(a)=:F(a,b)$ and with cocycles denoted by 
\begin{center}
\begin{tikzpicture}
\node (a) at (0,2) {$F(a,b)$};
\node (b) at (2.5,2) {$F(a',b)$};
\node (d) at (2.5,0) {$F(a'',b)$};
\draw[->] (a) to node[scale=.7] (f) [above] {$F_b(A_1)$} (b);
\draw[->] (b) to node[scale=.7] (G'f) [right]{$F_b(A_2)$} (d);
\draw[->] (a) to [bend right] node[scale=.7] (g'F) [left, xshift=-0.1cm]{$F_b(A_2A_1)$} (d);
\node (x) at (1.25,2) {};
\node (y) at (1.25,0) {};
\draw[-{Implies},double distance=1.5pt,shorten >=22pt,shorten <=14pt] (x) to node[scale=.7] [right, xshift=0.1cm, yshift=0.3cm] {$F^{A_2,A_1}_b$} (y);
\end{tikzpicture} \hspace{1cm}
\begin{tikzpicture}
\node (a) at (0,2) {$F(a,b)$};
\node (b) at (2.5,2) {$F(a,b')$};
\node (d) at (2.5,0) {$F(a,b'')$.};
\draw[->] (a) to node[scale=.7] (f) [above] {$F^a(B_1)$} (b);
\draw[->] (b) to node[scale=.7] (G'f) [right]{$F^a(B_2)$} (d);
\draw[->] (a) to [bend right] node[scale=.7] (g'F) [left, xshift=-0.1cm]{$F^a(B_2B_1)$} (d);
\node (x) at (1.25,2) {};
\node (y) at (1.25,0) {};
\draw[-{Implies},double distance=1.5pt,shorten >=22pt,shorten <=14pt] (x) to node[scale=.7] [right, xshift=0.1cm, yshift=0.3cm] {$F^{a}_{B_2,B_1}$} (y);
\end{tikzpicture} 
\end{center}
\item at 1-cells $A\colon a\rightarrow a'$ in $\ca$ and $B\colon b\rightarrow b'$ in $\cb$, a 2-cell in $\cc$
\begin{center}
\begin{tikzpicture}
\node (a) at (0,2) {$F(a,b)$};
\node (b) at (3,2) {$F(a',b)$};
\node (c) at (0,0) {$F(a,b')$};
\node (d) at (3,0) {$F(a',b')$.};
\draw[->] (a) to node[scale=.7] (f) [above] {$F_b(A)$} (b);
\draw[->] (b) to node[scale=.7] (G'f) [right]{$F^{a'}(B)$} (d);
\draw[->] (a) to node[scale=.7] (g'F) [left]{$F^a(B)$} (c);
\draw[->] (c) to node[scale=.7] (g'F) [below]{$F_{b'}(A)$} (d);
\node (x) at (1.5,2) {};
\node (y) at (1.5,0) {};
\draw[-{Implies},double distance=1.5pt,shorten >=18pt,shorten <=18pt] (x) to node[scale=.7] [right, xshift=0.1cm] {$A_B$} (y);
\end{tikzpicture} 
\end{center}

\item at
\[
\begin{tikzcd}
a \arrow[r, bend left, "A", ""{name=U, inner sep=1pt,below}] \arrow[r, bend right, "A'"{below}, ""{name=D, inner sep=1pt, above}] & a'
\arrow[Rightarrow, from=U, to=D, "\alpha"{xshift=0.1cm}]         
\end{tikzcd} \qquad 
\begin{tikzcd}
b \arrow[r, "B", ""{name=U, inner sep=1pt,below}]& b'
\end{tikzcd}
\]
a 3-cell in $\catC$ 
\begin{center}
\begin{tikzpicture}
\node (a) at (0,2) {$F^{a'}(B)F_b(A)$};
\node (b) at (4.5,2) {$F^{a'}(B)F_b(A')$};
\node (c) at (0,0) {$F_{b'}(A)F^{a}(B)$};
\node (d) at (4.5,0) {$F_{b'}(A')F^{a}(B)$.};
\draw[->] (a) to node[scale=.7] (x) [above] {$1\circ F_b(\alpha)$} (b);
\draw[->] (b) to node[scale=.7]  [right]{$A'_B$} (d);
\draw[->] (a) to node[scale=.7]  [left]{$A_B$} (c);
\draw[->] (c) to node[scale=.7] (y) [below]{$F_{b'}(\alpha)\circ 1$} (d);
\draw[-{Implies},double distance=1.5pt,shorten >=18pt,shorten <=18pt] (x) to node[scale=.7] [right, xshift=0.1cm] {$\alpha_B$} (y);
\end{tikzpicture} 
\end{center}

\item at
\[
\begin{tikzcd}
a \arrow[r, "A", ""{name=U, inner sep=1pt,below}] & a'
\end{tikzcd} \qquad 
\begin{tikzcd}
b \arrow[r, bend left, "B", ""{name=U, inner sep=1pt,below}] \arrow[r, bend right, "B'"{below}, ""{name=D, inner sep=1pt, above}] & b'
\arrow[Rightarrow, from=U, to=D, "\beta"{xshift=0.1cm}]         
\end{tikzcd}
\]
a 3-cell in $\catC$ 
\begin{center}
\begin{tikzpicture}
\node (a) at (0,2) {$F^{a'}(B)F_b(A)$};
\node (c) at (4.5,2) {$F_{b'}(A)F^{a}(B)$};
\node (b) at (0,0) {$F^{a'}(B')F_b(A)$};
\node (d) at (4.5,0) {$F_{b'}(A)F^{a}(B')$.};
\draw[->] (a) to node[scale=.7] [left] {$F^{a'}(\beta)\circ1$} (b);
\draw[->] (b) to node[scale=.7] (y) [below]{$A_{B'}$} (d);
\draw[->] (a) to node[scale=.7] (x) [above]{$A_B$} (c);
\draw[->] (c) to node[scale=.7]  [right]{$1\circ F^a(\beta)$} (d);
\draw[-{Implies},double distance=1.5pt,shorten >=18pt,shorten <=18pt] (x) to node[scale=.7] [right, xshift=0.1cm] {$A_\beta$} (y);
\end{tikzpicture} 
\end{center}

\item at 1-cells $a\xrightarrow{A_1}a'\xrightarrow{A_2}a''$ in $\catA$ and a 1-cell $B\colon b\rightarrow b'$ in $\catB$, an invertible 3-cell in $\catC$ 
\begin{center}
\begin{tikzpicture}
\node (T) at (0,4) {$F^{a''}(B)F_{b}(A_2)F_{b}(A_1)$};
\node (T') at (5,4) {$F^{a''}(B)F_{b}(A_2A_1)$};
\node (TS0X) at (0,2) {$F_{b'}(A_2)F^{a'}(B)F_{b}(A_1)$};
\node (ST') at (0,0) {$F_{b'}(A_2)F_{b'}(A_1)F^{a}(B)$};
\node (TS0') at (5,0) {$F_{b'}(A_2A_1)F^{a}(B)$.};
\path[->] 
(ST') edge node[scale=.7] (y) [below] {$F_{b'}^{A_2,A_1}\circ1$} (TS0')
(T) edge node[scale=.7] [left] {${(A_2)}_B\circ 1$} (TS0X)
	edge node[scale=.7] (x) [above] {$1\circ F^{A_2,A_1}_b$} (T')
(TS0X) edge node[scale=.7] [left] {$1\circ {(A_1)}_B$} (ST')
(T') edge node[scale=.7] [right] {$(A_2A_1)_B$} (TS0');
\draw[-{Implies},double distance=1.5pt,shorten >=45pt,shorten <=45pt] (x) to node[scale=.7] [right, xshift=0.1cm] {$F^{A_2,A_1}_B$} (y);
\end{tikzpicture} 
\end{center}

\item at a 1-cell $A\colon a\rightarrow a'$ in $\ca$ and 1-cells $b\xrightarrow{B_1}b'\xrightarrow{B_2}b''$ in $\cb$, an invertible 3-cell in $\cc$ 
\begin{center}
\begin{tikzpicture}
\node (T) at (0,2) {$F^{a'}(B_2)F^{a'}(B_1)F_{b}(A)$};
\node (T') at (5,2) {$F^{a'}(B_2)F_{b'}(A)F^{a}(B_1)$};
\node (TS0Y) at (10,2) {$F_{b''}(A)F^{a}(B_2)F^{a}(B_1)$};
\node (ST') at (0,0) {$F^{a'}(B_2B_1)F_{b}(A)$};
\node (TS0') at (10,0) {$F_{b''}(A)F^{a}(B_2B_1)$.};
\path[->] 
(ST') edge node[scale=.7] (y) [below] {$A_{(B_2B_1)}$} (TS0')
(T) edge node[scale=.7] [left] {$F^{a'}_{B_2,B_1}\circ1$} (ST')
	edge node[scale=.7] (x) [above] {$1\circ A_{B_1}$} (T')
(T') edge node[scale=.7] [above] {$A_{B_2}\circ1$} (TS0Y)
(TS0Y) edge node[scale=.7] [right] {$1\circ F^a_{B_2,B_1}$} (TS0');
\draw[-{Implies},double distance=1.5pt,shorten >=15
pt,shorten <=15pt] (T') to node[scale=.7] [left, xshift=-0.1cm] {$F^A_{B_2,B_1}$} (y);
\end{tikzpicture}
\end{center}
\end{enumerate}
satisfying the axioms described in \Cref{app:ax_bin}.
%
%
%
%
%
%
%

We say that $F$ is \emph{tight} if each $F_b$ is a Gray-functor (that is, the $2$-cells $F^{A_2,A_1}_b$ are identities) and if, moreover, the $3$-cells $F^{A_2,A_1}_{B}$ are identities.

\subsection{The correspondence between binary and unary maps}\label{sect:2closed}

In order to equip the sets $\GCatm_2^l(\ca,\cb;\cc)$ of loose binary maps with the structure of a functor  $$\GCatm_2^l(-,-;-)\colon(\GCatp^2)^{op} \times \GCatp \to \Set$$ we will use the following result, which again forms part of closedness.

\begin{prop}
\label{prop:binary maps}
There is a bijection $\lambda_2\colon\GCatm_1^l(\ca;\lax(\cb,\cc)) \cong \GCatm_2^l(\ca,\cb;\cc)$ which restricts to tight maps on both sides.
\end{prop}
 
\begin{proof}
The data for a pseudomap $\overline{F}\colon\ca \to \lax(\cb,\cc)$ is a $3$-graph map together with invertible $2$-cells $\overline{F}^2_{A_2,A_1}\colon\overline{F}(A_2)\overline{F}(A_1)\cong \overline{F}(A_2A_1)$. Let us firstly check that this data is in correspondence with the data for a binary map.
\begin{enumerate}
\item On objects, it assigns to $a \in \ca$ a pseudomap $\overline{F}(a)=F^a$.  This involves $5$ pieces of data: $F^a(b)$, $F^a(B)$, $F^a(\beta)$, $F^a(\Theta)$ and $F^a_{B_2,B_1}$ (here we are using notation consistent with the diagrams in \Cref{binary-maps}).
\item At $A\colon a \to a'$ we obtain a lax transformation $\overline{F}(A)\colon F^a \to F^{a'}$.  This involves four pieces of data: a $1$-cell $F^A_b\colon F^a_b \to F^{a'}_b$, at $B\colon b \to b'$ a $2$-cell $A_B$; at $\beta\colon B \to B'$ a $3$-cell $A_\beta$; at a composable pair $B_1,B_2$ an invertible $3$-cell $F^A_{B_2,B_1}$.
\item At $\alpha\colon A \to A'$ we obtain a modification $\overline{F}(\alpha)\colon F^A \to F^{A'}$.  This involves two pieces of data: at $b \in \cb$ a $2$-cell $\alpha_b\colon F^A_{b} \to F^{A'}_b$; at $B\colon b \to b'$ a $2$-cell $\alpha_B$.
\item At $\Lambda\colon\alpha \to \alpha'$ we obtain a perturbation $\overline{F}({\Lambda})\colon\overline{F}({\alpha}) \to \overline{F}(\alpha')$.  This involves a single piece of data: at $b \in \cb$ a $3$-cell $\Lambda_b$.
\item At $A_1\colon a\to a',A_2\colon a'\to a''$ the $2$-cell $\overline{F}^2_{A_2,A_1}\colon\overline{F}(A_2)\overline{F}(A_1)\cong \overline{F}(A_2A_1)$ is an invertible modification. This is specified by two pieces of data: an invertible $2$-cell $F^{A_2,A_1}_b \colon F^{A_2}_b F^{A_1}_b \cong F^{A_2A_1}_b$ and, at $B\colon b \to b'$, an invertible $3$-cell $F^{A_2,A_1}_B$.
\end{enumerate}
These are precisely the same pieces of data described in the definition of a binary map.  

Next we turn to the correspondence between axioms. These are displayed in \Cref{tab:cld-bij-2}. 

\begin{table}[ht]
\caption{Correspondence of axioms for $\lambda_2$}
\renewcommand\arraystretch{1.5} 
\begin{tabular}{|c | c|}
\hline
$\overline{F}\colon\ca\rightarrow\lax(\cb,\cc)$ pseudomap & $F\colon\ca,\cb\rightarrow\cc$ multimap \\
\hline
$\overline{F}(a) \in\lax(\cb,\cc)$ & $F^a:=\overline{F}(a) \colon\catB\rightarrow\catC$ pseudomap \\
\hdashline

$\overline{F}(A)$ lax transformation (\ref{ax:lax-tr.deg-2-cell}) & \eqref{ax:bin_D_A:1b} \\

" (\ref{ax:lax-tr.3-cell}) & \eqref{ax:bin_A:Theta} \\

" (\ref{ax:lax-tr.comp-2-cell}) & \eqref{ax:bin_A:bet1bet2} \\

" (\ref{ax:lax-tr.deg-3-cell}) & \eqref{ax:bin_D_A:1B}\\

" (\ref{ax:lax-tr.comp-1-cell}) & \eqref{ax:bin_A:B1B2B3}\\
" (\ref{ax:lax-tr.deg-alpha2}) & \eqref{ax:bin_D_A:B1B2} \\

" (\ref{ax:lax-tr.whisk-alpha2}) & \eqref{ax:bin_A:B1bet}, \eqref{ax:bin_A:betB2} \\
\hdashline

$\overline{F}(\alpha)$ modification (\ref{ax:mod.deg-3}) & \eqref{ax:bin_D_alph:1b} \\

" (\ref{ax:mod.comp-1}) & \eqref{ax:bin_alp:B1B2} \\

" (\ref{ax:mod.2-cell}) & \eqref{ax:bin_alp:bet} \\
\hdashline

$\overline{F}(\Lambda)$ perturbation & \eqref{ax:bin_Lamd:B} \\
\hline

$\overline{F}^2_{A_2,A_1}$ modification (\ref{ax:mod.deg-3}) & \eqref{ax:bin_D_A1A2:1b}  \\


" (\ref{ax:mod.comp-1}) & \eqref{ax:bin_A1A2:B1B2}   \\

" (\ref{ax:mod.2-cell}) & \eqref{ax:bin_A1A2:bet}  \\
\hline

$\overline{F}$ pseudomap & $F_b:=ev_b \circ \overline{F}$ pseudomap  \\
 " (\ref{ax:psfct.id-1cells}) & \eqref{ax:bin_D_1a:B}, \eqref{ax:bin_D_1a:bet}, \eqref{ax:bin_D_1a:B1B2}\\
" $(\ref{ax:psfct.loc-sesq})_{\leq 2}$ & \eqref{ax:bin_alp1alph2:B}, \eqref{ax:bin_D_1A:B} \\
" (\ref{ax:psfct.F2-1-cells}) & \eqref{ax:bin_A1A2A3:B}\\
"  (\ref{ax:psfct.deg-F2}) & \eqref{ax:bin_D_A1A2:B} \\
" (\ref{ax:psfct.F2-whisk-2}) & \eqref{ax:bin_A1alp:B}, \eqref{ax:bin_alpA2:B} \\
\hline 
\end{tabular}
\label{tab:cld-bij-2}
\end{table}

On the left hand side are the axioms for $\overline{F}\colon\ca \to \lax(\cb,\cc)$ to be a pseudomap; on the right are the corresponding axioms for the loose binary map $F\colon\ca,\cb \to \cc$. The axioms are split into three groups. 

The first group of axioms on the left capture that the $3$-graph map take values in pseudomaps, lax transformations, modifications and perturbations. 

The second group consist of the modification axioms for $\overline{F}^2_{A_2,A_1}$.  

The third group of axioms capture that $\overline{F}$, so defined, is actually a pseudomap, which requires a little more explanation. To say that $\overline{F}$ is a pseudomap involves 7 axioms, the final of two of which involve equalities of $3$-cells --- that is, perturbations in $\lax(\cb,\cc)$.  Since the evaluation maps $ev_b \colon \lax(\cb,\cc) \to \cc$ are jointly faithful on $3$-cells, these equations are already encoded by the corresponding axioms for $F_b = ev_b \circ F$ to be a pseudomap.  Axiom~(\ref{ax:psfct.loc-sesq}) says that $\overline{F}$ is locally a $2$-functor.  The parts of this axiom referring to equalities of $3$-cells are also encoded by $F_b$, so let us refer to Axiom~$(\ref{ax:psfct.loc-sesq})_{\leq 2}$ for the remaining equations involving equalities of dimension $\leq 2$.

 Thus $\overline{F}$ is a pseudomap just when each $F_b$ is a pseudomap and the axioms (\ref{ax:psfct.id-1cells}), (\ref{ax:psfct.loc-sesq})$_{\leq 2}$ and (\ref{ax:psfct.F2-1-cells}), (\ref{ax:psfct.deg-F2}) and (\ref{ax:psfct.F2-whisk-2}) hold.

%
%
%
%
\begin{itemize}
\item Axiom~(\ref{ax:psfct.id-1cells}) says that at $1_a\colon a \to a$ we have $\overline{F}(1_a) = 1_{\overline{F}(a)}$; that is, $F^{1_{a}} = 1_{F^{a}}$.  As an equality of lax transformations, this specifies the equality of four pieces of data.  Namely,
\begin{itemize}
\item $(F^{1_a})_b = 1_{F^a_b}$ which is part of the definition of $F_b$ being a pseudomap. 
\item At $B\colon b \to b'$ we have $({1_a)}_B = 1_{F^a_B}$, which is \Cref{ax:bin_D_1a:B}. 
\item At $\beta\colon B \to B'$ we have $({1_a)}_\beta = 1_{F^a_\beta}$,  which is \Cref{ax:bin_D_1a:bet}.
\item $F^{1_a}_{B_2,B_1} = 1_{F^a_{B_2,B_1}}$,  which is \Cref{ax:bin_D_1a:B1B2}. 
\end{itemize}
\end{itemize}
The remaining equations all involve the equality of modifications.  A modification involves two pieces of data --- one parametrised by an object $b$ and one by a morphism $B\colon b \to b'$.  In each case, the equation parametrised by the object $b$ is encoded by the fact that $F_b$ is a pseudomap.  Therefore, in each case we are left with the single equation parametrised by $B\colon b \to b'$.
\begin{itemize}
\item Axiom~(\ref{ax:psfct.loc-sesq})$_{\leq 2}$ says that $\overline{F}$ preserves vertical composition of $2$-cells --- this is \Cref{ax:bin_alp1alph2:B} -- and identity $2$-cells, which is \Cref{ax:bin_D_1A:B}.
\item Axiom~(\ref{ax:psfct.F2-1-cells}) is the cocycle equation for $\overline{F}^2$ concerning a composable triple in $\ca$, which is \Cref{ax:bin_A1A2A3:B}.
\item Axiom~(\ref{ax:psfct.deg-F2}) gives the two degeneracy equations for the cocycle $\overline{F}^2$, which are $F^{1_{a'},A}_B = 1_{A_B}$ and $F^{A,1_a}_B = 1_{A_{B}}$, and so  \Cref{ax:bin_D_A1A2:B}. 
\item Axiom~(\ref{ax:psfct.F2-whisk-2}) concerns naturality of $\overline{F}^{A_2,A_1}$ in $2$-cells $A_1 \to A_1'$ and in 2-cells $A_2 \to A_2'$ and yields Equations'\eqref{ax:bin_A1alp:B} and \eqref{ax:bin_alpA2:B}.

\end{itemize}

To see that the bijection restricts to tight maps, consider $\overline{F}\colon\ca \to \lax(\cb,\cc)$ tight --- that is, a strict map. This means simply that the modification $\overline{F}^2_{A_2,A_1}$ is an identity. To say that its component at $b \in \cb$ is an identity amounts to say that $F_b\colon\ca \to \cc$ is a strict map, whilst to say that its component on morphisms $B\colon b \to b'$ is an identity is to say precisely that each $F^{A_2,A_1}_B$ is an identity -- together, these say precisely that $F\colon\ca,\cb \to \cc$ is tight.
\end{proof}

\subsection{Compact definition for binary maps}\label{sect:comp-def-binary}

We can give an equivalent description of loose binary maps using just pseudomaps, lax and oplax transformations and modifications.  In these terms, a loose binary map $F\colon\catA,\catB\rightarrow\catC$ consists of:
\begin{enumerate}
\item[(1.1)] at an object $a\in\catA$, a pseudomap $F^a\colon\catB\to\catC$,
\item[(1.2)] at an object $b\in\catB$, a pseudomap $F_b\colon\catA\to\catC$,
\end{enumerate}
such that $F^a(b)=F_b(a)$, 
together with:
\begin{enumerate}
\item[(2.1)] at a 1-cell $A\colon a\to a'\in\catA$, a lax transformation $F^A\colon F^a\to F^{a'}$,
\item[(2.2)] at a 1-cell $B\colon b\to b'\in\catB$, an oplax transformation $F_B\colon F_b\to F_{b'}$, 
\end{enumerate}
such that 
\begin{center}
$F^A(b)=F_b(A)$, $F_B(a)=F^a(B)$, \\
$(F^A)_B=(F_B)_A=:A_B$.
\end{center}
Moreover we require the data
\begin{center}
$(F^\alpha)_b:=F_b(\alpha)$, $(F^\alpha)_B:=F_B(\alpha)$, \\
 $(F_\beta)^a:=F^a(\beta)$, $(F_\beta)^A:=(F^A)_\beta$, \\
\vspace{0.1cm}
$(F^{A_2,A_1})^b:=(F_b)^{A_2,A_1}$, $(F^{A_2,A_1})^B:={(F_B)^{A_2,A_1}}$, \\
$(F_{B_2,B_1})^a:=(F^a)_{B_2,B_1}$, ${(F_{B_2,B_1})^A}:=(F^A)_{B_2,B_1}$,
\end{center}
to form two modifications 
\begin{center}
$F^\alpha\colon F^A\to F^{A'}$ in $\lax(\cb,\cc)$, \\
$F_\beta\colon F_B\to F_{B'}$ in $\oplax(\ca,\cc)$, 
\end{center}
and two invertible modifications
\begin{center}
$F^{A_2,A_1}\colon F^{A_2}\cdot F^{A_1}\to F^{A_2A_1}$ in $\lax(\cb,\cc)$, \\
$F_{B_2,B_1}\colon F_{B_2}\cdot F_{B_1}\to F_{B_2B_1}$ in $\oplax(\ca,\cc)$.
\end{center} 

We refer to this as the \emph{compact definition} for loose binary maps.  Let us compare it with the earlier definition.  Both begin with pseudomaps $F^a$ and $F_b$.  Clearly the axioms for $F^A$, $F^\alpha$ and $F^{A_2,A_1}$ correspond to the same axioms as $\overline{F}(A)$, $\overline{F}(\alpha)$ and $\overline{F}^2_{A_2,A_1}$ in \Cref{tab:cld-bij-2}, respectively. In \Cref{tab:compact-def} we show the correspondence between the remaining axioms, where the axioms for the cells in $\oplax(\ca,\cc)$ refer to the axioms for the corresponding cells in $\lax(\ca^\co,\cc^\co)^\co$ under application of the isomorphism $(-)^+$.   

\begin{table}[ht]
\caption{Correspondence of axioms for compact definition}
\renewcommand\arraystretch{1.5} 
\begin{tabular}{|c | c|}
\hline
$F_B$ oplax transformation (\ref{ax:lax-tr.deg-2-cell}) & \eqref{ax:bin_D_1a:B} \\
" (\ref{ax:lax-tr.3-cell}) & \eqref{ax:bin_Lamd:B} \\
" (\ref{ax:lax-tr.comp-2-cell}) & \eqref{ax:bin_alp1alph2:B} \\
" (\ref{ax:lax-tr.deg-3-cell}) & \eqref{ax:bin_D_1A:B} \\
" (\ref{ax:lax-tr.comp-1-cell}) & \eqref{ax:bin_A1A2A3:B} \\
" (\ref{ax:lax-tr.deg-alpha2}) & \eqref{ax:bin_D_A1A2:B} \\
" (\ref{ax:lax-tr.whisk-alpha2}) & \eqref{ax:bin_A1alp:B}, \eqref{ax:bin_alpA2:B}\\
\hdashline
$F_\beta$ modification (\ref{ax:mod.deg-3}) & , \eqref{ax:bin_D_1a:bet}\\
\hdashline
$F_{B_2,B_1}$ modification (\ref{ax:mod.deg-3})& \eqref{ax:bin_D_1a:B1B2}\\
\hline
\end{tabular}
\label{tab:compact-def}
\end{table}

We have not mentioned axioms~' \eqref{ax:mod.comp-1} or \eqref{ax:mod.2-cell} for $F_\beta$ or $F_{B_2,B_1}$.  These are redundant, corresponding to Axioms~ \eqref{ax:bin_A1A2:bet},  \eqref{ax:bin_alp:bet}, \eqref{ax:bin_A1A2:B1B2} and \eqref{ax:bin_alp:B1B2} which are already encoded in \Cref{tab:cld-bij-2} by the axioms of the modifications $F^\alpha$ and $F^{A_2,A_1}$.  Thus the compact definition is overdetermined, containing some redundant information.

The useful feature of the compact definition is that it enables us to specify binary maps in a compact way --- it is enough to specify their associated pseudomaps, lax and oplax transformations.  We will use this later --- for instance, in \Cref{Tab:Compactsub} below.

\begin{rmk}
The compact definition also implies that we obtain two perturbations 
\begin{center}
for any 3-cell $\Lambda\colon \alpha\to \alpha'$ in $\ca$, $F^\Lambda\colon F^\alpha\to F^{\alpha'}$ in $\lax(\cb,\cc)$, \\
for any 3-cell $\Theta\colon \beta\to \beta'$ in $\cb$, $F_\Theta\colon F_\beta\to F_{\beta'}$ in $\oplax(\ca,\cc)$,
\end{center}
with components defined as
\begin{center}
$(F^\Lambda)_b:=F_b(\Lambda)$ and $(F_\Theta)_a:=F^a(\Theta)$.
\end{center}
The equation for $F^\Lambda$ to be a perturbation corresponds to Axiom (\ref{ax:lax-tr.3-cell}) for the oplax transformation $F_B$ and the equation for $F_\Theta$ corresponds to Axiom (\ref{ax:lax-tr.3-cell}) for the lax transformation $F^A$. 
\end{rmk}

\subsection{The functor of binary maps}\label{sect:2fun}

Consider again the bijection $$\lambda_2\colon\GCatm_1^l(\,\ca;\lax(\cb,\cc)\,) \cong \GCatm_2^l(\ca,\cb;\cc)$$ of Proposition~\ref{prop:binary maps}, which restricts to tight maps on either side.  The left hand side of the above bijection is functorial in $\ca,\cb$ and $\cc$ --- hence the sets $\GCatm_2^l(\ca,\cb;\cc)$ on the right hand side extend uniquely to a functor $$\GCatm_2^l(-,-;-)\colon(\GCatp^2)^{op} \times \GCatp \to \Set$$ such that the bijection $\lambda_2$ is natural in each variable.

This immediately defines the substitutions of loose unary and binary.   Using that $\lambda_2$ restricts to tight maps, plus the behaviour of $\lax(-,-)$ with respect to tight morphisms, as per Proposition~\ref{prop:asd}, we obtain that $F\circ_1 P$ is tight if $F$ and $P$ are, $F \circ_2 Q$ is tight if $F$ is, and $R\circ_1 F$ is tight if both $G$ and $F$ are, where $F$ is a binary map and $R,P$ and $Q$ unary.

We spell out the details of these three substitutions in \Cref{Tab:Compactsub} using the compact definition of binary maps described in \Cref{sect:comp-def-binary}. 

\begin{table}[ht]
\caption{Substitutions in compact form}
\begin{center}
\begin{tabular}{c:c c}
$F\circ_1P$ &
$(F\circ_1P)^a=F^{Pa}$ & $(F\circ_1P)_b=F_b\circ P$ \\
& $(F\circ_1P)^A=F^{PA}$ & $(F\circ_1P)_B=F_B\circ P$ \vspace{0.25cm} \\
$F\circ_2Q$ &
$(F\circ_2Q)^a=F^a\circ Q$ & $(F\circ_2Q)_b=F_{Qb}$ \\
& $(F\circ_2Q)^A=F^A\circ Q$ & $(F\circ_2Q)_B=F_{QB}$ \vspace{0.25cm}\\
$G\circ_1 F$ &
$(G\circ_1 F)^a=G\circ F^a$ & $(G\circ F)_b=G\circ F_b$ \\
& $(G\circ_1 F)^A=G\circ F^A$ & $(G\circ F)_B=G \circ F_B$ \\
\end{tabular}
\end{center}
  \label{Tab:Compactsub}
\end{table}

\subsection{The duality isomorphism in dimension $2$}

\begin{prop}
\label{prop:duality-bin}
There is a bijection 
$$d^2_{\ca,\cb,\cc}\colon\GCatm^l_2(\ca,\cb;\cc)\to\GCatm^l_2(\cb^\co,\ca^\co;\cc^\co),$$
natural in each variable.
\end{prop}
\begin{proof}
Given $F\colon\ca,\cb \to \cc$, we start by constructing $d^2F\colon\cb^{co},\ca^{co} \to \cc^{co}$ and, to this end, will use the compact definition of binary map.
\begin{itemize}
\item At $b \in \cb^{co}$, we define $d^2F^{b} = (F_b)^{co}\colon\ca^{co} \to \cc^{co}$;
\item at $a \in \ca^{co}$, we define $d^2F_{a} = (F^a)^{co}\colon\cb^{co} \to \cc^{co}$;
\end{itemize}
and we observe that $d^2F^{b}(a) = F_b(a) = F^a(b) = d^2F_{a}(b)$ as required.  Next, 
\begin{itemize}
\item at $B\colon b \to b' \in \cb^{co}$, we define $d^2F^{B}\colon d^2F^{b} \to d^2F^{b'}$ as the lax transformation $(F_B)^{+}\colon(F_b)^{co} \to (F_{b'})^{co}$;
\item at $A\colon a \to a' \in \ca^{co}$, we define $d^2F_{A}\colon d^2F_{a} \to d^2F_{a'}$ as the oplax transformation $({F^A})^{-}\colon(F^a)^{co} \to (F^{a'})^{co}$.
\end{itemize}
This describes the data for $d^2F$.  The equation $d^2F^{B}(a) = d^2F_{a}(B)$ holds since we have $(F_B)^{+}(a) = F_{B}(a) = F^{a}(B) = (F^{a})^{-}(B)$.  Similarly $d^2F_{A}(b) = (d^2F^{b})(A)$ and $(d^2F^{B})_A = (d^2F^{B})_A$ reduce to the corresponding equations $F^{A}(b) = F_{b}(A)$ and $(F_{B})_{A} = (F^{A})_{B}$ for $F$.
Lastly, we must prove that the data
$d^2F^\beta$, $d^2F_\alpha $, $d^2F^{B_2,B_1}$ and $d^2F_{A_2,A_1}$ as defined in Section~\ref{sect:comp-def-binary} yield modifications; but these are simply the modifications $(F_{\beta})^{+}$, $(F^{\alpha})^{-}$, $(F_{B_2,B_1})^{+}$ and $(F^{A_2,A_1})^-$ respectively, proving that $d^2F$ is a binary map.  

Naturality of $d^2_{\ca,\cb,\cc}$ in each variable follows formally using only the compact definitions of substitution, as per \Cref{Tab:Compactsub}, plus functoriality of $(-)^{co}\colon\GCat_p \to \GCat_p$ together with naturality of ${(-)^{+}}$ and $(-)^{-}$.

Observe that since $(-)^{+}$ and $(-)^{-}$ are inverse operations, it follows directly that $d^2_{\ca,\cb,\cc}\colon\GCatm^l_2(\ca,\cb;\cc)\to\GCatm^l_2(\cb^\co,\ca^\co;\cc^\co)$ has inverse 
$$d^2_{\ca^{co},\cb^{co},\cc^{co}}\colon\GCatm^l_2(\cb^\co,\ca^\co;\cc^\co) \to \GCatm^l_2(\ca,\cb;\cc).$$
In particular, $d^2$ is a natural bijection.
\end{proof}

\subsection{Substitution of nullary into binary}
At $F\colon\ca, \cb \to \cc$ and $a\colon \diamond \to \ca$ we define $F \circ_1 a :=F^a\colon\cb \to \cc$; similarly, at $b\colon \diamond \to \cb$, we define $F \circ_2 b = F_b\colon\ca \to \cc$.  By construction $F \circ_2 b$ is tight if $F$ is.   
 Naturality of $\circ_1$ and $\circ_2$ follow from the description of substitution of loose unary and binary in \Cref{Tab:Compactsub}. All of the cases are similar. For instance, consider $b\colon\diamond \to \cb'$ together with $Q$ and $F$ as  in \Cref{Tab:Compactsub}.  Then
$(F \circ_2 Q) \circ_2 b = (F \circ_2 Q)_b = F_{Qb}$ whilst $F \circ_{2} (Q \circ_1 b)  = F \circ_2 Qb =F_{Qb}$ using only definitional equalities and the description of $F \circ_2 Q$ in the table.

In Proposition~\ref{prop:ternary-maps}, we will also make of the following simple lemma.

\begin{lemma}\label{lemma:lambda12}
Given a nullary map $a\colon\diamond \to \ca$ and unary map $F\colon\ca \to \lax(\cb,\cc)$ the equality $\lambda_1(F \circ_1 a) = \lambda_2(F) \circ_1 a \in \lax(\cb,\cc)$ holds.
\end{lemma}
\begin{proof}
By definition of $\lambda_1$ and of substitution of nullary into unary, we have $\lambda_1(F \circ_1 a) = F(a)$.  We also have $\lambda_2(F) \circ_1 a = \lambda_2(F)^a = F(a)$ where the first equation holds by definition of substitution of nullary into binary and the second equation holds by definition of $\lambda_2(F)\colon\ca,\cb \to \cc$ as per \Cref{prop:binary maps}.
\end{proof}

%
%
%
%
%
%

\subsection{A $2$-ary skew multicategory}
We have defined all of the structure for the $2$-ary skew multicategory $\GCatm$ and verified all associativity equations involving unary maps.  The only case left to consider is the associativity equation $(F \circ_1 a) \circ_1 b = (F \circ_2 b) \circ_1 a$ for a binary map and two nullary maps.  This is simply the binary map equation $F^a(b)=F_b(a)$.

\subsection{Ternary maps}
A \emph{loose ternary map} $F\colon\catA,\catB,\catC\rightarrow\catD$ consists of:
\begin{enumerate}[(1)]
\item at an object $a\in\catA$ a loose binary map $F^a\colon\catB,\catC\to\catD$, at $b\in\catB$ a loose binary map $F\midscript{b}\colon\catA,\catC\to\catD$ and at $c\in\catC$ a loose binary map $F_c\colon\catA,\catB\to\catD$ such that
\begin{center} 
$F^a(b,-)=F\ts?{b}(a,-)=:F^a_b\quad F^a(-,c)=F_c(a,-)=:F^a_c$ \\ 
$F\midscript{b}(-,c)=F_c(-,b)=:F^b_c$,
\end{center}

\item at 1-cells $A\colon a\to a'$ in $\ca$, $B\colon b\to b'$ in $\cb$ and $C\colon c\to c'$ in $\cc$, a 3-cell 
\begin{center}
\begin{tikzpicture}
\node (T) at (0,4) {$F^{a'}_{b'}(C)F^{a'}_{c}(B)F^{b}_{c}(A)$};
\node (N) at (-2,2) {$F^{a'}_{c'}(B)F^{a'}_{b}(C)F^{b}_{c}(A)$};
\node (T') at (5,4) {$F^{a'}_{b'}(C)F^{b'}_{c}(A)F^{a}_{c}(B)$};
\node (TS0Y) at (7,2) {$F^{b'}_{c'}(A)F^{a}_{b'}(C)F^{a}_{c}(B)$};
\node (ST') at (0,0) {$F^{a'}_{c'}(B)F^{b}_{c'}(A)F^{a}_{b}(C)$};
\node (TS0') at (5,0) {$F^{b'}_{c'}(A)F^{a}_{c'}(B)F^{a}_{b}(C)$,};

\path[->] 
(ST') edge node[scale=.7] (y) [below] {${A_B}^{c'}\circ 1$} (TS0')
(T) edge node[scale=.7] [left, xshift=-0.2cm] {${B_C}^{a'}\circ 1$} (N)
	edge node[scale=.7] (x) [above] {$1\circ {A_B}^{c}$} (T')
(N) edge node[scale=.7] [left, xshift=-0.2cm] {$1\circ {A_C}^{b}$} (ST')
(T') edge node[scale=.7] [right, xshift=0.2cm] {${A_C}^{b'}\circ1$} (TS0Y)
(TS0Y) edge node[scale=.7] [right, xshift=0.2cm] {$1\circ {B_C}^{a}$} (TS0');
\draw[-{Implies},double distance=1.5pt,shorten >=45pt,shorten <=45pt] (x) to node[scale=.7] [right, xshift=0.1cm] {$(A\mid B\mid C)$} (y);
\end{tikzpicture}
\end{center}
\end{enumerate}
which we call the \emph{incubator}, satisfying the nine axioms described in \Cref{app:ax_tern}.

We say that the loose ternary map $F\colon\ca,\cb,\cc \to \cd$ is tight if $F\midscript{b}\colon\ca,\cb \to \cd$ and $F_c\colon\ca,\cc \to \cd$ are tight.

\subsection{The correspondence between ternary and binary maps}

\begin{prop}
\label{prop:ternary-maps}
There is a bijection $\lambda_3\colon\GCatm^l_2(\ca,\cb;\lax(\cc,\cd)) \cong \GCatm^l_3(\ca,\cb,\cc;\cd)$, which restricts to tight maps on both sides.
\end{prop}
\begin{proof}
Consider a binary map $\overline{F}\colon\catA,\catB\rightarrow\lax(\cc,\cd)$.  This is specified by the following data:
\begin{enumerate}[(1)]
\item Pseudomaps $\overline{F}^a\colon\catB\rightarrow\lax(\cc,\cd)$ and $\overline{F}_b\colon\catA\rightarrow\lax(\cc,\cd)$;
\item a modification $A_B\colon \overline{F}^{a'}(B)\cdot \overline{F}_{b}(A)\to \overline{F}_{b'}(A)\cdot \overline{F}^{a}(B)$ in $\lax(\cc,\cd)$;
\item four perturbations $\alpha_B$,  $A_\beta$, $\overline{F}^{A_2,A_1}_B$, $\overline{F}^A_{B_2,B_1} \in \lax(\cc,\cd)$, the latter two invertible;
\end{enumerate}
all satisfying the binary map equations. Now, we will describe the associated ternary map $F\colon\ca,\cb,\cc \to \cd$. It is convenient, for fixed objects $a,b,c$ to consider the diagram
\begin{equation}\label{eq:ternaryass}
\begin{gathered}
\begin{tikzpicture}[triangle/.style = {fill=yellow!50, regular polygon, regular polygon sides=3,rounded corners}]

\path
	(6.5,1.4) node [triangle,draw,shape border rotate=-90,inner sep=1pt] (b') {$a$} 
        (6.5,0.5) node [triangle,draw,shape border rotate=-90,inner sep=0pt] (q) {$b$} 
	(8.5,1) node [triangle,draw,shape border rotate=-90,label=135:$\ca$,label=230:$\cb$,inner sep=2pt] (c') {$\overline{F}$}
(12,1) node [triangle,draw,shape border rotate=-90,inner sep=0pt] (c'') {$ev_c$};

\draw [-] (q) .. controls +(right:1cm) and +(left:1cm).. (c'.223);
\draw [] (b') .. controls +(right:1cm) and +(left:1cm).. (c'.137);
\draw [] (c') to node[fill=white] {$\lax(\cc,\cd)$} (c'');
\draw [-] (c'') to node[fill=white] {$\cd$} (13.5,1);
\end{tikzpicture}
\end{gathered}
\end{equation}
since it encodes the three binary maps we wish to construct and the three equations relating them. Firstly, the two pseudomaps $\overline{F}^a = \overline{F} \circ_1 a$ and $\overline{F}_b=\overline{F} \circ_2 b$ correspond under the bijection $\lambda_2$ of Proposition~\ref{prop:binary maps} to binary maps $F^a  \colon\catB,\catC\to\catD, F\midscript{b}\colon\catA,\catC\to\catD$.  Also, for each $c \in \cc$ we  have the composite binary map $F_c:=ev_{c} \circ_1 F\colon\ca,\cb \to \lax(\cc,\cd) \to \cd$.

Next, we must verify that these satisfy the three compatibility equations for a ternary map.  To this end, observe that the associativity equation $(ev_c \circ_1 \overline{F}) \circ_1 a = ev_c \circ_1 (\overline{F} \circ_1 a)$ in the $2$-multicategory $\GCatm$ is precisely the equality of pseudomaps $F^a(-,c)=F_c(a,-)$.  Likewise the associativity $(ev_c \circ_1 \overline{F}) \circ_2 b = ev_c \circ_1 (\overline{F} \circ_2 b)$ yields $F\midscript{b}(-,c)=F_c(-,b)$.  Finally, consider the associativity equation $(\overline{F} \circ_1 a) \circ_1 b = (\overline{F} \circ_2 b) \circ_1 a$.  Applying $\lambda_1$ to both sides and using Lemma~\ref{lemma:lambda12}, this equation becomes the required equality of pseudomaps $F^a(b,-)=F\midscript{b}(a,-)$.


Consider the modification $A_B$ in (2) above.  Its component at $c \in \cc$ is encoded in the binary map $F_c$.  On the other hand, the value of $A_B$ at a $1$-cell $C\colon c \to c' \in \cc$ is a $3$-cell not encoded by any of the three binary maps --- this is the extra piece of data, the incubator $(A|B|C)$.  

The four perturbations in (3) are specified by their components for each $c \in \cc$ and so encoded by the binary maps $F_c$.

Finally we turn to the axioms, detailed in \Cref{tab:cld-bij-3}. Firstly, there are the three equations for $A_B$ to be a modification.    Next, there is the single perturbation equation for each of the four perturbations in (3) above.  Then we are left with the binary map equations for $\overline{F}$, excepting $(\overline{F}^a)(b) = (\overline{F}_b)(a)$, which we have dealt with already.  Observe that all of these, except Axioms' \eqref{ax:bin_D_A:1b} and \eqref{ax:bin_D_1a:B}, concern the equality of 3-cells in $\lax(\cc,\cd)$, that is --- perturbations.  Since perturbations are specified by their components for each $c \in \cc$, these equations all reduce to the corresponding equation for $F_c$ to be a binary map.  Therefore we are left with just the two equalities Axioms' \eqref{ax:bin_D_A:1b} and \eqref{ax:bin_D_1a:B} of modifications, both of which contribute a single equation each.   Combining these two equations with the three modification equations and four perturbation equations, we obtain nine equations --- these are the nine ternary map equations concerning the incubator $(A|B|C)$.  

\begin{table}[ht]
\caption{Correspondence of axioms for $\lambda_3$ concerning incubator}
\begin{center}
\renewcommand\arraystretch{1.7} 
\begin{tabular}{|c | c|}
\hline
$\overline{F}\colon\catA,\catB\rightarrow\lax(\cc,\cd)$ multimap & $F\colon\catA,\catB,\catC\rightarrow\catD$ multimap \\
\hline
$A_B$ modification (i) & \eqref{ax:tern_D_A:B:1c} \\
\hdashline

$A_B$ modification (ii) & \eqref{ax:tern_A:B:C1C2} \\
\hdashline

$A_B$ modification (iii) & \eqref{ax:tern_A:B:gam} \\
\hdashline

\eqref{ax:bin_D_A:1b} & \eqref{ax:tern_D_A:1b:C}  \\
\hdashline

$\alpha_B$ perturbation & \eqref{ax:tern_alp:B:C} \\
\hdashline

$A_\beta$ perturbation & \eqref{ax:tern_A:bet:C} \\
\hdashline

\eqref{ax:bin_D_1a:B}  & \eqref{ax:tern_D_1a:B:C} \\
\hdashline

$\overline{F}^{A_2,A_1}_B$ perturbation & \eqref{ax:tern_A1A2:B:C} \\
\hdashline

$\overline{F}^A_{B_2,B_1}$ perturbation & \eqref{ax:tern_A:B1B2:C} \\
\hline
\end{tabular}
\end{center}
\label{tab:cld-bij-3}
\end{table}

Now to say that $\overline{F}\colon\catA,\catB\rightarrow\lax(\cc,\cd)$ is tight is, by definition, to say that $\overline{F}_b\colon\catA \to\lax(\cc,\cd)$ is tight and that the perturbation $\overline{F}^{A_2,A_1}_{B}$ is an identity.  Since $\lambda_2$ respects tightness, the first statement says that $F\midscript{b}\colon\catA,\catC \to \catD$ is tight.  Note that this implies that $(F_c)^b \colon\catA \to \catD$ is tight.  That the perturbation $\overline{F}^{A_2,A_1}_{B}$ is an identity just says that it is so at each $c \in \cc$ --- combined with tightness of $(F_c)^b$, this says precisely that $F_c\colon\catA,\catB \to \catD$ is tight, as required.
\end{proof}

\subsection{The functor of ternary maps}
Consider again the bijection $$\lambda_3\colon\GCatm_3^l(\,\ca,\cb;\lax(\cc,\cd)\,) \cong \GCatm_3^l(\ca,\cb,\cc;\cd)$$ of Proposition~\ref{prop:binary maps}, which restricts to tight maps on either side.  Much as before, the sets $\GCatm_3^l(\ca,\cb,\cc;\cd)$ on the right hand side extend uniquely to a functor 
$$\GCatm_3^l(-,-,-;-)\colon(\GCatp^3)^{op} \times \GCatp \to \Set$$ 
such that the bijection $\lambda_3$ is natural in each variable.  

This immediately defines the substitution of loose ternary and loose unary maps. Furthermore, the behaviour of $\lax(-,-)$ with respect to strict morphisms, as per Proposition~\ref{prop:asd}, ensures that $F\circ_1 P$ is tight if $F$ and $P$ are,  $F\circ_2Q$ and $F\circ_3 R$ are tight if $F$ is, and $S\circ_1 F$ is tight if both $S$ and $F$ are, where $F$ is a ternary map and $P, Q, R$ and $S$ unary. 

Below, in \Cref{Tab:ternaryunary}, we spell out the details of substitution of ternary and unary maps. 


\begin{table}[ht]
\caption{Explicit substitutions of ternary and loose unary maps}
\begin{center}
\begin{tabular}{cc:c c c c c c c c}
$F\circ_1P$ &&
$F^{Pa}$ && $F\midscript{b} \circ_1 P$ && $F_{c} \circ_1 P$ && $(PA|B|C)$ \\
$F\circ_2Q$ &&
$F^{a} \circ_1 Q$ && $F\midscript{Qb}$ && $F_{c} \circ_2 Q$ && $(A|QB|C)$ \\
$F\circ_3R$ &&
$F^{a} \circ_2 R$ && $F\midscript{b}\circ_2 R$ && $F_{Qc}$ && $(A|B|RC)$ \\
$S \circ_1 F$ &&
$S \circ_1 F^{a}$ && $S \circ_1 F\midscript{b}$ && $S \circ_1 F_{c}$ && Diagram \eqref{eq:ternaryunary} below \\
\end{tabular}
\end{center}
  \label{Tab:ternaryunary}
\end{table}

In the case of post-composition by a loose map $S$, the incubator for $S\circ_1F$ is given by the $3$-cell

\begin{equation}\label{eq:ternaryunary}
\adjustbox{scale=.725}{$
[\,(SF^{b'}_{c'}(A)\circ {S^2_{F^{a}_{c'}(B),F^{a}_{b}(C)}}^{-1})\cdot {S^2_{F^{b'}_{c'}(A),F^{a}_{c'}(B)F^{a}_{b}(C)}}^{-1}\,]\cdot S(A\mid B\mid C)\cdot [\, S^2_{F^{a'}_{b'}(C)F^{a'}_{c}(B)F^{b}_{c}(A)}\cdot(SF^{a'}_{b'}(C)\circ S^2_{F^{a'}_{c}(B),F^{b}_{c}(A)})
\,]$}
\end{equation}

Ignoring subscripts for $S^2$, it is a $3$-cell whiskered on both sides as below.
\[\begin{tikzcd}[ampersand replacement=\&]
	\bullet \&\& \bullet \&\& \bullet \&\& \bullet \&\& \bullet \&\& \bullet
	\arrow["{SF^{a'}_{b'}(C)\circ S^2_{-,-}}", from=1-1, to=1-3]
	\arrow["{S^2_{-,--}}", from=1-3, to=1-5]
	\arrow[""{name=0, anchor=center, inner sep=0}, curve={height=-24pt}, from=1-5, to=1-7]
	\arrow[""{name=1, anchor=center, inner sep=0}, curve={height=24pt}, from=1-5, to=1-7]
	\arrow["{{S^2_{-,--}}^{-1}}", from=1-7, to=1-9]
	\arrow["{SF^{b'}_{c'}(A)\circ {S^2_{-,-}}^{-1}}", from=1-9, to=1-11]
	\arrow["{S(A\mid B\mid C)}"{description}, shorten <=6pt, shorten >=6pt, Rightarrow, from=0, to=1]
\end{tikzcd}\]

\subsection{The duality isomorphism in dimension $3$}

\begin{prop}
\label{prop:duality-tern}
There is a bijection $d^3\colon\GCatm^l_3(\ca,\cb,\cc;\cd) \to \GCatm^l_3(\cc^\co,\cb^\co,\ca^\co;\cd^\co)$, natural in each variable.
\end{prop}
\begin{proof}
Given $F \in \GCatm^l_3(\ca,\cb,\cc;\cd)$, we will define $d^3F = \overline{F} \in \GCatm^l_3(\cc^\co,\cb^\co,\ca^\co;\cd^\co)$.

To get started, we set
\begin{equation*}
\overline{F}^c := d^2(F_c), \overline{F}\midscript{b} := d^2(F\midscript{b}) \textnormal{ and } \overline{F}_a := d^2(F^a)
\end{equation*}

The three equations concerning the interaction of these binary maps follow from the corresponding equations for $F$.  For instance,
\begin{equation*}
(\overline{F}^c)_a :=  d^2(F_c)_a := ((F_c)^{a})^{co} \hspace{0.5cm} \textnormal{whilst} \hspace{0.5cm} (\overline{F}_a)^c :=  d^2(F^a)^c := ((F^a)_c)^{co}
\end{equation*}
and now the two right hand sides agree since $(F_c)^{a} = (F^a)_c$.  The incubator $(A|B|C)$ for $F$ is a $3$-cell in $\cd$, or equivalently a $3$-cell in $\cd^{co}$  
\begin{equation*}
\begin{tikzpicture}
\node (T) at (0,4) {$F^{b'}_{c'}(A)F^{a}_{c'}(B)F^{a}_{b}(C)$};
\node (N) at (-2,2) {$F^{a'}_{c'}(B)F^{b}_{c'}(A)F^{a}_{b}(C)$};
\node (T') at (5,4) {$F^{b'}_{c'}(A)F^{a}_{b'}(C)F^{a}_{c}(B)$}; 
\node (TS0Y) at (7,2) {$F^{a'}_{b'}(C)F^{b'}_{c}(A)F^{a}_{c}(B)$};
\node (ST') at (0,0) {$F^{a'}_{c'}(B)F^{a'}_{b}(C)F^{b}_{c}(A)$};
\node (TS0') at (5,0) {$F^{a'}_{b'}(C)F^{a'}_{c}(B)F^{b}_{c}(A)$,};

\path[->] 
(ST') edge node[scale=.7] (y) [below] {${B_C}^{a'}\circ 1$} (TS0')
(T) edge node[scale=.7] [left, xshift=-0.2cm] {${A_B}^{c'}\circ 1$} (N)
	edge node[scale=.7] (x) [above] {$1\circ {B_C}^{a}$} (T')
(N) edge node[scale=.7] [left, xshift=-0.2cm] {$1\circ {A_C}^{b}$} (ST')
(T') edge node[scale=.7] [right, xshift=0.2cm] {${A_C}^{b'}\circ1$} (TS0Y)
(TS0Y) edge node[scale=.7] [right, xshift=0.2cm] {$1\circ {A_B}^{c}$} (TS0');
\draw[-{Implies},double distance=1.5pt,shorten >=45pt,shorten <=45pt] (x) to node[scale=.7] [right, xshift=0.1cm] {$(A\mid B\mid C)$} (y);
\end{tikzpicture}
\end{equation*}
where we have reversed the $2$-cells in $\cd$ (that is, the $1$-cells in the diagram) and rotated the diagram by 180 degrees to move from top left to lower right as before.  In terms of the notation for $\overline{F}^c$, $\overline{F}\midscript{b}$ and $\overline{F}_a$, this is a $3$-cell

\begin{equation*}
\begin{tikzpicture}
\node (T) at (0,4) {$\overline{F}^{c'}_{b'}(A)\overline{F}_{a}^{c'}(B)\overline{F}_{a}^{b}(C)$};
\node (N) at (-2,2) {$\overline{F}_{a'}^{c'}(B)\overline{F}_{b}^{c'}(A)\overline{F}_{a}^{b}(C)$};
\node (T') at (5,4) {$\overline{F}_{b'}^{c'}(A)\overline{F}_{a}^{b'}(C)\overline{F}_{a}^{c}(B)$};
\node (TS0Y) at (7,2) {$\overline{F}_{a'}^{b'}(C)\overline{F}_{b'}^{c}(A)\overline{F}_{a}^{c}(B)$};
\node (ST') at (0,0) {$\overline{F}_{a'}^{c'}(B)\overline{F}_{a'}^{b}(C)\overline{F}_{b}^{c}(A)$};
\node (TS0') at (5,0) {$\overline{F}_{a'}^{b'}(C)\overline{F}_{a'}^{c}(B)\overline{F}_{b}^{c}(A)$,};

\path[->] 
(ST') edge node[scale=.7] (y) [below] {${\overline{C_B}}^{a'}\circ 1$} (TS0')
(T) edge node[scale=.7] [left, xshift=-0.2cm] {${\overline{B_A}}^{c'}\circ 1$} (N)
	edge node[scale=.7] (x) [above] {$1\circ {\overline{C_B}}^{a}$} (T')
(N) edge node[scale=.7] [left, xshift=-0.2cm] {$1\circ {\overline{C_A}}^{b}$} (ST')
(T') edge node[scale=.7] [right, xshift=0.2cm] {$\overline{{C_A}}^{b'}\circ1$} (TS0Y)
(TS0Y) edge node[scale=.7] [right, xshift=0.2cm] {$1\circ \overline{B_A}^{c}$} (TS0');
\draw[-{Implies},double distance=1.5pt,shorten >=45pt,shorten <=45pt] (x) to node[scale=.7] [right, xshift=0.1cm] {$(A \mid B \mid C)$} (y);
\end{tikzpicture}
\end{equation*}
which we define to be the incubator $(C \mid A \mid B)$ for $\overline{F}$.  The axioms for the incubator for $\overline{F}$ correspond to those for $F$ in the obvious way.

Naturality of $d^3$ follows from the descriptions of substitution of ternary and unary maps in Table~\ref{Tab:ternaryunary} and naturality of $d^2$.  For instance, let us verify naturality in the most complicated case of post-composition 
of a loose ternary map $F\colon\ca,\cb,\cc\to\cd$ by a loose unary map $S\colon\cd\to\cd'$.  We must show that $d^3(S \circ_1 F) = S^\co \circ_1 d^3 F$.  Naturality of $d^2$ ensures that they have the same underlying binary maps, so it remains to check that they have the same incubator.  Now the incubator of $S\circ_1 F$ is given by the 3-cell
\begin{center}
\begin{tikzcd}[ampersand replacement=\&]
	\bullet \&\& \bullet \&\& \bullet \&\& \bullet \&\& \bullet \&\& \bullet
	\arrow["{SF^{a'}_{b'}(C)\circ S^2_{-,-}}", from=1-1, to=1-3]
	\arrow["{S^2_{-,--}}", from=1-3, to=1-5]
	\arrow[""{name=0, anchor=center, inner sep=0}, curve={height=-24pt}, from=1-5, to=1-7]
	\arrow[""{name=1, anchor=center, inner sep=0}, curve={height=24pt}, from=1-5, to=1-7]
	\arrow["{{S^2_{-,--}}^{-1}}", from=1-7, to=1-9]
	\arrow["{SF^{b'}_{c'}(A)\circ {S^2_{-,-}}^{-1}}", from=1-9, to=1-11]
	\arrow["{S(A\mid B\mid C)}"{description}, shorten <=6pt, shorten >=6pt, Rightarrow, from=0, to=1]
\end{tikzcd} \hspace{0.5cm} in $\cd'$. 
\end{center}
Therefore the incubator of $d^3(S\circ_1F)$ is the 3-cell 
\begin{center}
\begin{tikzcd}[ampersand replacement=\&]
	\bullet \&\& \bullet \&\& \bullet \&\& \bullet \&\& \bullet \&\& \bullet
	\arrow["{SF^{b'}_{c'}(A)\circ {S^2_{-,-}}^{-1}}", from=1-1, to=1-3]
	\arrow["{{S^2_{-,--}}^{-1}}", from=1-3, to=1-5]
	\arrow[""{name=0, anchor=center, inner sep=0}, curve={height=-24pt}, from=1-5, to=1-7]
	\arrow[""{name=1, anchor=center, inner sep=0}, curve={height=24pt}, from=1-5, to=1-7]
	\arrow["{S^2_{-,--}}", from=1-7, to=1-9]
	\arrow["{SF^{a'}_{b'}(C)\circ S^2_{-,-}}", from=1-9, to=1-11]
	\arrow["{S(A\mid B\mid C)}"{description}, shorten <=6pt, shorten >=6pt, Rightarrow, from=0, to=1]
\end{tikzcd} \hspace{0.5cm} in ${\cd'}^\co$. 
\end{center}
This is exactly the incubator of $S^\co\circ_1d^3(F)$ since $(S^\co)^2=({S^2})^{-1}$. To see that $d^3$ is a natural bijection, simply observe that 
$$d^3_{\ca,\cb,\cc,\cd}\colon\GCatm^l_3(\ca,\cb,\cc;\cd) \to \GCatm^l_3(\cc^\co,\cb^\co,\ca^\co;\cd^\co)$$
has $d^3_{\cc^{co},\cb^{co},\ca^{co},\cd^{co}}$ as its inverse.

\end{proof}

 \subsection{Substitution of nullary into ternary:}
Consider a loose ternary map $F\colon\ca, \cb, \cc \to \cd$ and three nullary maps $a\colon\diamond \to \ca$, $b\colon\diamond \to \cb$ and $c\colon\diamond \to \cc$.   We define $F \circ_1 a :=F^a\colon\cb,\cc \to \cd$, $F\circ_2 b := F\midscript{b}\colon\catA,\catC\to\catD$ and $F \circ_3 c = F_c\colon\catA,\catB\to\catD$.
 
 By definition of tight ternary maps, if $F$ is tight, so are $F \circ_2 b$ and $F \circ_3 c$, as required.   
 

Naturality of the above three substitutions follow from the description of substitution of ternary and unary in Table~\ref{Tab:ternaryunary}.
All of the cases are similar.  For example, consider $b\colon\diamond \to \cb'$ together with $Q$ and $F$ as in Table~\ref{Tab:ternaryunary}.  Then
$(F \circ_2 Q) \circ_2 b := (F \circ_2 Q)\midscript{b} := F\midscript{Qb}$ whilst similarly $F \circ_{2} (Q \circ_1 b) := F \circ_2 Qb := F\midscript{Qb}$.

By analogy with Lemma~\ref{lemma:lambda12}, we have the following trivial result relating $\lambda_3$ with $\lambda_2$ in the context of a nullary map.  As in Lemma~\ref{lemma:lambda12}, the proof is simply a matter of tracing through the definitions. 

\begin{lemma}\label{lemma:lambda23}
Given a nullary map $a_i\colon\diamond \to \ca_i$ and binary map $F\colon\ca_1,\ca_2 \to \lax(\cb,\cc)$ we have an equality of binary maps $\lambda_2(F \circ_i a) = \lambda_3(F) \circ_i a$.
\end{lemma}

\subsection{Substitution of binary into binary}
\subsection*{First variable}

Given $F\colon\ca,\cb \to \cx$ and $G\colon\cx,\cc \to \cd$ we define the substitution $$G \circ_1 F := \lambda_3({\lambda_2}^{-1}(G) \circ_1 F)$$ where ${\lambda_2}^{-1}(G)\colon\cx \to \lax(\cc,\cd)$ is the adjoint map.  Naturality of this substitution follows from the naturality of the components $\lambda_3$, $\lambda_2$ and $\circ_1$ (for binary and unary maps) involved in its construction.

Observe also that if $F$ and $G$ are tight, then since $\lambda_2$ respects tightness, the binary map $\lambda_2^{-1}(G)$ is tight.  Therefore so is the composite $\lambda_2^{-1}(G) \circ_1 F$, and since $\lambda_3$ also respects tightness, we have that $G \circ_1 F$ is tight, as required.
%
%



\subsection*{Second variable}
Consider $F\colon\cb,\cc \to \cx$ and $G\colon\ca,\cx \to \cd$.  Firstly, we use the duality bijection $d^2$ to form $$d^2G\colon\cx^{co},\ca^{co} \to \cd^{co} \hspace{0.5cm} \textnormal{and} \hspace{0.5cm} d^2F\colon\cc^{co},\cb^{co} \to \cx^{co}.$$
Then, we form $d^2G \circ_1 d^2F\colon\cc^{co},\cb^{co},\ca^{co} \to \cd^{co}$ before dualising again to define $$G\circ_2 F:= d^3(d^2G \circ_1 d^2F)\colon\ca,\cb,\cc \to \ce.$$

It follows from naturality of $d^2$, $d^3$ and $\circ_1$ for binary maps that $\circ_2$ is natural in each variable.

%

\Cref{Tab:binbin} below describes the substitutions of binary into binary.

\begin{table}[ht]
\caption{Explicit substitutions of binary into binary}
\begin{center}
\begin{tabular}{cc:c c c c c c c c}
$G \circ_1 F$ &&
$G \circ_1 F^a$ && $G \circ_1 F_b$ && $G_c \circ_1 F$ && $\textnormal{Incubator in Appendix}~\ref{app:inc-sub-b-in-b-1st}$ \\
$G \circ_2 F$ &&
$G^a \circ_1 F$ && $G \circ_2 F^b$ && $G \circ_2 F_c$ && $\textnormal{Incubator in Appendix}~\ref{app:inc-sub-b-in-b-2nd}$ \\
\end{tabular}
\end{center}
  \label{Tab:binbin}
\end{table}

Since we will not require any explicit information about the incubators, we put them in the appendix.  

Suppose that $G$ is tight.  We must show that $G \circ_2 F$ is tight, which is to say that $(G \circ_2 F) \midscript{b}$ and $(G \circ_2 F)_{c}$ are tight binary maps.  We cannot argue by analogy with $\circ_1$ since neither $d^2$ nor $d^3$ respect tightness --- instead, we use the description in Table~\ref{Tab:binbin}.  In that case, $G \circ_2 F^b$ and $G \circ_2 F_c$ are always tight since $G$ is a tight binary map.

\subsection{A $3$-ary skew multicategory}
We have defined all of the structure for the $3$-ary skew multicategory $\GCatm$ and verified all associativity equations, except those relating nullary maps with the two above substitutions $\circ_1$ and $\circ_2$ of binary into binary.  In each case there are three equations, and all six follow directly from the description in Table~\ref{Tab:binbin}.

\subsection{4-ary maps}\label{sect:4ary}
A \emph{loose 4-ary map} $F\colon\catA,\catB,\catC\rightarrow\catD$ consists of:
\begin{itemize}
\item at $a\in\catA$, a loose ternary map $F^a\colon\catB,\catC,\catD\to\catE$;
\item at $b\in\catB$, a loose ternary map $F\midscript{b}\colon\catA,\catC,\catD\to\catE$;
\item at $c\in\catC$, a loose ternary map $F\midscript{c}\colon\catA,\catB,\catD\to\catE$;
\item at $d\in\catD$, a loose ternary map $F_d\colon\catA,\catB,\catC\to\catE$;
\end{itemize}
such that the following six equations between loose binary maps hold
\begin{center}
$F^a(b,-,-)=F\midscript{b}(a,-,-)=:F^a_{b}\quad F^a(-,c,-)=F\midscript{c}(a,-,-)=:F^a_{c}$ \\ \vspace{0.1cm}
$F^a(-,-,d)=F_{d}(a,-,-)=:F^a_{d} \quad F\midscript{b}(-,c,-)=F\midscript{c}(-,b,-)=:F^b_c$ \\ \vspace{0.1cm}
$F\midscript{b}(-,-,d)=F_{d}(-,b,-)=:F^{b}_{d} \quad F\midscript{c}(-,-,d)=F_{d}(-,-,c)=:F^{c}_d$
\end{center}
and such that the \emph{mecon axiom} is satisfied (see Equation~\eqref{ax:4-ary_mec} in \Cref{app:ax_4-ary}).

A loose $4$-ary map $F\colon\ca,\cb,\cc,\cd \to \ce$ is tight if the ternary maps $F\midscript{b}$, $F\midscript{c}$ and $F_d$ are tight for all $b \in \cb, c \in \cc$ and $d \in \cd$.

\begin{rmk}\label{rmk:redundant}
There is some redundancy in the above definition of tight $4$-ary maps. In fact, any two of $F\midscript{b}$, $F\midscript{c}$ and $F_d$ being tight implies the third is tight. Indeed if $F\midscript{b}$ are $F\midscript{c}$ tight, consider $F_d$: we must show that $F_d(b,-,-)$ and $F_d(-,c,-)$ are tight binary maps, but these equal $F\midscript{b}(-,-,d)$ and $F\midscript{c}(-,-,d)$ which are tight by assumption.
\end{rmk}

\subsection{The correspondence between $4$-ary and $3$-ary maps}

\begin{prop}
\label{lemma:4-ary-maps}
There is a bijection $$\lambda_4\colon \GCatm^l_3(\,\ca,\cb,\cc;\lax(\cd,\ce)\,) \to \GCatm^l_4(\ca,\cb,\cc,\cd;\ce),$$
which restricts to tight maps on either side.
\end{prop}
\begin{proof}
Consider a ternary map $\overline{F}\colon\catA,\catB,\catC \rightarrow\lax(\cd,\ce)$. This is specified by:
\begin{enumerate}[(1)]
\item three binary maps $\overline{F}^a\colon \catB,\catC \rightarrow \lax(\cd,\ce)$, $\overline{F}\midscript{b}\colon\catA,\catC \rightarrow\lax(\cd,\ce)$ and $\overline{F}_{c}\colon\ca,\cb\rightarrow\lax(\cd,\ce)$;
\item a perturbation $(A|B|C)$ in $\lax(\cd,\ce)$;
\end{enumerate}
all satisfying the ternary map equations. Now, we will describe the associated $4$-ary map $F\colon\ca,\cb,\cc \to \cd$. It is convenient, for fixed objects $a,b,c,d$ to consider the diagram below.

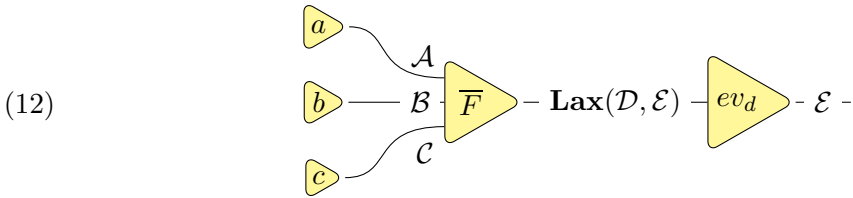
\begin{equation}\label{eq:ternaryass}
\begin{gathered}
\begin{tikzpicture}[triangle/.style = {fill=yellow!50, regular polygon, regular polygon sides=3,rounded corners}]

\path
	(6.5,2) node [triangle,draw,shape border rotate=-90,inner sep=1pt] (b') {$a$} 
	(6.5,1) node [triangle,draw,shape border rotate=-90,inner sep=0pt] (b) {$b$} 
        (6.5,-0) node [triangle,draw,shape border rotate=-90,inner sep=1pt] (q) {$c$} 
	(8.5,1) node [triangle,draw,shape border rotate=-90,label=135:$\ca$,label=230:$\cc$,inner sep=2pt] (c') {$\overline{F}$}
	(12,1) node [triangle,draw,shape border rotate=-90,inner sep=0pt] (c'') {$ev_d$};

\draw [-] (b) to node[fill=white,near end] {$\cb$} (c');
\draw [-] (q) .. controls +(right:1cm) and +(left:1cm).. (c'.223);
\draw [] (b') .. controls +(right:1cm) and +(left:1cm).. (c'.137);
\draw [] (c') to node [fill=white] {$\lax(\cd,\ce)$} (c'');
\draw [-] (c'') to node[fill=white] {$\ce$} (13.5,1);
\end{tikzpicture}
\end{gathered}
\end{equation}
Firstly, the three ternary maps $\overline{F}^a = \overline{F} \circ_1 a$, $\overline{F}\midscript{b}=\overline{F} \circ_2 b$ and $\overline{F}_{c}=\overline{F} \circ_3 c$ correspond under the bijection $\lambda_3$ of Proposition~\ref{prop:ternary-maps} to $F^a  \colon\catB,\catC,\catD \to\catE, F\midscript{b}\colon\catA,\catC,\catD \to\catE$ and $F_c\colon \catA,\catB,\catD \to \catE$.  Also, for each $d \in \cd$ we have the composite ternary map $F_d:=ev_{d} \circ_1 F\colon\ca,\cb,\cc \to \lax(\cd,\ce) \to \ce$.

We must establish the six equations relating the four ternary maps, which are parametrised by a choice of two elements from $\{a,b,c,d\}$.  Each of these corresponds to one of the six associativity equations visible in the above diagram, involving the ternary map and any two of the other four multimaps (nullary or unary). The argument here is identical in form to the corresponding part of the proof of Proposition~\ref{prop:ternary-maps}, but using Lemma~\ref{lemma:lambda23} instead of Lemma~\ref{lemma:lambda12}, and we leave the details to the reader.

The perturbation axiom for $(A|B|C)$ is precisely the mecon axiom.

Finally, $\overline{F}$ is tight just when $\overline{F}\midscript{b}$ and $\overline{F}\midscript{c}$ are tight, but since $\lambda_3$ respects tightness, this is precisely to say that $F\midscript{b}$ and $F\midscript{c}$ are tight.  By Remark~\ref{rmk:redundant}, these two conditions fully capture tightness of $4$-maps --- that is, tightness of $F_d$ follows automatically.
 \end{proof}
 
 \subsection{The functor of $4$-ary maps}
 Consider again the bijection $$\lambda_4\colon \GCatm^l_3(\,\ca,\cb,\cc;\lax(\cd,\ce)\,) \to \GCatm^l_4(\ca,\cb,\cc,\cd;\ce).$$
 As in the cases of $\lambda_2$ and $\lambda_3$, the right hand side admits the unique structure of a functor
 $$\GCatm^l_4(-;-)\colon(\GCat_{p}^{op})^4 \times \GCat_p \to \Set$$
  such that $\lambda_4$ becomes natural in each variable.  This defines the substitution of $4$-ary and unary maps.  
  
 Furthermore, the behaviour of $\lax(-,-)$, as per Proposition~\ref{prop:asd}, ensures that $F\circ_1 P$ is tight if $F$ and $P$ are, $F \circ_i P$ is always tight if $F$ is and $i > 1$, and $Q \circ_1 F$ is tight if both $Q$ and $F$ are, where $F$ is a 4-ary map and $P$ and $Q$ unary. 

In \Cref{Tab:4-aryunary} below, we spell out the details of substitution of $4$-ary and unary. 

\begin{table}[ht]
\caption{Explicit substitutions of $4$-ary and unary maps}
\begin{center}
\begin{tabular}{cc:c c c c c c c c}
$F\circ_1P$ &&
$F^{Pa}$ && $F\midscript{b} \circ_1 P$ && $F\midscript{c} \circ_1 P$ && $F_d \circ_1 P$ \\
$F\circ_2P$ &&
$F^{a} \circ_1 P$ && $F\midscript{Pb}$ && $F\midscript{c} \circ_2 P$ && $F_d \circ_2 P$ \\
$F\circ_3P$ &&
$F^{a} \circ_2 P$ && $F\midscript{b} \circ_2 P$ && $F\midscript{Pc}$ && $F_d \circ_3 P$ \\
$F\circ_4P$ &&
$F^{a} \circ_3 P$ && $F\midscript{b} \circ_3 P$ && $F\midscript{c} \circ_3 P$ && $F_{Pd}$ \\
$Q \circ_1 F$ &&
$Q\circ_1 F^{a}$ && $Q \circ_1 F\midscript{b}P$ && $Q\circ_1 F\midscript{c}P$ && $Q \circ_1 F_{d}$ \\
\end{tabular}
\end{center}
  \label{Tab:4-aryunary}
\end{table}

%

\subsection{The duality isomorphism in dimension $4$}
\begin{prop}\label{prop:duality4}
There is a bijection
$$d^3\colon\GCatm^l_4(\ca,\cb,\cc,\cd;\ce) \to \GCatm^l_4(\cd^\co,\cc^\co,\cb^\co,\ca^\co;\ce^\co),$$
natural in each variable.
 \end{prop}
\begin{proof}
Given $F \in \GCatm^l_4(\ca,\cb,\cc,\cd;\ce)$, we define $d^4F \in \GCatm^l_4(\cd^\co,\cc^\co,\cb^\co,\ca^\co;\cd^\co)$.

We set
\begin{equation*}
d^4F^d := d^3(F_d), d^4F^d := d^3(F\midscript{c}), d^4F\midscript{b} := d^3(F\midscript{b}) \textnormal{ and } d^4F_a := d^3(F^a).
\end{equation*}

The six equations concerning the interaction of these ternary maps follow from the corresponding equations for $F$.  For instance,
\begin{equation*}
(d^4F^d)(-,-,a) :=  d^3(F_d)_a := d^2(F_d(a-,-))
\end{equation*} whilst
\begin{equation*}
(d^4F_a)(d,-,-) :=  d^3(F^a)^d := d^2(F^a(-,-,d))
\end{equation*}
and now the two right hand sides agree since $(F_d)(a,-,-) = (F^a)(-,-,d)$.  The mecon axiom for $F$, which involves the equality of two $3$-cells in $\cd$, coincides with the mecon axiom for $d^4F$ viewing these $3$-cells as belonging to $\cd^{\co}$ by reversing the orientations of their bounding $2$-cells.

Naturality of $d^4$ follows from naturality of $d^3$ and to see that $d^4$ is a natural bijection, simply observe that 
$d^4_{\ca,\cb,\cc,\cd}\colon\GCatm^l_4(\ca,\cb,\cc,\cd;\ce) \to \GCatm^l_4(\cd^\co,\cc^\co,\cb^\co,\ca^\co;\ce^\co)$ has $d^4_{\cd^{\co},\cc^{\co},\cb^{\co},\ca^{\co},\ce^{\co}}$ as its inverse.
\end{proof}

\subsection{Substitution of nullary into $4$-ary}
Consider a loose $4$-ary map $F\colon\ca, \cb,\cc,\cd \to \ce$ and four nullary maps $a\colon\diamond \to \ca$, $b\colon\diamond \to \cb$, $c\colon\diamond \to \cc$ and $d\colon\diamond \to \cd$.   We define $F \circ_1 a :=F^a$, $F\circ_2 b := F\midscript{b}$, $F \circ_3 c = F\midscript{c}$ and $F \circ_4 d = F_d$.

If $F$ is tight, so are $F \circ_2 b$, $F \circ_3 c$ and $F \circ_4 d$ by definition of these substitutions.

The six associativity equations involving two nullary maps and a $4$-ary map reduce to the six equalities of binary maps in the definition of a $4$-ary map.  As before, naturality of the four substitutions of nullary into $4$-ary follows from the description of substitution of $4$-ary and unary in Table~\ref{Tab:4-aryunary}.

\subsection{Substitution of ternary into binary}
\subsection*{First variable}

Given $F\colon\ca,\cb,\cc \to \cx$ and $G\colon\cx,\cd \to \ce$ we define the substitution $$G \circ_1 F := \lambda_4({\lambda_2}^{-1}(G) \circ_1 F)$$ where ${\lambda_2}^{-1}(G)\colon\cx \to \lax(\cd,\ce)$ is the adjoint map.  

Naturality of this substitution follows from naturality of $\lambda_4$, $\lambda_2$ and the substitution $\circ_1$ of ternary maps.

The multimap $G \circ_1 F$ is tight whenever $F$ and $G$ are.  Indeed, if $G$ is tight, then as $\lambda_2$ respects tightness, so is $(\lambda_2)^{-1}(G)$.  Therefore as $\circ_1$ for binary maps respects tightness, $(\lambda_2)^{-1}(G) \circ_1 F$ is tight.  Finally, since $\lambda_4$ respects tightness, $G \circ_1 F$ is tight.




\subsection*{Second variable} 

Consider $G\colon\ca,\cx \to \ce$ and $F\colon\cb,\cc,\cd \to \cx$.  Firstly, we use the duality bijections to form $$d^2G\colon\cx^{co},\ca^{co} \to \ce^{co} \hspace{0.5cm} \textnormal{and} \hspace{0.5cm} d^2F\colon\cd^{co},\cc^{co},\cb^{co} \to \cx^{co}.$$
Then, we form $d^2G \circ_1 d^2F\colon \cd^{co},\cc^{co},\cb^{co},\ca^{co} \to \ce^{co}$ before dualising again using $d^4$ to define $$G\circ_2 F:= d^4(d^2G \circ_1 d^3F)\colon\ca,\cb,\cc,\cd \to \ce.$$

It follows from naturality of $d^2$, $d^3$, $d^4$ and naturality of $\circ_1$ that $\circ_2$ is natural in each variable.

%

Tracing through the details of these bijections one obtains the formulae for the substitutions of ternary into binary maps in \Cref{Tab:ternbin} below.

\begin{table}[ht]
\caption{Explicit substitutions of ternary into binary}
\begin{center}
\begin{tabular}{cc:c c c c c c c c}
$G \circ_1 F$ &&
$G \circ_1 F^a$ && $G \circ_1 F\midscript{b}$ && $G \circ_1 F_c$ && $G_d \circ_1 F$ \\
$G \circ_2 F$ &&
$G^a \circ_2 F$ && $G \circ_2 F^b$ && $G \circ_2 F\midscript{c}$ && $G \circ_2 F_d$ \\
\end{tabular}
\end{center}
  \label{Tab:ternbin}
\end{table}

We must show that $G \circ_2 F$ is tight if $G$ is.  This follows from the description in \Cref{Tab:ternbin}, since, as $\GCatm$ is a $3$-ary skew multicategory, each of the ternary maps $G \circ_2 F^b, G \circ_2 F\midscript{c}$ and $G \circ_2 F_d$ is tight regardless of $F$.


\subsection{Substitution of binary into ternary}

\subsection*{First variable} 
Given  $F\colon\ca,\cb \to \cx$ and $G\colon\cx,\cc,\cd \to \ce$ we define the substitution $$G \circ_1 F := \lambda_4(\lambda_3^{-1}(G) \circ_1 F)$$ where $\lambda_3^{-1}(G)\colon\cx,\cc \to \lax(\cd,\ce)$ is the adjoint map.

It follows from naturality of the bijections $\lambda_3$ and $\lambda_4$ and naturality of $\circ_1$ for two binary maps that the above substitution $G \circ_1 F$ is natural in each variable.

The multimap $G \circ_1 F$ is tight whenever $F$ and $G$ are.  Indeed, if $G$ is tight, then as $\lambda_3$ respects tightness, so is $\lambda_3^{-1}(G)$.  Therefore as $\circ_1$ for binary maps respects tightness, $\lambda_3^{-1}(G)\circ_1 F$ is tight.  Finally since $\lambda_4$ respects tightness, $G \circ_1 F$ is tight.

%

\subsection*{Second variable}
Given $F\colon\cb,\cc \to \cx$ and $G\colon\ca, \cx,\cd \to \ce$ we define the substitution $$G \circ_2 F := \lambda_4(\lambda_3^{-1}(G) \circ_2 F)$$ where $\lambda_3^{-1}(G)\colon\ca,\cx \to \lax(\cd,\ce)$ is the adjoint map.

It follows from naturality of the bijections $\lambda_3$ and $\lambda_4$ and naturality of $\circ_2$ for two binary maps that the above substitution $G \circ_1 F$ is natural.

To see that $G \circ_2 F$ is tight if $G$ is, we follow the same line of argument as above --- since $\lambda_4$ and $\lambda_3$ respect tightness, it follows from the corresponding claim concerning $\circ_2$ for binary maps.
%

\subsection*{Third variable}
Consider $F\colon\cc,\cd \to \cx$ and $G\colon\ca,\cb,\cx \to \ce$.  We firstly use the duality bijections to form $$d^2F:\cd^\co,\cc^\co \to \cx^\co \hspace{0.5cm} \textnormal{and} \hspace{0.5cm} d^3G\colon\cx^{co},\cb^{co},\ca^{co} \to \ce^{co}.$$
Then we form $d^3G \circ_1 d^2(F)\colon\cd^{co},\cc^{co},\cb^{co},\ca^{co} \to \ce^{co}$ before dualising again to obtain
 $$G \circ_3 F := {d^4}(d^3(G) \circ_1 d^2(F)).$$

It follows from naturality of $d^2, d^3$ and $d^4$ plus naturality of $\circ_1$ that $\circ_3$ is natural.


%

Tracing through the defining bijections, we obtain the descriptions for the substitutions of binary into ternary in \Cref{Tab:bintern} below.

\begin{table}[ht]
\caption{Explicit substitutions of binary into ternary}
\begin{center}
\begin{tabular}{cc:c c c c c c c c}
$G \circ_1 F$ &&
$G \circ_1 F^a$ && $G \circ_1 F_b$ && $G \circ_2 F\midscript{c}$ && $G \circ_2 F_d$ \\
$G \circ_2 F$ &&
$G^a \circ_1 F$ && $G \circ_2 F^b$ && $G \circ_2 F_c$ && $G_d \circ_2 F$ \\
$G \circ_3 F$ &&
$G^a \circ_2 F$ && $G\midscript{b} \circ_2 F$ && $G \circ_3 F^c$ && $G \circ_3 F_d$ \\
\end{tabular}
\end{center}
  \label{Tab:bintern}
\end{table}

As in the previous cases of operations defined using duality, we use the description in the above table to show that $G \circ_3 F$ is tight if $G$ is.

\subsection{A $4$-ary skew multicategory}
We have defined all of the structure for the $4$-ary skew multicategory $\GCatm$ and verified all associativity equations except those involving (1) a binary map, a ternary map and a nullary map or (2) three binary maps.  The first group of  associativity equations follow directly from the definitions of substitution of ternary and binary given in Tables~\ref{Tab:ternbin} and ~\ref{Tab:bintern}.

It remains to consider the second group of equations involving three binary maps.  In fact, given that $4$-ary maps are determined by their four ternary components, the remaining equations follow easily from what we have already done. For a representative example, consider 
$F\colon\cb,\cc \to \cx$, $G\colon\cx,\cd \to \cy$ and $H\colon\ca, \cy \to \ce$.  

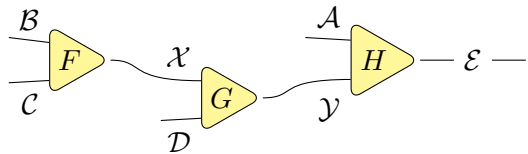
\begin{center}
\begin{tikzpicture}[triangle/.style = {fill=yellow!50, regular polygon, regular polygon sides=3,rounded corners}]
\path 
	(0,0.5) node [triangle,draw,shape border rotate=-90,inner sep=1pt,label=135:$\cb$,label=230:$\cc$] (b) {$F$} 
	(2,0) node [triangle,draw,shape border rotate=-90,inner sep=1pt,label=135:$\cx$,label=230:$\cd$] (a) {$G$}
	(4,0.5) node [triangle,draw,shape border rotate=-90,inner sep=1pt,label=135:$\ca$,label=230:$\cy$] (c) {$H$};

\draw [-] (-0.8,0.8)  to (b.139);
\draw [-] (-0.8,0.2) to (b.225);
\draw [-] (3.1,0.8) to (c.139);
\draw [-] (a) .. controls +(right:1cm) and +(left:1cm).. (c.220);
\draw [-] (b) .. controls +(right:1cm) and +(left:1cm).. (a.140);
\draw [-] (c) to node [fill=white] {$\ce$} (6,0.5);
\draw [-] (1.2,-0.3)  to (a.225);
\end{tikzpicture} 
\end{center}

We must prove that $(H \circ_2 G) \circ_2 F = H \circ_2 (G \circ_1 F)$.  To prove these $4$-ary maps are equal, we must show that their four defining ternary maps are the same.
We have 
$$((H \circ_2 G) \circ_2 F)^a= (H \circ_2 G)^a \circ_2 F = (H^a \circ_1 G) \circ_2 F$$
and similarly
$$H \circ_2 (G \circ_1 F)^a = H^a \circ_1 (G \circ_1 F)$$ 
where all equalities follow from the explicit descriptions of substitutions given in Tables \ref{Tab:binbin}, \ref{Tab:ternbin} and \ref{Tab:bintern}. The equality of the right hand sides is just naturality of $\circ_2$ in unary maps.

Similarly, the necessary equalities in $b,c$ and $d$ reduce to the associativity equations
$(H \circ_2 G) \circ_1 F^b = H \circ_2 (G \circ_1 F^b)$, $(H \circ_2 G) \circ_1 F_c = H \circ_2 (G \circ_1 F_c)$ and $(H \circ_2 G_d) \circ_2 F = H \circ_2 (G_d \circ_1 F)$
all of which are just instances of naturality in unary maps, which hold in the $3$-multicategory $\GCatm$.

Putting the results of this section together, we have proven:

\begin{theorem}
With the structure described in this section, $\GCatm$ forms a $4$-ary skew multicategory.
\end{theorem}

\section{The skew monoidal closed structure capturing lax transformations}\label{sect:skewmonlax}

In this section we will use the $4$-ary skew multicategory $\GCatm$ of Section~\ref{sect:shortmult} to construct the closed skew monoidal category $(\GCat,\otimes_l,1)$ with internal hom $\lax(\ca,\cb)$.  In order to do this, we must start by establishing further properties of $\GCatm$ --- firstly, we establish closedness and then left representability.  Combining these results, in \Cref{thm:laxskewmonoidal} we obtain the skew monoidal structure of $(\GCat,\otimes_l,1)$.

\subsection{Closedness}

To this end, recall the natural bijection $\lambda_2\colon\GCatm_1^l(\ca;\lax(\cb,\cc)) \cong \GCatm_2^l(\ca,\cb;\cc)$ of Section~\ref{sect:2closed} and Section~\ref{sect:2fun} which restricts to tight maps on either side.  Its unit is a binary map $$ev = \lambda_2(1_{\lax(\cb,\cc)}) \colon\lax(\cb,\cc),\cb \to \cc$$ whose universal property is that given a binary map $F\colon\ca,\cb \to \cc$ there exists a unique unary map $\overline{F}\colon\ca \to \lax(\cb,\cc)$ such that $ev \circ_1 \overline{F} = F$.  Note that $ev\colon\lax(\cb,\cc),\cb \to \cc$ is \emph{tight}, since it corresponds to the identity on $\lax(\cb,\cc)$ which is tight.  

\begin{theorem}
The tight binary maps $ev \colon\lax(\cb,\cc),\cb \to \cc$ exhibit $\GCatm$ as a closed $4$-ary skew multicategory.
\end{theorem}
\begin{proof}
Consider the bijections  
$$\lambda_{n+1} \colon\GCatm_n^l(\ca_1,\ldots,\ca_n;\lax(\cb,\cc)) \cong \GCatm_{n+1}^l(\ca_1,\ldots,\ca_n,\cb;\cc)$$
of \Cref{sect:shortmult}.  We have constructed these for $n=0,1,2$ and $3$ and shown that, for $n > 0$, these restrict to bijections on tight maps.  What remains to show, in each case, is that $\lambda_{n+1}(F) = ev \circ_1 F$.  

Firstly, consider $\lambda_1\colon\GCatm_0^l(\diamond;\lax(\cb,\cc)) \cong \GCatm_1^l(\cb;\cc)$ and $F\colon\diamond \to \lax(\cb,\cc)$.  Then $$\lambda_1 (F) = \lambda_1(1_{\lax(\cb,\cc)} \circ_1 F) = \lambda_2(1_{\lax(\cb,\cc)}) \circ_1 F = ev \circ_1 F$$
where the first equality holds by functoriality, the second by Lemma~\ref{lemma:lambda12} and the third by definition of $ev$.

For $\lambda_2\colon\GCatm_1^l(\ca;\lax(\cb,\cc)) \cong \GCatm_2^l(\ca,\cb;\cc)$ we have that $\lambda_2(F) = ev \circ_1 F$ by definition of $ev$ as the unit of the natural bijection.

Next, consider $\lambda_3\colon\GCatm^l_2(\ca_1,\ca_2;\lax(\cb,\cc)) \to \GCatm^l_3(\ca_1,\ca_2,\cb;\cc)$. We have
$$\lambda_3(F) = \lambda_3(1_{\lax(\cb,\cc)} \circ_1 F) = \lambda_3({\lambda_2}^{-1}(ev) \circ_1 F) = ev \circ_1 F$$ where
the first equality holds by functoriality, the second by definition of $ev_{\cc,\cd}$ and the third by definition of $\circ_1$ for binary maps.  The case of $\lambda_4$ is identical in form to that of $\lambda_3$.

\end{proof}


If we forget the skew aspects, we have a still stronger result.

\begin{theorem}\label{thm:closed}
The underlying $4$-ary multicategory $\GCatm^l$ of loose multimaps is biclosed, with right internal hom $\oplax(\cb,\cc)$.
\end{theorem}
\begin{proof}
Given the natural isomorphism $(-)^{+}_{\cb,\cc}\colon\oplax(\cb,\cc) \to \lax(\cb^{co},\cc^{co})^{co}$, the problem is equally to prove the corresponding result for the homs $\lax(\cb^{co},\cc^{co})^{co}$.  To this end, we observe that applying the duality $d^2$ to $ev \colon \lax(\cb^{co},\cc^{co}),\cb^{co} \to \cc^{co}$ give rise to a binary map 
$$e := d^2(ev) \colon \cb, \lax(\cb^{co},\cc^{co})^{co} \to \cc$$
and we will show that the induced function
$$e \circ_2 - \colon\GCatm_n^l(\ca_1,\ldots,\ca_n;\lax(\cb^{co},\cc^{co})^{co}) \cong \GCatm_{n+1}^l(\cb, \ca_1,\ldots,\ca_n;\cc)$$
is a bijection for $n=0,1,2,3$.  In fact the cases $n=2,3$ are easiest.  For $n=2$ and a binary map $F$, we have 
$$e \circ_2 F := d^3(d^2 e \circ_1 d^2 F) = d^3(ev \circ_1 d^2 F).$$  In particular, $e \circ_2 -$ is the composite of three bijections $d^2$, $ev \circ_1 -$ and $d^3$ and so a bijection too.  The case $n=3$ is identical in form since substitution of ternary maps into binary maps is defined using duality and $\circ_1$ in the same way.

For $n=1$, $e \circ_2 F := d^2(ev) \circ_2 F = d^2(ev \circ_1 F^{co})$ where the second equality holds by naturality of $d^2$.  In particular, $e \circ_2 -$ is again a composite of three bijections $d^2$, $ev \circ_1 -$ and $(-)^{co}$ and so a bijection.

For $n=0$, observe that for a nullary map $x$ (viewed as an object of its codomain) we have $e \circ_2 x := d^2(ev) \circ_2 x = d^2(ev)_x = (ev)^x = (ev) \circ_1 x$ where the second equality is by definition of $d^2$.  Hence, $e \circ_2 - $ is a bijection since $ev \circ_1 -$ is such on nullary maps.
\end{proof}

\subsection{Left representability}

Following Section~\ref{sect:closedrep}, to show that $\GCatm^l$ is left representable, we must show that it has a nullary map classifier and tight binary map classifiers.

Let us first deal with the nullary map classifier, which is simply the terminal Gray-category $\mathbf{1}$.

\begin{prop}
The $4$-ary skew multicategory $\GCatm$ has nullary map classifier $u\colon \diamond \to \mathbf{1}$, where $u$ is the nullary map selecting the single object $\bullet$ of the terminal Gray-category $\mathbf{1}$.
\end{prop}
\begin{proof}
The function $- \circ_1 u\colon\GCatm^t_1(\mathbf{1};\ca) \to \GCatm^l_0(\diamond;\ca)$ sends $F\colon \mathbf{1} \to \ca$ to its value at $\bullet$ --- this function is clearly a bijection.
\end{proof}

\begin{rmk}
In fact, $u\colon\diamond \to \mathbf{1}$ is also a nullary map classifier in the $4$-ary multicategory $\GCatm^l$ of loose multimaps, since, because of the normality conditions for a pseudomap, every pseudomap out of $\mathbf{1}$ is strict.   
\end{rmk}

The tight binary map classifiers $\ca \otimes_l \cb$ can be shown to exist in two ways.  Following Section~\ref{sect:closedrep}, the first approach is to show each $\lax(\ca,-)$ has a left adjoint, which can be done by appealing to an adjoint functor theorem.  This approach emphasises that the skew monoidal structure can be constructed without requiring any explicit description of the tensor product $\ca \otimes_l \cb$. The second approach is to construct $\ca \otimes_l \cb$ using a \emph{presentation}.  We describe both approaches here, beginning with the first.  The reader more interested in the approach via presentations may skip directly to Section~\ref{sect:presentationlax}.

In order to prove that each $\lax(\ca,-)$ has a left adjoint, it will be helpful to first consider Gohla's construction of $\lax(\ca,\cb)$ using his \emph{path space} functor.

\begin{rmk}\label{rmk:PathSpace2}

Gohla's construction of $\lax(\ca,\cb)$ in Section 5 of \cite{GohlaB:mapp} begins with the construction of a path space functor $P\colon \GCat \to \GCat$, which comes equipped with natural transformations $i\colon 1 \to P$ and $s,t\colon P \rightrightarrows 1$ satisfying $s \circ i  = t \circ i =1$.  In fact $P =\lax(\mathbf 2,-)$ and his construction of $\lax(\ca,\cb)$ builds on this.  We can look at his construction more abstractly as follows.

A \emph{path object functor} on a category $\cc$ can be defined as a functor $P\colon \cc \to \cc$ equipped with natural transformations $i\colon 1 \to P$ and $s,t\colon P \to 1$ satisfying $s \circ i = 1 =t \circ i$.  This defines a functor $P\colon\cc \to [\mathbb G_1^{op},\cc]$ to the category of reflexive graphs in $\cc$ and, if $\cc$ has finite limits, it can be extended inductively to a functor $P_{\bullet}\colon \cc \to [\mathbb G_n^{op},\cc]$, the category of reflexive $n$-graphs in $\cc$.  Here $P_{n+2}\ca$ is defined as the pullback in the square left below.
\begin{equation}\label{eq:po}
\xymatrix{
P_{n+2}\ca \ar@/^2pc/[rr]^{\langle s_{n+1},t_{n+1} \rangle} \ar[d]_{} \ar[r]^-{} & P(P_{n+1}\ca)^2  \ar[d]^{Ps_n \times Pt_n} \ar[r]^-{s \times t} & (P_{n+1}\ca)^{2} \ar[d]^{s_{n} \times t_{n}} \\
(P_{n}\ca)^2 \ar[r]_-{i \times i} & P(P_{n}\ca)^2 \ar[r]_-{s \times t} & (P_n\ca)^{2}}
\end{equation}
The maps $s_{n+1}$ and $t_{n+1}$ are obtained as the composite on the top row, as depicted, and the globularity relations follow by chasing the above diagram under post-composition with the two product projections. The reflexivity map $i_{n+1}\colon P_{n+1}\ca \to P_{n+2}\ca$ is the unique map to the pullback induced by $\langle i, i \rangle \colon  P_{n+1}\ca \to P(P_{n+1}\ca)^2$ and $\langle s_{n+1},t_{n+1} \rangle \colon  P_{n+1}\ca \to (P_{n}\ca)^2$.

Doing this, Gohla produces a functor $P_{\bullet}\colon \GCat \to [\mathbb G_3^{op},\GCat]$ and shows, in fact, that each $3$-globular object $P_{\bullet}\cb$ underlies an \emph{internal Gray-category} in $\GCat_p$ --- here, the pseudomaps are essential.  Accordingly, given a Gray-category $\ca$, the hom globular set $\GCat_p(\ca,P_{\bullet}\cb)$ underlies a Gray-category.  This is exactly $\lax(\ca,\cb)$ --- in particular, the $n$-cells are $\lax(\ca,\cb)$ correspond to pseudomaps $\ca \to P_{n}\cb$.

\end{rmk}

 \begin{prop}\label{prop:mappingspaceadjoint}
 For each Gray-category $\ca$, the functor $\lax(\ca,-)\colon\GCat \to \GCat$ has a left adjoint. 
 \end{prop}
 \begin{proof}
 Our proof will use that the category $\GCat$ is locally finitely presentable (l.f.p.) because such categories permit an easily verifiable adjoint functor theorem.  Namely a functor between l.f.p. categories has a left adjoint just when it preserves limits and is accessible (that is, preserves $\lambda$-filtered colimits for some regular cardinal $\lambda$).  This is Theorem 1.66 of \cite{Adamek1994Locally}, to which we also refer the reader as a reference on l.f.p. categories.
 
To show that $\GCat$  is l.f.p, we use Theorem 4.5 of \cite{Kelly2001VCat}, which asserts that if the monoidal biclosed category $\cv$ is l.f.p., then $\cv$-$\Cat$ is too. Starting with the cartesian closed l.f.p. category of sets this establishes that $\GCat$ is l.f.p. in two further steps, passing through the cartesian closed category of categories, followed by the Gray-tensor product of $2$-categories.   
 
Given the above, it remains to prove that each $\lax(\ca,-)\colon\GCat \to \GCat$ preserves limits and is accessible.  Since a Gray-functor is invertible just when its underlying morphism of $3$-graphs is, it follows that $\GCat$ has strong generator $\{D_i :i =0,1,2,3\}$ where $D_i$ denotes the free $i$-cell.  Since each $D_i$ is a finitely presentable (indeed, finite) Gray-category, the representable $\GCat(D_i,-)$ preserves limits and filtered colimits.  Being jointly conservative, they also reflect limits and filtered colimits.  Therefore $\lax(\ca,-)$ will preserve limits and be accessible just when each composite $\GCat(D_i,\lax(\ca,-))$ has these properties.  But following Remark~\ref{rmk:PathSpace2}, we have a natural isomorphism 
 $$\GCat(D_i,\lax(\ca,-)) \cong \GCat_p(A,P_i(-)) \cong \GCat(QA,P_i(-)) = \GCat(QA,-) \circ P_i.$$  By Examples 2.17(1) of \cite{Adamek1994Locally}, each representable preserves limits and is accessible.  Therefore it suffices to prove that each $P_i\colon\GCat \to \GCat$ has these properties --- we will prove that each preserves limits and all filtered colimits. 

Certainly, this is true for the identity functor $P_0$ and let us suppose that we can prove it for $P_1 = P$.  Then by the left pullback square in Equation~\eqref{eq:po}, we will obtain $P_2$ and $P_3$ inductively as finite limits of functors, each of which preserve limits and filtered colimits.  Since finite limits commute with limits (in any category) and filtered colimits (in any l.f.p. one) it will follow that $P_2$ and $P_3$ has these preservation properties too.  Therefore it remains to prove that $P$ preserves limits and filtered colimits.

Arguing as before, we must prove that $P(-)_i \colon\GCat \to \Set$ has these properties for $i = 0,1,2,3$.  In fact each of these functors is represented by a finite, and so finitely presentable, Gray-category.  Now the elements of $P(\ca)_0$ are the $1$-cells in $\ca$ --- in particular, we have a natural isomorphism $P(\ca)_0 \cong \GCat(D_1,\ca)$.  The elements of $P(\ca)_1$ are the lax squares, or $1$-cylinders, in $\ca$.  Hence we have a natural isomorphism $P(\ca)_1 \cong \GCat(C_1,\ca)$, where $C_1$ is the free lax square drawn below.

\begin{equation*}
\xy
(0,0)*+{0}="00";(15,0)*+{1}="10";
 (0,-15)*+{2}="01";(15,-15)*+{3}="11";
{\ar^{} "00"; "10"}; 
{\ar_{} "00"; "01"}; 
{\ar_{} "01"; "11"}; 
{\ar^{} "10"; "11"}; 
{\ar@{=>}^{}(7,-5)*{};(7,-10)*{}}; 
\endxy
\end{equation*}

The elements of $P(\ca)_2$ and $P(\ca)_3$ are $2$- and $3$-cylinders, as described in detail in \cite[Definition~3.1]{GohlaB:mapp}, and are represented by the free $2$/$3$-cylinders.  Whilst, we omit writing them down here, we note that they are both finite $3$-categories (no horizontally composable $2$-cells arise) whose descriptions can easily be extracted from \cite[Definition~3.1]{GohlaB:mapp}.

\end{proof}

In particular, putting the above results together, we have:
\begin{theorem}\label{thm:summarylaxcase}
The $4$-ary skew multicategory $\GCatm$ is left representable and closed, with internal hom $\lax(\ca,\cb)$.
\end{theorem}

\subsection{A presentation of the Gray-category $\ca \otimes_l \cb$}\label{sect:presentationlax}
Let us begin by defining \emph{presentations of Gray-categories.}  Informally, firstly, a presentation of a Gray-category $\ca$ consists of 
\begin{enumerate}
\item a set $X_0$ of $0$-cells;
\item a set $X_1$ of $1$-cells with specified source and target in $X_0$;
\item a set $X_2$ of $2$-cells with specified source and target given by formal composites of $1$-cells in $X_1$ (that is, parallel $1$-cells in the Gray-category $\ca_1$ freely generated by the data $X_0$ and $X_1$);
\item a set $X_3$ of $3$-cells with specified source and target given by formal composites of $2$-cells in $X_2$ (that is, parallel $2$-cells in the Gray-category $\ca_2$ freely generated by the data $X_0$, $X_1$ and $X_2$);
\item a set $E_i$ of pairs $(s,t)$ of $i$-cells in $\ca_3$ for $i \in\{0,1,2,3\}$ ,which are forced to become equal in the presented Gray-category $\ca = \ca_4$.
\end{enumerate}

The \emph{freely generated} structures in the above informal description will be defined as certain colimits of $\GCat$.  Note that such colimits exist since, as explained in the proof of Proposition~\ref{prop:mappingspaceadjoint}, $\GCat$ is locally presentable and therefore cocomplete.

Now for $i=0,1,2,3$, let $D_i \in \GCat$ denote the free $i$-cell and $k_i \colon B_{i-1} \to D_{i}$ denote the inclusion of the free parallel pair of $(i-1)$-cells. (Here we set $B_{-1} = \varnothing$.)   We also write $\bigtriangledown_i\colon  D_i + D_i \to D_i$ for the codiagonal.


Formally, a presentation of $\ca$ consists of, for each $i=0,1,2,3$,  a pushout square in $\GCat$ 
 as on the left below
\begin{equation}\label{eq:presentations}
\xymatrix{
X_i . B_{i-1} \ar[d]_{X_i . k_i} \ar[r]^{} & \ca_{i-1} \ar[d] \\
X_i. D_{i} \ar[r] & \ca_{i}
}
\hspace{2cm}
\xymatrix{
\sum_{i=0}^{3} E_i.(D_{i} + D_{i}) \ar[d]_{\sum_{i=0}^{3} \bigtriangledown_i} \ar[r]^{} & \ca_{3} \ar[d] \\
\sum_{i=0}^{3} E_i.D_{i} \ar[r] & \ca_{4}=A
}
\end{equation}
where $\ca_{-1} = \varnothing$, together with a final pushout square as on the right above.

The morphism $X_i . B_{i-1} \to \ca_{i-1}$ from the copower assigns to each $\alpha \in X_i$ a parallel pair of $(i-1)$-cells in $\ca_{i-1}$; the pushout $\ca_i$ is then obtained by freely adjoining an $i$-cell for each such $\alpha$ with the parallel pair comprising its source and target.  Being a Gray-category, this pushout also contains all the formal composites of such generating $i$-cells that must exist in any Gray-category.  For instance $\ca_0$ is the discrete Gray-category with objects $X_0$ whilst the morphism $X_1 . B_0 = X_1 + X_1  \to X_0$ specifies a graph $X_1 \rightrightarrows X_0$ on which $\ca_1$ is the free category, viewed as a Gray-category with only identity $2$-cells and $3$-cells.  
The morphism $\sum_{i=0}^{3} E_i.(D_{i} + D_{i}) \to \ca_3$ assigns to each element of $E_i$ a pair $(s,t)$ of $i$-cells in $\ca_3$ which are forced to become equal in the final pushout $\ca_4=\ca$.

Now let $\ca$ and $\cb$ be Gray-categories.  In the presentation below, we use our standard notation (see start of \Cref{sect:shortmult}) such as $a,A,\alpha$ and $b,B,\beta$  for cells in $\ca$ and $\cb$ without further comment.  The tight binary map classifier $\ca \otimes_l \cb$ is the Gray-category presented by the following data and equations:

\begin{itemize}
\item[\underline{$0$-cells}:]
	\begin{itemize}
	\item[$\bullet$] at $0$-cells $a$ and $b$, a $0$-cell $T(a,b)$, also denoted by $T^a(b)$ or $T_b(a)$ where convenient;
	\end{itemize}\vspace{0.1cm}
\item[\underline{$1$-cells}:]
	\begin{itemize}
	\item[$\bullet$] at $A\colon a \to a'$ and $b$, a $1$-cell $T_b(A)\colon T_b(a) \to T_b(a')$; 
	\item[$\bullet$] at $a$ and $B\colon b \to b'$, a $1$-cell $T^a(B)\colon T^a(b) \to T^{a}(b')$; 
	\end{itemize}\vspace{0.1cm}
\item[\underline{$2$-cells}:]
	\begin{itemize}
	\item[$\bullet$] at $\alpha\colon A \to A'$ and $b$, a $2$-cell $T_b(\alpha)\colon T_{b}(A) \to T_{b}(A')$;
	\item[$\bullet$] at $a$ and $\beta\colon B \to B'$, a $2$-cell $T^a(\beta)\colon T^a(B)\to T^a(B')$ ;
	\item[$\bullet$] at $a$ and $(B_1\colon b \to b',B_2\colon b' \to b'')$, a $2$-cell $T^a_{B_2,B_1}\colon T^a(B_2)T^a(B_1) \to T^a(B_2B_1)$ (and a further $2$-cell $(T^a_{B_2,B_1})^{-1}\colon  T^a(B_2B_1) \to T^a(B_2)T^a(B_1)$ which will later be forced to be its inverse);
	\item[$\bullet$] at $(A_1\colon a \to a',A_2\colon a' \to a'')$ and $b$, a $2$-cell $T_{b}^{A_2,A_1}\colon T_b(A_2)T_{b}(A_1) \to T_{b}(A_2A_1)$ (which will later be forced to be an identity);
	\item[$\bullet$] at $A\colon a \to a'$ and $B\colon b \to b'$, a $2$-cell $A_B \colon T^{a'}(B)T_b(A) \to T_{b'}(A)T^{a}(B)$;
	\end{itemize}\vspace{0.1cm}
\item[\underline{$3$-cells}:]
	\begin{itemize}
	\item[$\bullet$] at $\Lambda\colon \alpha \to \alpha'$ and $b$, a $3$-cell $T_b(\Lambda)\colon T_b(\alpha) \to T_b(\alpha')$;
	\item[$\bullet$] at $a$ and $\Theta\colon \beta \to \beta'$, a $3$-cell $T^a(\Theta)\colon T^a(\beta) \to T^a(\beta')$;
	\item[$\bullet$] further generating $3$-cells $\alpha_B, A_{\beta}, T^{A_2,A_1}_{B}$ and $T_{B_2,B_1}^A$  whose domains and codomains are as in the definition of a loose binary map (as well as a $3$-cell $(T_{B_2,B_1}^A)^{-1}$).
	\end{itemize}\vspace{0.1cm}
\item[\underline{Equations}:] 
        \begin{itemize}
	\item[$\bullet$] The pseudomap equations of \Cref{subsec:ax-pseudofunct}, interpreted for $T^a$ and $T_b$, plus the loose binary map equations of \Cref{app:ax_bin}, interpreted for $T$.  Furthermore, we require the equations capturing $(T^a_{B_2,B_1})^{-1}$ and $(T_{B_2,B_1}^A)^{-1}$ are in fact the appropriate inverses.  To these we add the equations for a tight binary map: namely that $T_{b}^{A_2,A_1}$ and $T_{B}^{A_2,A_1}$ are identities.
	\end{itemize}
\end{itemize}

Comparing with the formal description in \eqref{eq:presentations} 
above, the universal property of the presented Gray-category $\ca \otimes_l\cb$, as an iterated pushout, is that we have an isomorphism 
$$\GCat(\ca \otimes_l \cb,\cc) \cong\GCatm_2^t(\ca,\cb;\cc)$$
 natural in $\cc$ --- indeed, $T:\ca,\cb \to \ca \otimes_l \cb$, as constructed above, is precisely the universal tight binary map.
%
%
%

\subsection{The closed skew monoidal structure}\label{sect:lskewmonoidal}

Given Theorem~\ref{thm:summarylaxcase}, the construction of Section~\ref{sect:From} immediately yields: 

\begin{theorem}\label{thm:laxskewmonoidal}
There is a closed skew monoidal structure $(\GCat, \otimes_l, \mathbf{1})$ whose internal hom is $\lax(\ca,\cb)$.
\end{theorem}

Let us describe the skew monoidal structure in a little more detail and explore some of its properties.

\begin{itemize}
\item The tensor product $\ca \otimes_l \cb$ is the tight binary map classifier, and the unit is the terminal Gray-category $\mathbf{1}$.
\item The left unit map $l_{\ca}\colon\mathbf{1} \otimes_l \ca \to \ca$ is, as described in Section~\ref{sect:From}, the counit of the adjunction with right adjoint the inclusion $j\colon\GCat \to \GCat_p$ of tight unary into loose unary.  Accordingly, this is the counit $p_{\ca}\colon Q\ca \to \ca$ of the pseudomorphism classifier comonad of Remark~\ref{rmk:comonad}.  Since the counit $p_{\ca}$ is a triequivalence, we conclude that $l_{\ca}\colon \mathbf{1} \otimes_l \ca \to \ca$ is a triequivalence for all $\ca$.
\item By construction, the right unit map $r_{\ca}$ is given by $\theta_{\ca,\mathbf{1}} \circ_2 u \colon\ca \to \ca \otimes_l \mathbf{1}$.  We will show that $r_{\ca}$ is invertible, so that the skew monoidal structure is right normal.  Indeed, by Yoneda, this is equivalent to show that the composite 
\begin{equation*}
\xymatrix{
\GCatm_1^t(\ca \otimes_l \mathbf{1};\cb) \ar[r]^{- \circ \theta_{\ca,\mathbf{1}}} & \GCatm_2^t(\ca,\mathbf{1};\cb) \ar[r]^{- \circ_2 u} & \GCatm_1^t(\ca;\cb)
}
\end{equation*}
is invertible.  Since its first component is so, we must show that $-\circ_2 u$
 is invertible.  But this is clearly the case, since a tight binary map $\ca,\mathbf{1} \to \cb$ just amounts to a strict map $\ca \to \cb$.
\item 
Let us show that the associator is not, in general, invertible.  One of the axioms for a skew monoidal category gives the commutative diagram below.
\begin{equation*}
\xymatrix{
(\mathbf{1} \otimes_l \mathbf{1}) \otimes_l \ca \ar[rr]^{\alpha_{\mathbf{1},\mathbf{1},\ca}} && \mathbf{1} \otimes_l (\mathbf{1} \otimes_l \ca) \ar[d]^{\mathbf{1} \otimes_l l_{\ca}} \\
\mathbf{1} \otimes_l \ca \ar[u]^{r_{\mathbf{1}} \otimes_l \ca} \ar[rr]_{1} && \mathbf{1} \otimes_l \ca
}
\end{equation*}
Since the left vertical is invertible the $\alpha_{\mathbf{1},\mathbf{1},\ca}$ is invertible just when the right vertical is.  This is $Qp_{\ca}\colon Q^2\ca \to Q\ca$. So, as long as $\ca$ has a composable pair of non-identity maps $f\colon x \to y$ and $g\colon y \to z$, $Qp$ will identify the parallel pair $[[gf]],\,[[g],[f]]\colon x \rightrightarrows z \in Q^2\ca$, and hence fail to be invertible.  In particular, the associator is not in general invertible.  

\end{itemize}

\section{The skew monoidal closed structure capturing pseudo-transformations}\label{sect:pseudo}

In this section, we introduce the notion of \emph{pseudo-transformation} and the Gray-category of pseudo-transformations $\psd(\ca,\cb)$.  We then describe a closed $4$-ary skew multicategory $\GCatmp$ of \emph{pseudo-multimaps} with internal hom $\psd(\ca,\cb)$.  Using this, we construct in Theorem~\ref{thm:pseudoskewmonoidal} a closed skew monoidal structure $(\GCat,\otimes_p,\mathbf{1})$ with internal hom $\psd(\ca,\cb)$.

Both pseudo-transformations and pseudo-binary maps enhance their lax counterparts by equipping them with certain adjoint equivalences.  As such, let us begin by fixing a little notation concerning adjoint equivalences in $2$-categories.

\subsection{Adjoint equivalences in $2$-categories}

Let $\ck$ be a $2$-category.  The category $\adje(\ck)$ of adjoint equivalences in $\ck$ has the same objects as $\ck$ whilst a morphism $(\epsilon, f \dashv g, \eta)\colon x \to y$ is an adjoint equivalence $(\epsilon, f \dashv g, \eta)$ with $g\colon x \to y$ the right adjoint.  
\begin{itemize}
\item Given adjoint equivalences $(\epsilon,f\dashv g,\eta)\colon x \to y$ and $(\epsilon',f'\dashv g',\eta')\colon y \to z$ the composite adjoint equivalence is given by $(\epsilon\cdot f\epsilon' g,\,ff'\dashv g'g,\,g'\eta f'\cdot\eta')\colon x\to z$.  
\item For $x \in \ck$, the identity adjoint equivalence is $1_x \dashv 1_x\colon x \to x$.
\end{itemize} 
A 2-functor $F\colon \ck \to \cl$ induces a functor $\adje(F) \colon \adje(\ck) \to \adje(\cl)$ sending $f \dashv g\colon x \to y$ to $Ff \dashv Fg\colon Fx \to Fy$.


\subsection{The Gray-category of pseudo-transformations}

Let us now define the Gray-category $\psd(\ca,\cb)$ of pseudo-transformations between two Gray-categories $\ca$ and $\cb$.  Its cells are defined as follows.
\begin{itemize}
\item \emph{0-cells:} pseudomaps $F\colon\ca\to\cb$. 
\item \emph{1-cells:} pseudo-transformations --- these are lax transformations $\alpha\colon F\to G$ equipped with, for any 1-cell $f\colon x\to y\in\ca$, an adjoint equivalence $\alpha^o_f \dashv \alpha_f\colon Gf\alpha_x \to \alpha_yFf$ such that:
	\begin{enumerate}[(a)]
	\item at an object $x\in\ca$, $\alpha^o_{1_x} \dashv \alpha_{1_x}$ is the identity adjoint equivalence;
	\item at a $2$-cell $\phi\colon f\to f'\in\ca$, the 3-cell $\alpha_\phi$ is invertible.
	\end{enumerate}
\item \emph{2-cells:} pseudo-modifications --- these are modifications $\Lambda\colon \alpha\to \beta$ such that, for any 1-cell $f\colon x\to y\in\ca$, the 3-cell $\Lambda_f$ is invertible. 
\item \emph{3-cells:} perturbations $\theta\colon\Lambda\to\Gamma$. 
\end{itemize}

\begin{rmk}
The notation $\alpha^o_f$ for the left adjoint of the adjoint equivalence is intended to suggest that $\alpha^o_f$ is the component of an \emph{oplax transformation}.  This hints at a more symmetric presentation of the notion of pseudo-transformation, which we will explore further in a followup paper. 
\end{rmk}

Our aim is now to equip $\psd(\ca,\cb)$ with the structure of a Gray-category such that the forgetful $3$-graph map $U\colon \psd(\ca,\cb) \to \lax(\ca,\cb)$ becomes a strict Gray-functor.  

The identity $1\colon F \to F$ pseudo-transformation is given by the identity in $\lax(\ca,\cb)$ equipped with the identity adjoint equivalences in each component.  Given a composable pair $\alpha\colon F\to G$ and $\beta\colon G\to H$ of pseudo-transformations, their composite $\beta \alpha\colon F \to H$ in $\psd(\ca,\cb)$ is given by their composite in $\lax(\ca,\cb)$ equipped with
\begin{itemize}
\item at $f\colon x \to y \in \ca$, the composite adjoint equivalence 
\[\begin{tikzcd}[ampersand replacement=\&]
	{Hf\beta_x\alpha_x} \&\& {\beta_y Gf \alpha_x} \&\& {\beta_y \alpha_y Ff.}
	\arrow[""{name=0, anchor=center, inner sep=0}, "{\beta_{f} \circ 1}", curve={height=-12pt}, from=1-1, to=1-3]
	\arrow[""{name=1, anchor=center, inner sep=0}, "{\beta^{o}_{f} \circ 1}", curve={height=-12pt}, from=1-3, to=1-1]
	\arrow["\top"{description}, draw=none, from=1, to=0]
	\arrow[""{name=2, anchor=center, inner sep=0}, "{1 \circ \alpha_f}", curve={height=-12pt}, from=1-3, to=1-5]
	\arrow[""{name=3, anchor=center, inner sep=0}, "{1 \circ \alpha^o_f}", curve={height=-12pt}, from=1-5, to=1-3]
	\arrow["\top"{description}, draw=none, from=3, to=2]
\end{tikzcd}\]
\end{itemize}
To see that $\beta \alpha$ is in fact a pseudo-transformation, firstly suppose that $f$ is an identity $1$-cell.  As $\alpha$ and $\beta$ are pseudo-transformations, $\beta^o_f \dashv \beta_f$ and $\alpha^o_f \dashv \alpha_f$ are identity adjoint equivalences.  Applying the whiskering $2$-functors $-\circ\alpha_x$ and $\beta_y\circ -$ produces identity adjoint equivalences, the two components depicted above, whose composite is therefore the identity adjoint equivalence too, as required.  Furthermore, the $3$-cell $(\beta \alpha)_{\phi}$ is invertible since, as described explicitly in Equation \eqref{comp-lax-2-cell-part} of \Cref{app:lax-gray-cat-str}, it is a composite of whiskers of the invertible 3-cells $\beta_\phi$ and $\alpha_\phi$.

Associativity of $1$-cell composition in $\psd(\ca,\cb)$ follows from associativity in $\lax(\ca,\cb)$, associativity of composition of adjoint equivalences in a $2$-category plus $2$-functoriality of whiskering by $1$-cells.  The unital laws easily follow, hence we get a category structure on $\psd(\ca,\cb)$ making $U$ a functor. With regards $2$-cells, we must show that pseudo-modifications are closed in $\lax(\ca,\cb)$ under vertical composition, whiskering by pseudomaps and identities.  In each case, we simply need to show that the $3$-cell components of the relevant modifications in $\lax(\ca,\cb)$ are invertible --- this follows directly from the descriptions of these operations on modifications in \Cref{app:lax-gray-cat-str}.  Since $3$-cells in $\psd(\ca,\cb)$ are just perturbations, we define their composites and the interchange maps just as in $\lax(\ca,\cb)$.  This establishes that $\psd(\ca,\cb)$ is a Gray-category and $U\colon \psd(\ca,\cb) \to \lax(\ca,\cb)$ a strict Gray-functor.

Next, let $F\colon\ca\to\ca'$ and $G\colon\cb\to\cb'$ be pseudo-functors between Gray-categories. Lifting the action of $\lax(-,-)$ and using that $F$ and $G$ are locally 2-functors (and hence preserve local adjoint equivalences), we obtain:

\begin{itemize}
 \item \label{item:strict1} A strict map $\psd(F,\cb)\colon\psd(\ca',\cb)\to\psd(\ca,\cb)$.
 \item \label{item:strict2} A pseudomap $\psd(\ca,G)\colon\psd(\ca,\cb)\to\psd(\ca,\cb')$ which is strict if $G$ is.
\end{itemize}

Both $\psd(F,\cb)$ and $\psd(\ca,G)$ act on $0$-cells, that is pseudomaps, by composition.  Consider a pseudo-transformation $\beta\colon G \to G' \in \psd(\ca',\cb)$. Then $\psd(F,\cb)(\beta)$ is the lax transformation $\beta F$ equipped with, at $f\colon x\to y \in\ca$, the adjoint equivalence $\beta^o_{Ff} \dashv \beta_{Ff}$.  Given $\alpha\colon F\to F' \in \psd(\ca,\cb)$, the pseudo-transformation $\psd(\ca,G)(\alpha)$ is the lax transformation $G\alpha$ equipped with, for $f\colon x\to y \in\ca$, the composite adjoint equivalence 
\[\begin{tikzcd}[ampersand replacement=\&]
	{GF'fG\alpha_x} \&\& {G(F'f\alpha_x)} \&\& {G(\alpha_yFf)} \&\& {G\alpha_yGFf}
	\arrow[""{name=0, anchor=center, inner sep=0}, "{G^2}", curve={height=-12pt}, from=1-1, to=1-3]
	\arrow[""{name=1, anchor=center, inner sep=0}, "{(G^2)^{-1}}", curve={height=-12pt}, from=1-3, to=1-1]
	\arrow[""{name=2, anchor=center, inner sep=0}, "{G(\alpha_f)}", curve={height=-12pt}, from=1-3, to=1-5]
	\arrow[""{name=3, anchor=center, inner sep=0}, "{G({\alpha^o_f})}", curve={height=-12pt}, from=1-5, to=1-3]
	\arrow[""{name=4, anchor=center, inner sep=0}, "{{G^2}^{-1}}", curve={height=-12pt}, from=1-5, to=1-7]
	\arrow[""{name=5, anchor=center, inner sep=0}, "{G^2}", curve={height=-12pt}, from=1-7, to=1-5]
	\arrow["\top"{description}, draw=none, from=1, to=0]
	\arrow["\top"{description}, draw=none, from=3, to=2]
	\arrow["\top"{description}, draw=none, from=5, to=4]
\end{tikzcd}\] 
in which the isomorphisms and their inverses, on the left and the right, are viewed as adjoint equivalences with identity counit and unit.  

The higher-dimensional structure of $\psd(F,\cb)$ and $\psd(\ca,G)$ is exactly as for $\lax(F,\cb)$ and $\lax(\ca,G)$.  With this description, we obtain a Gray-functor $\psd(-,-)\colon \GCat_p^{op} \times \GCat_p \to \GCat_p$ and natural transformation $U\colon \psd(-,-) \to \lax(-,-)$.  We leave to the reader the routine verifications of naturality, which, again, amount to checking the equality of a few adjoint equivalences.  

\subsection{The $4$-ary skew multicategory of Gray-categories and pseudo-multimaps}\label{sect:corrected}

We will construct the set of $n$-ary pseudo-multimaps inductively, together with a forgetful function
$\mathbb{U}_n\colon \GCatmp^{l}_{n}(\ca_1,\ldots,\ca_n,\cb) \to \GCatm^{l}_{n}(\ca_1,\ldots,\ca_n,\cb)$, and will sometimes refer to the elements of $\GCatm^{l}_{n}(\ca_1,\ldots,\ca_n,\cb)$ as \emph{lax} multimaps.

Nullary pseudo-multimaps are just objects as before; that is, we have 
$\GCatmp^l_0(\diamond;\ca):=\GCatm^l_0(\diamond;\ca)$
 with $\mathbb{U}_0$ the identity function.  For higher $n$, they are defined via the pullback
\begin{equation}
\label{eq:psd-pull-closedness0}
\xymatrix{
\GCatmp^{l}_{n+1}(\ca_1,\ldots,\ca_n,\cb;\cc)   \ar[d]_{\mathbb{U}_{n+1}} \ar[rr]^-{(\lambda_{n+1}^{p})^{-1}} && \GCatmp^{l}_n(\,\ca_1,\ldots,\ca_n;\psd(\cb,\cc)\,) \ar[d]^-{(U_{\cb,\cc} \circ_1 -) \circ \mathbb{U}_{n}} \\
\GCatm^{l}_{n+1}(\ca_1,\ldots,\ca_n,\cb;\cc)  \ar[rr]_-{(\lambda_{n+1})^{-1}} && \GCatm^l_n(\,\ca_1,\ldots,\ca_n;\lax(\cb,\cc)\,)
}
\end{equation}
where the lower horizontal is the inverse of the closedness bijection for lax multimaps constructed in \Cref{sect:shortmult}.  

Unpacking the above yields the following descriptions:

\begin{itemize}
\item A pseudo unary map $\mathbf F\colon\ca \to \cb$ is just a pseudomap $F\colon\ca \to \cb$.
\item A pseudo binary map $\mathbf F\colon\ca,\cb\to\cc$ consists of a lax binary map $F\colon\ca,\cb\to\cc$ in $\GCatm^l$ equipped with adjoint equivalences 

\[\begin{tikzcd}[ampersand replacement=\&]
	{F^{a'}_{B}F^{A}_{b}} \&\& {F^{A}_{b'}F^{a}_{B}}
	\arrow[""{name=0, anchor=center, inner sep=0}, "{A_B}", curve={height=-12pt}, from=1-1, to=1-3]
	\arrow[""{name=1, anchor=center, inner sep=0}, "{{A^o_B}}", curve={height=-12pt}, from=1-3, to=1-1]
	\arrow["\top"{description}, draw=none, from=1, to=0]
\end{tikzcd}\]
in $\cc(F^a_b,F^{a'}_{b'})$ such that:
\begin{enumerate}[(i)]
\item given objects $a\in\ca$ and $b\in\cb$, the adjoint equivalences ${A^o_{1_b}}\dashv A_{1_b}$ and ${{(1_a)}^o_B}\dashv{(1_a)}_B$ are identities;
\item given 1-cells $A\in\ca$ and $B\in\cb$ and 2-cells $\alpha\in\ca$ and $\beta\in\cb$, the $3$-cells $\alpha_B$ and $A_\beta$ in $\cc$ are invertible. 
\end{enumerate}
\item A pseudo ternary map $\mathbf F\colon\ca,\cb,\cc\to\cd$ consists of a lax ternary map $\mathbf F\colon\ca,\cb,\cc\to\cd$ with invertible incubators $(A\mid B\mid C)$ together with pseudo-binary maps $\mathbf{F}^a$, $\mathbf{F}\midscript{b}$ and $\mathbf{F}_{c}$ whose underlying lax binary maps are ${F}^a$, ${F}\midscript{b}$ and ${F}_{c}$
\item A pseudo 4-ary map $\mathbf{F}\colon\ca,\cb,\cc,\cd\to\ce$ consists of a lax 4-ary map $F\colon\ca,\cb,\cc,\cd\to\ce$  together with pseudo-ternary maps $\mathbf{F}^a$, $\mathbf{F}\midscript{b}$, $\mathbf{F}\midscript{c}$ and $\mathbf{F}_d$ satisfying the  six compatibility equations of Section~\ref{sect:4ary} (but with $\mathbf{F}$ substituted for $F$) and whose underlying lax ternary maps are $F^a$, $F\midscript{b}$, $F\midscript{c}$ and $F_d$.
\end{itemize}

When $n=0$, the sets $\GCatmp^l_0(\diamond;\ca)$ form a functor $\GCatmp^l_0(\diamond;-)\colon \GCat_p^{op} \to \Set$ for which the (identity) components $\mathbb{U}_0\colon  \GCatmp^l_0(\diamond;\ca) \to \GCatm^l_0(\diamond;\ca)$ are natural in $\ca$.  Using the established naturality of $U\colon \psd(-,-) \to \lax(-,-)$, by induction the sets $\GCatmp^{l}_{n}(\ca_1,\ldots,\ca_n,\cb)$ obtain the unique structure of a functor $\GCatmp^l_{n}(-;-)\colon (\GCat_p^n)^{op} \times \GCat_p \to \Set$ for which the pullback projections in \eqref{eq:psd-pull-closedness0} are natural in each variable.  Replacing the horizontal bijections in the pullback square by their inverses, this is equally to say that the pullback projections in
\begin{equation}
\label{eq:psd-pull-closedness}
\xymatrix{
\GCatmp^{l}_n(\,\ca_1,\ldots,\ca_n;\psd(\cb,\cc)\,) \ar[d]_{(U_{\cb,\cc} \circ_1 -) \circ \mathbb{U}_{n}} \ar[r]^-{\lambda_{n+1}^{p}} & \GCatmp^{l}_{n+1}(\ca_1,\ldots,\ca_n,\cb;\cc) \ar[d]^-{\mathbb{U}_{n+1}} \\
\GCatm^l_n(\,\ca_1,\ldots,\ca_n;\lax(\cb,\cc)\,) \ar[r]_
-{\lambda_{n+1}} & \GCatm^{l}_{n+1}(\ca_1,\ldots,\ca_n,\cb;\cc)
}
\end{equation}
are natural.

Now that we have the $n$-ary multimap functors $\GCatmp^l_{n}(-;-)$ for all $n$, it remains to describe the substitutions for the $4$-ary multicategory.
Let us split them into two groups.  All are straightforward enhancements of the lax case --- none involve checking any equations between adjoint equivalences, but only verifying that some $3$-cells are invertible. 
\begin{itemize}
\item Firstly, we consider substitutions involving nullary maps.  Given $\mathbf{F}\colon \ca,\cb \to \cc$, we define $\mathbf{F} \circ_1 a = F^a$ and $\mathbf{F} \circ_2 b = F_b$, noting that unary maps are the same in the lax and pseudo-cases.  Given $\mathbf{F}\colon \ca,\cb,\cc \to \cd$, we define $\mathbf{F} \circ_1 a = \mathbf{F}^a$, $\mathbf{F} \circ_2 b = \mathbf{F} \midscript{b}$ and $\mathbf{F} \circ_3 c = \mathbf{F}_{c}$.  Similarly, given $\mathbf{F}\colon \ca,\cb,\cc ,\cd\to \ce$, we define $\mathbf{F} \circ_1 a = \mathbf{F}^a$, $\mathbf{F} \circ_2 b = \mathbf{F} \midscript{b}$, $\mathbf{F} \circ_3 c = \mathbf{F}\midscript{c}$ and $\mathbf{F} \circ_4 d = \mathbf{F}_{d}$. Naturality is routine.  The associativity equations involving two nullary and a pseudo-binary, pseudo-ternary or pseudo-4-ary map hold as a special case of the corresponding equations for $\GCatm$.
\item In the cases of the remaining substitutions, we note that those involving unary maps are already defined --- via the $n$-ary multimaps functors $\GCatmp^l_{n}(-;-)$.  It remains to consider the higher substitutions --- namely, binary into binary, ternary into binary and binary into ternary.  These are defined just as in the tables --- Table~\ref{Tab:binbin}, \ref{Tab:ternbin} and \ref{Tab:bintern} --- from the case of $\GCatm$.  For instance, we define substitutions of binary into binary via the following table, which modifies \Cref{Tab:binbin} only by replacing the lax maps $F$ and $G$ with their enhanced pseudomaps $\mathbf{F}$ and $\mathbf{G}$.  

\begin{table}[ht]
\begin{center}
\begin{tabular}{cc:c c c c c c c c}
$\mathbf{G} \circ_1 \mathbf{F}$ &&
$\mathbf{G} \circ_1 \mathbf{F}^a$ && $\mathbf{G} \circ_1 \mathbf{F}_b$ && $\mathbf{G}_c \circ_1 \mathbf{F}$ && $\textnormal{Incubator in Appendix}~\ref{app:inc-sub-b-in-b-1st}$ \\
$\mathbf{G} \circ_2 \mathbf{F}$ &&
$\mathbf{G}^a \circ_1 \mathbf{F}$ && $\mathbf{G} \circ_2 \mathbf{F}^b$ && $\mathbf{G} \circ_2 \mathbf{F}_c$ && $\textnormal{Incubator in Appendix}~\ref{app:inc-sub-b-in-b-2nd}$ \\
\end{tabular}
\end{center}
\end{table}

Note that the substitutions of pseudo-binary and unary in the table have been defined above, so that for these to give pseudo-ternary maps, it suffices to check that the incubators in the cited appendix are invertible, which, as the composite of invertible $3$-cells, is indeed the case.
\end{itemize}

The remaining associativity equations then follow from the descriptions in these tables, just as for the case of $\GCatm^l$ in the preceding section.  In particular, we obtain a $4$-ary multicategory $\GCatmp^l$ and $4$-ary multifunctor $\mathbb{U}^l\colon\GCatmp^l \to \GCatm^l$.  


Turning to tight pseudo-multimaps, a tight pseudo-unary map $\mathbf F \colon \ca \to \cb$ is defined to be a Gray-functor.  In order that the resulting skew multicategory be closed, for higher $n$, we are then forced to inductively define $(n+1)$-ary tight pseudo-multimaps via the pullback.

\begin{equation*}
\label{eq:psd-pull-closednesstight}
\begin{tikzcd}
	{\GCatmp^{t}_{n+1}(\ca_1,\ldots,\ca_n,\cb;\cc)} && {\GCatmp^{t}_n(\,\ca_1,\ldots,\ca_n;\psd(\cb,\cc)\,)} \\
	\\
	{\GCatmp^{l}_{n+1}(\ca_1,\ldots,\ca_n,\cb;\cc)} && {\GCatmp^l_n(\,\ca_1,\ldots,\ca_n;\psd(\cb,\cc)\,)}
	\arrow["{(\lambda_{n+1}^{p})^{-1}}", from=1-1, to=1-3]
	\arrow[hook, from=1-1, to=3-1]
	\arrow[hook, from=1-3, to=3-3]
	\arrow["{(\lambda_{n+1}^{p})^{-1}}", from=3-1, to=3-3]
\end{tikzcd}
\end{equation*}

Unpacking the above yields the following.
\begin{itemize}
\item A pseudo binary map $\mathbf F\colon\ca,\cb\to\cc$ is tight just when its underlying lax binary map is tight and, moreover, we have an equality of adjoint equivalences
\begin{equation}\label{eq:adjointequivalence}
\begin{tikzcd}
	{(A_2A_1)^o_B  \dashv (A_2A_1)_B =} & {F^{a''}_{B}F^{A_2}_{b}F^{A_1}_{b}} & {F_{b'}^{A_2}F^{a'}_{B}F^{A_1}_{b}} & {F_{b'}^{A_2}F_{b'}^{A_1}F^{a}_{B}}
	\arrow[""{name=0, anchor=center, inner sep=0}, "{(A_2)_B \circ 1}", curve={height=-12pt}, from=1-2, to=1-3]
	\arrow[""{name=1, anchor=center, inner sep=0}, "{(A_2)^o_B \circ 1}", curve={height=-12pt}, from=1-3, to=1-2]
	\arrow[""{name=2, anchor=center, inner sep=0}, "{1 \circ (A_1)_B }", curve={height=-12pt}, from=1-3, to=1-4]
	\arrow[""{name=3, anchor=center, inner sep=0}, "{1 \circ (A_1)^o_B }", curve={height=-12pt}, from=1-4, to=1-3]
	\arrow["\top"{description}, draw=none, from=1, to=0]
	\arrow["\top"{description}, draw=none, from=2, to=3]
\end{tikzcd}
\end{equation}

To see this, first observe that the associated pseudomap $\overline{\mathbf F} := (\lambda^p_2)^{-1} (\mathbf F) \colon \ca \to \psd(\cb,\cc)$ is strict just when the modifications $$\overline{\mathbf F}^2_{A_2,A_1} \colon \overline{\mathbf F}(A_2)\overline{\mathbf F}(A_1) \to \overline{\mathbf F}(A_2A_1)$$ are identities.  Now to say that the components $F^{A_2,A_1}_b$ and $F^{A_2,A_1}_B$ of the modification are identities --- i.e. that $F$ is tight as a lax bimap ---is equally to say that the modification is an identity between the underlying \emph{lax transformations}.  However, this may be the case without it being an identity between \emph{pseudo-transformations}.  For this, we further need the equality $ \overline{\mathbf F}(A_2)\overline{\mathbf F}(A_1) = \overline{\mathbf F}(A_2A_1)$ of pseudo-transformations, which requires that both sides have the same adjoint equivalence components (not only the same right adjoints).   This is the condition that \eqref{eq:adjointequivalence} imposes (noting that composition of $1$-cells in $\psd(\cb,\cc)$ involves composition of adjoint equivalences).

\item A pseudo ternary map $\mathbf F\colon\ca,\cb,\cc\to\cd$ is tight just when the pseudo-binary maps $\mathbf{F}\midscript{b}$ and $\mathbf{F}_{c}$ are tight.
\item A pseudo 4-ary map $\mathbf{F}\colon\ca,\cb,\cc,\cd\to\ce$ is tight just when its underlying pseudo-ternary maps $\mathbf{F}\midscript{b}$, $\mathbf{F}\midscript{c}$ and $\mathbf{F}_d$ are tight.
\end{itemize}
In order to show that $\GCatmp$ is a $4$-ary \emph{skew} multicategory, we must verify that $\mathbf{F} \circ_i \mathbf{G}$ is tight if $\mathbf{F}$ and $\mathbf{G}$ are tight or $\mathbf{F}$ is tight and $i>1$.

The case of substitution where $\mathbf{G}$ is nullary is trivial since a tight pseudo-multimap is, by definition, tight in each variable except the first one.

We turn to the substitution of unary $\mathbf{G}$ --- i.e. a pseudomap $G$.  Substitution of tight unary into tight unary is just composition in $\GCat_s$.  Thus let $\mathbf{F} \colon \ca,\cb \to \cc$ be tight binary and consider $\mathbf{G} \colon \ca' \to \ca$ tight.  Then, by definition, $\mathbf{F} \circ_1 \mathbf{G} = \lambda^p_2 ((\lambda_2^p)^{-1}\mathbf{F} \circ_1 \mathbf{G})$ for  $(\lambda_2^p)^{-1}\mathbf{F} \circ_1 \mathbf{G} \colon \ca' \to \ca \to \psd(\cb,\cc)$.  Now $(\lambda_2^p)^{-1}\mathbf{F}$ is tight unary, by definition of tight binary, so that $(\lambda_2^p)^{-1}\mathbf{F} \circ_1 G$ is tight as a composite of tight unary maps.  Then since the closedness bijections respect tightness, by definition of tight pseudo-multimaps, the transpose $\mathbf{F} \circ_1 \mathbf{G}$ is tight.  Substitution of unary in the other two variables combines composition of tight unary as before with the fact that the pseudomap $\psd(G,\cc)$ is always strict and $\psd(\cb,G)$ is strict whenever $G$ is.

For substitutions producing multimaps of arity $3$ or $4$, this follows from the binary case using associativity in $\GCatmp^{l}$ since the tightness of a $3/4$-ary pseudo-multimap is determined by that of its component pseudo-binary maps.  

\begin{theorem}\label{thm:summarypseudocase}
The $4$-ary skew multicategory $\GCatmp$ is left representable and closed, with internal hom $\psd(\cb,\cc)$.
\end{theorem}
\begin{proof}
We will first establish closedness.  To this end, observe that evaluating the natural bijection $\lambda_2^p$ in Equation~\ref{eq:psd-pull-closedness} at $1 \colon \psd(\cb,\cc) \to \psd(\cb,\cc)$ yields a pseudo binary map $\mathbf{ev}\colon \psd(\cb,\cc),\cb \to \cc$ such that $\lambda^p_2 =  \mathbf{ev} \circ_1 -$.  Moreover, chasing the identity around the commutative square of Equation~\ref{eq:psd-pull-closedness} we see that $\mathbb{U}\mathbf(\mathbf{ev}) = ev \circ_1 U_{\cb,\cc} \colon \psd(\cb,\cc),\cb \to \cc$ where  $ev \colon \lax(\cb,\cc),\cb \to \cc$ is the universal lax binary map of Theorem~\ref{thm:closed}.
%
%
This equation, plus commutativity of Equation~\ref{eq:psd-pull-closedness}, ensure that both functions $$\mathbf{ev} \circ_1 - , \lambda^p_{n+1}\colon \GCatmp^{l}_n(\,\ca_1,\ldots,\ca_n;\psd(\cb,\cc)\,) \rightrightarrows \GCatmp^{l}_{n+1}(\ca_1,\ldots,\ca_n,\cb;\cc)$$ coincide upon post-composition by $\mathbb{U}_{n+1}$.
Our task is to prove that they coincide in general.  The case $n=0$ is trivial since $\mathbb{U}_{1}$ is the identity function whilst the case $n=1$ is true by definition of $\mathbf{ev}$.  It remains to consider the cases $n=2,3$.

Consider then $\mathbf{F} \in \GCatmp^{l}_n(\ca_1,\ca_2;\psd(\cb,\cc))$.  By the above, we know that $\mathbf{ev} \circ_1 \mathbf{F}$ and $\lambda^p_3 \mathbf{F}$ have the same underlying lax ternary maps.  Therefore it remains to show that their associated pseudo-binary maps coincide.  To prove this, we use that, much as in the lax case of Proposition~\ref{prop:ternary-maps}, the ternary map $\lambda^p_3 \mathbf{F}$ has components $\lambda^p_2 (\mathbf{F}^{a_1}),\lambda^p_2 (\mathbf{F}_{a_2})$ and $\mathbf{ev} _b \circ_1 \mathbf{F}$ where $\mathbf{ev}_b\colon\psd(\cb,\cc) \to \cc$ denotes the evaluation Gray-functor.  

We then have the first equality in the sequence $$(\lambda^p_3 \mathbf{F})^{a_1} = \lambda^p_2 (\mathbf{F}^{a_1}) = \mathbf{ev} \circ_1 (\mathbf{F} \circ_1 a_1) =  (\mathbf{ev} \circ_1 \mathbf{F}) \circ_1 a_1 = (\mathbf{ev} \circ_1 \mathbf{F})^{a_1}.$$
The second equality holds by definition of $\mathbf{ev}$, the third by associativity and the fourth by definition of substitution of nullary maps.  Similarly $(\lambda^p_3 \mathbf{F})\midscript{a_2} = (\mathbf{ev} \circ_1 \mathbf{F})\midscript{a_2}$.  Finally, we have $(\lambda^p_3 \mathbf{F})_{b} =\mathbf{ev} _b \circ_1 \mathbf{F} = (\mathbf{ev} \circ_2 b) \circ_1 \mathbf{F} = (\mathbf{ev} \circ_1 \mathbf{F}) \circ_3 b = (\mathbf{ev} \circ_1 \mathbf{F})_b$.  This completes the case $n=2$ and the case $n=3$  is similar.  

This completes the proof that $\GCatmp^l$ is a closed $4$-ary multicategory.  Furthermore it follows immediately that $\GCatmp$  is a closed $4$-ary skew multicategory since, by definition, the tight pseudo-multimaps are preserved by the closedness bijections.  

Having established closedness, we turn to left representability.  The nullary map classifier $u\colon (-) \to \mathbf{1}$ is as for $\GCatm$.  Therefore, as in that setting, it remains to prove that the internal hom functor, now $\psd(\ca,-)\colon \GCat \to \GCat$, has a left adjoint.  Arguing as in the proof of Proposition~\ref{prop:mappingspaceadjoint}, it suffices to show that $\psd(\ca,-)_i\colon \GCat \to \Set$ preserves limits and is accessible for $i=0,1,2,3$.  Since $\psd(\ca,-)_0 = \lax(\ca,-)_0 \cong \GCat(QA,-)$, the case $i=0$ is as before.  For $i=1$, we write $P_{+}:= \psd(\mathbf 2,-)\colon \GCat \to \GCat$ for the \emph{pseudo} path space functor.  It is straightforward to show that we have a bijection $\psd(\ca,\cb)_1 \cong \GCat_p(\ca,P_{+}\cb)$ natural in $\cb$ --- equally, a natural bijection $\psd(\ca,\cb)_1 \cong \GCat(Q\ca,P_{+}\cb)$ so that, again arguing as in the lax case, it suffices to prove that $P_{+}$ preserves limits and is accessible.  As for $P$, this is the case since $P_{+}$ is representable in each component.

It remains to prove that $\psd(\ca,-)_i$ preserves limits and is accessible for $i=2,3$.  To this end, consider the diagram below left.  
\begin{equation*}
\adjustbox{scale=0.9}{
\xymatrix{
\psd(\ca,\cb)_2 \ar[dd]_{} \ar[dr] \ar[rr]^{} && \Set(\ca_1,Iso(\cb_2)) \ar[d]^{incl} \\
& \lax(\ca,\cb)_2 \ar[d]_{\langle s, t \rangle} \ar[r]^{k_{\ca,\cb}} & \Set(\ca_1,\cb_2) \\
(\psd(\ca,\cb)_1)^2 \ar[r]_{(U_{\ca,\cb})_1^2} & (\lax(\ca,\cb)_1)^2
}
\hspace{0.3cm}
\xymatrix{
\psd(\ca,\cb)_3 \ar[d]_{} \ar[r]^{} & \lax(\ca,\cb)_3 \ar[d]^{\langle s,t \rangle} \\
(\psd(\ca,\cb)_2)^2 \ar[r]_{(U_{\ca,\cb})_2^2} & (\lax(\ca,\cb)_2)^2
}
}
\end{equation*}

Here $k_{\ca,\cb} \colon \lax(\ca,\cb)_2 \to \Set(\ca_1,\cb_2)$ sends a modification $\theta$ to the $\ca_1$-indexed family of $2$-cells $(\theta_f)_{f \in \ca_1}$.  We denote by $incl \colon Iso(\cb) \hookrightarrow \cb_2$ the inclusion of the set of invertible $2$-cells in $\cb$ and note that the functor $Iso(-)\colon \GCat \to \Set$ is represented by the free invertible $2$-cell.  Since a pseudo-modification $\Lambda\colon\alpha \to \beta$ is an ordinary modification for which each $\Lambda_f$ is invertible, and for which $\alpha$ and $\beta$ are pseudo-transformations, the diagram above left exhibits $\psd(\ca,\cb)_2$ as a wide pullback, naturally in $\cb \in \GCat$.  The functors $\Set(\ca_1,Iso(-_2))$ and $\Set(\ca_1,(-)_2)$ preserve limits and are accessible since both are composite of representables.  The functor $\psd(\ca,-)_1$ has been covered above, whilst $\lax(\ca,-)_2$ and $\lax(\ca,-)_1$ have the same preservation properties, as established in the proof of Proposition~\ref{prop:mappingspaceadjoint}.  As such, the finite limit $\psd(\ca,-)_2$ has the same preservation properties since finite limits commute with limits and filtered colimits in the l.f.p. category $\GCat$.

Finally, the simple case of $i=3$ follows from the pullback square depicted above right. 
\end{proof}

\subsection{The associated skew monoidal closed structure}\label{sect:pskewmonoidal}

Given Theorem~\ref{thm:summarypseudocase}, the construction of Section~\ref{sect:From} again yields: 

\begin{theorem}\label{thm:pseudoskewmonoidal}
There is a closed skew monoidal structure $(\GCat, \otimes_p, \mathbf{1})$ whose internal hom is $\psd(\ca,\cb)$.
\end{theorem}

 Let us describe the skew monoidal structure in a little more detail.

\begin{itemize}
\item The unit is the terminal Gray-category $\mathbf{1}$ as before.
\item The tensor product $\ca \otimes_p \cb$ is the tight pseudo-binary map classifier.  It can be presented by adapting the presentation of $\ca \otimes_l \cb$ in Section~\ref{sect:presentationlax} by adjoining further:
\vspace{0.1cm}
\begin{itemize}
\item[\underline{$2$-cells}:] 
	\begin{itemize}
	\item[$\bullet$] at $A\colon a \to a'$ and $B\colon b \to b'$, a $2$-cell $A^{o}_B \colon T_{b'}(A)T^{a}(B) \to T^{a'}(B)T_b(A)$;
	\end{itemize}
\item[\underline{$3$-cells}:] 
	\begin{itemize}\vspace{0.1cm}
	\item[$\bullet$] at $A\colon a \to a'$ and $B\colon b \to b'$, $3$-cells $\epsilon_{A,B} \colon A^{o}_B A_B \to 1$ and $\eta_{A,B} \colon 1 \to A_B A^{o}_B$, along with $3$-cells $(\epsilon_{A,B})^{-1}$ and $(\eta_{A,B})^{-1}$;
	\item[$\bullet$] Further $3$-cells $(T^{\alpha}_B)^{-1}$ and $(T^{A}_{\beta})^{-1}$.
	\end{itemize}
\item[\underline{Equations}:] 
        \begin{itemize}\vspace{0.1cm}
\item[$\bullet$] The triangle equations for an adjunction $A^{o}_B \dashv A_B$ plus the equations that $(\epsilon_{A,B})^{-1}$ and $(\eta_{A,B})^{-1}$ are the intended inverses so that we have an adjoint equivalence.  Furthermore the equations asserting that $(T^{\alpha}_B)^{-1}$ and $(T^{A}_{\beta})^{-1}$ are the intended inverses.  To this we add the equations that $A^{o}_B, \epsilon_{A,B}$ and $\eta_{A,B}$ are identities whenever $A$ or $B$ is an identity.  Finally, we add the equations requiring that the adjoint equivalence $(A_2A_1)^o_B  \dashv (A_2A_1)_B$ on the composite is the composite adjoint equivalence \eqref{eq:adjointequivalence}.
\end{itemize}
	\end{itemize}

\item The left unit map $l_{\ca}\colon \mathbf{1} \otimes_p \ca\to\ca$ is again the counit $p_{\ca}\colon Q\ca \to \ca$ whilst the right unit map $r_\ca\colon \ca \to \ca \otimes_p \mathbf{1}$ is, as before, invertible. The example of Section~\ref{sect:pskewmonoidal} again shows that the associator is not invertible.
\end{itemize}

In a followup paper, we will show that $(\GCat, \otimes_p, \mathbf{1})$ forms a symmetric skew monoidal structure in the sense of \cite{BourkeLack:symmetric}.
\section{Restricting to strict maps in $\lax(\ca,\cb)$ and $\psd(\ca,\cb)$}\label{sect:sharp}

The reader may have noticed that the full sub-Gray-categories $\lax(\ca,\cb)_{s}$ and $\psd(\ca,\cb)_{s}$ of $\lax(\ca,\cb)$ and $\psd(\ca,\cb)$ containing the strict maps form Gray-categories too.  In this section, we show that these form the internal homs of closed left representable $4$-ary multicategories $\GCatm^{\sharp}$ and $\GCatmp^{\sharp}$, and so give rise to skew monoidal closed structures on $\GCat$ which are both left and right normal.  

In fact the $4$-ary multicategories $\GCatm^{\sharp}$ and $\GCatmp^{\sharp}$ are obtained in a universal way from $\GCatm$ and $\GCatmp$, by applying the construction sending a $4$-ary skew multicategory to its $4$-ary multicategory of \emph{sharp maps}, which we introduce below.  

Let $\mc$ be a $k$-ary skew multicategory, and let us re-emphasise that we require that the functions $\mc^t_n(x_1,\ldots,x_n;y) \subseteq \mc^l_n(x_1,\ldots,x_n;y)$ are subset inclusions, since this condition is used in the constructions below.  The notion of sharp multimap is defined inductively, as follows. 
\begin{itemize}
\item Each nullary map is sharp.
\item For $n>0$, a multimap $f\colon a_1,\ldots, a_n \to b$ is sharp if it is tight and $f \circ _i x$ is sharp for each nullary $x \colon \diamond \to a_i$ and $1 \leq i \leq n$.
\end{itemize}

\begin{eg}\label{eg:sharp}
A binary map $F\colon\ca,\cb\to\cc \in \GCatm$ (that is, a lax binary map) is sharp if it is tight and, moreover, each $F^a$ is a Gray-functor.  In more elementary terms, a lax binary map $F\colon\ca,\cb\to\cc$ is sharp just when each of $F^a$ and $F_b$ are Gray-functors and the $3$-cells $F^{A_2,A_1}_{B}$ are identities.  

A binary map $\mathbf F \colon\ca,\cb\to\cc \in \GCatmp$ (that is, a pseudo binary map) is sharp when it is tight and each $F^a$ is a Gray-functor.  Equivalently, just when its underlying lax binary map is sharp and we have the equality of adjoint equivalences \eqref{eq:adjointequivalence}.
\end{eg}

We let $\mc^{\sharp}$ have the same objects as $\mc^l$ and have multimaps the sharp ones.

\begin{prop}
$\mc^{\sharp}$ is a $k$-ary sub-multicategory of $\mc^{l}$.
\end{prop}
\begin{proof}
The unary identities are tight and so sharp.  Therefore it remains to prove that $g \circ_i f$ is sharp if $g$ and $f$ are.  We argue by induction on the degree $n$ of $g \circ_i f$.  If $n=0$, it is trivial, since all nullary maps are sharp.  Suppose then that $g \circ_i f$ has degree $n+1$.  
Firstly, we show it is tight.  If $f$ is of degree $0$, then $g$ has degree at least $2$.  Hence since $g \circ_i f$ is sharp of degree $\geq 1$, it is tight.  Otherwise both $g$ and $f$ have degree at least $1$ and so are both tight by assumption.  Hence their composite $g \circ_i f$ is tight.  

It remains to show that $(g \circ_i f) \circ_j x$ is sharp for each nullary $x$.   Let $f$ have degree $m$.  There are two cases.  If $i \leq j \leq i + m$ (so that $x$ is substituted inside $f$) then associativity gives us $(g \circ_i f) \circ_j x = g \circ_{p} (f \circ_q x)$ for the appropriate $p$ and $q$.  In particular, the right hand side is a composite of sharp multimaps of composite degree $\leq n$ and so sharp by assumption.  Otherwise ($x$ substituted outside of $f$) associativity gives $(g \circ_i f) \circ_j x = (g \circ_{p} x) \circ_{q} f$ for suitable $p$ and $q$ which, again, has composite degree $\leq n$ and so is sharp by induction.

%
\end{proof}

The $k$-ary multicategory $\mc^{\sharp}$ has a universal property.  To see this, consider the full inclusion $I \colon \kmulti \hookrightarrow \skkmulti$ viewing a $k$-ary multicategory as a $k$-ary skew multicategory in which all loose morphisms of arity $>0$ are tight.  It restricts to a full inclusion $I \colon \kmulti^{0} \hookrightarrow \skkmulti^{0}$ where, on either side, we take the \emph{$0$-fully faithful morphisms}: those $F \colon \mc \to \md$ for which each function $ \mc_0(\diamond;a) \to \md_0(\diamond;Fa)$ is a bijection.


\begin{prop}
The inclusion $I \colon \kmulti^{0} \hookrightarrow \skkmulti^{0}$ has a right adjoint, whose value at a $k$-ary skew multicategory $\mc$ is the $k$-ary multicategory $\mc^{\sharp}$.
\end{prop}
\begin{proof}
We have an identity on objects $0$-fully faithful multifunctor $I(\mc^{\sharp}) \to \mc$ which we will show the be the counit of the adjunction. This amounts to showing, given a $k$-ary multicategory $\md$ and a multifunctor $F \colon I(\md) \to \mc$, that each multimorphism in the image of $F$ is sharp and we prove this by induction. The case $n=0$ is trivial since each nullary multimorphism is sharp. Suppose it is true for multimaps of degree $\leq n$, and consider $f \colon a_1, \ldots, a_{n+1} \to b\in I(\md)$.  Since $f$ is tight, so is $Ff$.  Given $y \colon \diamond \to Fa_i\in\mc$, it remains to show that $Ff \circ_i y$ is sharp.  But since $F$ is $0$-fully faithful, there exists a unique nullary map $x \colon \diamond \to a_i\in I(\md)$ such that $Fx = y$.  Then $Ff \circ_i y = Ff \circ_i Fx = F(f \circ_i x)$ which is sharp by induction, since $f \circ_i x$ has degree $n$.
\end{proof}
In order to prove that $\GCatm^{\sharp}$ and $\GCatmp^{\sharp}$ are left representable, we will use the following result.

\begin{prop}\label{prop:sharp2}
Let $k \geq 2$.  If $\mc$ is a left representable $k$-ary skew multicategory and the category of tight unary maps $\mc^t_1$ is cocomplete, then $\mc^{\sharp}$ admits binary and nullary map classifiers. 
\end{prop}
\begin{proof}
The nullary map classifier $i$ is unchanged.  Now a tight binary map $f\colon a,b \to c$ is sharp just when for each nullary map $x\colon \diamond \to a$, the composite $f \circ_1 x\colon  b \to c$ is tight.  Let $\overline{f}\colon a \otimes b \to c$ denote the corresponding map out of the tight binary map classifier and $\overline{x}\colon i \to a$ the corresponding map out of $i$.  Expressed in these terms, the condition that $f \circ_1 x\colon b \to c$ is tight says that the (upper horizontal, right vertical) composite in the diagram below factors through $l_b\colon i \otimes b \to c$ as depicted.
\begin{equation*}
\xymatrix{i \otimes b \ar[d]_{l_b} \ar[r]^{\overline{x} \otimes 1} & a \otimes b \ar[d]^{\overline{f}} \\
b \ar[r]_{\exists} & c}
\end{equation*}
Here $l_b\colon i \otimes b \to b$ is the left unit map of the corresponding skew monoidal structure.  Note that $l_b$ is epi since, following Section~\ref{sect:From}, it is the counit of the adjunction $i\otimes- \dashv j$ with faithful right adjoint $j \colon \mc^t_1 \subseteq \mc^l_1$.  In particular, such a factorisation is unique.  

We can express the condition for all $\overline{x}$ simultaneously by requiring that the (upper horizontal, right vertical) composite in the diagram below left
\begin{equation*}
\xymatrix{\mc^t_1(i,a) . i \otimes b \ar[d]_{ \mc^t_1(i,a) . l_b } \ar[r]^-{\epsilon} & a \otimes b \ar[d]^{\overline{f}} \\
\mc^t_1(i,a) . b \ar[r]_{\exists} & c}
\hspace{1cm}
\xymatrix{\mathbb C^t_1(i,a) . i \otimes b \ar[d]_{\mc^t_1(i,a) . l_b } \ar[r]^-{\epsilon} & a \otimes b \ar[d]^{} \\
\mc^t_1(i,a) . b \ar[r] & a \star b}
\end{equation*}

factors through the left vertical epimorphism,  where $\epsilon$ is the morphism from the copower whose component at $\overline{x}\colon i \to a$ is $\overline{x} \otimes 1\colon i \otimes b \to a \otimes b$.  Accordingly, such an $\overline{f}$ bijectively corresponds to a map out of the pushout $a \star b$ in the right diagram --- that is, the pushout classifies sharp binary multimaps.

\end{proof}

\begin{theorem}\label{thm:sharpmulti}
Both the $4$-ary multicategories $\GCatm^{\sharp}$ and $\GCatmp^{\sharp}$ are left representable and closed, with internal homs $\lax(\cb,\cc)_{s}$ and $\psd(\cb,\cc)_{s}$.
\end{theorem}
\begin{proof}
Since both $\GCatm$ and $\GCatmp$ are left representable, and since $\GCat$ is cocomplete, by Proposition~\ref{prop:sharp2} both $\GCatm^{\sharp}$ and $\GCatmp^{\sharp}$ admit nullary and binary map classifiers.  Therefore, by Section~\ref{sect:closedrep} it remains to show that both are closed.  The arguments in both cases are identical in form, and we will give the proof in the case of $\lax(\cb,\cc)_s$.


%

Let us write $I_{\cb,\cc}\colon\lax(\cb,\cc)_{s} \hookrightarrow \lax(\cb,\cc)$ for the inclusion.  We will prove that the closedness bijections for $\GCatm^l$ on the bottom row below

\begin{equation*}
\xymatrix{
\GCatm^{\sharp}_n(\,\ca_1,\ldots,\ca_n;\lax(\cb,\cc)_{s}\,) \ar[d]_{I_{\cb,\cc} \circ_1 -} \ar[r]^-{\lambda_{n+1}^{\sharp}} & \GCatm^{\sharp}_{n+1}(\ca_1,\ldots,\ca_n,\cb;\cc) \ar[d]^-{} \\
\GCatm^{l}_n(\,\ca_1,\ldots,\ca_n;\lax(\cb,\cc)\,) \ar[r]_
-{\lambda_{n+1}} & \GCatm^{l}_{n+1}(\ca_1,\ldots,\ca_n,\cb;\cc)
}
\end{equation*}
restrict to bijections as on the top row --- this will imply closedness, since it implies that $\lambda_{n+1}^{\sharp}(F) = \lambda_{n+1}( I_{\cb,\cc} \circ F) = (ev \circ_1 I_{\cb,\cc}) \circ_1 F$, and so exhibits $ev \circ_1 I_{\cb,\cc} \colon \lax(\cb,\cc)_{s},\cb \to \cc$ as witnessing the bijections on the top row.

For $n=0$, this holds since the sharp unary maps are the strict maps.  For $n\geq 1$ each sharp $n$-ary multimap is tight.  Therefore it is equally to show that the closedness bijections

\begin{equation*}
\xymatrix{
\GCatm^{\sharp}_n(\,\ca_1,\ldots,\ca_n;\lax(\cb,\cc)_{s}\,) \ar[d]_{I_{\cb,\cc} \circ_1 -} \ar[r]^-{\lambda_{n+1}^{\sharp}} & \GCatm^{\sharp}_{n+1}(\ca_1,\ldots,\ca_n,\cb;\cc) \ar[d]^-{} \\
\GCatm^{t}_n(\,\ca_1,\ldots,\ca_n;\lax(\cb,\cc)\,) \ar[r]_
-{\lambda_{n+1}} & \GCatm^{t}_{n+1}(\ca_1,\ldots,\ca_n,\cb;\cc)
}
\end{equation*}
restrict to bijections on the top row.   For $n=1$, consider a tight binary map $\overline{F}\colon \ca_1 \to \lax(\cb,\cc)$.  The associated tight binary map $F\colon \ca_1,\cb \to \cc$ will be sharp just when $F^{a_1}$ and $F_{b}$ are strict.  Since $F$ is tight, this is just to say that each $F^{a_1}$ is strict --- that is, belongs to $\lax(\cb,\cc)_{s}$ --- which is to say that $\overline{F}$ factors uniquely through $\lax(\cb,\cc)_{s}$.

The case $n=0$ is trivial.  For $n=1$, consider a tight binary map $\overline{F}\colon \ca_1 \to \lax(\cb,\cc)$.  The associated tight binary map $F\colon \ca_1,\cb \to \cc$ will be sharp just when $F^{a_1}$ and $F_{b}$ are strict.  Since $F$ is tight, this is just to say that each $F^{a_1}$ is strict --- that is, belongs to $\lax(\cb,\cc)_{s}$ --- which is to say that $\overline{F}$ factors uniquely through $\lax(\cb,\cc)_{s}$.

For $n=2$, consider a tight binary map $\overline{F}\colon \ca_1,\ca_2 \to \lax(\cb,\cc)$ with $F\colon \ca_1,\ca_2,\cb \to \cc$ the associated tight ternary map.  To say that $\overline{F}$ factors through $\lax(\cb,\cc)_s \hookrightarrow \lax(\cb,\cc)$ via a sharp binary map is, since the inclusion is full, equally to say that the unary maps $\overline{F}^{a_1}$ and $\overline{F}_{a_2}$ lift to strict maps $\ca_i \to \lax_s(\cb,\cc)$.  By the preceding part, this is equally to say that the corresponding tight binary maps $F^{a_1}$ and $F\midscript{a_2}$ are sharp.  It remains, therefore, to show that the condition that $F_b$ is sharp is redundant.  Indeed, since $F$ is tight, so is $F_b$, so we must only show that $(F_b)^{a_1}$ and $(F_b)_{a_2}$ are strict, but these equal $(F^{a_1})_b$ and $(F \midscript{a_2})_b$ which are strict since $F^{a_1}$ and $F\midscript{a_2}$ are sharp.

For $n=3$, consider a tight ternary map $\overline{F}\colon \ca_1,\ca_2, \ca_3 \to \lax(\cb,\cc)$ with associated tight $4$-ary map $F\colon \ca_1,\ca_2,\ca_3,\cb \to \cc$.  Arguing as before, to say that $\overline{F}$ factors through $\lax(\cb,\cc)_s \hookrightarrow \lax(\cb,\cc)$ via a sharp ternary map is just to say that each of the ternary maps $F^{a_1}$, $F\midscript{a_2}$ and $F\midscript{a_3}$ are sharp.  Therefore it suffices to show the condition that $F_b$ is sharp is redundant.  Since $F_b$ is tight, this follows from the fact that $(F_b)^{a_1} = (F^{a_1})_b$, $(F_b)\midscript{a_2} = (F\midscript{a_2})_b$ and $(F_b)_{a_3} = (F\midscript{a_3})_b$ and each of these are sharp.

\end{proof}

\subsection{The associated skew monoidal closed structure}\label{sect:sharpskewmonoidal}

As usual, combining Theorem~\ref{thm:sharpmulti} with Section~\ref{sect:From}, we obtain:

\begin{theorem}\label{thm:sharpskewmonoidal}
There are closed skew monoidal structures $(\GCat, \otimes^{\sharp}_l, \mathbf{1})$ and $(\GCat, \otimes^{\sharp}_p, \mathbf{1})$, having internal homs $\lax(\ca,\cb)_s$ and $\psd(\ca,\cb)_s$.  
\end{theorem}

\begin{itemize}
\item The tensor products $\ca \otimes^{\sharp}_l \cb$ and $\ca \otimes^{\sharp}_p \cb$ classify sharp lax binary maps and sharp pseudo binary maps, respectively. Both admit presentations, obtained by modifying the presentations in Section~\ref{sect:presentationlax} and Section~\ref{sect:pskewmonoidal} for $\ca \otimes_l \cb$  and $\ca \otimes_p \cb$  respectively simply by adding the equation asserting that $T^a_{B_2,B_1}$ is an identity 2-cell.
\item Both skew monoidal structures are right normal as before but now they are also left normal.  Indeed, $\mathbf{1} \otimes^{\sharp}_l \ca$ classifies loose unary maps in $\GCatm^{\sharp}$ but, in this restricted setting, these are just the strict maps, and so classified by $\ca$ itself.  The associator for neither skew monoidal closed structure is invertible --- for instance, $(\mathbf 2 \otimes^{\sharp}_l \mathbf 2) \otimes^{\sharp}_l \mathbf 2$ and $\mathbf 2 \otimes^{\sharp}_l (\mathbf 2 \otimes^{\sharp}_l \mathbf 2)$ are not isomorphic. 
\item Both skew monoidal closed structures give rise to skew closed structures in the sense of Street \cite{Street2013Skew} --- in fact, since these skew monoidal structures are left and right normal, the corresponding skew closed structures are in fact closed categories in the sense of Eilenberg and Kelly \cite{EilenbergS:cloc}.
\item In a followup paper, we will show that the skew structure $(\GCat, \otimes^{\sharp}_p, \mathbf{1})$ is homotopically well behaved with respect to Lack's model structure on $\GCat$ \cite{Lack2011A-quillen}.  Furthermore, we will show that it descends to a symmetric monoidal closed structure on the homotopy category of Gray-categories.  This implies, for instance, that the associators $(\ca \otimes^{\sharp}_p \cb) \otimes^{\sharp}_p \cc \to \ca \otimes^{\sharp}_p (\cb \otimes^{\sharp}_p \cc)$ are triequivalences whenever $\ca,\cb$ and $\cc$ are cofibrant Gray-categories.
\end{itemize}

The reader is entitled to wonder whether we could give a simplified construction of the $4$-ary multicategories of sharp multimaps, without considering pseudomaps and the larger skew multicategories $\GCatm$ and $\GCatmp$ at all.  It is not clear that this is the case.  

The reason is that if $F\colon \ca,\cb \to \cc$ and $G\colon \cx,\cc \to \cd$ are sharp lax binary maps, our construction of $G \circ_2 F$ will still pass through pseudomaps.  Indeed, our construction used the dual $d_2 G\colon \cc^{co},\cx^{co} \to \cd^{co}$ but this need not be tight, so that the adjoint map $\cc^{co} \to \lax(\cx^{co},\cd^{co})$ is only a pseudomap. Since the definition of $G \circ_2 F$ ultimately uses this, pseudomaps cannot be easily eliminated!

A closely related issue is that since a sharp map $\ca,\cb \to \cc$ only induces a pseudomap $\cb \to \oplax(\ca,\cc)$, the multicategory $\GCatm^{\sharp}$ is not biclosed.  One only obtains biclosedness in the larger $4$-ary multicategory $\GCatm^l$ of loose multimaps.  But this suffers from not being left representable.  In between $\GCatm^{\sharp}$ and $\GCatm^{l}$, we have the $4$-ary skew multicategory $\GCatm$ which, in addition to being left representable and closed, contains both $\GCatm^{\sharp}$ and $\GCatm^l$ as canonically associated $4$-ary multicategories.


\appendix

\section{Axioms for $\lax(\ca,\cb)$}
\label{app:ax-gohla-mapp}

\subsection{Pseudomaps}
\label{subsec:ax-pseudofunct}
A pseudomap $F\colon\ca\to\cb$ has to satisfy the following axioms. 

\begin{enumerate}[(i)]
\item\label{ax:psfct.id-1cells}\emph{Preservation of identity $1$-cells:} at an object $x \in \ca$ we have $F(1_x) = 1_{Fx}$.
\item\label{ax:psfct.loc-sesq} \emph{Locally a $2$-functor:} at objects $x,y\in\ca$, $F_{x,y}\colon\ca(x,y)\to\cb(Fx,Fy)$ is a \emph{2-functor}.

\item\label{ax:psfct.F2-1-cells} \emph{Compatibility of $F^2_{-,-}$ with composition of 1-cells:} at $1$-cells $f_1\colon x\to y$, $f_2\colon y\to z$ and $f_3\colon z\to$~$w$ $\ca$, we have:
\begin{equation*}
F^2_{f_3,f_2f_1}\cdot (1\circ F^2_{f_2,f_1})=F^2_{f_3f_2,f_1}\cdot ( F^2_{f_2,f_1}\circ1).
\end{equation*}

\item\label{ax:psfct.deg-F2}  \emph{Degeneracy for $F^2_{f_2,f_1}$:} at a $1$-cell $f\colon x\to y$ in $\ca$, we have
\begin{equation*}
F^2_{1_y,f}=F^2_{f,1_x}=1_{Ff}.
\end{equation*}

\item\label{ax:psfct.F2-whisk-2} \emph{Compatibility of $F^2_{-,-}$ with whiskering of 2-cells:}
In the context of the diagrams in $\ca$ on the left below, we have the two equations on the right:
\begin{center}
(L)\hspace{0.25cm}
\begin{tikzcd}
	x & y & z,
	\arrow["{f_1}", from=1-1, to=1-2]
	\arrow[""{name=0, anchor=center, inner sep=0}, "{f_2}", curve={height=-12pt}, from=1-2, to=1-3]
	\arrow[""{name=1, anchor=center, inner sep=0}, "{f_2'}"', curve={height=12pt}, from=1-2, to=1-3]
	\arrow["\psi", shorten <=3pt, shorten >=3pt, Rightarrow, from=0, to=1]
\end{tikzcd}
\hspace{0.25cm}
$F(\psi\circ f_1)\cdot F^2_{f_2,f_1}=F^2_{f'_2,f_1}\cdot (F\psi\circ Ff_1);$
\end{center}
\begin{center}
(R) \hspace{0.25cm}
\begin{tikzcd}
	x & y & z,
	\arrow[""{name=0, anchor=center, inner sep=0}, "{f_1}", curve={height=-12pt}, from=1-1, to=1-2]
	\arrow["{f_2}", from=1-2, to=1-3]
	\arrow[""{name=1, anchor=center, inner sep=0}, "{f'_1}"', curve={height=12pt}, from=1-1, to=1-2]
	\arrow["\phi", shorten <=3pt, shorten >=3pt, Rightarrow, from=0, to=1]
\end{tikzcd}\hspace{0.25cm}
$F(f_2\circ\phi)\cdot F^2_{f_2,f_1}=F^2_{f_2,f'_1}\cdot (Ff_2\circ F\phi).$
\end{center}

\item\label{ax:psfct.F2-whisk-3} \emph{Compatibility of $F^2$ with whiskering of 3-cells:}
In the context of the diagrams in $\ca$ on the left below, we have the two equations on the right:
\begin{center}
(L)\hspace{0.25cm}
\begin{tikzcd}[ampersand replacement=\&]
	x \& y \&\& z
	\arrow["{f_1}", from=1-1, to=1-2]
	\arrow[""{name=0, anchor=center, inner sep=0}, "{f_2}", curve={height=-30pt}, from=1-2, to=1-4]
	\arrow[""{name=1, anchor=center, inner sep=0}, "{f_2'}"', curve={height=30pt}, from=1-2, to=1-4]
	\arrow[""{name=2, anchor=center, inner sep=0}, "{\psi'}", shift left=5, shorten <=8pt, shorten >=8pt, Rightarrow, from=0, to=1]
	\arrow[""{name=3, anchor=center, inner sep=0}, "\psi"', shift right=5, shorten <=8pt, shorten >=8pt, Rightarrow, from=0, to=1]
	\arrow["\Delta"{yshift=0.15cm}, shorten <=4pt, shorten >=4pt, Rightarrow, scaling nfold=3, from=3, to=2]
\end{tikzcd}\hspace{0.25cm}
$F(\Delta\circ f_1)\cdot F^2_{f_2,f_1}=F^2_{f'_2,f_1}\cdot (F\Delta\circ Ff_1);$
\end{center}
\begin{center}
(R) \hspace{0.25cm}
\begin{tikzcd}[ampersand replacement=\&]
	x \&\& y \& z
	\arrow[""{name=0, anchor=center, inner sep=0}, "{f_1}", curve={height=-30pt}, from=1-1, to=1-3]
	\arrow["{f_2}", from=1-3, to=1-4]
	\arrow[""{name=1, anchor=center, inner sep=0}, "{f'_1}"', curve={height=30pt}, from=1-1, to=1-3]
	\arrow[""{name=2, anchor=center, inner sep=0}, "{\phi'}", shift left=5, shorten <=8pt, shorten >=8pt, Rightarrow, from=0, to=1]
	\arrow[""{name=3, anchor=center, inner sep=0}, "\phi"', shift right=5, shorten <=8pt, shorten >=8pt, Rightarrow, from=0, to=1]
	\arrow["\Gamma"{yshift=0.15cm}, shorten <=4pt, shorten >=4pt, Rightarrow, scaling nfold=3, from=3, to=2]
\end{tikzcd}  
\hspace{0.25cm}
$F(f_2\circ\Gamma)\cdot F^2_{f_2,f_1}=F^2_{f_2,f'_1}\cdot (Ff_2\circ F\Gamma).$
\end{center}

\item\label{ax:psfct.F.int} \emph{Compatibility of $F$ with the interchange:}
Given a diagram in $\ca$
\[\begin{tikzcd}[ampersand replacement=\&]
	x \& y \& z
	\arrow[""{name=0, anchor=center, inner sep=0}, "{f_1}", curve={height=-12pt}, from=1-1, to=1-2]
	\arrow[""{name=1, anchor=center, inner sep=0}, "{f_2}", curve={height=-12pt}, from=1-2, to=1-3]
	\arrow[""{name=2, anchor=center, inner sep=0}, "{f_2'}"', curve={height=12pt}, from=1-2, to=1-3]
	\arrow[""{name=3, anchor=center, inner sep=0}, "{f_1'}"', curve={height=12pt}, from=1-1, to=1-2]
	\arrow["\psi"{xshift=0.1cm}, shorten <=3pt, shorten >=3pt, Rightarrow, from=1, to=2]
	\arrow["\phi"{xshift=0.1cm}, shorten <=3pt, shorten >=3pt, Rightarrow, from=0, to=3]
\end{tikzcd}\]
we require the following equality
\begin{center}
\adjustbox{scale=0.9}{\begin{tikzcd}
	& {Ff_2Ff'_1} \\
	{Ff_2Ff_1} && {Ff'_2Ff'_1} \\
	& {F(f_2f'_1)} \\
	{F(f_2f_1)} && {F(f'_2f'_1)} \\
	& {F(f'_2f_1)}
	\arrow["F(1\circ\phi)"{description}, curve={height=-12pt}, from=4-1, to=3-2]
	\arrow["F(\psi\circ1)"{description}, curve={height=-12pt}, from=3-2, to=4-3]
	\arrow["F(\psi\circ1)"', curve={height=12pt}, from=4-1, to=5-2]
	\arrow["F(1\circ\phi)"', curve={height=12pt}, from=5-2, to=4-3]
	\arrow["{F^2_{f_2,f_1}}"', from=2-1, to=4-1]
	\arrow["{1\circ F\phi}", curve={height=-12pt}, from=2-1, to=1-2]
	\arrow["{F^2_{f_2,f'_1}}"{description}, from=1-2, to=3-2]
	\arrow["F\psi\circ1", curve={height=-12pt}, from=1-2, to=2-3]
	\arrow["{F^2_{f'_2,f'_1}}", from=2-3, to=4-3]
	\arrow["{F(\psi:\phi)}"{description}, shorten <=10pt, shorten >=10pt, Rightarrow, from=3-2, to=5-2]
	\arrow["{}"{description}, draw=none, from=1-2, to=4-1]
	\arrow["{}"{description}, Rightarrow, draw=none, from=4-3, to=1-2]
\end{tikzcd} $=$
\begin{tikzcd}
	& {Ff_2Ff'_1} \\
	{Ff_2Ff_1} && {Ff'_2Ff'_1} \\
	& {Ff'_2Ff_1} \\
	{F(f_2f_1)} && {F(f'_2f'_1).} \\
	& {F(f'_2f_1)}
	\arrow["F(\psi\circ1)"', curve={height=12pt}, from=4-1, to=5-2]
	\arrow["F(1\circ\phi)"', curve={height=12pt}, from=5-2, to=4-3]
	\arrow["{F^2_{f_2,f_1}}"', from=2-1, to=4-1]
	\arrow["{1\circ F\phi}", curve={height=-12pt}, from=2-1, to=1-2]
	\arrow["F\psi\circ1", curve={height=-12pt}, from=1-2, to=2-3]
	\arrow["{F^2_{f'_2,f'_1}}", from=2-3, to=4-3]
	\arrow["F\psi\circ1"{description}, curve={height=12pt}, from=2-1, to=3-2]
	\arrow["{1\circ F\phi}"{description}, curve={height=12pt}, from=3-2, to=2-3]
	\arrow[from=3-2, to=5-2]
	\arrow["{}"{description}, draw=none, from=2-3, to=5-2]
	\arrow["{}"{description}, draw=none, from=2-1, to=5-2]
	\arrow["{F\psi:F\phi}"{description}, shorten <=10pt, shorten >=10pt, Rightarrow, from=1-2, to=3-2]
\end{tikzcd}}
\end{center}
where the unlabelled regions commute by \eqref{ax:psfct.F2-whisk-2}. 

\item\label{ax:psfct.F2.int} \emph{Compatibility of $F^2$ with the interchange map:}
at 1-cells $f_1\colon x\to y$, $f_2\colon y\to z$, $f_3\colon z\to w$ and $f_4\colon w\to u$ in $\ca$ we have the diagram 
\[\begin{tikzcd}
	& Fy && Fw \\
	Fx && Fz && Fu.
	\arrow["{Ff_1}", curve={height=-6pt}, from=2-1, to=1-2]
	\arrow["{Ff_2}", curve={height=-6pt}, from=1-2, to=2-3]
	\arrow["{Ff_3}", curve={height=-6pt}, from=2-3, to=1-4]
	\arrow["{Ff_4}", curve={height=-6pt}, from=1-4, to=2-5]
	\arrow[""{name=0, anchor=center, inner sep=0}, "{F(f_2f_1)}"', curve={height=18pt}, from=2-1, to=2-3]
	\arrow[""{name=1, anchor=center, inner sep=0}, "{F(f_4f_3)}"', curve={height=18pt}, from=2-3, to=2-5]
	\arrow["{F^2_{f_2,f_1}}"{description}, shorten <=8pt, shorten >=8pt, Rightarrow, from=1-2, to=0]
	\arrow["{F^2_{f_4,f_3}}"{description}, shorten <=8pt, shorten >=8pt, Rightarrow, from=1-4, to=1]
\end{tikzcd}\]
and require the equality of $3$-cells $F^2_{f_4,f_3}:F^2_{f_2,f_1}=1$, remarking that the lhs is automatically an endo-$3$-cell by \eqref{ax:psfct.F2-1-cells}.

\item \label{ax:psfct.F2.int.2-cells} \emph{Compatibility of $F^2$ with 2-cells and the interchange:}
Given the diagram below in $\ca$
\[\begin{tikzcd}
	x & y & z & w,
	\arrow[""{name=0, anchor=center, inner sep=0}, "{f_1}", curve={height=-12pt}, from=1-1, to=1-2]
	\arrow[""{name=1, anchor=center, inner sep=0}, "{f_1'}"', curve={height=12pt}, from=1-1, to=1-2]
	\arrow[""{name=2, anchor=center, inner sep=0}, "{f_3}", curve={height=-12pt}, from=1-3, to=1-4]
	\arrow[""{name=3, anchor=center, inner sep=0}, "{f_3'}"', curve={height=12pt}, from=1-3, to=1-4]
	\arrow["{f_2}", from=1-2, to=1-3]
	\arrow["\phi", shorten <=3pt, shorten >=3pt, Rightarrow, from=0, to=1]
	\arrow["\varphi", shorten <=3pt, shorten >=3pt, Rightarrow, from=2, to=3]
\end{tikzcd}\]
we require the equations
 $F^2_{f_3,f_2}:F\phi=1$ and 
 $F\varphi:F^2_{f_2,f_1}=1$,
 remarking again that both $3$-cells on the lhs of equations are endomorphisms by three instances of \eqref{ax:psfct.F2-whisk-2} and two of \eqref{ax:psfct.F2-1-cells}.
\end{enumerate}

\subsection{Lax Transformations}
\label{subsec:ax-lax-transf}
A lax transformation $\alpha\colon F\to G$ has to satisfy the following axioms. 
\begin{enumerate}[(i)]

\item\label{ax:lax-tr.deg-2-cell} \emph{Degeneracy for the 2-cell component:}
at an object $x\in\ca$, then $\alpha_{1_x}=1_{\alpha_x}$. 

\item\label{ax:lax-tr.3-cell} \emph{Compatibility with 3-cells:}
at $1$-cells $f,f'\colon x\to y$, $2$-cells $\phi,\phi' \colon f\to f'$ and a $3$-cell $\Gamma\colon\phi\to\phi'$, we require:
\[\begin{tikzcd}
	{Gf\alpha_x} && {\alpha_y Ff} &&&& {Gf\alpha_x} && {\alpha_y Ff} \\
	&&&& {=} \\
	{Gf'\alpha_x} && {\alpha_y Ff'} &&&& {Gf'\alpha_x} && {\alpha_y Ff'.}
	\arrow["{\alpha_f}", from=1-7, to=1-9]
	\arrow[""{name=0, anchor=center, inner sep=0}, "{1\circ F\phi}", curve={height=-18pt}, from=1-9, to=3-9]
	\arrow["{\alpha_{f'}}"', from=3-7, to=3-9]
	\arrow[""{name=1, anchor=center, inner sep=0}, "{G\phi'\circ 1}"', curve={height=18pt}, from=1-7, to=3-7]
	\arrow[""{name=2, anchor=center, inner sep=0}, "{1\circ F\phi'}"{description}, curve={height=18pt}, from=1-9, to=3-9]
	\arrow["{\alpha_{f'}}"', from=3-1, to=3-3]
	\arrow["{\alpha_f}", from=1-1, to=1-3]
	\arrow[""{name=3, anchor=center, inner sep=0}, "{1\circ F\phi}", curve={height=-18pt}, from=1-3, to=3-3]
	\arrow[""{name=4, anchor=center, inner sep=0}, "{G\phi'\circ 1}"', curve={height=18pt}, from=1-1, to=3-1]
	\arrow[""{name=5, anchor=center, inner sep=0}, "{G\phi\circ 1}"{description}, curve={height=-18pt}, from=1-1, to=3-1]
	\arrow["{1\circ F\Gamma}"'{yshift=0.1cm}, shorten <=11pt, shorten >=11pt, Rightarrow, from=0, to=2]
	\arrow["{\alpha_{\phi'}}"'{yshift=0.1cm}, shorten <=28pt, shorten >=28pt, Rightarrow, from=2, to=1]
	\arrow["{\alpha_{\phi}}"'{yshift=0.1cm}, shorten <=28pt, shorten >=28pt, Rightarrow, from=3, to=5]
	\arrow["G\Gamma\circ1"'{yshift=0.1cm}, shorten <=11pt, shorten >=11pt, Rightarrow, from=5, to=4]
\end{tikzcd}\]

\item\label{ax:lax-tr.comp-2-cell} \emph{Compatibility with vertical composition of 2-cells:}
at $1$-cells $f,f',f''\colon x\to y$ and $2$-cells $\phi_1\colon f\to f',\phi_2\colon f'\to f''$ in $\ca$, we require:
\[\begin{tikzcd}
	{Gf\alpha_x} && {\alpha_yFf} && {Gf\alpha_x} && {\alpha_yFf} \\
	\\
	{Gf'\alpha_x} && {\alpha_yFf'} & {=} \\
	\\
	{Gf''\alpha_x} && {\alpha_yFf''} && {Gf''\alpha_x} && {\alpha_yFf''.}
	\arrow["{G(\phi_2\cdot\phi_1)\circ1}"{description}, from=1-5, to=5-5]
	\arrow["{1\circ F(\phi_2\cdot\phi_1)}", from=1-7, to=5-7]
	\arrow[""{name=0, anchor=center, inner sep=0}, "{\alpha_{f}}", from=1-1, to=1-3]
	\arrow[""{name=1, anchor=center, inner sep=0}, "{\alpha_f}", from=1-5, to=1-7]
	\arrow[""{name=2, anchor=center, inner sep=0}, "{\alpha_{f'}}"{description}, from=3-1, to=3-3]
	\arrow[""{name=3, anchor=center, inner sep=0}, "{\alpha_{f''}}"', from=5-1, to=5-3]
	\arrow[""{name=4, anchor=center, inner sep=0}, "{\alpha_{f''}}"', from=5-5, to=5-7]
	\arrow["{1\circ F\phi_1}", from=1-3, to=3-3]
	\arrow["{1\circ F\phi_2}", from=3-3, to=5-3]
	\arrow["{G\phi_1\circ1}"', from=1-1, to=3-1]
	\arrow["{G\phi_2\circ1}"', from=3-1, to=5-1]
	\arrow["{\alpha_{\phi_1}}"{xshift=0.1cm}, shorten <=17pt, shorten >=17pt, Rightarrow, from=0, to=2]
	\arrow["{\alpha_{\phi_2}}"{xshift=0.1cm}, shorten <=17pt, shorten >=17pt, Rightarrow, from=2, to=3]
	\arrow["{\alpha_{\phi_2\cdot\phi_1}}"{xshift=0.1cm}, shorten <=34pt, shorten >=34pt, Rightarrow, from=1, to=4]
\end{tikzcd}\]

\item \label{ax:lax-tr.deg-3-cell}\emph{Degeneracy for the 3-cell component:} at a $1$-cell $f \colon x \to y$ we require $\alpha_{1_f}=1_{\alpha_f}$.

\item\label{ax:lax-tr.comp-1-cell} \emph{Compatibility with composition of 1-cells:} at composable $1$-cells $f_1\colon w\to x$, $f_2\colon x\to y$ and $f_3\colon y\to z$ in $\ca$, we require:
\[
\adjustbox{scale=0.85}{
\begin{tikzcd}
	{Gf_3Gf_2Gf_1\alpha_w} && {Gf_3Gf_2\alpha_xFf_1} && {Gf_3\alpha_yFf_2Ff_1} && {\alpha_zFf_3Ff_2Ff_1} \\
	\\
	{Gf_3G(f_2f_1)\alpha_w} && {} & {} & {Gf_3\alpha_yF(f_2f_1)} && {\alpha_zFf_3F(f_2f_1)} \\
	\\
	{G(f_3f_2f_1)\alpha_w} &&&&&& {\alpha_zF(f_3f_2f_1)} \\
	&&& {=}
	\arrow["{1\circ\alpha_{f_1}}", from=1-1, to=1-3]
	\arrow["{1\circ\alpha_{f_2f_1}}"', from=3-1, to=3-5]
	\arrow["{1\circ G^2_{f_2,f_1}\circ1}"', from=1-1, to=3-1]
	\arrow["{1\circ\alpha_{f_2}\circ1}", from=1-3, to=1-5]
	\arrow["{1\circ F^2_{f_2,f_1}}"{description}, from=1-5, to=3-5]
	\arrow[""{name=0, anchor=center, inner sep=0}, "{\alpha_{f_3}\circ1}", from=1-5, to=1-7]
	\arrow[""{name=1, anchor=center, inner sep=0}, "{\alpha_{f_3}\circ1}"', from=3-5, to=3-7]
	\arrow["{1\circ F^2_{f_2,f_1}}", from=1-7, to=3-7]
	\arrow["{ G^2_{f_3,f_2f_1}\circ1}"', from=3-1, to=5-1]
	\arrow["{1\circ F^2_{f_3,f_2f_1}}", from=3-7, to=5-7]
	\arrow[""{name=2, anchor=center, inner sep=0}, "{\alpha_{f_3f_2f_1}}"', from=5-1, to=5-7]
	\arrow["{1\circ \alpha^2_{f_2,f_1}}"{xshift=0.1cm}, shorten <=10pt, shorten >=10pt, Rightarrow, from=1-3, to=3-3]
	\arrow[shorten <=13pt, shorten >=13pt, Rightarrow, dashed, from=0, to=1]
	\arrow["{\alpha^2_{f_3,f_2f_1}}"{xshift=0.1cm}, shorten <=11pt, shorten >=11pt, Rightarrow, from=3-4, to=2]
\end{tikzcd}}\]
\[\adjustbox{scale=0.85}{\begin{tikzcd}[ampersand replacement=\&]
	{Gf_3Gf_2Gf_1\alpha_w} \&\& {Gf_3Gf_2\alpha_xFf_1} \&\& {Gf_3\alpha_yFf_2Ff_1} \&\& {\alpha_zFf_3Ff_2Ff_1} \\
	\\
	{G(f_3f_2)Gf_1\alpha_w} \&\& {G(f_3f_2)\alpha_xFf_1} \& {} \& {} \&\& {\alpha_zF(f_3f_2)Ff_1} \\
	\\
	{G(f_3f_2f_1)\alpha_w} \&\&\&\&\&\& {\alpha_zF(f_3f_2f_1).}
	\arrow[""{name=0, anchor=center, inner sep=0}, "{1\circ\alpha_{f_1}}", from=1-1, to=1-3]
	\arrow["{G^2_{f_3,f_2}\circ1}"', from=1-1, to=3-1]
	\arrow["{1\circ\alpha_{f_2}\circ1}", from=1-3, to=1-5]
	\arrow["{\alpha_{f_3}\circ1}", from=1-5, to=1-7]
	\arrow["{1\circ F^2_{f_3,f_2}\circ1}", from=1-7, to=3-7]
	\arrow["{ G^2_{f_3f_2,f_1}\circ1}"', from=3-1, to=5-1]
	\arrow["{1\circ F^2_{f_3f_2,f_1}}", from=3-7, to=5-7]
	\arrow[""{name=1, anchor=center, inner sep=0}, "{\alpha_{f_3f_2f_1}}"', from=5-1, to=5-7]
	\arrow["{G^2_{f_3,f_2}\circ1}"{description}, from=1-3, to=3-3]
	\arrow[""{name=2, anchor=center, inner sep=0}, "{1\circ\alpha_{f_1}}"', from=3-1, to=3-3]
	\arrow["{\alpha_{f_3f_2}\circ1}"', from=3-3, to=3-7]
	\arrow["{\alpha^2_{f_3,f_2}\circ1}"{xshift=0.1cm}, shorten <=10pt, shorten >=10pt, Rightarrow, from=1-5, to=3-5]
	\arrow[shorten <=13pt, shorten >=13pt, Rightarrow, dashed, from=0, to=2]
	\arrow["{\alpha^2_{f_3f_2,f_1}}"{xshift=0.1cm}, shorten <=11pt, shorten >=11pt, Rightarrow, from=3-4, to=1]
\end{tikzcd}}\]

\item\label{ax:lax-tr.deg-alpha2} \emph{Degeneracy for $\alpha^2_{f_2,f_1}$:} at a $1$-cell $f\colon x\to y$ in $\ca$, we require
$\alpha^2_{1_y,f}=\alpha^2_{f,1_x}=1_{\alpha f}$.

\item\label{ax:lax-tr.whisk-alpha2} \emph{Compatibility of $\alpha^2_{-,-}$ with whiskering of 2-cells:} at a diagram in $\ca$ as below
\[\begin{tikzcd}
	x & y & z,
	\arrow["{f_1}", from=1-1, to=1-2]
	\arrow[""{name=0, anchor=center, inner sep=0}, "{f_2}", curve={height=-12pt}, from=1-2, to=1-3]
	\arrow[""{name=1, anchor=center, inner sep=0}, "{f_2'}"', curve={height=12pt}, from=1-2, to=1-3]
	\arrow["\psi"{xshift=0.1cm}, shorten <=3pt, shorten >=3pt, Rightarrow, from=0, to=1]
\end{tikzcd}\] we require:
\begin{center}
(L)
\begin{tikzcd}[ampersand replacement=\&]
	{Gf_2Gf_1\alpha_x} \&\& {Gf_2\alpha_yFf_1} \&\& {\alpha_ZFf_2Ff_1} \\
	{Gf'_2Gf_1\alpha_x} \&\& {Gf'_2\alpha_yFf_1} \&\& {\alpha_ZFf'_2Ff_1} \\
	{G(f'_2f_1)\alpha_x} \&\&\&\& {\alpha_ZF(f'_2f_1)}
	\arrow[""{name=0, anchor=center, inner sep=0}, "{1\circ\alpha_{f_1}}", from=1-1, to=1-3]
	\arrow[""{name=1, anchor=center, inner sep=0}, "{\alpha_{f_2}\circ1}", from=1-3, to=1-5]
	\arrow["G\psi\circ1"', from=1-1, to=2-1]
	\arrow["{G^2_{f_2',f_1}\circ1}"', from=2-1, to=3-1]
	\arrow["{1\circ F^2_{f_2',f_1}}", from=2-5, to=3-5]
	\arrow["{1\circ F\psi\circ1}", from=1-5, to=2-5]
	\arrow["G\psi\circ1"{description}, from=1-3, to=2-3]
	\arrow[""{name=2, anchor=center, inner sep=0}, "{1\circ\alpha_{f_1}}"{description}, from=2-1, to=2-3]
	\arrow[""{name=3, anchor=center, inner sep=0}, "{\alpha_{f'_2}\circ1}"{description}, from=2-3, to=2-5]
	\arrow[""{name=4, anchor=center, inner sep=0}, "{\alpha_{f'_2f_1}}"', from=3-1, to=3-5]
	\arrow["{\alpha_\psi\circ1}"{xshift=0.1cm}, shorten <=9pt, shorten >=9pt, Rightarrow, from=1, to=3]
	\arrow[shorten <=9pt, shorten >=9pt, Rightarrow, dashed, from=0, to=2]
	\arrow["{\alpha^2_{f'_2,f_1}}"{xshift=0.1cm}, shorten <=5pt, shorten >=5pt, Rightarrow, from=2-3, to=4]
\end{tikzcd} \\
$=$
\begin{tikzcd}[ampersand replacement=\&]
	{Gf_2Gf_1\alpha_x} \&\& {Gf_2\alpha_yFf_1} \&\& {\alpha_zFf_2Ff_1} \\
	{G(f_2f_1)\alpha_x} \&\&\&\& {\alpha_zF(f_2f_1)} \\
	{G(f'_2f_1)\alpha_x} \&\&\&\& {\alpha_zF(f'_2f_1).}
	\arrow["{1\circ\alpha_{f_1}}", from=1-1, to=1-3]
	\arrow["{\alpha_{f_2}\circ1}", from=1-3, to=1-5]
	\arrow["{G^2_{f_2,f_1}\circ1}"', from=1-1, to=2-1]
	\arrow["{G(\psi f_1)\circ1}"', from=2-1, to=3-1]
	\arrow["{1\circ F(\psi f_1)}", from=2-5, to=3-5]
	\arrow["{1\circ F^2_{f_2,f_1}}", from=1-5, to=2-5]
	\arrow[""{name=0, anchor=center, inner sep=0}, "{\alpha_{f'_2f_1}}"', from=3-1, to=3-5]
	\arrow[""{name=1, anchor=center, inner sep=0}, "{\alpha_{f_2f_1}}"{description}, from=2-1, to=2-5]
	\arrow["{\alpha^2_{f_2,f_1}}"{xshift=0.1cm}, shorten <=6pt, shorten >=6pt, Rightarrow, from=1-3, to=1]
	\arrow["{\alpha_{\psi f_1}}"{xshift=0.1cm}, shorten <=9pt, shorten >=9pt, Rightarrow, from=1, to=0]
\end{tikzcd}
\end{center}
noting that the vertical paths in the two diagrams are equal by \eqref{ax:psfct.F2-whisk-2} for $G$ and $F$.  

Similarly, given a diagram in $\ca$ as below
\[\begin{tikzcd}
	x & y & z,
	\arrow[""{name=0, anchor=center, inner sep=0}, "{f_1}", curve={height=-12pt}, from=1-1, to=1-2]
	\arrow["{f_2}", from=1-2, to=1-3]
	\arrow[""{name=1, anchor=center, inner sep=0}, "{f'_1}"', curve={height=12pt}, from=1-1, to=1-2]
	\arrow["\phi"{xshift=0.1cm}, shorten <=3pt, shorten >=3pt, Rightarrow, from=0, to=1]
\end{tikzcd}\]
we require:
\begin{center}
(R)
\begin{tikzcd}[ampersand replacement=\&]
	{Gf_2Gf_1\alpha_X} \&\& {Gf_2\alpha_YFf_1} \&\& {\alpha_ZFf_2Ff_1} \\
	{Gf_2Gf'_1\alpha_X} \&\& {Gf_2\alpha_YFf'_1} \&\& {\alpha_ZFf_2Ff'_1} \\
	{G(f_2f'_1)\alpha_X} \&\&\&\& {\alpha_ZF(f'_2f_1)}
	\arrow[""{name=0, anchor=center, inner sep=0}, "{1\circ\alpha_{f_1}}", from=1-1, to=1-3]
	\arrow[""{name=1, anchor=center, inner sep=0}, "{\alpha_{f_2}\circ1}", from=1-3, to=1-5]
	\arrow["{1\circ G\phi\circ1}"', from=1-1, to=2-1]
	\arrow["{G^2_{f_2,f_1'}\circ1}"', from=2-1, to=3-1]
	\arrow["{1\circ F^2_{f_2,f_1'}}", from=2-5, to=3-5]
	\arrow["{1\circ F\phi}", from=1-5, to=2-5]
	\arrow["{1\circ F\phi}"{description}, from=1-3, to=2-3]
	\arrow[""{name=2, anchor=center, inner sep=0}, "{1\circ\alpha_{f'_1}}"{description}, from=2-1, to=2-3]
	\arrow[""{name=3, anchor=center, inner sep=0}, "{\alpha_{f_2}\circ1}"{description}, from=2-3, to=2-5]
	\arrow[""{name=4, anchor=center, inner sep=0}, "{\alpha_{f_2f'_1}}"', from=3-1, to=3-5]
	\arrow[shorten <=9pt, shorten >=9pt, Rightarrow, dashed, from=1, to=3]
	\arrow["{1\circ\alpha_\phi}"{xshift=0.1cm}, shorten <=9pt, shorten >=9pt, Rightarrow, from=0, to=2]
	\arrow["{\alpha^2_{f_2,f'_1}}"{xshift=0.1cm}, shorten <=6pt, shorten >=6pt, Rightarrow, from=2-3, to=4]
\end{tikzcd} \\
	$=$
\begin{tikzcd}[ampersand replacement=\&]
	{Gf_2Gf_1\alpha_x} \&\& {Gf_2\alpha_yFf_1} \&\& {\alpha_zFf_2Ff_1} \\
	{G(f_2f_1)\alpha_x} \&\&\&\& {\alpha_zF(f_2f_1)} \\
	{G(f_2f'_1)\alpha_x} \&\&\&\& {\alpha_zF(f_2f'_1).}
	\arrow["{1\circ\alpha_{f_1}}", from=1-1, to=1-3]
	\arrow["{\alpha_{f_2}\circ1}", from=1-3, to=1-5]
	\arrow["{G^2_{f_2,f_1}\circ1}"', from=1-1, to=2-1]
	\arrow["{G(f_2\phi)\circ1}"', from=2-1, to=3-1]
	\arrow["{1\circ F(f_2\phi)}", from=2-5, to=3-5]
	\arrow["{1\circ F^2_{f_2,f_1}}", from=1-5, to=2-5]
	\arrow[""{name=0, anchor=center, inner sep=0}, "{\alpha_{f_2f'_1}}"', from=3-1, to=3-5]
	\arrow[""{name=1, anchor=center, inner sep=0}, "{\alpha_{f_2f_1}}"{description}, from=2-1, to=2-5]
	\arrow["{\alpha^2_{f_2,f_1}}"{xshift=0.1cm}, shorten <=6pt, shorten >=6pt, Rightarrow, from=1-3, to=1]
	\arrow["{\alpha_{f_2\phi}}"{xshift=0.1cm}, shorten <=9pt, shorten >=9pt, Rightarrow, from=1, to=0]
\end{tikzcd}
\end{center}

\end{enumerate}

\subsection{Modifications}
\label{subsec:ax-mod}
A modification $\Gamma\colon\alpha\to\beta$ has to satisfy the following axioms.

\begin{enumerate}[(i)]
\item\label{ax:mod.deg-3} \emph{Degeneracy for the 3-cell component:} at an object $x\in\ca$, $\Gamma_{1_x}=1_{\Gamma_x}$. 

\item\label{ax:mod.comp-1} \emph{Compatibility with composition of 1-cells:}
at $1$-cells $f_1\colon x\to y$ and $f_2\colon y\to z$ in $\ca$, 
\begin{center}
\adjustbox{scale=0.8}{
\begin{tikzcd}[ampersand replacement=\&]
	{Gf_2Gf_1\alpha_x} \&\& {Gf_2Gf_1\beta_x} \\
	\& {Gf_2\alpha_yFf_1} \&\& {Gf_2\beta_yFf_1} \\
	{G(f_2f_1)\alpha_x} \&\& {\alpha_zFf_2Ff_1} \&\& {\beta_zFf_2Ff_1} \\
	\&\& \textcolor{white}{Gf_2\alpha_yFf_1} \\
	\&\& {\alpha_zF(f_2f_1)} \&\& {\beta_zF(f_2f_1)}
	\arrow["{1\circ\alpha_{f_1}}"{description}, from=1-1, to=2-2]
	\arrow["{\alpha_{f_2}\circ1}"{description}, from=2-2, to=3-3]
	\arrow[""{name=0, anchor=center, inner sep=0}, "{\Gamma_z\circ1}"', from=3-3, to=3-5]
	\arrow["{1\circ F^2_{f_2',f_1}}", from=3-5, to=5-5]
	\arrow[""{name=1, anchor=center, inner sep=0}, "{1\circ\Gamma_x}", from=1-1, to=1-3]
	\arrow["{1\circ\beta_{f_1}}", from=1-3, to=2-4]
	\arrow["{\beta_{f_2}\circ1}", from=2-4, to=3-5]
	\arrow[""{name=2, anchor=center, inner sep=0}, "{1\circ\Gamma_y\circ1}"{description}, from=2-2, to=2-4]
	\arrow["{1\circ F^2_{f_2',f_1}}"{description}, from=3-3, to=5-3]
	\arrow["{G^2_{f_2',f_1}\circ1}"{description}, from=1-1, to=3-1]
	\arrow[""{name=3, anchor=center, inner sep=0}, "{\alpha_{f_2f_1}}"', from=3-1, to=5-3]
	\arrow[""{name=4, anchor=center, inner sep=0}, "{\Gamma_z\circ1}"', from=5-3, to=5-5]
	\arrow["{\alpha^2_{f_2,f_1}}"{xshift=0.1cm}, shorten <=17pt, shorten >=17pt, Rightarrow, from=2-2, to=3]
	\arrow[shorten <=20pt, shorten >=20pt, Rightarrow, dashed, from=0, to=4]
	\arrow["{1\circ\Gamma_{f_1}}"{xshift=0.1cm,yshift=-0.15cm}, shorten <=25pt, shorten >=25pt, Rightarrow, from=1, to=2]
	\arrow["{\Gamma_{f_2}\circ1}"{xshift=0.1cm,yshift=-0.15cm}, shorten <=25pt, shorten >=25pt, Rightarrow, from=2, to=0]
\end{tikzcd}} \\ $=$
\adjustbox{scale=0.8}{
\begin{tikzcd}[ampersand replacement=\&]
	{Gf_2Gf_1\alpha_x} \& \textcolor{white}{Gf_2\beta_yFf_1} \& {Gf_2Gf_1\beta_x} \\
	\&\&\& {Gf_2\beta_yFf_1} \\
	{G(f_2f_1)\alpha_x} \&\& {G(f_2f_1)\beta_x} \&\& {\beta_zFf_2Ff_1} \\
	\&\&\& \textcolor{white}{Gf_2\beta_yFf_1} \\
	\&\& {\alpha_zF(f_2f_1)} \&\& {\beta_zF(f_2f_1).}
	\arrow["{1\circ F^2_{f_2',f_1}}", from=3-5, to=5-5]
	\arrow[""{name=0, anchor=center, inner sep=0}, "{1\circ\Gamma_x}", from=1-1, to=1-3]
	\arrow["{1\circ\beta_{f_1}}", from=1-3, to=2-4]
	\arrow["{\beta_{f_2}\circ1}", from=2-4, to=3-5]
	\arrow["{G^2_{f_2',f_1}\circ1}"{description}, from=1-1, to=3-1]
	\arrow["{\alpha_{f_2f_1}}"', from=3-1, to=5-3]
	\arrow[""{name=1, anchor=center, inner sep=0}, "{\Gamma_z\circ1}"', from=5-3, to=5-5]
	\arrow["{G^2_{f_2',f_1}\circ1}"{description}, from=1-3, to=3-3]
	\arrow[""{name=2, anchor=center, inner sep=0}, from=3-3, to=5-5]
	\arrow[""{name=3, anchor=center, inner sep=0}, "{1\circ\Gamma_x}", from=3-1, to=3-3]
	\arrow[shorten <=20pt, shorten >=20pt, Rightarrow, dashed, from=0, to=3]
	\arrow["{\beta^2_{f_2,f_1}}"{xshift=0.1cm}, shorten <=15pt, shorten >=15pt, Rightarrow, from=2-4, to=2]
	\arrow["{\Gamma_{f_2f_1}}", shorten <=55pt, shorten >=55pt, Rightarrow, from=3, to=1]
\end{tikzcd}}
\end{center}

\item\label{ax:mod.2-cell} \emph{Compatibility with 2-cells:}
at $1$-cells $f,f'\colon x\to y$ and a $2$-cell $\phi\colon f\to f'$, we require:
\begin{center}
\adjustbox{scale=0.85}{\begin{tikzcd}
	{Gf\alpha_x} && {Gf\beta_x} \\
	&& {\alpha_yFf} && {\beta_yFf} \\
	{Gf'\alpha_x} \\
	&& {\alpha_yFf'} && {\beta_yFf'}
	\arrow[""{name=0, anchor=center, inner sep=0}, "{1\circ\Gamma_x}", from=1-1, to=1-3]
	\arrow[""{name=1, anchor=center, inner sep=0}, "{\alpha_f}"{description}, from=1-1, to=2-3]
	\arrow["{\beta_f}", from=1-3, to=2-5]
	\arrow[""{name=2, anchor=center, inner sep=0}, "{\Gamma_y\circ1}"{description}, from=2-3, to=2-5]
	\arrow["{1\circ F\phi}"{description}, from=2-3, to=4-3]
	\arrow["{1\circ F\phi}", from=2-5, to=4-5]
	\arrow["G\phi\circ1"', from=1-1, to=3-1]
	\arrow[""{name=3, anchor=center, inner sep=0}, "{\alpha_{f'}}"', from=3-1, to=4-3]
	\arrow[""{name=4, anchor=center, inner sep=0}, "{\Gamma_y\circ1}"', from=4-3, to=4-5]
	\arrow["{\Gamma_f}"{xshift=0.1cm,yshift=-0.1cm}, shorten <=30pt, shorten >=30pt, Rightarrow, from=0, to=2]
	\arrow[shorten <=20pt, shorten >=20pt, Rightarrow, dashed, from=2, to=4]
	\arrow["{\alpha_\phi}"{xshift=0.1cm}, shorten <=20pt, shorten >=20pt, Rightarrow, from=1, to=3]
\end{tikzcd} $=$
\begin{tikzcd}[ampersand replacement=\&]
	{Gf\alpha_x} \&\& {Gf\beta_x} \\
	\&\&\&\& {\beta_yFf} \\
	{Gf'\alpha_x} \&\& {Gf'\beta_x} \\
	\&\& {\alpha_yFf'} \&\& {\beta_yFf'.}
	\arrow[""{name=0, anchor=center, inner sep=0}, "{1\circ\Gamma_x}", from=1-1, to=1-3]
	\arrow[""{name=1, anchor=center, inner sep=0}, "{\beta_f}", from=1-3, to=2-5]
	\arrow["{1\circ F\phi}", from=2-5, to=4-5]
	\arrow["{\alpha_{f'}}"', from=3-1, to=4-3]
	\arrow[""{name=2, anchor=center, inner sep=0}, "{\Gamma_Y\circ1}"', from=4-3, to=4-5]
	\arrow[""{name=3, anchor=center, inner sep=0}, "{1\circ\Gamma_x}", from=3-1, to=3-3]
	\arrow["G\phi\circ1"{description}, from=1-3, to=3-3]
	\arrow[""{name=4, anchor=center, inner sep=0}, "{\beta_{f'}}"{description}, from=3-3, to=4-5]
	\arrow["G\phi\circ1"', from=1-1, to=3-1]
	\arrow[shorten <=20pt, shorten >=20pt, Rightarrow, dashed, from=0, to=3]
	\arrow["{\beta_\phi}"{xshift=0.1cm}, shorten <=20pt, shorten >=20pt, Rightarrow, from=1, to=4]
	\arrow["{\Gamma_{f'}}"{xshift=0.1cm,yshift=-0.1cm}, shorten <=30pt, shorten >=30pt, Rightarrow, from=3, to=2]
\end{tikzcd}}
\end{center}

\end{enumerate}

\subsection{Perturbations}
\label{subsec:ax-pert}
A perturbation $\theta\colon\Gamma\to\Lambda$ has to satisfy the following equality for any 1-cell $f\colon x\to y$ in $\ca$:
\[\begin{tikzcd}
	{Gf\alpha_x} && {Gf\beta_x} && {Gf\alpha_x} && {Gf\beta_x} \\
	&&& {=} \\
	{\alpha_yFf} && {\beta_yFf} && {\alpha_yFf} && {\beta_yFf.}
	\arrow[""{name=0, anchor=center, inner sep=0}, "{1\circ\Gamma_x}", curve={height=-18pt}, from=1-5, to=1-7]
	\arrow["{\alpha_f}"{description}, from=1-5, to=3-5]
	\arrow["{\beta_f}", from=1-7, to=3-7]
	\arrow[""{name=1, anchor=center, inner sep=0}, "{\Gamma_y\circ1}"{description}, curve={height=-18pt}, from=3-5, to=3-7]
	\arrow[""{name=2, anchor=center, inner sep=0}, "{\Lambda_y\circ1}"', curve={height=18pt}, from=3-5, to=3-7]
	\arrow[""{name=3, anchor=center, inner sep=0}, "{1\circ\Gamma_x}", curve={height=-18pt}, from=1-1, to=1-3]
	\arrow[""{name=4, anchor=center, inner sep=0}, "{1\circ\Lambda_x}"{description}, curve={height=18pt}, from=1-1, to=1-3]
	\arrow["{\beta_f}", from=1-3, to=3-3]
	\arrow["{\alpha_f}"', from=1-1, to=3-1]
	\arrow[""{name=5, anchor=center, inner sep=0}, "{\Lambda_y\circ1}"', curve={height=18pt}, from=3-1, to=3-3]
	\arrow["{\Gamma_f}"{description}, shorten <=17pt, shorten >=17pt, Rightarrow, from=0, to=1]
	\arrow["{\theta_y\circ1}"{description}, shorten <=5pt, shorten >=5pt, Rightarrow, from=1, to=2]
	\arrow["{1\circ\theta_x}"{description}, shorten <=5pt, shorten >=5pt, Rightarrow, from=3, to=4]
	\arrow["{\Gamma_f}"{description}, shorten <=17pt, shorten >=17pt, Rightarrow, from=4, to=5]
\end{tikzcd}\]

\section{The Gray-categorical structure of $\lax(\ca,\cb)$}
\label{app:lax-gray-cat-str}

Let $\ca$ and $\cb$ be Gray-categories.  We now describe the operations equipping $\lax(\ca,\cb)$ with the structure of a Gray-category. 

\subsection{The category of pseudo-functors and lax transformations}
\label{app:lax(A,B)-cat-fct-transf}
Given a pseudomap $F\colon\ca\to\cb$, the identity lax transformation $1_F\colon F\to F\in\lax(\ca,\cb)$ is given by the identity in each component. 

Let $\alpha\colon F\to G$ and $\beta\colon G\to H$ be lax transformations. The lax transformation $\beta\cdot\alpha\colon F\to H$ is defined as follows:
\begin{itemize}
\item at an object $x\in\ca$, $(\beta\cdot\alpha)_x:=\beta_x\alpha_x\colon Fx\to Gx\to Hx$;
\item at a 1-cell $f\colon x\to y$ in $\ca$, the 2-cell $(\beta\cdot\alpha)_f\colon Hf(\beta\cdot\alpha)_x\to(\beta\cdot\alpha)_yFf$ is the composite
\[\begin{tikzcd}[ampersand replacement=\&]
	{Hf(\beta\cdot\alpha)_x=Hf\beta_x\alpha_x} \& {\beta_yGf\alpha_x} \& {\beta_y\alpha_yFf=(\beta\cdot\alpha)_yFf;}
	\arrow["{\beta_f\circ1}", from=1-1, to=1-2]
	\arrow["{1\circ\alpha_f}", from=1-2, to=1-3]
\end{tikzcd}\]
	
\item at a 2-cell $\phi\colon f\to f'$, $(\beta\cdot\alpha)_\phi$ is the composite
\begin{equation}\label{comp-lax-2-cell-part}
\begin{tikzcd}[ampersand replacement=\&]
	{Hf\beta_x\alpha_x} \&\& {\beta_yGf\alpha_x} \&\& {\beta_y\alpha_yFf} \\
	{Hf'\beta_x\alpha_x} \&\& {\beta_yGf'\alpha_x} \&\& {\beta_y\alpha_yFf';}
	\arrow[""{name=0, anchor=center, inner sep=0}, "{\beta_f\circ1}", from=1-1, to=1-3]
	\arrow[""{name=1, anchor=center, inner sep=0}, "{1\circ\alpha_f}", from=1-3, to=1-5]
	\arrow["H\phi\circ1"', from=1-1, to=2-1]
	\arrow["{1\circ G\phi\circ1}"{description}, from=1-3, to=2-3]
	\arrow["{1\circ F\phi}", from=1-5, to=2-5]
	\arrow[""{name=2, anchor=center, inner sep=0}, "{\beta_{f'}\circ1}"', from=2-1, to=2-3]
	\arrow[""{name=3, anchor=center, inner sep=0}, "{1\circ\alpha_{f'}}"', from=2-3, to=2-5]
	\arrow["{\beta_\phi\circ1}"{description}, shorten <=6pt, shorten >=6pt, Rightarrow, from=0, to=2]
	\arrow["{1\circ\alpha_\phi}"{description}, shorten <=6pt, shorten >=6pt, Rightarrow, from=1, to=3]
\end{tikzcd}
\end{equation}

\item at 1-cells $f_1\colon x\to y$ and $f_2\colon y\to z$ in $\ca$, the invertible 3-cell $(\beta\cdot\alpha)^2_{f_2,f_1}$ is the composite
\[\scalebox{0.8}{
\begin{tikzcd}[ampersand replacement=\&]
	{Hf_2Hf_1\beta_x\alpha_x} \& {Hf_2\beta_yGf_1\alpha_x} \& {Hf_2\beta_y\alpha_yFf_1} \& {\beta_zGf_2\alpha_yFf_1} \& {\beta_z\alpha_zFf_2Ff_1} \\
	\& {} \& {\beta_zGf_2Gf_1\alpha_x} \& {} \\
	\& {} \&\& {} \\
	{H(f_2f_1)\beta_x\alpha_x} \&\& {\beta_zG(f_2f_1)\alpha_x} \&\& {\beta_z\alpha_zF(f_2f_1).}
	\arrow["{1\circ\beta_{f_1}\circ1}", from=1-1, to=1-2]
	\arrow["{1\circ\alpha_{f_1}}", from=1-2, to=1-3]
	\arrow["{\beta_{f_2}\circ1}", from=1-3, to=1-4]
	\arrow["{1\circ\alpha_{f_2}\circ1}", from=1-4, to=1-5]
	\arrow["{H^2_{f_2,f_1}\circ1}"', from=1-1, to=4-1]
	\arrow["{\beta_{f_2f_1}\circ1}"', from=4-1, to=4-3]
	\arrow["{1\circ\alpha_{f_2f_1}}"', from=4-3, to=4-5]
	\arrow["{1\circ F^2_{f_2,f_1}}", from=1-5, to=4-5]
	\arrow["{\beta_{f_2}\circ1}"', from=1-2, to=2-3]
	\arrow["{1\circ\alpha_{f_1}}"'{xshift=0.1cm}, from=2-3, to=1-4]
	\arrow[shorten <=3pt, shorten >=3pt, Rightarrow, dashed, from=1-3, to=2-3]
	\arrow["{1\circ G^2_{f_2,f_1}\circ1}"{description}, from=2-3, to=4-3]
	\arrow["{\beta^2_{f_2,f_1}\circ1}"'{xshift=-0.1cm}, Rightarrow, from=2-2, to=3-2]
	\arrow["{1\circ\alpha^2_{f_2,f_1}}", Rightarrow, from=2-4, to=3-4]
\end{tikzcd}
}\]
\end{itemize}

\subsection{Hom-2-categories $\lax(\ca,\cb)[F,G]$}

\subsubsection{Composition of modifications}
\label{app:lax-comp-mod}
Given a lax transformation $\alpha\colon F\to G\in\lax(\ca,\cb)$, the identity modification $1_\alpha\colon\alpha\to\alpha$ is given by the identity in each component. 

Let $\alpha,\beta,\gamma\colon F\to G$ be lax transformations and $\Lambda\colon\alpha\to\beta$ and $\Lambda'\colon\beta\to\gamma$ be modifications. Then the modification $\Lambda'\ast\Lambda\colon\alpha\to\gamma$ is defined by:
\begin{itemize}
	\item at $x\in\ca$, $(\Lambda'\ast\Lambda)_x:=\Lambda'_x\cdot\Lambda_x\colon\alpha_x\to\beta_x\to\gamma_x$;
	\item at $f\colon x\to y\in\ca$, $(\Lambda'\ast\Lambda)_f$ is the composite
\[\begin{tikzcd}[ampersand replacement=\&]
	{Gf\alpha_x} \&\& {Gf\beta_x} \&\& {Gf\gamma_x} \\
	{\alpha_yFf} \&\& {\beta_yFf} \&\& {\gamma_yFf.}
	\arrow["{\alpha_f}"', from=1-1, to=2-1]
	\arrow[""{name=0, anchor=center, inner sep=0}, "{\Lambda_y\circ1}"', from=2-1, to=2-3]
	\arrow[""{name=1, anchor=center, inner sep=0}, "{\Lambda'_y\circ1}"', from=2-3, to=2-5]
	\arrow[""{name=2, anchor=center, inner sep=0}, "{1\circ\Lambda_x}", from=1-1, to=1-3]
	\arrow[""{name=3, anchor=center, inner sep=0}, "{1\circ\Lambda'_x}", from=1-3, to=1-5]
	\arrow["{\gamma_f}", from=1-5, to=2-5]
	\arrow["{\beta_f}"{description}, from=1-3, to=2-3]
	\arrow["{\Lambda_f}"{description}, shorten <=5pt, shorten >=5pt, Rightarrow, from=2, to=0]
	\arrow["{\Lambda'_f}"{description}, shorten <=5pt, shorten >=5pt, Rightarrow, from=3, to=1]
\end{tikzcd}\]
\end{itemize}

\subsubsection{Composition of perturbations}
\label{app:lax-comp-pert}

Given a modification $\Lambda\colon\alpha\to\beta\in\lax(\ca,\cb)[F,G]$, the identity perturbation $1_\Lambda\colon\Lambda\to\Lambda$  has $(1_\Lambda)_x:=1_{\Lambda_x}$ at $x\in\ca$.

Let $\theta\colon\Lambda\to\Gamma$ and $\theta'\colon\Gamma\to\Omega$ be composable perturbations in $\lax(\ca,\cb)$. The perturbation $\theta'\theta\colon\Lambda\to\Omega$ has component, at an object $x \in \ca$, the 3-cell 
$$(\theta'\theta)_x:=\theta'_x\ast\theta_x\colon\Lambda_x\to\Gamma_x\to\Omega_x.$$

\subsubsection{Whiskering of perturbations by modifications}

Consider the following diagram in $\lax(\ca,\cb)[F,G]$.
\[\begin{tikzcd}[ampersand replacement=\&]
	\alpha \& \beta \& \gamma
	\arrow[""{name=0, anchor=center, inner sep=0}, "\Lambda", curve={height=-12pt}, from=1-1, to=1-2]
	\arrow[""{name=1, anchor=center, inner sep=0}, "{\Lambda'}"', curve={height=12pt}, from=1-1, to=1-2]
	\arrow[""{name=2, anchor=center, inner sep=0}, "\Gamma", curve={height=-12pt}, from=1-2, to=1-3]
	\arrow[""{name=3, anchor=center, inner sep=0}, "{\Gamma'}"', curve={height=12pt}, from=1-2, to=1-3]
	\arrow["\nu"{xshift=0.1cm}, shorten <=6pt, shorten >=6pt, Rightarrow, from=2, to=3]
	\arrow["\theta"{xshift=0.1cm}, shorten <=6pt, shorten >=6pt, Rightarrow, from=0, to=1]
\end{tikzcd}\]

\begin{itemize}
\item The whiskering $\nu\ast\Lambda\colon\Gamma\ast\Lambda\to\Gamma'\ast\Lambda$ is given by, at $x\in\ca$, $$(\nu\ast\Lambda)_x:=\nu_x\cdot\Lambda_x\colon\Gamma_x\cdot\Lambda_x\to\Gamma'_x\cdot\Lambda_x.$$

\item The whiskering $\Gamma\ast\theta\colon\Gamma\ast\Lambda\to\Gamma\ast\Lambda'$ is given by, at $x\in\ca$, $$(\Gamma\ast\theta)_x:=\Gamma_x\cdot\theta_x\colon\Gamma_x\cdot\Lambda_x\to\Gamma_x\cdot\Lambda'_x.$$
\end{itemize}

\subsection{Whiskering by lax transformations}

\subsubsection{Whiskering of modifications}
\label{app:lax-whisk-mod-fct}
Consider the following diagram below in $\lax(\ca,\cb)$, 
\[\begin{tikzcd}[ampersand replacement=\&]
	F \& G \& H
	\arrow[""{name=0, anchor=center, inner sep=0}, "\alpha", curve={height=-12pt}, from=1-1, to=1-2]
	\arrow[""{name=1, anchor=center, inner sep=0}, "{\alpha'}"', curve={height=12pt}, from=1-1, to=1-2]
	\arrow[""{name=2, anchor=center, inner sep=0}, "\beta", curve={height=-12pt}, from=1-2, to=1-3]
	\arrow[""{name=3, anchor=center, inner sep=0}, "{\beta'}"', curve={height=12pt}, from=1-2, to=1-3]
	\arrow["\Lambda"{xshift=0.1cm}, shorten <=6pt, shorten >=6pt, Rightarrow, from=0, to=1]
	\arrow["\Gamma"{xshift=0.1cm}, shorten <=6pt, shorten >=6pt, Rightarrow, from=2, to=3]
\end{tikzcd}\]
The whiskerings $\Gamma\cdot\alpha\colon \beta\alpha\to\beta'\alpha$ and $\beta\cdot\Lambda\colon \beta\alpha\to\beta\alpha'$ are defined as, 
\begin{itemize}
\item at $x\in\ca$, $(\Gamma\cdot\alpha)_x:=\Gamma_x\circ \alpha_x \colon\beta_x\alpha_x\to\beta'_x\alpha_x$ and $(\beta\cdot\Lambda)_x:=\beta_x \circ\Lambda_x\colon\beta_x\alpha_x\to\beta_x\alpha'_x$;
	\item at $f\colon x\to y\in\ca$, $(\Gamma\cdot\alpha)_f$ and $(\beta\cdot\Lambda)_f$ are defined as the composites on the left and on the right below, respectively:
\begin{center}
\begin{tikzcd}[ampersand replacement=\&]
	{Hf\beta_x\alpha_x} \&\& {Hf\beta'_x\alpha_x} \\
	{\beta_yGf\alpha_x} \&\& {\beta'_yGf\alpha_x} \\
	{\beta_y\alpha_yFf} \&\& {\beta'_y\alpha_yFf}
	\arrow[""{name=0, anchor=center, inner sep=0}, "{1\circ\Gamma_x\circ1}", from=1-1, to=1-3]
	\arrow["{\beta_f\circ1}"', from=1-1, to=2-1]
	\arrow["{1\circ\alpha_f}"', from=2-1, to=3-1]
	\arrow["{\beta'_f\circ1}", from=1-3, to=2-3]
	\arrow["{1\circ\alpha_f}", from=2-3, to=3-3]
	\arrow[""{name=1, anchor=center, inner sep=0}, "{\Gamma_y\circ1}"{description}, from=2-1, to=2-3]
	\arrow[""{name=2, anchor=center, inner sep=0}, "{\Gamma_y\circ1}"', from=3-1, to=3-3]
	\arrow[shorten <=9pt, shorten >=9pt, Rightarrow, dashed, from=1, to=2]
	\arrow["{\Gamma_f\circ1}"{xshift=0.1cm}, shorten <=9pt, shorten >=9pt, Rightarrow, from=0, to=1]
\end{tikzcd} \hspace{0.5cm} \hspace{0.5cm}
\begin{tikzcd}[ampersand replacement=\&]
	{Hf\beta_x\alpha_x} \&\& {Hf\beta_x\alpha_x} \\
	{\beta_yGf\alpha_x} \&\& {\beta_yGf\alpha'_x} \\
	{\beta_y\alpha_yFf} \&\& {\beta_y\alpha'_yFf.}
	\arrow[""{name=0, anchor=center, inner sep=0}, "{1\circ\Lambda_x}", from=1-1, to=1-3]
	\arrow["{\beta_f\circ1}"', from=1-1, to=2-1]
	\arrow["{1\circ\alpha_f}"', from=2-1, to=3-1]
	\arrow["{\beta_f\circ1}", from=1-3, to=2-3]
	\arrow["{1\circ\alpha'_f}", from=2-3, to=3-3]
	\arrow[""{name=1, anchor=center, inner sep=0}, "{1\circ\Lambda_y\circ1}"', from=3-1, to=3-3]
	\arrow[""{name=2, anchor=center, inner sep=0}, "{1\circ\Lambda_x}"{description}, from=2-1, to=2-3]
	\arrow[shorten <=9pt, shorten >=9pt, Rightarrow, dashed, from=0, to=2]
	\arrow["{1\circ\Lambda_f}"{xshift=0.1cm}, shorten <=9pt, shorten >=9pt, Rightarrow, from=2, to=1]
\end{tikzcd}
\end{center}
\end{itemize}

\subsubsection{Whiskering of perturbations}
Given the diagram in $\lax(\ca,\cb)$ 
\[\begin{tikzcd}[ampersand replacement=\&]
	F \&\& G \&\& H
	\arrow[""{name=0, anchor=center, inner sep=0}, "\alpha", curve={height=-30pt}, from=1-1, to=1-3]
	\arrow[""{name=1, anchor=center, inner sep=0}, "\beta", curve={height=-30pt}, from=1-3, to=1-5]
	\arrow[""{name=2, anchor=center, inner sep=0}, "{\beta'}"', curve={height=30pt}, from=1-3, to=1-5]
	\arrow[""{name=3, anchor=center, inner sep=0}, "{\alpha'}"', curve={height=30pt}, from=1-1, to=1-3]
	\arrow[""{name=4, anchor=center, inner sep=0}, "{\Gamma'}", shift left=5, shorten <=8pt, shorten >=8pt, Rightarrow, from=1, to=2]
	\arrow[""{name=5, anchor=center, inner sep=0}, "\Gamma"', shift right=5, shorten <=8pt, shorten >=8pt, Rightarrow, from=1, to=2]
	\arrow[""{name=6, anchor=center, inner sep=0}, "{\Lambda'}", shift left=5, shorten <=8pt, shorten >=8pt, Rightarrow, from=0, to=3]
	\arrow[""{name=7, anchor=center, inner sep=0}, "\Lambda"', shift right=5, shorten <=8pt, shorten >=8pt, Rightarrow, from=0, to=3]
	\arrow["\nu"{yshift=0.15cm}, shorten <=4pt, shorten >=4pt, Rightarrow, scaling nfold=3,  from=5, to=4] 
	\arrow["\theta"{yshift=0.15cm}, shorten <=4pt, shorten >=4pt, Rightarrow, scaling nfold=3, from=7, to=6]
\end{tikzcd}\] 
\begin{itemize}
\item the whiskering $\nu\cdot\alpha\colon\Gamma\cdot\alpha\to\Gamma'\cdot\alpha$ is defined by, for any $x\in\ca$, $$(\nu\cdot\alpha)_x:=\nu_x\circ\alpha_x\colon\Gamma_x\circ\alpha_x\to\Gamma'_x\circ\alpha_x.$$

\item the whiskering $\beta\cdot\theta\colon\beta\cdot\Lambda\to\beta\cdot\Lambda'$ is defined by, for any $x\in\ca$, $$(\beta\cdot\theta)_x:=\beta_x\circ\theta_x\colon\beta_x\circ\Lambda_x\to\beta_x\circ\Lambda'_x.$$
\end{itemize}

\subsection{Interchange}
Given the diagram below in $\lax(\ca,\cb)$
\[\begin{tikzcd}[ampersand replacement=\&]
	F \& G \& H,
	\arrow[""{name=0, anchor=center, inner sep=0}, "\alpha", curve={height=-12pt}, from=1-1, to=1-2]
	\arrow[""{name=1, anchor=center, inner sep=0}, "{\alpha'}"', curve={height=12pt}, from=1-1, to=1-2]
	\arrow[""{name=2, anchor=center, inner sep=0}, "\beta", curve={height=-12pt}, from=1-2, to=1-3]
	\arrow[""{name=3, anchor=center, inner sep=0}, "{\beta'}"', curve={height=12pt}, from=1-2, to=1-3]
	\arrow["\Lambda"{xshift=0.1cm}, shorten <=6pt, shorten >=6pt, Rightarrow, from=0, to=1]
	\arrow["\Gamma"{xshift=0.1cm}, shorten <=6pt, shorten >=6pt, Rightarrow, from=2, to=3]
\end{tikzcd}\]

the interchange perturbation $\Gamma:\Lambda$ is defined, at $x\in\ca$, as $(\Gamma:\Lambda)_x:=\Gamma_x:\Lambda_x$ in $\cb$. 

\section{The action of $\lax(-,-)$ on pseudomaps}\label{sect:operations}

In this section, given pseudomaps $F\colon\ca\to\ca'$ and $G\colon\cb\to\cb'$, we will spell out the action of the strict map $\lax(F,\cb)\colon\lax(\ca',\cb)\to\lax(\ca,\cb)$ and the pseudomap $\lax(\ca,G)\colon\lax(\ca,\cb)\to\lax(\ca,\cb')$. These two maps are defined using left and right whiskering respectively. 

\subsection{The strict map $\lax(F,\cb)$}
\label{app:lax.F.B}

For any pseudomap $F\colon\ca\to\ca'$, $\lax(F,\cb)$ is a strict map, so it is specified by its action on pseudomaps, lax transformations, modifications and perturbations.  At a pseudomap $G \colon \ca' \to \cb$, we have $\lax(F,\cb)G = GF$.

\subsubsection{Left whiskering of lax transformations}
\label{app:lax.F.B-transf}

Given the following diagram in $\lax(\ca,\cb)$
\[\begin{tikzcd}[ampersand replacement=\&]
	\ca \& {\ca'} \& \cb,
	\arrow["F", from=1-1, to=1-2]
	\arrow[""{name=0, anchor=center, inner sep=0}, "G", curve={height=-12pt}, from=1-2, to=1-3]
	\arrow[""{name=1, anchor=center, inner sep=0}, "{G'}"', curve={height=12pt}, from=1-2, to=1-3]
	\arrow["\beta"{xshift=0.1cm}, shorten <=6pt, shorten >=6pt, Rightarrow, from=0, to=1]
\end{tikzcd}\]
 the lax transformation $\lax(F,\cb)(\beta)=\beta F\colon GF\to G'F$ is defined by:

\begin{itemize}
\item at an object $x\in\ca$, $(\beta F)_x:=\beta_{Fx}\colon GFx\to G'Fx$;

\item at a 1-cell $f\colon x\to y$ in $\ca$, $(\beta F)_f:=\beta_{Ff}\colon G'Ff\beta_{Fx}\to \beta_{Fy}GFf$; 

\item at a 2-cell $\phi\colon f\to f'$ in $\ca$, $(\beta F)_\phi$ is defined as
\[\begin{tikzcd}
	{G'Ff\beta_{Fx}} & {\beta_{Fy}GFf} \\
	{G'Ff'\beta_{Fx}} & {\beta_{Fy}GFf';}
	\arrow[""{name=0, anchor=center, inner sep=0}, "{\beta_{Ff}}", from=1-1, to=1-2]
	\arrow[""{name=1, anchor=center, inner sep=0}, "{\beta_{Ff'}}"', from=2-1, to=2-2]
	\arrow["{G'F\phi\circ1}"', from=1-1, to=2-1]
	\arrow["{1\circ GF\phi}", from=1-2, to=2-2]
	\arrow["{\beta_{F\phi}}"{xshift=0.1cm}, shorten <=8pt, shorten >=8pt, Rightarrow, from=0, to=1]
\end{tikzcd}\]

\item at composable 1-cells $f\colon x\to y$ and $g\colon y\to z$ in $\ca$, the invertible 3-cell $(\beta F)^2_{g,f}$ in $\cb$ is defined as 
\[\begin{tikzcd}
	{G'FgG'Ff\beta_{Fx}} && {G'Fg\beta_{Fy} GFf} && {\beta_{Fz}GFgGFf} \\
	{G'(FgFf)\beta_{Fx}} && {} && {\beta_{Fz}G(FgFf)} \\
	{GF(gf) \beta_{Fx}} && {} && {\beta_{Fz}GF(gf),}
	\arrow["{1\circ\beta_{Ff}}", from=1-1, to=1-3]
	\arrow["{\beta_{Fg}\circ1}", from=1-3, to=1-5]
	\arrow[""{name=0, anchor=center, inner sep=0}, "{\beta_{F(gf)}}"', from=3-1, to=3-5]
	\arrow["{{G'}^2_{Fg,Ff}\circ1}"', from=1-1, to=2-1]
	\arrow["{G({F}^2_{g,f})\circ1}"', from=2-1, to=3-1]
	\arrow["{1\circ G^2_{Fg,Ff}}", from=1-5, to=2-5]
	\arrow["{1\circ G(F^2_{g,f})}", from=2-5, to=3-5]
	\arrow[""{name=1, anchor=center, inner sep=0}, "{\beta_{FgFf}}"{description}, from=2-1, to=2-5]
	\arrow["{\beta_{F^2_{g,f}}}"{xshift=0.1cm}, shorten <=8pt, shorten >=8pt, Rightarrow, from=1, to=0]
	\arrow["{\beta^2_{Fg,Ff}}"{xshift=0.1cm}, shorten <=8pt, shorten >=10pt, Rightarrow, from=1-3, to=1]
\end{tikzcd}\]
	where the 3-cell $\beta_{F^2_{g,f}}$ is invertible because $F^2_{g,f}$ is an invertible 2-cell (this follows from Axiom \eqref{ax:lax-tr.comp-2-cell} of a lax transformation). 
\end{itemize} 

\subsubsection{Left whiskering of modifications}
\label{app:lax.F.B-mod}
Let $\Theta\colon\beta\to\beta'$ be a modification in $\lax(\ca',\cb)$. The modification $\lax(F,\cb)(\Theta)=\Theta F\colon \beta F\to \beta' F$ is given by: 
	\begin{itemize}
	\item at an object $x\in\ca$ the 2-cell $(\Theta F)_x:=\Theta_{Fx}\colon \beta_{Fx}\to \beta'_{Fx}$ in $\cb$;
	\item at a 1-cell $f\colon x\to y$ in $\ca$, the 3-cell $(\Theta F)_f$ in $\cb$ is defined as 
\[\begin{tikzcd}
	{GFf\beta_{Fx}} & {GFf\beta'_{Fx}} \\
	{\beta_{Fy} GFf} & {\beta'_{Fy} GFf.}
	\arrow[""{name=0, anchor=center, inner sep=0}, "{1\circ \Theta_{Fx}}", from=1-1, to=1-2]
	\arrow[""{name=1, anchor=center, inner sep=0}, "{\Theta_{Fy}\circ1}"', from=2-1, to=2-2]
	\arrow["{\beta_{Ff}}"', from=1-1, to=2-1]
	\arrow["{\beta'_{Ff}}", from=1-2, to=2-2]
	\arrow["{\Theta_{Ff}}"{xshift=0.1cm}, shorten <=8pt, shorten >=8pt, Rightarrow, from=0, to=1]
\end{tikzcd}\]
	\end{itemize} 

\subsubsection{Left whiskering of perturbations}
\label{app:lax.F.B-pert}

Let $\delta\colon\Theta\to\Lambda$ be a perturbation in $\lax(\ca',\cb)$. The perturbation $\lax(F,\cb)(\delta)=\delta F\colon \Theta F\to\Lambda F$ is defined, at an object $x\in\ca$, as $(\delta F)_x:=\delta_{Fx}\colon \Theta_{Fx}\to\Lambda_{Fx}$.


\subsection{The pseudomap $\lax(\ca,G)$}
\label{app:lax.A.G}
To describe the pseudomap $\lax(\ca,G)\colon\lax(\ca,\cb)\to\lax(\ca,\cb')$, we must specify its action on pseudomaps, lax transformations, modifications, perturbations and also describe the cocycles $\lax(\ca,G)^2$.  At a pseudomap $F \colon \ca \to \cb$, it is defined by $\lax(\ca,G)(F) = GF$.

\subsubsection{Right whiskering of lax transformations}
\label{app:lax.A.G-transf}
Given the following diagram in $\lax(\ca,\cb')$
\[\begin{tikzcd}[ampersand replacement=\&]
	\ca \& \cb \& {\cb',}
	\arrow[""{name=0, anchor=center, inner sep=0}, "F", curve={height=-12pt}, from=1-1, to=1-2]
	\arrow[""{name=1, anchor=center, inner sep=0}, "{F'}"', curve={height=12pt}, from=1-1, to=1-2]
	\arrow["G", from=1-2, to=1-3]
	\arrow["\alpha"{xshift=0.1cm}, shorten <=6pt, shorten >=6pt, Rightarrow, from=0, to=1]
\end{tikzcd}\]
the lax transformation $\lax(\ca,G)(\alpha)=G\alpha\colon GF\to GF'$ is defined by:
\begin{itemize}
\item at an object $x\in\ca$, $(G\alpha)_x:=G(\alpha_x)\colon GFx\to GF'x$.
\item at a 1-cell $f\colon x\to y$ in $\ca$, $(G\alpha)_f$ is defined as
\[\begin{tikzcd}
	{GF'f\circ G\alpha_x} && {G(F'f\circ \alpha_x)} && {G(\alpha_y\circ Ff)} && {G\alpha_y\circ GFf.}
	\arrow["{G^2_{F'f,\alpha_x}}", from=1-1, to=1-3]
	\arrow["{G(\alpha_f)}", from=1-3, to=1-5]
	\arrow["{{G^2_{\alpha_y, Ff}}^{-1}}", from=1-5, to=1-7]
\end{tikzcd}\]
Note that this composite is just $G(\alpha_f)$ if $G$ is a Gray-functor. 

\item at a 2-cell $\phi\colon f\to f'$ in $\ca$, $(G\alpha)_\phi$ is defined as the 3-cell
\[\begin{tikzcd}
	{GF'f\circ G\alpha_x} & {G(F'f\circ\alpha_x)} && {G(\alpha_y\circ Ff)} & {G\alpha_y\circ GFf} \\
	{GF'f'\circ G\alpha_x} & {G(F'f'\circ\alpha_x)} && {G(\alpha_y\circ Ff')} & {G\alpha_y\circ GFf',}
	\arrow["{{G^2_{\alpha_y,Ff}}^{-1}}", from=1-4, to=1-5]
	\arrow[""{name=0, anchor=center, inner sep=0}, "{G(\alpha_f)}", from=1-2, to=1-4]
	\arrow["{G^2_{F'f,\alpha_x}}", from=1-1, to=1-2]
	\arrow[""{name=1, anchor=center, inner sep=0}, "{1\circ GF\phi}", from=1-5, to=2-5]
	\arrow[""{name=2, anchor=center, inner sep=0}, "{G(1\circ F\phi)}"{description}, from=1-4, to=2-4]
	\arrow["{{G^2_{\alpha_y,Ff'}}^{-1}}"', from=2-4, to=2-5]
	\arrow[""{name=3, anchor=center, inner sep=0}, "{G(F'\phi\circ1)}"{description}, from=1-2, to=2-2]
	\arrow[""{name=4, anchor=center, inner sep=0}, "{GF'\phi\circ1}"', from=1-1, to=2-1]
	\arrow[""{name=5, anchor=center, inner sep=0}, "{G(\alpha_{f'})}"', from=2-2, to=2-4]
	\arrow["{{G^2_{F'f',\alpha_x}}}"', from=2-1, to=2-2]
	\arrow["{}"{description, pos=0.4}, Rightarrow, draw=none, from=1, to=2]
	\arrow["{G(\alpha_\phi)}"{xshift=0.1cm}, shorten <=8pt, shorten >=8pt, Rightarrow, from=0, to=5]
	\arrow["{}"{description, pos=0.6}, Rightarrow, draw=none, from=3, to=4]
\end{tikzcd}\]
where the left and right square commute by Axiom \eqref{ax:psfct.F2-whisk-2} for $GF$ and $GF'$ as pseudomaps;
	\item at composable 1-cells $f\colon x\to y$ and $g\colon y\to z$ in $\ca$, the invertible 3-cell $(G\alpha)^2_{g,f}$ in $\cb'$ is defined as
\[\begin{tikzcd}
	{GF'gGF'fG\alpha_x} \\
	{G(F'gF'f)G\alpha_x} & {G(F'gF'f\alpha_x)} & {G(F'g\alpha_yFf)} & {G(\alpha_zFgFf)} & {G\alpha_zG(FgFf)} \\
	{GF'(gf)G\alpha_x} & {G(F'(gf)\alpha_x)} && {G(\alpha_zF(gf))} & {G\alpha_zGF(gf).}
	\arrow[""{name=0, anchor=center, inner sep=0}, "{G(\alpha_{gf})}"', from=3-2, to=3-4]
	\arrow["{G({F'}^2_{g,f})\circ1}"', from=2-1, to=3-1]
	\arrow["{{G^2}_{F'(gf),\alpha_x}}"', from=3-1, to=3-2]
	\arrow["{{G^2}^{-1}_{\alpha_z,F(gf)}}"', from=3-4, to=3-5]
	\arrow["{1\circ G({F}^2_{g,f})}", from=2-5, to=3-5]
	\arrow["{G^2_{F'g,F'f}\circ1}"', from=1-1, to=2-1]
	\arrow["{G({F'}^2_{g,f}\circ1)}"{description}, from=2-2, to=3-2]
	\arrow["{G(1\circ {F}^2_{g,f})}"{description}, from=2-4, to=3-4]
	\arrow["{G(1\circ \alpha_f)}", from=2-2, to=2-3]
	\arrow["{G(\alpha_g\circ1)}", from=2-3, to=2-4]
	\arrow["{G^2}", from=2-1, to=2-2]
	\arrow["{{G^2}^{-1}}", from=2-4, to=2-5]
	\arrow["{G(\alpha^2_{g,f})}"{xshift=0.1cm}, shorten <=6pt, shorten >=8pt, Rightarrow, from=2-3, to=0]
\end{tikzcd}\]
\end{itemize}

\subsubsection{Right whiskering of modifications}
\label{app:lax.A.G-mod}
Let $\Lambda\colon\alpha\to\alpha'$ be a modification in $\lax(\ca,\cb)$. The modification $\lax(\ca,G)(\Lambda)=G\Lambda\colon G\alpha\to G\alpha'$ is given by: 

	\begin{itemize}
	\item at an object $x\in\ca$ the 2-cell $(G\Lambda)_x:=G\Lambda_x\colon G\alpha_x\to G\alpha_x$ in $\cb'$.
	\item at a 1-cell $f\colon x\to y$ in $\ca$ the 3-cell $(G\Lambda)_f$ in $\cb'$ is defined as
\[\begin{tikzcd}[ampersand replacement=\&]
	{GF'fG\alpha_x} \&\& {GF'fG\alpha'_x} \\
	{G(F'f\alpha_x)} \&\& {G(F'f\alpha'_x)} \\
	{G(\alpha_y Ff)} \&\& {G(\alpha'_y Ff)} \\
	{G\alpha_y GFf} \&\& {G\alpha'_y GFf,}
	\arrow["{1\circ G\Lambda_x}", from=1-1, to=1-3]
	\arrow["{G\Lambda_y\circ1}"', from=4-1, to=4-3]
	\arrow["{G\alpha'_{f}}", from=2-3, to=3-3]
	\arrow["{G\alpha_f}"', from=2-1, to=3-1]
	\arrow[""{name=0, anchor=center, inner sep=0}, "{G(1\circ \Lambda_x)}", from=2-1, to=2-3]
	\arrow[""{name=1, anchor=center, inner sep=0}, "{G(\Lambda_y\circ1)}"', from=3-1, to=3-3]
	\arrow["{G^2}"', from=1-1, to=2-1]
	\arrow["{G^2}", from=1-3, to=2-3]
	\arrow["{{G^2}^{-1}}"', from=3-1, to=4-1]
	\arrow["{{G^2}^{-1}}", from=3-3, to=4-3]
	\arrow["{G(\Lambda_f)}"{xshift=0.1cm}, shorten <=10pt, shorten >=10pt, Rightarrow, from=0, to=1]
\end{tikzcd}\]
where the top and bottom squares commute by Axiom \eqref{ax:psfct.F2-whisk-2} for $G$. 
	\end{itemize}

\subsubsection{Right whiskering of perturbations}
\label{app:lax.A.G-pert}
Let $\theta\colon\Delta\to\Gamma$ be a perturbation in $\lax(\ca,\cb)$. The perturbation $\lax(\ca,G)(\theta)=G\theta\colon G\Delta\to G\Gamma$ is defined, at an object $x\in\ca$, as $(G\theta)_x:=G\theta_x\colon G\Delta_x\to G\Gamma_x$.

\subsubsection{Cocycle for $\lax(\ca,G)$}
\label{app:lax.A.G-cocyle}

At composable 1-cells $\alpha\colon F\to F'$ and $\alpha'\colon F'\to F''$ in $\lax(\ca,\cb)$, the cocyle invertible modification $\lax(\ca,G)^2_{\alpha',\alpha}\colon G\alpha'\cdot G\alpha\to G(\alpha'\cdot \alpha)$ is defined by:

\begin{itemize}
	\item at an object $x\in\ca$, $(\lax(\ca,G)^2_{\alpha',\alpha})_x$ is the invertible 2-cell $$G^2_{\alpha'_x,\alpha_x}\colon G\alpha'_x\circ G\alpha_x\to G(\alpha'_x\circ \alpha_x);$$
	\item at a 1-cell $f\colon x\to y\in\ca$, $(\lax(\ca,G)^2_{\alpha',\alpha})_f$ is the identity. 
\end{itemize}

\section{Axioms for multimaps in $\GCatm$}

\subsection{Axioms for binary maps}
\label{app:ax_bin}

\emph{Compatibility of $A_\beta$ with 3-cells $\Theta\colon \beta\rightarrow\beta'$ in $\cb$}, 
\begin{equation}\label{ax:bin_A:Theta}
\tag{$A:\Theta$}
\adjustbox{scale=0.9}{
\begin{tikzcd}[ampersand replacement=\&]
	{F^{a'}(B)F_b(A)} \& {F_{b'}(A)F^{a}(B)} \\
	\\
	{F^{a'}(B')F_b(A)} \& {F_{b'}(A)F^{a}(B')}
	\arrow[""{name=0, anchor=center, inner sep=0}, "{F^{a'}(\beta)\circ 1}"{description, pos=0.6}, curve={height=-24pt}, from=1-1, to=3-1]
	\arrow["{A_B}", from=1-1, to=1-2]
	\arrow["{A_{B'}}"', from=3-1, to=3-2]
	\arrow[""{name=1, anchor=center, inner sep=0}, "{1\circ F^{a}(\beta)}"{description}, curve={height=-18pt}, from=1-2, to=3-2]
	\arrow[""{name=2, anchor=center, inner sep=0}, "{F^{a'}(\beta')\circ 1}"', curve={height=24pt}, from=1-1, to=3-1]
	\arrow["{F^{a'}(\Theta)\circ 1}"', shorten <=15pt, shorten >=15pt, Rightarrow, from=0, to=2]
	\arrow["{A_{\beta}}"', shorten <=30pt, shorten >=30pt, Rightarrow, from=1, to=0]
\end{tikzcd}
	\hspace{0.25cm}=\hspace{0.25cm}
\begin{tikzcd}[ampersand replacement=\&]
	{F^{a'}(B)F_b(A)} \& {F_{b'}(A)F^{a}(B)} \\
	\\
	{F^{a'}(B')F_b(A)} \& {F_{b'}(A)F^{a}(B')}
	\arrow["{A_B}", from=1-1, to=1-2]
	\arrow[""{name=0, anchor=center, inner sep=0}, "{1\circ F^{a}(\beta)}", curve={height=-24pt}, from=1-2, to=3-2]
	\arrow[""{name=1, anchor=center, inner sep=0}, "{F^{a'}(\beta')\circ 1}"{description}, curve={height=24pt}, from=1-1, to=3-1]
	\arrow[""{name=2, anchor=center, inner sep=0}, "{1\circ F^{a}(\beta')}"{description, pos=0.7}, curve={height=24pt}, from=1-2, to=3-2]
	\arrow["{A_{B'}}"', from=3-1, to=3-2]
	\arrow["{1\circ F^{a'}(\Theta) }"', shorten <=15pt, shorten >=15pt, Rightarrow, from=0, to=2]
	\arrow["{A_{\beta'}}"', shorten <=30pt, shorten >=30pt, Rightarrow, from=2, to=1]
\end{tikzcd}}
\end{equation}
and \emph{of $\alpha_B$  with 3-cells $\Lambda\colon \alpha\rightarrow\alpha'$ in $\ca$},
\begin{equation}\label{ax:bin_Lamd:B}
\tag{$\Lambda:B$}
\adjustbox{scale=0.9}{\begin{tikzcd}
	{F^{a'}(B)F_b(A)} && {F^{a'}(B)F_b(A')} & {F^{a'}(B)F_b(A)} && {F^{a'}(B)F_b(A')} \\
	\\
	{F_{b'}(A)F^{a}(B)} && {F_{b'}(A')F^{a}(B)} & {F_{b'}(A)F^{a}(B)} && {F_{b'}(A')F^{a}(B).}
	\arrow[""{name=0, anchor=center, inner sep=0}, "{1\circ F_b(\alpha)}", curve={height=-18pt}, from=1-1, to=1-3]
	\arrow[""{name=1, anchor=center, inner sep=0}, "{F_{b'}(\alpha)\circ 1}"{description}, curve={height=-18pt}, from=3-1, to=3-3]
	\arrow["{A_B}"', from=1-1, to=3-1]
	\arrow[""{name=2, anchor=center, inner sep=0}, "{A'_B}", from=1-3, to=3-3]
	\arrow[""{name=3, anchor=center, inner sep=0}, "{F_{b'}(\alpha')\circ 1}"', curve={height=18pt}, from=3-1, to=3-3]
	\arrow[""{name=4, anchor=center, inner sep=0}, "{1\circ F_b(\alpha)}", curve={height=-18pt}, from=1-4, to=1-6]
	\arrow[""{name=5, anchor=center, inner sep=0}, "{A_B}"', from=1-4, to=3-4]
	\arrow[""{name=6, anchor=center, inner sep=0}, "{F_{b'}(\alpha')\circ 1}"', curve={height=18pt}, from=3-4, to=3-6]
	\arrow[""{name=7, anchor=center, inner sep=0}, "{1\circ F_{b}(\alpha')}"{description}, curve={height=18pt}, from=1-4, to=1-6]
	\arrow["{A'_B}", from=1-6, to=3-6]
	\arrow["{F_{b'}(\Lambda)\circ1}"{description}, shorten <=5pt, shorten >=4pt, Rightarrow, from=1, to=3]
	\arrow["{\alpha_B}"{xshift=0.1cm}, shorten <=20pt, shorten >=20pt, Rightarrow, from=0, to=1]
	\arrow["{\alpha'_B}"{xshift=0.1cm}, shorten <=20pt, shorten >=20pt, Rightarrow, from=7, to=6]
	\arrow["{1\circ F_{b}(\Lambda)}"{description,yshift=0.1cm}, shorten <=4pt, shorten >=5pt, Rightarrow, from=4, to=7]
	\arrow[shorten <=45pt, shorten >=45pt, Rightarrow, no head, from=2, to=5]
\end{tikzcd}}
\end{equation}
\emph{Compatibility of $A_\beta$ with composition of 2-cells $B\xrightarrow{\beta_1}B'\xrightarrow{\beta_2}B''$ in $\cb$},
\begin{equation}\label{ax:bin_A:bet1bet2}
\tag{$A:\beta_1,\beta_2$}
\begin{tikzcd}[ampersand replacement=\&]
	{F^{a'}(B)F_b(A)} \& {F_b'(A)F^{a}(B)} \\
	{F^{a'}(B')F_b(A)} \& {F_b'(A)F^{a}(B')} \\
	{F^{a'}(B'')F_b(A)} \& {F_b'(A)F^{a}(B'')}
	\arrow["{F^{a'}(\beta_1)\circ 1}"', from=1-1, to=2-1]
	\arrow[""{name=0, anchor=center, inner sep=0}, "{A_B}", from=1-1, to=1-2]
	\arrow[""{name=1, anchor=center, inner sep=0}, "{A_{B'}}"{description}, from=2-1, to=2-2]
	\arrow["{1\circ F^{a'}(\beta_1)}"{description}, from=1-2, to=2-2]
	\arrow["{F^{a'}(\beta_2)\circ 1}"', from=2-1, to=3-1]
	\arrow["{1\circ F^{a'}(\beta_2)}"{description}, from=2-2, to=3-2]
	\arrow[""{name=2, anchor=center, inner sep=0}, "{A_{B''}}"', from=3-1, to=3-2]
	\arrow["{A_{\beta_2}}"', shorten <=9pt, shorten >=9pt, Rightarrow, from=1, to=2]
	\arrow["{A_{\beta_1}}"', shorten <=9pt, shorten >=9pt, Rightarrow, from=0, to=1]
\end{tikzcd}
	\hspace{0.5cm}=\hspace{0.5cm}
\begin{tikzcd}[ampersand replacement=\&]
	{F^{a'}(B)F_b(A)} \& {F_b'(A)F^{a}(B)} \\
	\\
	{F^{a'}(B'')F_b(A)} \& {F_b'(A)F^{a}(B''),}
	\arrow[""{name=0, anchor=center, inner sep=0}, "{A_B}", from=1-1, to=1-2]
	\arrow["{F^{a'}(\beta_2\cdot\beta_1)\circ 1}"{description}, from=1-1, to=3-1]
	\arrow["{1\circ F^{a'}(\beta_2\cdot\beta_1)}", from=1-2, to=3-2]
	\arrow[""{name=1, anchor=center, inner sep=0}, "{A_{B''}}"', from=3-1, to=3-2]
	\arrow["{A_{\beta_2\cdot\beta_1}}", shorten <=17pt, shorten >=17pt, Rightarrow, from=0, to=1]
\end{tikzcd}
\end{equation}
and \emph{of $\alpha_B$ with composition of 2-cells $A\xrightarrow{\alpha_1}A'\xrightarrow{\alpha_2}A''$ in $\ca$},
\begin{equation}\label{ax:bin_alp1alph2:B}
\tag{$\alpha_1,\alpha_2:B$}
\adjustbox{scale=.8}{
\begin{tikzcd}[ampersand replacement=\&]
	{F^{a'}(B)F_b(A)} \&\& {F^{a'}(B)F_b(A')} \&\& {F^{a'}(B)F_b(A'')} \\
	\\
	{F_{b'}(A)F^{a}(B)} \&\& {F_{b'}(A')F^{a}(B)} \&\& {F_{b'}(A'')F^{a}(B)}
	\arrow[""{name=0, anchor=center, inner sep=0}, "{1\circ F_b({\alpha_1})}", from=1-1, to=1-3]
	\arrow[""{name=1, anchor=center, inner sep=0}, "{F_{b'}({\alpha_1})\circ 1}"', from=3-1, to=3-3]
	\arrow["{A_B}"', from=1-1, to=3-1]
	\arrow["{A'_B}"{description}, from=1-3, to=3-3]
	\arrow[""{name=2, anchor=center, inner sep=0}, "{1\circ F_b({\alpha_2})}", from=1-3, to=1-5]
	\arrow[""{name=3, anchor=center, inner sep=0}, "{F_{b'}({\alpha_2})\circ 1}"', from=3-3, to=3-5]
	\arrow["{A''_B}", from=1-5, to=3-5]
	\arrow["{{\alpha_1}_{ B}}"{xshift=0.1cm}, shorten <=17pt, shorten >=17pt, Rightarrow, from=0, to=1]
	\arrow["{{\alpha_2}_{ B}}"{xshift=0.1cm}, shorten <=17pt, shorten >=17pt, Rightarrow, from=2, to=3]
\end{tikzcd}}=
\adjustbox{scale=1}{
\begin{tikzcd}[ampersand replacement=\&]
	{F^{a'}(B)F_b(A)} \&\& {F^{a'}(B)F_b(A'')} \\
	{F_{b'}(A)F^{a}(B)} \&\& {F_{b'}(A'')F^{a}(B).}
	\arrow["{ A_B}"', from=1-1, to=2-1]
	\arrow["{ A''_B}", from=1-3, to=2-3]
	\arrow[""{name=0, anchor=center, inner sep=0}, "{F_{b'}(\alpha_2\cdot\alpha_1)\circ 1}"', from=2-1, to=2-3]
	\arrow[""{name=1, anchor=center, inner sep=0}, "{1\circ F_{b}(\alpha_2\cdot\alpha_1)}", from=1-1, to=1-3]
	\arrow["{{(\alpha_2\cdot\alpha_1)}_g}"{xshift=0.1cm}, shorten <=9pt, shorten >=9pt, Rightarrow, from=1, to=0]
\end{tikzcd}}
\end{equation}

\emph{Compatibility of $A_\beta$ with composition of 1-cells $a\xrightarrow{A_1}a'\xrightarrow{A_2}a''$ in $\ca$},
\begin{equation}\label{ax:bin_A1A2:bet}
\tag{$A_1,A_2:\beta$}
\begin{split}\hspace*{-3.5cm}
\adjustbox{scale=.7}{
\begin{tikzcd}[ampersand replacement=\&]
	\&\& {F^{a''}(B)F_b(A_2A_1)} \\
	{F^{a''}(B)F_b(A_2)F_b(A_1)} \&\&\&\& {F_{b'}(A_2A_1)F^{a}(B)} \\
	\& {F_{b'}(A_2)F^{a'}(B)F_b(A_1)} \&\& {F_{b'}(A_2)F_{b'}(A_1)F^{a}(B)} \\
	{F^{a''}(B')F_b(A_2)F_b(A_1)} \&\&\&\& {F_{b'}(A_2A_1)F^{a}(B)} \\
	\& {F_{b'}(A_2)F^{a'}(B')F_b(A_1)} \&\& {F_{b'}(A_2)F_{b'}(A_1)F^{a}(B')}
	\arrow[""{name=0, anchor=center, inner sep=0}, "{1\circ {A_1}_B}", from=3-2, to=3-4]
	\arrow[""{name=1, anchor=center, inner sep=0}, "{1\circ {A_1}_{B'}}"', from=5-2, to=5-4]
	\arrow["{1\circ F^{a'}(\beta)}"{description}, from=3-4, to=5-4]
	\arrow["{1\circ F^{a'}(\beta)\circ 1}"{description}, from=3-2, to=5-2]
	\arrow[""{name=2, anchor=center, inner sep=0}, "{{A_2}_B\circ1}"{description}, from=2-1, to=3-2]
	\arrow["{F^{a'}(\beta)\circ 1}"', from=2-1, to=4-1]
	\arrow[""{name=3, anchor=center, inner sep=0}, "{{A_2}_{B'}\circ1}"', from=4-1, to=5-2]
	\arrow["{1\circ F^{A_2,A_1}_b}", from=2-1, to=1-3]
	\arrow["{(A_2A_1)_B}", from=1-3, to=2-5]
	\arrow[""{name=4, anchor=center, inner sep=0}, "{F^{A_2,A_1}_{b'}\circ1}"{description}, from=3-4, to=2-5]
	\arrow["{1\circ F^{a'}(\beta)}", from=2-5, to=4-5]
	\arrow[""{name=5, anchor=center, inner sep=0}, "{F^{A_2,A_1}_{b'}\circ1}"{description}, from=5-4, to=4-5]
	\arrow["{1\circ {A_1}_{\beta}}"{description}, shorten <=17pt, shorten >=17pt, Rightarrow, from=0, to=1]
	\arrow["{{A_2}_\beta\circ1}"{description}, shorten <=17pt, shorten >=17pt, Rightarrow, from=2, to=3]
	\arrow[shorten <=17pt, shorten >=17pt, Rightarrow, dashed, from=4, to=5]
	\arrow["{F^{A_2,A_1}_B}"{description}, shorten <=15pt, shorten >=15pt, Rightarrow, from=1-3, to=0]
\end{tikzcd}} \\
\hspace*{-1.5cm} = 
\adjustbox{scale=.7}{
\begin{tikzcd}[ampersand replacement=\&]
	\&\& {F^{a''}(B)F_b(A_2A_1)} \\
	{F^{a''}(B)F_b(A_2)F_b(A_1)} \&\&\&\& {F_{b'}(A_2A_1)F^{a}(B)} \\
	\&\& {F_{b'}(A_2A_1)F^{a}(B)} \\
	{F^{a''}(B')F_b(A_2)F_b(A_1)} \&\&\&\& {F_{b'}(A_2A_1)F^{a}(B).} \\
	\& {F_{b'}(A_2)F^{a'}(B')F_b(A_1)} \&\& {F_{b'}(A_2)F_{b'}(A_1)F^{a}(B')}
	\arrow[""{name=0, anchor=center, inner sep=0}, "{1\circ {A_1}_{B'}}"', from=5-2, to=5-4]
	\arrow["{F^{a'}(\beta)\circ 1}"', from=2-1, to=4-1]
	\arrow["{{A_2}_{B'}\circ1}"', from=4-1, to=5-2]
	\arrow[""{name=1, anchor=center, inner sep=0}, "{1\circ F^{A_2,A_1}_b}", from=2-1, to=1-3]
	\arrow[""{name=2, anchor=center, inner sep=0}, "{(A_2A_1)_B}", from=1-3, to=2-5]
	\arrow["{1\circ F^{a'}(\beta)}", from=2-5, to=4-5]
	\arrow["{F^{A_2,A_1}_{b'}\circ1}"{description}, from=5-4, to=4-5]
	\arrow[""{name=3, anchor=center, inner sep=0}, "{1\circ F^{A_2,A_1}_b}"', from=4-1, to=3-3]
	\arrow["{F^{a'}(\beta)\circ 1}"{description}, from=1-3, to=3-3]
	\arrow[""{name=4, anchor=center, inner sep=0}, "{(A_2A_1)_{B'}}"', from=3-3, to=4-5]
	\arrow["{(A_2A_1)_\beta}"{description}, shorten <=17pt, shorten >=17pt, Rightarrow, from=2, to=4]
	\arrow[shorten <=17pt, shorten >=17pt, Rightarrow, dashed, from=1, to=3]
	\arrow["{F^{A_2,A_1}_{B'}}"{description}, shorten <=15pt, shorten >=15pt, Rightarrow, from=3-3, to=0]
\end{tikzcd}}
\end{split}
\end{equation}
and \emph{of $\alpha_B$ with composition of 1-cells $b\xrightarrow{B_1}b'\xrightarrow{B_2}b''$ in $\cb$,}
\begin{equation}\label{ax:bin_alp:B1B2}
\tag{$\alpha:B_1,B_2$}
\begin{split}
\adjustbox{scale=0.8}{
\begin{tikzcd}[ampersand replacement=\&]
	{F^{a'}(B_2)F^{a'}(B_1)F_b(A)} \& {} \& {F^{a'}(B_2)F^{a'}(B_1)F_b(A')} \\
	\& {F^{a'}(B_2)F_{b'}(A)F^{a}(B_1)} \&\& {F^{a'}(B_2)F_{b'}(A')F^{a}(B_1)} \\
	\\
	\& {F_{b''}(A)F^{a}(B_2)F^{a}(B_1)} \&\& {F_{b''}(A')F^{a}(B_2)F^{a}(B_1)} \\
	{F^{a'}(B_2B_1)F_b(A)} \\
	\& {F_{b''}(A)F^{a}(B_2B_1)} \&\& {F_{b''}(A')F^{a}(B_2B_1)}
	\arrow[""{name=0, anchor=center, inner sep=0}, "{1\circ F_b(\alpha)}", from=1-1, to=1-3]
	\arrow[""{name=1, anchor=center, inner sep=0}, "{1\circ F_{b'}(\alpha)\circ 1}"{description}, from=2-2, to=2-4]
	\arrow[""{name=2, anchor=center, inner sep=0}, "{1\circ A_{B_1}}"{description}, from=1-1, to=2-2]
	\arrow["{1\circ A'_{B_1}}", from=1-3, to=2-4]
	\arrow["{A_{B_2}\circ1}"{description}, from=2-2, to=4-2]
	\arrow["{1\circ F^a_{B_2,B_1}}"{description}, from=4-2, to=6-2]
	\arrow["{F^{a'}_{B_2,B_1}\circ1}"', from=1-1, to=5-1]
	\arrow[""{name=3, anchor=center, inner sep=0}, "{A_{(B_2B_1)}}"', from=5-1, to=6-2]
	\arrow[""{name=4, anchor=center, inner sep=0}, "{F_{b''}(\alpha)\circ 1}"{description}, from=6-2, to=6-4]
	\arrow["{1\circ A'_{B_2}\circ1}"{description}, from=2-4, to=4-4]
	\arrow["{1\circ F^a_{B_2,B_1}}", from=4-4, to=6-4]
	\arrow[""{name=5, anchor=center, inner sep=0}, "{F_{b''}(\alpha)\circ 1}"{description}, from=4-2, to=4-4]
	\arrow["{1\circ\alpha_{B_1}}"{pos=0.7}, shorten <=45pt, shorten >=45pt, Rightarrow, from=0, to=1]
	\arrow["{F^A_{B_2,B_1}}"'{xshift=-0.1cm}, shorten <=45pt, shorten >=45pt, Rightarrow, from=2, to=3]
	\arrow["{\alpha_{B_2}\circ1}"{xshift=0.1cm}, shorten <=20pt, shorten >=20pt, Rightarrow, from=1, to=5]
	\arrow[shorten <=25pt, shorten >=25pt, Rightarrow, dashed, from=5, to=4]
\end{tikzcd}
} \\
\hspace*{1cm}=
\adjustbox{scale=0.8}{
\begin{tikzcd}[ampersand replacement=\&]
	{F^{a'}(B_2)F^{a'}(B_1)F_b(A)} \&\& {F^{a'}(B_2)F^{a'}(B_1)F_b(A')} \\
	\&\&\& {F^{a'}(B_2)F_{b'}(A')F^{a}(B_1)} \\
	\\
	\&\&\& {F_{b''}(A')F^{a}(B_2)F^{a}(B_1)} \\
	{F^{a'}(B_2B_1)F_b(A)} \&\& {F^{a'}(B_2B_1)F_b(A')} \\
	\& {F_{b''}(A)F^{a}(B_2B_1)} \&\& {F_{b''}(A')F^{a}(B_2B_1).}
	\arrow[""{name=0, anchor=center, inner sep=0}, "{1\circ F_b(\alpha)}", from=1-1, to=1-3]
	\arrow[""{name=1, anchor=center, inner sep=0}, "{1\circ A'_{B_1}}", from=1-3, to=2-4]
	\arrow["{F^{a'}_{B_2,B_1}\circ1}"', from=1-1, to=5-1]
	\arrow["{A_{(B_2,B_1)}}"', from=5-1, to=6-2]
	\arrow[""{name=2, anchor=center, inner sep=0}, "{F_{b''}(\alpha)\circ 1}"', from=6-2, to=6-4]
	\arrow["{1\circ A'_{ B_2}\circ1}"{description}, from=2-4, to=4-4]
	\arrow["{1\circ F^a_{B_2,B_1}}", from=4-4, to=6-4]
	\arrow["{F^{a'}_{B_2,B_1}\circ1}"', from=1-3, to=5-3]
	\arrow[""{name=3, anchor=center, inner sep=0}, "{1\circ F_b(\alpha)}", from=5-1, to=5-3]
	\arrow[""{name=4, anchor=center, inner sep=0}, "{A'_{(B_2B_1)}}", from=5-3, to=6-4]
	\arrow[shorten <=45pt, shorten >=45pt, Rightarrow, dashed, from=0, to=3]
	\arrow["{F^{A'}_{B_2,B_1}}"'{xshift=-0.1cm}, shorten <=45pt, shorten >=45pt, Rightarrow, from=1, to=4]
	\arrow["{\alpha_{(B_2B_1)}}"{pos=0.7}, shorten <=45pt, shorten >=45pt, Rightarrow, from=3, to=2]
\end{tikzcd}
}
\end{split}
\end{equation}

\emph{Compatibility of $A_\beta$ and $\alpha_B$:}

\begin{equation}\label{ax:bin_alp:bet}
\tag{$\alpha:\beta$}
\begin{split}
\hspace*{-2cm}\adjustbox{scale=.9}{
\begin{tikzcd}[ampersand replacement=\&]
	{F^{a'}(B)F_b(A)} \&\& {F^{a'}(B)F_b(A')} \\
	\& {F_{b'}(A)F^{a}(B)} \&\& {F_{b'}(A')F^{a}(B)} \\
	{F^{a'}(B')F_b(A)} \\
	\& {F_{b'}(A)F^{a}(B')} \&\& {F_{b'}(A')F^{a}(B')}
	\arrow[""{name=0, anchor=center, inner sep=0}, "{1\circ F_b(\alpha)}", from=1-1, to=1-3]
	\arrow[""{name=1, anchor=center, inner sep=0}, "{F_{b'}(\alpha)\circ 1}"{description}, from=2-2, to=2-4]
	\arrow[""{name=2, anchor=center, inner sep=0}, "{A_B}"{description}, from=1-1, to=2-2]
	\arrow["{A'_B}", from=1-3, to=2-4]
	\arrow["{F^{a'}(\beta)\circ 1}"', from=1-1, to=3-1]
	\arrow["{1\circ F^{a}(\beta)}"{description}, from=2-2, to=4-2]
	\arrow[""{name=3, anchor=center, inner sep=0}, "{F_{b'}(\alpha)\circ 1}"', from=4-2, to=4-4]
	\arrow["{1\circ F^{a}(\beta)}", from=2-4, to=4-4]
	\arrow[""{name=4, anchor=center, inner sep=0}, "{A_{B'}}"', from=3-1, to=4-2]
	\arrow["{\alpha_B}", shorten <=30pt, shorten >=30pt, Rightarrow, from=0, to=1]
	\arrow[shorten <=25pt, shorten >=25pt, Rightarrow, dashed, from=1, to=3]
	\arrow["{A_\beta}"'{xshift=-0.1cm}, shorten <=25pt, shorten >=25pt, Rightarrow, from=2, to=4]
\end{tikzcd}} \\
\hspace*{1cm} =
\adjustbox{scale=.9}{
\begin{tikzcd}[ampersand replacement=\&]
	{F^{a'}(B)F_b(A)} \&\& {F^{a'}(B)F_b(A')} \\
	\&\&\& {F_{b'}(A')F^{a}(B)} \\
	{F^{a'}(B')F_b(A)} \&\& {F^{a'}(B')F_b(A')} \\
	\& {F_{b'}(A)F^{a}(B')} \&\& {F_{b'}(A')F^{a}(B').}
	\arrow[""{name=0, anchor=center, inner sep=0}, "{1\circ F_b(\alpha)}", from=1-1, to=1-3]
	\arrow[""{name=1, anchor=center, inner sep=0}, "{A'_B}", from=1-3, to=2-4]
	\arrow["{F^{a'}(\beta)\circ 1}"', from=1-1, to=3-1]
	\arrow[""{name=2, anchor=center, inner sep=0}, "{F_{b'}(\alpha)\circ 1}"', from=4-2, to=4-4]
	\arrow["{1\circ F^{a}(\beta)}", from=2-4, to=4-4]
	\arrow["{A_{B'}}"', from=3-1, to=4-2]
	\arrow["{F^{a'}(\beta)\circ 1}"', from=1-3, to=3-3]
	\arrow[""{name=3, anchor=center, inner sep=0}, "{1\circ F_b(\alpha)}", from=3-1, to=3-3]
	\arrow[""{name=4, anchor=center, inner sep=0}, "{A'_{B'}}", from=3-3, to=4-4]
	\arrow[shorten <=25pt, shorten >=25pt, Rightarrow, dashed, from=0, to=3]
	\arrow["{A'_\beta}"{xshift=0.1cm}, shorten <=25pt, shorten >=25pt, Rightarrow, from=1, to=4]
	\arrow["{\alpha_{B'}}", shorten <=30pt, shorten >=30pt, Rightarrow, from=3, to=2]
\end{tikzcd}}
\end{split}
\end{equation}

\emph{Associativity for $F^A_{B_2,B_1}$: }

\begin{equation}\label{ax:bin_A:B1B2B3}
\tag{$A:B_1,B_2,B_3$}\hspace*{-4.5cm}
\begin{split}
\begin{tikzcd}[ampersand replacement=\&]
	{F^{a'}_{B_3}F^{a'}_{B_2}F^{a'}_{B_1}F_b^A} \&\& {F^{a'}_{B_3}F^{a'}_{B_2}F_{b'}^AF^{a}_{B_1}} \&\& {F^{a'}_{B_3}F_{b''}^AF^{a}_{B_2}F^{a}_{B_1}} \&\& {F_{b'''}^AF^{a}_{B_3}F^{a}_{B_2}F^{a}_{B_1}} \\
	\\
	{F^{a'}_{B_3}F^{a'}_{B_2B_1}F_b^A} \&\& {} \& {} \& {F^{a'}_{B_3}F_{b''}^AF^{a}_{B_2B_1}} \&\& {F_{b'''}^AF^{a}_{B_3}F^{a}_{B_2B_1}} \\
	\\
	{F^{a'}_{B_3B_2B_1}F_{b}^A} \&\&\& {} \&\&\& {F_{b'''}^AF^{a}_{B_3B_2B_1}}
	\arrow["{1\circ A_{B_1}}", from=1-1, to=1-3]
	\arrow["{1\circ A_{B_2}\circ1}", from=1-3, to=1-5]
	\arrow["{1\circ F^a_{B_2,B_1}}"', from=1-5, to=3-5]
	\arrow["{1\circ F^{a'}_{B_2,B_1}\circ1}"', from=1-1, to=3-1]
	\arrow["{1\circ A_{(B_2B_1)}}"{description}, from=3-1, to=3-5]
	\arrow[""{name=0, anchor=center, inner sep=0}, "{A_{B_3}\circ 1}", from=1-5, to=1-7]
	\arrow[""{name=1, anchor=center, inner sep=0}, "{A_{B_3}\circ 1}"', from=3-5, to=3-7]
	\arrow["{1\circ F^a_{B_2,B_1}}", from=1-7, to=3-7]
	\arrow["{F^{a'}_{B_3,B_2B_1}\circ1}"', from=3-1, to=5-1]
	\arrow["{A_{(B_3B_2B_1)}}"', from=5-1, to=5-7]
	\arrow["{1\circ F^{a}_{B_3,B_2B_1}}", from=3-7, to=5-7]
	\arrow["{F^A_{B_3,B_2B_1}}"', shorten <=13pt, shorten >=13pt, Rightarrow, from=3-4, to=5-4]
	\arrow["{1\circ F^A_{B_2,B_1}}"', shorten <=13pt, shorten >=13pt, Rightarrow, from=1-3, to=3-3]
	\arrow[shorten <=17pt, shorten >=17pt, Rightarrow, dashed, from=0, to=1]
\end{tikzcd}
\\
= 
\begin{tikzcd}[ampersand replacement=\&]
	{F^{a'}_{B_3}F^{a'}_{B_2}F^{a'}_{B_1}F_b^A} \&\& {F^{a'}_{B_3}F^{a'}_{B_2}F_{b'}^AF^{a}_{B_1}} \&\& {F^{a'}_{B_3}F_{b''}^AF^{a}_{B_2}F^{a}_{B_1}} \&\& {F_{b'''}^AF^{a}_{B_3}F^{a}_{B_2}F^{a}_{B_1}} \\
	\\
	{F^{a'}_{B_3B_2}F^{a'}_{B_1}F_b^A} \&\& {F^{a'}_{B_3B_2}F_{b'}^AF^{a}_{B_1}} \& {} \& {} \&\& {F_{b'''}^AF^{a}_{B_3B_2}F^{a}_{B_1}} \\
	\\
	{F^{a'}_{B_3B_2B_1}F_{b}^A} \&\&\& {} \&\&\& {F_{b'''}^AF^{a}_{B_3B_2B_1}}
	\arrow["{A_{(B_3B_2B_1)}}"', from=5-1, to=5-7]
	\arrow["{1\circ F^{a}_{B_3B_2,B_1}}", from=3-7, to=5-7]
	\arrow["{1\circ F^{a}_{B_3,B_2}\circ1}", from=1-7, to=3-7]
	\arrow["{A_{B_3}\circ 1}", from=1-5, to=1-7]
	\arrow["{1\circ A_{B_2}\circ1}", from=1-3, to=1-5]
	\arrow[""{name=0, anchor=center, inner sep=0}, "{1\circ A_{B_1}}", from=1-1, to=1-3]
	\arrow["{F^{a'}_{B_3,B_2}\circ1}"', from=1-1, to=3-1]
	\arrow["{F^{a'}_{B_3B_2,B_1}\circ1}"', from=3-1, to=5-1]
	\arrow["{F^{a'}_{B_3,B_2}\circ1}", from=1-3, to=3-3]
	\arrow[""{name=1, anchor=center, inner sep=0}, "{1\circ A_{B_1}}"', from=3-1, to=3-3]
	\arrow["{A_{(B_3B_2)}\circ1}"', from=3-3, to=3-7]
	\arrow["{F^A_{B_3B_2,B_1}}", shorten <=13pt, shorten >=13pt, Rightarrow, from=3-4, to=5-4]
	\arrow["{F^A_{B_3,B_2}\circ1}", shorten <=13pt, shorten >=13pt, Rightarrow, from=1-5, to=3-5]
	\arrow[shorten <=17pt, shorten >=17pt, Rightarrow, dashed, from=0, to=1]
\end{tikzcd}
\end{split}
\end{equation}

where we note that the left and right sides of the two diagrams coincide by Axiom (\ref{ax:psfct.F2-1-cells}) for the pseudomaps $F^a$ and $F^{a'}$.

\emph{Associativity for $F^{A_2,A_1}_B$:}

\begin{equation}\label{ax:bin_A1A2A3:B}
\tag{$A_1,A_2,A_3:B$}
\begin{split}
\adjustbox{scale=.85}{
\begin{tikzcd}[ampersand replacement=\&]
	{F^{a'''}_BF_b^{A_3}F_b^{A_2}F_b^{A_1}} \&\& {F^{a'''}_BF_b^{A_3}F_b^{A_2A_1}} \&\& {F^{a'''}_BF_b^{A_3A_2A_1}} \\
	{F_{b'}^{A_3}F^{a''}_BF_b^{A_2}F_b^{A_1}} \&\& {F_{b'}^{A_3}F^{a''}_BF_b^{A_2A_1}} \\
	{F_{b'}^{A_3}F_{b'}^{A_2}F^{a'}_BF_b^{A_1}} \\
	{F_{b'}^{A_3}F_{b'}^{A_2}F_{b'}^{A_1}F^{a}_B} \&\& {F_{b'}^{A_3}F_{b'}^{A_2A_1}F^{a}_B} \&\& {F_{b'}^{A_3A_2A_1}F^{a}_B}
	\arrow[""{name=0, anchor=center, inner sep=0}, "{1\circ F^{A_2,A_1}_{b}}", from=1-1, to=1-3]
	\arrow["{{A_3}_B\circ1}", from=1-3, to=2-3]
	\arrow["{1\circ {A_1}_B}"', from=3-1, to=4-1]
	\arrow["{1\circ (A_2A_1)_B}"{pos=0.7}, from=2-3, to=4-3]
	\arrow[""{name=1, anchor=center, inner sep=0}, "{1\circ F^{A_2,A_1}_{b'}\circ1}"', from=4-1, to=4-3]
	\arrow[""{name=2, anchor=center, inner sep=0}, "{1\circ F^{A_3,A_2A_1}_{b'}}", from=1-3, to=1-5]
	\arrow["{(A_3A_2A_1)_B}", from=1-5, to=4-5]
	\arrow[""{name=3, anchor=center, inner sep=0}, "{F^{A_3,A_2A_1}_{b'}\circ1}"', from=4-3, to=4-5]
	\arrow["{{A_3}_B\circ1}"', from=1-1, to=2-1]
	\arrow["{1\circ {A_2}_B\circ1}"', from=2-1, to=3-1]
	\arrow[""{name=4, anchor=center, inner sep=0}, "{1\circ F^{h,f}_{b}}"', shift right=1, from=2-1, to=2-3]
	\arrow[shorten <=9pt, shorten >=9pt, Rightarrow, dashed, from=0, to=4]
	\arrow["{F^{A_3,A_2A_1}_B}", shorten <=26pt, shorten >=26pt, Rightarrow, from=2, to=3]
	\arrow["{1\circ F^{A_2,A_1}_B}", shorten <=17pt, shorten >=17pt, Rightarrow, from=4, to=1]
\end{tikzcd}}\\
= 
\adjustbox{scale=.85}{
\begin{tikzcd}[ampersand replacement=\&]
	{F^{a'''}_BF_b^{A_3}F_b^{A_2}F_b^{A_1}} \&\& {F^{a'''}_BF_b^{A_3A_2}F_b^{A_1}} \&\& {F^{a'''}_BF_b^{A_3A_2A_1}} \\
	{F_{b'}^{A_3}F^{a''}_BF_b^{A_2}F_b^{A_1}} \\
	{F_{b'}^{A_3}F_{b'}^{A_2}F^{a'}_BF_b^{A_1}} \&\& {F_{b'}^{A_3A_2}F^{a'}_BF_b^{A_1}} \\
	{F_{b'}^{A_3}F_{b'}^{A_2}F_{b'}^{A_1}F^{a}_B} \&\& {F_{b'}^{A_3A_2}F_{b'}^{A_1}F^{a}_B} \&\& {F_{b'}^{A_3A_2A_1}F^{a}_B}
	\arrow["{{A_3}_B\circ1}"', from=1-1, to=2-1]
	\arrow["{1\circ {A_2}_B\circ1}"', from=2-1, to=3-1]
	\arrow["{1\circ {A_1}_B}", from=3-3, to=4-3]
	\arrow["{(A_3A_2A_1)_B}", from=1-5, to=4-5]
	\arrow[""{name=0, anchor=center, inner sep=0}, "{F^{A_2,A_1}_{b'}\circ1}", from=3-1, to=3-3]
	\arrow[""{name=1, anchor=center, inner sep=0}, "{F^{A_3A_2,A_1}_{b'}\circ1}"', from=4-3, to=4-5]
	\arrow["{(A_3A_2)_B\circ1}"{pos=0.3}, from=1-3, to=3-3]
	\arrow[""{name=2, anchor=center, inner sep=0}, "{1\circ F^{A_3,A_2}_{b}\circ1}", from=1-1, to=1-3]
	\arrow[""{name=3, anchor=center, inner sep=0}, "{1\circ F_b^{A_3A_2,A_1}}", from=1-3, to=1-5]
	\arrow["{1\circ {A_1}_B}"', from=3-1, to=4-1]
	\arrow[""{name=4, anchor=center, inner sep=0}, "{F^{A_3,A_2}_{b'}\circ1}"', from=4-1, to=4-3]
	\arrow["{F^{A_3,A_2}_B\circ1}", shorten <=17pt, shorten >=17pt, Rightarrow, from=2, to=0]
	\arrow["{F^{A_3A_2,A_1}_B}", shorten <=26pt, shorten >=26pt, Rightarrow, from=3, to=1]
	\arrow[shorten <=9pt, shorten >=9pt, Rightarrow, dashed, from=0, to=4]
\end{tikzcd}
}
\end{split}
\end{equation}
where we note that the top and bottom of the two diagrams in the equation coincide by Axiom (\ref{ax:psfct.F2-1-cells}) for the pseudomaps $F_b$ and $F_{b'}$. \\

\emph{Left whiskering for $F^A_{B_2,B_1}$:}

\begin{equation}\label{ax:bin_A:B1bet}
\tag{$A:B_1,\beta$}
\begin{split}
\adjustbox{scale=.85}{
\begin{tikzcd}[ampersand replacement=\&]
	{F^{a'}_{B_2}F^{a'}_{B_1}F_b^A} \&\& {F^{a'}_{B_2}F_{b'}^AF^{a}_{B_1}} \&\& {F_{b''}^AF^{a}_{B_2}F^{a}_{B_1}} \\
	{F^{a'}_{B_2B_1}F_b^A} \&\&\&\& {F_{b''}^AF^{a}_{B_2B_1}} \\
	{F^{a'}_{B_2'B_1}F_b^A} \&\&\&\& {F_{b''}^AF^{a}_{B_2'B_1}}
	\arrow["{1\circ A_{B_1}}", from=1-1, to=1-3]
	\arrow["{A_{B_2}\circ 1}", from=1-3, to=1-5]
	\arrow["{F^{a'}_{B_2,B_1}\circ1}"', from=1-1, to=2-1]
	\arrow[""{name=0, anchor=center, inner sep=0}, "{A_{(B_2B_1)}}"{description}, from=2-1, to=2-5]
	\arrow["{1\circ F^{a}_{B_2,B_1}}", from=1-5, to=2-5]
	\arrow["{F^{a'}_{\beta\circ B_1}\circ1}"', from=2-1, to=3-1]
	\arrow["{1\circ F^{a}_{\beta\circ B_1}}", from=2-5, to=3-5]
	\arrow[""{name=1, anchor=center, inner sep=0}, "{A_{(B_2'B_1)}}"', from=3-1, to=3-5]
	\arrow["{F^A_{B_2,B_1}}", shorten <=6pt, shorten >=6pt, Rightarrow, from=1-3, to=0]
	\arrow["{A_{\beta\circ B_1}}", shorten <=9pt, shorten >=9pt, Rightarrow, from=0, to=1]
\end{tikzcd}
} \\
= 
\adjustbox{scale=.85}{
\begin{tikzcd}[ampersand replacement=\&]
	{F^{a'}_{B_2}F^{a'}_{B_1}F_b^A} \&\& {F^{a'}_{B_2}F_{b'}^AF^{a}_{B_1}} \&\& {F_{b''}^AF^{a}_{B_2}F^{a}_{B_1}} \\
	{F^{a'}_{B_2'}F^{a'}_{B_1}F_b^A} \&\& {F^{a'}_{B_2'}F_{b'}^AF^{a}_{B_1}} \&\& {F_{b''}^AF^{a}_{B_2'}F^{a}_{B_1}} \\
	{F^{a'}_{B_2'B_1}F_b^A} \&\& {} \&\& {F_{b''}^AF^{a}_{B_2'B_1}}
	\arrow[""{name=0, anchor=center, inner sep=0}, "{1\circ A_{B_1}}"', from=2-1, to=2-3]
	\arrow["{F^{a'}_\beta\circ1}"{description}, from=1-3, to=2-3]
	\arrow[""{name=1, anchor=center, inner sep=0}, "{1\circ A_{B_1}}", from=1-1, to=1-3]
	\arrow["{F^{a'}_\beta\circ1}"', from=1-1, to=2-1]
	\arrow["{F^{a'}_{B_2',B_1}\circ1}"', from=2-1, to=3-1]
	\arrow["{A_{(B_2'B_1)}}"', from=3-1, to=3-5]
	\arrow[""{name=2, anchor=center, inner sep=0}, "{A_{B_2}\circ 1}", from=1-3, to=1-5]
	\arrow[""{name=3, anchor=center, inner sep=0}, "{A_{B_2'}\circ 1}"', from=2-3, to=2-5]
	\arrow["{1\circ F^{a}_\beta\circ1}", from=1-5, to=2-5]
	\arrow["{1\circ F^{a}_{B_2',B_1}}", from=2-5, to=3-5]
	\arrow["{F^A_{B_2',B_1}}", shorten <=4pt, shorten >=4pt, Rightarrow, from=2-3, to=3-3]
	\arrow["{A_\beta\circ1}", shorten <=9pt, shorten >=9pt, Rightarrow, from=2, to=3]
	\arrow[shorten <=9pt, shorten >=9pt, Rightarrow, dashed, from=1, to=0]
\end{tikzcd}
}
\end{split}
\end{equation}
We note that the left and right side of the two diagrams in the equation are equal by Axiom (\ref{ax:psfct.F2-whisk-2}.L) for the pseudomaps $F^a$ and $F^{a'}$. 

\emph{Right whiskering for $F^A_{B_2,B_1}$.}
\begin{equation}\label{ax:bin_A:betB2}
\tag{$A:\beta,B_2$}
\begin{split}
\adjustbox{scale=.85}{
\begin{tikzcd}[ampersand replacement=\&]
	{F^{a'}_{B_2}F^{a'}_{B_1}F_b^A} \&\& {F^{a'}_{B_2}F_{b'}^AF^{a}_{B_1}} \&\& {F_{b''}^AF^{a}_{B_2}F^{a}_{B_1}} \\
	{F^{a'}_{B_2B_1}F_b^A} \&\&\&\& {F_{b''}^AF^{a}_{B_2B_1}} \\
	{F^{a'}_{B_2B_1'}F_b^A} \&\&\&\& {F_{b''}^AF^{a}_{B_2B_1'}}
	\arrow["{1\circ A_{B_1}}", from=1-1, to=1-3]
	\arrow["{A_{B_2}\circ 1}", from=1-3, to=1-5]
	\arrow["{F^{a'}_{B_2,B_1}\circ1}"', from=1-1, to=2-1]
	\arrow[""{name=0, anchor=center, inner sep=0}, "{A_{(B_2B_1)}}"{description}, from=2-1, to=2-5]
	\arrow["{F^{a'}_{B_2\circ\beta}\circ1}"', from=2-1, to=3-1]
	\arrow["{1\circ F^{a}_{B_2\circ\beta}}", from=2-5, to=3-5]
	\arrow[""{name=1, anchor=center, inner sep=0}, "{A_{(B_2B_1')}}"', from=3-1, to=3-5]
	\arrow["{1\circ F^{a}_{B_2,B_1}}", from=1-5, to=2-5]
	\arrow["{F^A_{B_2,B_1}}", shorten <=6pt, shorten >=6pt, Rightarrow, from=1-3, to=0]
	\arrow["{A_{B_2\circ\beta}}", shorten <=9pt, shorten >=9pt, Rightarrow, from=0, to=1]
\end{tikzcd}
} \\
= 
\adjustbox{scale=.85}{
\begin{tikzcd}[ampersand replacement=\&]
	{F^{a'}_{B_2}F^{a'}_{B_1}F_b^A} \&\& {F^{a'}_{B_2}F_{b'}^AF^{a}_{B_1}} \&\& {F_{b''}^AF^{a}_{B_2}F^{a}_{B_1}} \\
	{F^{a'}_{B_2}F^{a'}_{B_1'}F_b^A} \&\& {F^{a'}_{B_2}F_{b'}^AF^{a}_{B_1'}} \&\& {F_{b''}^AF^{a}_{B_2}F^{a}_{B_1'}} \\
	{F^{a'}_{B_2B_1'}F_b^A} \&\& {} \&\& {F_{b''}^AF^{a}_{B_2B_1'}}
	\arrow[""{name=0, anchor=center, inner sep=0}, "{1\circ A_{B_1'}}"', from=2-1, to=2-3]
	\arrow["{1\circ F^{a}_\beta}"{description}, from=1-3, to=2-3]
	\arrow[""{name=1, anchor=center, inner sep=0}, "{1\circ A_{B_1}}", from=1-1, to=1-3]
	\arrow["{1\circ F^{a'}_\beta\circ1}"', from=1-1, to=2-1]
	\arrow["{A_{(B_2B_1')}}"', from=3-1, to=3-5]
	\arrow[""{name=2, anchor=center, inner sep=0}, "{A_{B_2}\circ 1}", from=1-3, to=1-5]
	\arrow[""{name=3, anchor=center, inner sep=0}, "{A_{B_2}\circ 1}"', from=2-3, to=2-5]
	\arrow["{1\circ F^{a}_\beta}", from=1-5, to=2-5]
	\arrow["{F^A_{B_2,B_1'}}", shorten <=4pt, shorten >=4pt, Rightarrow, from=2-3, to=3-3]
	\arrow["{F^{a'}_{B_2,B_1'}\circ1}"', from=2-1, to=3-1]
	\arrow["{1\circ F^{a}_{B_2,B_1'}}", from=2-5, to=3-5]
	\arrow[shorten <=9pt, shorten >=9pt, Rightarrow, dashed, from=2, to=3]
	\arrow["{1\circ A_\beta}"', shorten <=9pt, shorten >=9pt, Rightarrow, from=1, to=0]
\end{tikzcd}
}
\end{split}
\end{equation}
We note that the left and right sides of the two diagrams in the equation are equal by Axiom (\ref{ax:psfct.F2-whisk-2}.R) for the pseudomaps $F^a$ and $F^{a'}$. 

\emph{Left whiskering for $F^{A_2,A_1}_B$:}
\begin{equation}\label{ax:bin_A1alp:B}
\tag{$A_1,\alpha:B$}
\adjustbox{scale=.85}{
\begin{tikzcd}[ampersand replacement=\&]
	{F^{a''}_BF_b^{A_2}F_b^{A_1}} \&\& {F^{a''}_BF_b^{A_2A_1}} \&\& {F^{a''}_BF_b^{A_2'A_1}} \\
	{F_{b'}^{A_2}F^{a'}_BF_{b}^{A_1}} \\
	{F_{b'}^{A_2}F_{b'}^{A_1}F^{a'}_B} \&\& {F_{b'}^{A_2A_1}F^{a}_B} \&\& {F_{b'}^{A_2'A_1}F^{a}_B}
	\arrow["{1\circ {A_1}_B}"', from=2-1, to=3-1]
	\arrow["{(A_2A_1)_B}"{description}, from=1-3, to=3-3]
	\arrow[""{name=0, anchor=center, inner sep=0}, "{F^{A_2,A_1}_{b'}\circ1}"', from=3-1, to=3-3]
	\arrow[""{name=1, anchor=center, inner sep=0}, "{1\circ F^{A_2,A_1}_{b}}", shift right=1, from=1-1, to=1-3]
	\arrow[""{name=2, anchor=center, inner sep=0}, "{1\circ F_b^{\alpha\circ A_1}}", from=1-3, to=1-5]
	\arrow[""{name=3, anchor=center, inner sep=0}, "{F_{b'}^{\alpha\circ A_1}\circ1}"', from=3-3, to=3-5]
	\arrow["{{A_2}_B\circ1}"', from=1-1, to=2-1]
	\arrow["{(A_2'A_1)_B}", from=1-5, to=3-5]
	\arrow["{F^{A_2,A_1}_B}", shorten <=17pt, shorten >=17pt, Rightarrow, from=1, to=0]
	\arrow["{(\alpha\circ A_1)_B}", shorten <=17pt, shorten >=17pt, Rightarrow, from=2, to=3]
\end{tikzcd}
} 
\end{equation}
\begin{equation*}
= 
\adjustbox{scale=.85}{
\begin{tikzcd}[ampersand replacement=\&]
	{F^{a''}_BF_b^{A_2}F_b^{A_1}} \&\& {F^{a''}_BF_b^{A_2'}F_b^{A_1}} \&\& {F^{a''}_BF_b^{A_2'A_1}} \\
	{F_{b'}^{A_2}F^{a'}_BF_{b}^{A_1}} \&\& {F_{b'}^{A_2'}F^{a'}_BF_{b}^{A_1}} \\
	{F_{b'}^{A_2}F_{b'}^{A_1}F^{a'}_B} \&\& {F_{b'}^{A_2'}F_{b'}^{A_1}F^{a'}_B} \&\& {F_{b'}^{A_2'A_1}F^{a}_B}
	\arrow["{1\circ {A_1}_B}", from=2-3, to=3-3]
	\arrow["{(A_2'A_1)_B}", from=1-5, to=3-5]
	\arrow["{{A_2'}_B\circ1}", from=1-3, to=2-3]
	\arrow[""{name=0, anchor=center, inner sep=0}, "{1\circ F^{A_2',A_1}_{b}}", from=1-3, to=1-5]
	\arrow[""{name=1, anchor=center, inner sep=0}, "{F^{A_2',A_1}_{b'}\circ1}"', from=3-3, to=3-5]
	\arrow["{{A_2}_B\circ1}"', from=1-1, to=2-1]
	\arrow[""{name=2, anchor=center, inner sep=0}, "{F_{b'}^\alpha\circ1}"{description}, from=2-1, to=2-3]
	\arrow[""{name=3, anchor=center, inner sep=0}, "{1\circ F_{b'}^\alpha\circ1}", from=1-1, to=1-3]
	\arrow[""{name=4, anchor=center, inner sep=0}, "{F_{b'}^\alpha\circ1}"', from=3-1, to=3-3]
	\arrow["{1\circ {A_1}_B}"', from=2-1, to=3-1]
	\arrow["{F^{A_2',A_1}_B}", shorten <=17pt, shorten >=17pt, Rightarrow, from=0, to=1]
	\arrow[shorten <=9pt, shorten >=9pt, Rightarrow, dashed, from=2, to=4]
	\arrow["{\alpha_B\circ1}", shorten <=9pt, shorten >=9pt, Rightarrow, from=3, to=2]
\end{tikzcd}
}
\end{equation*}
We note that the top and bottom of the two diagrams coincide by Axiom (\ref{ax:psfct.F2-whisk-2}.L) for the pseudomaps $F_b$ and $F_{b'}$. 

\emph{Right whiskering for $F^{A_2,A_1}_B$:}
\begin{equation}\label{ax:bin_alpA2:B}
\tag{$\alpha,A_2:B$}
\begin{split}
\adjustbox{scale=.85}{
\begin{tikzcd}[ampersand replacement=\&]
	{F^{a''}_BF_b^{A_2}F_b^{A_1}} \&\& {F^{a''}_BF_b^{A_2A_1}} \&\& {F^{a''}_BF_b^{A_2A_1'}} \\
	{F_{b'}^{A_2}F^{a'}_BF_{b}^{A_1}} \\
	{F_{b'}^{A_2}F_{b'}^{A_1}F^{a'}_B} \&\& {F_{b'}^{A_2A_1}F^{a}_B} \&\& {F_{b'}^{A_2A_1'}F^{a}_B}
	\arrow["{1\circ {A_1}_B}"', from=2-1, to=3-1]
	\arrow["{(A_2A_1)_B}"{description}, from=1-3, to=3-3]
	\arrow[""{name=0, anchor=center, inner sep=0}, "{F^{A_2,A_1}_{b'}\circ1}"', from=3-1, to=3-3]
	\arrow["{{A_2}_B\circ1}"', from=1-1, to=2-1]
	\arrow[""{name=1, anchor=center, inner sep=0}, "{1\circ F^{A_2,A_1}_{b}}", shift right=1, from=1-1, to=1-3]
	\arrow[""{name=2, anchor=center, inner sep=0}, "{1\circ F_b^{A_2\circ\alpha}}", from=1-3, to=1-5]
	\arrow[""{name=3, anchor=center, inner sep=0}, "{F_{b'}^{A_2\circ\alpha}\circ1}"', from=3-3, to=3-5]
	\arrow["{(A_2A_1')_B}", from=1-5, to=3-5]
	\arrow["{F^{A_2,A_1}_B}", shorten <=17pt, shorten >=17pt, Rightarrow, from=1, to=0]
	\arrow["{(A_2\circ\alpha)_B}", shorten <=17pt, shorten >=17pt, Rightarrow, from=2, to=3]
\end{tikzcd}
} \\
= 
\adjustbox{scale=.85}{
\begin{tikzcd}[ampersand replacement=\&]
	{F^{a''}_BF_b^{A_2}F_b^{A_1}} \&\& {F^{a''}_BF_b^{A_2}F_b^{A_1'}} \&\& {F^{a''}_BF_b^{A_2A_1'}} \\
	{F_{b'}^{A_2}F^{a'}_BF_{b}^{A_1}} \&\& {F_{b'}^{A_2}F^{a'}_BF_{b}^{A_1'}} \\
	{F_{b'}^{A_2}F_{b'}^{A_1}F^{a'}_B} \&\& {F_{b'}^{A_2}F_{b'}^{A_1'}F^{a'}_B} \&\& {F_{b'}^{A_2A_1'}F^{a}_B}
	\arrow["{(A_2A_1')_B}", from=1-5, to=3-5]
	\arrow["{{A_2}_B\circ1}", from=1-3, to=2-3]
	\arrow[""{name=0, anchor=center, inner sep=0}, "{1\circ F^{A_2,A_1'}_{b}}", from=1-3, to=1-5]
	\arrow[""{name=1, anchor=center, inner sep=0}, "{F^{A_2,A_1'}_{b'}\circ1}"', from=3-3, to=3-5]
	\arrow["{{A_2}_B\circ1}"', from=1-1, to=2-1]
	\arrow[""{name=2, anchor=center, inner sep=0}, "{1\circ F_b^\alpha}"{description}, from=2-1, to=2-3]
	\arrow[""{name=3, anchor=center, inner sep=0}, "{1\circ F_{b'}^\alpha}", from=1-1, to=1-3]
	\arrow[""{name=4, anchor=center, inner sep=0}, "{1\circ F_b^\alpha\circ1}"', from=3-1, to=3-3]
	\arrow["{1\circ {A_1}_B}"', from=2-1, to=3-1]
	\arrow["{1\circ {A_1'}_B}", from=2-3, to=3-3]
	\arrow[shorten <=9pt, shorten >=9pt, Rightarrow, dashed, from=3, to=2]
	\arrow["{F^{A_2,A_1'}_B}", shorten <=17pt, shorten >=17pt, Rightarrow, from=0, to=1]
	\arrow["{1\circ\alpha_B}", shorten <=9pt, shorten >=9pt, Rightarrow, from=2, to=4]
\end{tikzcd}
}
\end{split}
\end{equation}
We note that the top and bottom of the two diagrams coincide by Axiom (\ref{ax:psfct.F2-whisk-2}.R) for the pseudomaps $F_b$ and $F_{b'}$. 

\emph{Compatibility of $F^{A_2,A_1}_{B_1}$, $F^{A_2,A_1}_{B_2}$, $F^{A_1}_{B_2,B_1}$ and $F^{A_2}_{B_2,B_1}$:}
\begin{equation}\label{ax:bin_A1A2:B1B2}
\tag{$A_1,A_2:B_1,B_2$}
\adjustbox{scale=0.7}{
\begin{tikzcd}[ampersand replacement=\&]
	{F^{a''}_{B_2}F^{a''}_{B_1}F_b^{A_2}F_b^{A_1}} \&\& {F^{a''}_{B_2}F^{a''}_{B_1}F_b^{A_2A_1}} \\
	\& {F^{a''}_{B_2}F_{b'}^{A_2}F^{a'}_{B_1}F_b^{A_1}} \&\& \textcolor{white}{F^{a''}_{B_2}F_{b'}^{A_2}F^{a'}_{B_1}F_b^{A_1}} \\
	\&\& {F^{a''}_{B_2}F_{b'}^{A_2}F_{b'}^{A_1}F^{a}_{B_1}} \&\& {F^{a''}_{B_2}F_{b'}^{A_2A_1}F^{a}_{B_1}} \\
	\& {F_{b''}^{A_2}F^{a'}_{B_2}F^{a'}_{B_1}F_b^{A_1}} \\
	\&\& {F_{b''}^{A_2}F^{a'}_{B_2}F_{b'}^{A_1}F^{a}_{B_1}} \\
	\\
	{F^{a''}_{B_2B_1}F_{b}^{A_2}F_{b}^{A_1}} \&\& {F_{b''}^{A_2}F_{b''}^{A_1}F^{a}_{B_2}F^{a}_{B_1}} \&\& {F_{b''}^{A_2A_1}F^{a}_{B_2}F^{a}_{B_1}} \\
	\& {F_{b''}^{A_2}F^{a'}_{B_2B_1}F_{b}^{A_1}} \\
	\&\& {F_{b''}^{A_2}F_{b''}^{A_1}F^{a}_{B_2B_1}} \&\& {F_{b''}^{A_2A_1}F^{a}_{B_2B_1}}
	\arrow[""{name=0, anchor=center, inner sep=0}, "{1\circ {A_1}_{B_1}}"{description}, from=2-2, to=3-3]
	\arrow["{1\circ(A_2A_1)_{B_1}}", from=1-3, to=3-5]
	\arrow[""{name=1, anchor=center, inner sep=0}, "{1\circ F^{A_2,A_1}_{b'}\circ1}"', from=3-3, to=3-5]
	\arrow[""{name=2, anchor=center, inner sep=0}, "{1\circ F^{A_2,A_1}_{b}}", shift right=1, from=1-1, to=1-3]
	\arrow[""{name=3, anchor=center, inner sep=0}, "{{A_2}_{B_1}\circ1}"{description}, shift left=1, from=1-1, to=2-2]
	\arrow["{{A_2}_{B_2}\circ1}"{description}, from=3-3, to=5-3]
	\arrow["{1\circ {A_1}_{B_2}\circ1}"{description}, from=5-3, to=7-3]
	\arrow["{1\circ F^{a}_{B_2,B_1}}"{description}, from=7-3, to=9-3]
	\arrow["{{A_2}_{B_2}\circ1}"{description}, from=2-2, to=4-2]
	\arrow[""{name=4, anchor=center, inner sep=0}, "{1\circ {A_1}_{B_1}}"{description}, from=4-2, to=5-3]
	\arrow["{1\circ F^{a'}_{B_2,B_1}\circ1}"{description}, from=4-2, to=8-2]
	\arrow[""{name=5, anchor=center, inner sep=0}, "{1\circ {A_1}_{(B_2B_1)}}"'{pos=0.3}, from=8-2, to=9-3]
	\arrow["{F^{a''}_{B_2,B_1}\circ1}"', from=1-1, to=7-1]
	\arrow[""{name=6, anchor=center, inner sep=0}, "{{A_2}_{(B_2B_1)}\circ1}"'{pos=0.3}, from=7-1, to=8-2]
	\arrow[""{name=7, anchor=center, inner sep=0}, "{F^{A_2,A_1}_{b''}\circ1}"', from=9-3, to=9-5]
	\arrow[""{name=8, anchor=center, inner sep=0}, "{F^{A_2,A_1}_{b''}\circ1}", from=7-3, to=7-5]
	\arrow["{1\circ F^{a}_{B_2,B_1}}", tail reversed, from=7-5, to=9-5]
	\arrow["{(A_2A_1)_{B_2}\circ1}", from=3-5, to=7-5]
	\arrow["{1\circ F^{A_2,A_1}_{B_1}}", shorten <=52pt, shorten >=52pt, Rightarrow, from=2, to=1]
	\arrow["{F^{A_2,A_1}_{B_2}\circ1}", shorten <=34pt, shorten >=34pt, Rightarrow, from=1, to=8]
	\arrow[shorten <=17pt, shorten >=17pt, Rightarrow, dashed, from=8, to=7]
	\arrow["{1\circ F^{A_1}_{B_2,B_1}}"', shorten <=34pt, shorten >=34pt, Rightarrow, from=4, to=5]
	\arrow["{F^{A_2}_{B_2,B_1}\circ1}"', shorten <=52pt, shorten >=52pt, Rightarrow, from=3, to=6]
	\arrow[shorten <=17pt, shorten >=17pt, Rightarrow, dashed, from=0, to=4]
\end{tikzcd}
}
\end{equation}
\begin{equation*}
= 
\adjustbox{scale=0.7}{
\begin{tikzcd}[ampersand replacement=\&]
	{F^{a''}_{B_2}F^{a''}_{B_1}F_b^{A_2}F_b^{A_1}} \&\& {F^{a''}_{B_2}F^{a''}_{B_1}F_b^{A_2A_1}} \\
	\\
	\&\&\& \textcolor{white}{F_{b''}^{A_2}F^{a'}_{B_2B_1}F_{b}^{A_1}} \& {F^{a''}_{B_2}F_{b'}^{A_2A_1}F^{a}_{B_1}} \\
	\\
	\\
	\\
	{F^{a''}_{B_2B_1}F_{b}^{A_2}F_{b}^{A_1}} \&\& {F^{a''}_{B_2B_1}F_{b}^{A_2A_1}} \&\& {F_{b''}^{A_2A_1}F^{a}_{B_2}F^{a}_{B_1}} \\
	\& {F_{b''}^{A_2}F^{a'}_{B_2B_1}F_{b}^{A_1}} \\
	\&\& {F_{b''}^{A_2}F_{b''}^{A_1}F^{a}_{B_2B_1}} \&\& {F_{b''}^{A_2A_1}F^{a}_{B_2B_1}}
	\arrow[""{name=0, anchor=center, inner sep=0}, "{1\circ(A_2A_1)_{B_1}}", from=1-3, to=3-5]
	\arrow[""{name=1, anchor=center, inner sep=0}, "{1\circ F^{A_2,A_1}_{b}}", shift right=1, from=1-1, to=1-3]
	\arrow["{1\circ {A_1}_{(B_2B_1)}}"'{pos=0.2}, from=8-2, to=9-3]
	\arrow["{F^{a''}_{B_2,B_1}\circ1}"', from=1-1, to=7-1]
	\arrow["{{A_2}_{(B_2B_1)}\circ1}"'{pos=0.2}, from=7-1, to=8-2]
	\arrow[""{name=2, anchor=center, inner sep=0}, "{F^{A_2,A_1}_{b''}\circ1}"', from=9-3, to=9-5]
	\arrow["{1\circ F^{a}_{ B_2,B_1}}", from=7-5, to=9-5]
	\arrow["{F^{a''}_{B_2,B_1}\circ1}"', from=1-3, to=7-3]
	\arrow[""{name=3, anchor=center, inner sep=0}, "{(A_2A_1)_{(B_2B_1)}}", from=7-3, to=9-5]
	\arrow[""{name=4, anchor=center, inner sep=0}, "{1\circ F^{A_2,A_1}_{b}}", from=7-1, to=7-3]
	\arrow["{(A_2A_1)_{B_2}\circ1}", from=3-5, to=7-5]
	\arrow["{F^{A_2,A_1}_{B_2B_1}}", shorten <=49pt, shorten >=49pt, Rightarrow, from=4, to=2]
	\arrow[shorten <=51pt, shorten >=51pt, Rightarrow, dashed, from=1, to=4]
	\arrow["{F^{A_2A_1}_{B_2,B_1}}", shorten <=51pt, shorten >=51pt, Rightarrow, from=0, to=3]
\end{tikzcd}
}
\end{equation*}
\emph{Degeneracy equations: }

\begin{minipage}{0.5\textwidth}
\begin{equation}\label{ax:bin_D_A1A2:b}
\tag{D-$A_1,A_2:b$}
F_b^{1_{a'},A}=F_b^{A,1_a}=1_{F^A_b}
\end{equation} 
\begin{equation}\label{ax:bin_D_1a:B}
\tag{D-$1_a:B$} (1_a)_B=1_{F^a_B}
\end{equation}
\begin{equation}\label{ax:bin_D_1A:B}
\tag{D-$1_A:B$} (1_A)_B=1_{A_B}
\end{equation}
\begin{equation}\label{ax:bin_D_1a:bet}
\tag{D-$1_a:\beta$} (1_a)_\beta=1_{F^a_\beta}
\end{equation}
\begin{equation}\label{ax:bin_D_A1A2:B}
\tag{D-$A_1,A_2:B$}
F_B^{1_{a'},A}=F_B^{A,1_a}=1_{A_B}
\end{equation} 
\begin{equation}\label{ax:bin_D_A1A2:1b}
\tag{D-$A_1,A_2:1_b$}
F_{1_b}^{A_2,A_1}=1_{F_{b}^{A_2,A_1}}
\end{equation} 
\end{minipage}
\begin{minipage}{0.5\textwidth}
\begin{equation}\label{ax:bin_D_a:B1B2}
\tag{D-$a:B_1,B_2$}
F^a_{1_{b'},B}=F^a_{B,1_b}=1_{F^a_B}
\end{equation}
\begin{equation}\label{ax:bin_D_A:1b}
\tag{D-$A:1_b$} A_{1_b}=1_{F^A_b}
\end{equation}
\begin{equation}\label{ax:bin_D_alph:1b}
\tag{D-$\alpha:1_b$} \alpha_{1_b}=1_{F^\alpha_b}
\end{equation}
\begin{equation}\label{ax:bin_D_A:1B}
\tag{D-$A:1_B$} A_{1_B}=1_{A_B}
\end{equation}
\begin{equation}\label{ax:bin_D_A:B1B2}
\tag{D-$A:B_1,B_2$}
F^A_{1_{b'},B}=F^A_{B,1_b}=1_{A_B}
\end{equation}
\begin{equation}\label{ax:bin_D_1a:B1B2}
\tag{D-$1_a:B_1,B_2$}
F^{1_a}_{B_2,B_1}=1_{F^{a}_{B_2,B_1}}
\end{equation}
\end{minipage}


\subsection{Axioms for ternary maps}
\label{app:ax_tern}

\emph{Compatibility of $(A\mid B\mid C)$ with composition of 1-cells $A_2\circ A_1$:}
\begin{equation}\label{ax:tern_A1A2:B:C}
\tag{$A_1,A_2:B:C$}\hspace*{-2cm}
\resizebox{19cm}{4.5cm}{
\begin{tikzcd}[ampersand replacement=\&]
	{F^{a''}_{b'}(C)F^{a''}_{c}(B)F^{b}_{c}(A_2)F^{b}_{c}(A_1)} \&\& {F^{a''}_{b'}(C)F^{a''}_{c}(B)F^{b}_{c}(A_2A_1)} \\
	\\
	{F^{a''}_{c'}(B)F^{a''}_{b}(C)F^{b}_{c}(A_2)F^{b}_{c}(A_1)} \& {F^{a''}_{b'}(C)F^{b'}_{c}(A_2)F^{a'}_{c}(B)F^{b}_{c}(A_1)} \&\& \textcolor{white}{F^{a''}_{b'}(C)F^{b'}_{c}(A_2)F^{a'}_{c}(B)F^{b}_{c}(A_1)} \\
	\\
	{F^{a''}_{c'}(B)F^{b}_{c'}(A_2)F^{a'}_{b}(C)F^{b}_{c}(A_1)} \& {F^{b'}_{c'}(A_2)F^{a'}_{b'}(C)F^{a'}_{c}(B)F^{b}_{c}(A_1)} \& {F^{a''}_{b'}(C)F^{b'}_{c}(A_2)F^{b'}_{c}(A_1)F^{a}_{c}(B)} \&\& {F^{a''}_{b'}(C)F^{b'}_{c}(A_2A_1)F^{a}_{c}(B)} \\
	\\
	{F^{a''}_{c'}(B)F^{b}_{c'}(A_2)F^{b}_{c'}(A_2)F^{a}_{b}(C)} \& {F^{b'}_{c'}(A_2)F^{a'}_{c'}(B)F^{a'}_{b}(C)F^{b}_{c}(A_1)} \& {F^{b'}_{c'}(A_2)F^{a'}_{b'}(C)F^{b'}_{c}(A_1)F^{a}_{c}(B)} \\
	\\
	\& {F^{b'}_{c'}(A_2)F^{a'}_{c'}(B)F^{b}_{c'}(A_1)F^{a}_{b}(C)} \& {F^{b'}_{c'}(A_2)F^{b'}_{c'}(A_1)F^{a}_{b'}(C)F^{a}_{c}(B)} \&\& {F^{b'}_{c'}(A_2A_1)F^{a}_{b'}(C)F^{a}_{c}(B)} \\
	\\
	\&\& {F^{b'}_{c'}(A_2)F^{b'}_{c'}(A_1)F^{a}_{c'}(B)F^{a}_{b}(C)} \&\& {F^{b'}_{c'}(A_2A_1)F^{a}_{c'}(B)F^{a}_{b}(C)}
	\arrow[""{name=0, anchor=center, inner sep=0}, "{1\circ (F_c)^{A_2,A_1}_b}", from=1-1, to=1-3]
	\arrow[""{name=1, anchor=center, inner sep=0}, "{1\circ {{A_1}_B}^c}"{description}, from=3-2, to=5-3]
	\arrow[""{name=2, anchor=center, inner sep=0}, "{1\circ(F_c)^{A_2,A_1}_{b'}\circ1}", from=5-3, to=5-5]
	\arrow[""{name=3, anchor=center, inner sep=0}, "{1\circ {{A_2}_B}^{c}\circ1}"{description}, from=1-1, to=3-2]
	\arrow["{{{A_2}_C}^{b'}\circ1}"{description}, from=5-3, to=7-3]
	\arrow["{1\circ {{A_1}_C}^{b'}\circ1}"{description}, from=7-3, to=9-3]
	\arrow["{1\circ{B_C}^{a}}"{description}, from=9-3, to=11-3]
	\arrow["{1\circ{B_C}^{a}}", from=9-5, to=11-5]
	\arrow[""{name=4, anchor=center, inner sep=0}, "{(F\midscript{b'})^{A_2,A_1}_{c'}\circ1}"', from=11-3, to=11-5]
	\arrow["{{{A_2}_C}^{b'}\circ1}"{description}, from=3-2, to=5-2]
	\arrow["{1\circ{B_C}^{a'}}"{description}, from=5-2, to=7-2]
	\arrow[""{name=5, anchor=center, inner sep=0}, "{{{A_2}_B}^{c'}\circ1}"{description}, from=5-1, to=7-2]
	\arrow[""{name=6, anchor=center, inner sep=0}, "{1\circ {{A_1}_B}^c}"{description}, from=5-2, to=7-3]
	\arrow[""{name=7, anchor=center, inner sep=0}, "{(F\midscript{b'})^{A_2,A_1}_{c'}\circ1}", from=9-3, to=9-5]
	\arrow["{1\circ {(A_2A_1)_B}^c}", from=1-3, to=5-5]
	\arrow["{{B_C}^{a''}\circ1}"', from=1-1, to=3-1]
	\arrow["{1\circ {{A_2}_C}^{b}\circ1}"', from=3-1, to=5-1]
	\arrow["{(A_2A_1)^{b'}_C\circ1}", from=5-5, to=9-5]
	\arrow["{1\circ{{A_1}_C}^{b}}"', from=7-2, to=9-2]
	\arrow[""{name=8, anchor=center, inner sep=0}, "{1\circ {{A_1}_B}^{c'}\circ1}"', from=9-2, to=11-3]
	\arrow["{1\circ{{A_1}_C}^{b}}"', from=5-1, to=7-1]
	\arrow[""{name=9, anchor=center, inner sep=0}, "{{{A_2}_B}^{c'}\circ1}"', from=7-1, to=9-2]
	\arrow["{1\circ(F\midscript{b'})^{A_2,A_1}_C\circ1}"{description}, shorten <=34pt, shorten >=34pt, Rightarrow, from=2, to=7]
	\arrow[shorten <=17pt, shorten >=17pt, Rightarrow, dashed, from=7, to=4]
	\arrow["{1\circ (F_c)^{A_2,A_1}_B}"{description}, shorten <=90pt, shorten >=90pt, Rightarrow, from=0, to=2]
	\arrow["{1\circ(A_1\mid B\mid C)}"{description}, shorten <=34pt, shorten >=34pt, Rightarrow, from=6, to=8]
	\arrow["{(A_2\mid B\mid C)\circ1}"{description}, shorten <=26pt, shorten >=26pt, Rightarrow, from=3, to=5]
	\arrow[shorten <=9pt, shorten >=9pt, Rightarrow, dashed, from=5, to=9]
	\arrow[shorten <=17pt, shorten >=17pt, Rightarrow, dashed, from=1, to=6]
\end{tikzcd}
}\end{equation}
\begin{equation*}\hspace*{-1.5cm}
=
\resizebox{19cm}{4.5cm}{
\begin{tikzcd}[ampersand replacement=\&]
	{F^{a''}_{b'}(C)F^{a''}_{c}(B)F^{b}_{c}(A_2)F^{b}_{c}(A_1)} \&\& {F^{a''}_{b'}(C)F^{a''}_{c}(B)F^{b}_{c}(A_2A_1)} \\
	\\
	{F^{a''}_{c'}(B)F^{a''}_{b}(C)F^{b}_{c}(A_2)F^{b}_{c}(A_1)} \&\& {F^{a''}_{c'}(B)F^{a''}_{b}(C)F^{b}_{c}(A_2A_1)} \\
	\\
	{F^{a''}_{c'}(B)F^{b}_{c'}(A_2)F^{a'}_{b}(C)F^{b}_{c}(A_1)} \&\&\& \textcolor{white}{F^{b'}_{c'}(A_2)F^{a'}_{c'}(B)F^{b}_{c'}(A_1)F^{a}_{b}(C)} \& {F^{a''}_{b'}(C)F^{b'}_{c}(A_2A_1)F^{a}_{c}(B)} \\
	\\
	{F^{a''}_{c'}(B)F^{b}_{c'}(A_2)F^{b}_{c'}(A_2)F^{a}_{b}(C)} \&\& {F^{a''}_{c'}(B)F^{b}_{c'}(A_2A_1)F^{a}_{b}(C)} \\
	\\
	\& {F^{b'}_{c'}(A_2)F^{a'}_{c'}(B)F^{b}_{c'}(A_1)F^{a}_{b}(C)} \&\&\& {F^{b'}_{c'}(A_2A_1)F^{a}_{b'}(C)F^{a}_{c}(B)} \\
	\\
	\&\& {F^{b'}_{c'}(A_2)F^{b'}_{c'}(A_1)F^{a}_{c'}(B)F^{a}_{b}(C)} \&\& {F^{b'}_{c'}(A_2A_1)F^{a}_{c'}(B)F^{a}_{b}(C)}
	\arrow[""{name=0, anchor=center, inner sep=0}, "{1\circ (F\midscript{b})^{A_2,A_1}_c}", from=1-1, to=1-3]
	\arrow["{1\circ{B_C}^{a}}", from=9-5, to=11-5]
	\arrow[""{name=1, anchor=center, inner sep=0}, "{(F_{c'})^{A_2,A_1}_{b'}\circ1}"', from=11-3, to=11-5]
	\arrow["{{B_C}^{a''}\circ1}"', from=1-1, to=3-1]
	\arrow["{1\circ {{A_2}_C}^{b}\circ1}"', from=3-1, to=5-1]
	\arrow["{(A_2A_1)^{b'}_C\circ1}", from=5-5, to=9-5]
	\arrow["{1\circ {{A_1}_B}^{c'}\circ1}"', from=9-2, to=11-3]
	\arrow["{1\circ{{A_1}_C}^{b}}"', from=5-1, to=7-1]
	\arrow["{{{A_2}_B}^{c'}\circ1}"', from=7-1, to=9-2]
	\arrow["{{B_C}^{a''}\circ1}"{description}, from=1-3, to=3-3]
	\arrow[""{name=2, anchor=center, inner sep=0}, "{1\circ (F\midscript{b})^{A_2,A_1}_c}"{description}, from=3-1, to=3-3]
	\arrow[""{name=3, anchor=center, inner sep=0}, "{1\circ (F\midscript{b})^{A_2,A_1}_{c'}\circ1 =\\1\circ (F_{c'})^{A_2,A_1}_{b}\circ1}"{description}, from=7-1, to=7-3]
	\arrow[""{name=4, anchor=center, inner sep=0}, "{{(A_2A_1)_B}^{c'}\circ1}"{description}, from=7-3, to=11-5]
	\arrow[""{name=5, anchor=center, inner sep=0}, "{1\circ {(A_2A_1)_B}^c}", from=1-3, to=5-5]
	\arrow["{1\circ{(A_2A_1)_C}^{c}}"{description}, from=3-3, to=7-3]
	\arrow["{(F_{c'})^{A_2,A_1}_B\circ1}"{description}, shorten <=88pt, shorten >=88pt, Rightarrow, from=3, to=1]
	\arrow["{1\circ(F\midscript{b})^{A_2,A_1}_C}"{description}, shorten <=17pt, shorten >=17pt, Rightarrow, from=2, to=3]
	\arrow["{(A_2A_1\mid B\mid C)}"{description}, shorten <=51pt, shorten >=51pt, Rightarrow, from=5, to=4]
	\arrow[shorten <=9pt, shorten >=9pt, Rightarrow, dashed, from=0, to=2]
\end{tikzcd}
}
\end{equation*}

\emph{Compatibility of $(A\mid B\mid C)$ with composition of 1-cells $B_2\circ B_1$:}
\begin{equation}\label{ax:tern_A:B1B2:C}
\tag{$A:B_1,B_2:C$}
\hspace*{-2cm}\resizebox{18.5cm}{4.5cm}{
\begin{tikzcd}[ampersand replacement=\&]
	{F^{a'}_{b''}(C)F^{a'}_{c}(B_2)F^{a'}_{c}(B_1)F^{b}_{c}(A)} \& \textcolor{white}{F^{a'}_{b''}(C)F^{a'}_{c}(B_2)F^{b'}_{c}(A)F^{a}_{c}(B_1)} \& {F^{a'}_{b''}(C)F^{a'}_{c}(B_2)F^{b'}_{c}(A)F^{a}_{c}(B_1)} \\
	\&\&\& {F^{a'}_{b''}(C)F^{b''}_{c}(A)F^{a}_{c}(B_2)F^{a}_{c}(B_1)} \\
	{F^{a'}_{c'}(B_2)F^{a'}_{b'}(C)F^{a'}_{c}(B_1)F^{b}_{c}(A)} \&\& {F^{a'}_{b''}(C)F^{a'}_{c}(B_2B_1)F^{b}_{c}(A)} \&\& {F^{a'}_{b''}(C)F^{b''}_{c}(A)F^{a}_{c}(B_2B_1)} \\
	\\
	{F^{a'}_{c'}(B_2)F^{a'}_{c'}(B_1)F^{a'}_{b}(C)F^{b}_{c}(A)} \\
	\\
	{F^{a'}_{c'}(B_2)F^{a'}_{c'}(B_1)F^{b}_{c'}(A)F^{a}_{b}(C)} \&\& {F^{a'}_{c'}(B_2B_1)F^{a'}_{b}(C)F^{b}_{c}(A)} \&\& {F^{b''}_{c'}(A)F^{a}_{b''}(C)F^{a}_{c}(B_2B_1)} \\
	\\
	\&\& {F^{a'}_{c'}(B_2B_1)F^{b}_{c'}(A)F^{a}_{b}(C)} \&\& {F^{b''}_{c'}(A)F^{a}_{c'}(B_2B_1)F^{a}_{b}(C)}
	\arrow[""{name=0, anchor=center, inner sep=0}, "{1\circ {A_{B_1}}^c}", from=1-1, to=1-3]
	\arrow["{1\circ{A_{B_2}}^c\circ1}"{pos=0.7}, from=1-3, to=2-4]
	\arrow["{1\circ(F_c)^a_{B_2,B_1}}"{pos=0.7}, from=2-4, to=3-5]
	\arrow[""{name=1, anchor=center, inner sep=0}, "{1\circ(F_c)^{a'}_{B_2,B_1}\circ1=1\circ(F^{a'})_c^{B_2,B_1}\circ1}"{description}, from=1-1, to=3-3]
	\arrow[""{name=2, anchor=center, inner sep=0}, "{1\circ {A_{(B_2B_1)}}^c}"{description}, from=3-3, to=3-5]
	\arrow["{{A_C}^{b'}\circ1}", from=3-5, to=7-5]
	\arrow["{1\circ{(B_2B_1)_C}^a}", from=7-5, to=9-5]
	\arrow["{{(B_2B_1)_C}^{a'}\circ1}"{description}, from=3-3, to=7-3]
	\arrow["{1\circ{A_C}^{b}}"{description}, from=7-3, to=9-3]
	\arrow[""{name=3, anchor=center, inner sep=0}, "{{A_{(B_2B_1)}}^{c'}\circ1}"', from=9-3, to=9-5]
	\arrow["{{{B_2}_C}^{a'}\circ1}"', from=1-1, to=3-1]
	\arrow["{1\circ {{B_1}_C}^{a'}\circ1}"{description}, from=3-1, to=5-1]
	\arrow[""{name=4, anchor=center, inner sep=0}, "{(F^{a'})_c^{B_2,B_1}\circ1}"{description}, from=5-1, to=7-3]
	\arrow["{1\circ{A_C}^{b}}"{description}, from=5-1, to=7-1]
	\arrow[""{name=5, anchor=center, inner sep=0}, "{(F^{a'})_c^{B_2,B_1}\circ1}"{description}, from=7-1, to=9-3]
	\arrow["{1\circ (F_c)^A_{B_2,B_1}}"{description}, shorten <=66pt, shorten >=66pt, Rightarrow, from=0, to=2]
	\arrow["{(A\mid B_2B_1\mid C)}"{description}, shorten <=51pt, shorten >=51pt, Rightarrow, from=2, to=3]
	\arrow[shorten <=17pt, shorten >=17pt, Rightarrow, dashed, from=4, to=5]
	\arrow["{(F^{a'})^{B_2,B_1}_C\circ1}"{description}, shorten <=26pt, shorten >=26pt, Rightarrow, from=1, to=4]
\end{tikzcd}
} = 
\end{equation}
\begin{equation*}
\hspace*{-2cm}
\resizebox{19cm}{4.5cm}{
\begin{tikzcd}[ampersand replacement=\&]
	{F^{a'}_{b''}(C)F^{a'}_{c}(B_2)F^{a'}_{c}(B_1)F^{b}_{c}(A)} \&\&\& {F^{a'}_{b''}(C)F^{a'}_{c}(B_2)F^{b'}_{c}(A)F^{a}_{c}(B_1)} \\
	\&\&\&\& {F^{a'}_{b''}(C)F^{b''}_{c}(A)F^{a}_{c}(B_2)F^{a}_{c}(B_1)} \\
	{F^{a'}_{c'}(B_2)F^{a'}_{b'}(C)F^{a'}_{c}(B_1)F^{b}_{c}(A)} \&\&\& {F^{a'}_{c'}(B_2)F^{a'}_{b'}(C)F^{b'}_{c}(A)F^{a}_{c}(B_1)} \&\& {F^{a'}_{b''}(C)F^{b''}_{c}(A)F^{a}_{c}(B_2B_1)} \\
	\&\&\&\& {F^{b''}_{c'}(A)F^{a}_{b''}(C)F^{a}_{c}(B_2)F^{a}_{c}(B_1)} \\
	{F^{a'}_{c'}(B_2)F^{a'}_{c'}(B_1)F^{a'}_{b}(C)F^{b}_{c}(A)} \&\&\& {F^{a'}_{c'}(B_2)F^{b'}_{c'}(A)F^{a}_{b'}(C)F^{a}_{c}(B_1)} \&\& {F^{b''}_{c'}(A)F^{a}_{b''}(C)F^{a}_{c}(B_2B_1)} \\
	\&\&\&\& {F^{b''}_{c'}(A)F^{a}_{b''}(C)F^{a}_{c}(B_2)F^{a}_{c}(B_1)} \\
	{F^{a'}_{c'}(B_2)F^{a'}_{c'}(B_1)F^{b}_{c'}(A)F^{a}_{b}(C)} \&\&\& {F^{a'}_{c'}(B_2)F^{b'}_{c'}(A)F^{a}_{c'}(B_1)F^{a}_{b}(C)} \\
	\&\&\&\& {F^{b''}_{c'}(A)F^{a}_{c'}(B_2)F^{a}_{c'}(B_1)F^{a}_{b}(C)} \\
	\&\& {F^{a'}_{c'}(B_2B_1)F^{b}_{c'}(A)F^{a}_{b}(C)} \&\&\& {F^{b''}_{c'}(A)F^{a}_{c'}(B_2B_1)F^{a}_{b}(C)}
	\arrow[""{name=0, anchor=center, inner sep=0}, "{1\circ {A_{B_1}}^c}", from=1-1, to=1-4]
	\arrow[""{name=1, anchor=center, inner sep=0}, "{1\circ{A_{B_2}}^c\circ1}"{pos=0.7}, from=1-4, to=2-5]
	\arrow[""{name=2, anchor=center, inner sep=0}, "{1\circ(F_c)^a_{B_2,B_1}}"{pos=0.7}, from=2-5, to=3-6]
	\arrow["{{A_C}^{b'}\circ1}", from=3-6, to=5-6]
	\arrow["{1\circ{(B_2B_1)_C}^a}", from=5-6, to=9-6]
	\arrow[""{name=3, anchor=center, inner sep=0}, "{{A_{(B_2B_1)}}^{c'}\circ1}"', from=9-3, to=9-6]
	\arrow["{{{B_2}_C}^{a'}\circ1}"', from=1-1, to=3-1]
	\arrow["{1\circ {{B_1}_C}^{a'}\circ1}"{description}, from=3-1, to=5-1]
	\arrow["{1\circ{A_C}^{b}}"{description}, from=5-1, to=7-1]
	\arrow["{(F^{a'})_c^{B_2,B_1}\circ1}"{description}, from=7-1, to=9-3]
	\arrow[""{name=4, anchor=center, inner sep=0}, "{1\circ {A_{B_1}}^c}"{description}, from=3-1, to=3-4]
	\arrow["{{{B_2}_C}^{a'}\circ1}"', from=1-4, to=3-4]
	\arrow[""{name=5, anchor=center, inner sep=0}, "{1\circ {A_{B_1}}^{c'}\circ1}"{description}, from=7-1, to=7-4]
	\arrow["{1\circ{A_C}^{b'}\circ1}"{description}, from=3-4, to=5-4]
	\arrow["{1\circ {{B_1}_C}^{a}}"{description}, from=5-4, to=7-4]
	\arrow["{{A_C}^{b'}\circ1}"{description}, from=2-5, to=4-5]
	\arrow["{1\circ{{B_2}_C}^{a}\circ1}"{description}, from=4-5, to=6-5]
	\arrow[""{name=6, anchor=center, inner sep=0}, "{{A_{B_2}}^{c'}\circ1}"{description}, from=5-4, to=6-5]
	\arrow[""{name=7, anchor=center, inner sep=0}, "{{A_{B_2}}^{c'}\circ1}"{description}, from=7-4, to=8-5]
	\arrow["{1\circ {{B_1}_C}^{a}}"{description}, from=6-5, to=8-5]
	\arrow[""{name=8, anchor=center, inner sep=0}, "{1\circ (F^a)_{c'}^{B_2,B_1}\circ1}"{description}, from=8-5, to=9-6]
	\arrow[""{name=9, anchor=center, inner sep=0}, "{1\circ(F_c)^a_{B_2,B_1}}"{description}, from=4-5, to=5-6]
	\arrow[shorten <=9pt, shorten >=9pt, Rightarrow, dashed, from=0, to=4]
	\arrow["{1\circ(A\mid B_1\mid C)}"{description}, shorten <=26pt, shorten >=26pt, Rightarrow, from=4, to=5]
	\arrow["{(A\mid B_2\mid C)\circ1}"{description}, shorten <=26pt, shorten >=26pt, Rightarrow, from=1, to=6]
	\arrow["{1\circ(F^a)^{B_2,B_1}_C}"{description}, shorten <=26pt, shorten >=26pt, Rightarrow, from=9, to=8]
	\arrow[shorten <=9pt, shorten >=9pt, Rightarrow, dashed, from=6, to=7]
	\arrow[shorten <=9pt, shorten >=9pt, Rightarrow, dashed, from=2, to=9]
	\arrow["{(F_{c'})^A_{B_2,B_1}}"{description}, shorten <=52pt, shorten >=52pt, Rightarrow, from=5, to=3]
\end{tikzcd}
}
\end{equation*}

\emph{Compatibility of $(A\mid B\mid C)$ with composition of 1-cells $C_2\circ C_1$:}
\begin{equation}\label{ax:tern_A:B:C1C2}
\tag{$A:B:C_1,C_2$}
\hspace*{-2.5cm}\resizebox{19.5cm}{4cm}{
\begin{tikzcd}[ampersand replacement=\&]
	{F^{a'}_{b'}(C_2)F^{a'}_{b'}(C_1)F^{a'}_{c}(B)F^{b}_{c}(A)} \&\& {F^{a'}_{b'}(C_2)F^{a'}_{b'}(C_1)F^{b'}_{c}(A)F^{a}_{c}(B)} \\
	\& {F^{a'}_{b'}(C_2)F^{a'}_{c'}(B)F^{a'}_{b}(C_1)F^{b}_{c}(A)} \&\& {F^{a'}_{b'}(C_2)F^{b'}_{c'}(A)F^{a}_{b'}(C_1)F^{a}_{c}(B)} \\
	\&\& {F^{a'}_{b'}(C_2)F^{a'}_{c'}(B)F^{b}_{c'}(A)F^{a}_{b}(C_1)} \&\& {F^{a'}_{b'}(C_2)F^{b'}_{c'}(A)F^{a}_{c'}(B)F^{a}_{b}(C_1)} \\
	\& {F^{a'}_{c''}(B)F^{a'}_{b}(C_2)F^{a'}_{b}(C_1)F^{b}_{c}(A)} \\
	\&\& {F^{a'}_{c''}(B)F^{a'}_{b}(C_2)F^{b}_{c'}(A)F^{a}_{b}(C_1)} \&\& {F^{b'}_{c''}(A)F^{a}_{b'}(C_2)F^{a}_{c'}(B)F^{a}_{b}(C_1)} \\
	\\
	{F^{a'}_{b'}(C_2C_1)F^{a'}_{c}(B)F^{b}_{c}(A)} \&\& {F^{a'}_{c''}(B)F^{b}_{c''}(A)F^{a}_{b}(C_2)F^{a}_{b}(C_1)} \&\& {F^{b'}_{c''}(A)F^{a}_{c''}(B)F^{a}_{b}(C_2)F^{a}_{b}(C_1)} \\
	\& {F^{a'}_{c''}(B)F^{a'}_{b}(C_2C_1)F^{b}_{c}(A)} \\
	\&\& {F^{a'}_{c''}(B)F^{b}_{c''}(A)F^{a}_{b}(C_2C_1)} \&\& {F^{b'}_{c''}(A)F^{a}_{c''}(B)F^{a}_{b}(C_2C_1)}
	\arrow[""{name=0, anchor=center, inner sep=0}, "{1\circ {A_B}^c}", from=1-1, to=1-3]
	\arrow["{1\circ{A_{C_1}}^{b'}\circ1}", from=1-3, to=2-4]
	\arrow["{1\circ {B_{C_1}}^{a}}", from=2-4, to=3-5]
	\arrow[""{name=1, anchor=center, inner sep=0}, "{1\circ {A_{C_1}}^b}"{description}, from=2-2, to=3-3]
	\arrow[""{name=2, anchor=center, inner sep=0}, "{1\circ{A_B}^{c'}\circ1}"{description}, from=3-3, to=3-5]
	\arrow[""{name=3, anchor=center, inner sep=0}, "{1\circ {B_{C_1}}^{a'}\circ1}"{description}, from=1-1, to=2-2]
	\arrow["{{B_{C_2}}^{a'}\circ1}"{description}, from=3-3, to=5-3]
	\arrow["{1\circ{A_{C_2}}^b\circ1}"{description}, from=5-3, to=7-3]
	\arrow["{1\circ(F\midscript{b})^a_{C_2,C_1}}"{description}, from=7-3, to=9-3]
	\arrow["{{A_{C_2}}^{b'}\circ1}", from=3-5, to=5-5]
	\arrow["{1\circ{B_{C_2}}^a\circ1}", from=5-5, to=7-5]
	\arrow["{1\circ(F\midscript{b})^a_{C_2,C_1}}", from=7-5, to=9-5]
	\arrow[""{name=4, anchor=center, inner sep=0}, "{{A_B}^{c''}\circ1}"', from=9-3, to=9-5]
	\arrow["{(F^{a'})^{b}_{C_2,C_1}\circ1}"', from=1-1, to=7-1]
	\arrow["{{B_{C_2}}^{a'}\circ1}"{description}, from=2-2, to=4-2]
	\arrow["{1\circ (F^{a'})^b_{C_2,C_1}\circ1}"', from=4-2, to=8-2]
	\arrow[""{name=5, anchor=center, inner sep=0}, "{{B_{(C_2C_1)}}^{a'}\circ1}"', from=7-1, to=8-2]
	\arrow[""{name=6, anchor=center, inner sep=0}, "{1\circ{A_{(C_2C_1)}}^{b}}"', from=8-2, to=9-3]
	\arrow[""{name=7, anchor=center, inner sep=0}, "{1\circ {A_{C_1}}^b}"', from=4-2, to=5-3]
	\arrow[""{name=8, anchor=center, inner sep=0}, "{{A_B}^{c''}\circ1}"{description}, from=7-3, to=7-5]
	\arrow["{(A\mid B\mid C_2)\circ1}", shorten <=34pt, shorten >=34pt, Rightarrow, from=2, to=8]
	\arrow[shorten <=17pt, shorten >=17pt, Rightarrow, dashed, from=1, to=7]
	\arrow[shorten <=17pt, shorten >=17pt, Rightarrow, dashed, from=8, to=4]
	\arrow["{(F^{a'})^B_{C_2,C_1}\circ1}"', shorten <=51pt, shorten >=51pt, Rightarrow, from=3, to=5]
	\arrow["{1\circ(A\mid B\mid C_1)}", shorten <=89pt, shorten >=89pt, Rightarrow, from=0, to=2]
	\arrow["{1\circ(F\midscript{b})^A_{C_2,C_1}}"', shorten <=34pt, shorten >=34pt, Rightarrow, from=7, to=6]
\end{tikzcd}
}
= 
\end{equation}
\begin{equation*}
\hspace*{-2cm}\resizebox{19.5cm}{4cm}{
\begin{tikzcd}[ampersand replacement=\&]
	{F^{a'}_{b'}(C_2)F^{a'}_{b'}(C_1)F^{a'}_{c}(B)F^{b}_{c}(A)} \&\& {F^{a'}_{b'}(C_2)F^{a'}_{b'}(C_1)F^{b'}_{c}(A)F^{a}_{c}(B)} \\
	\&\&\& {F^{a'}_{b'}(C_2)F^{b'}_{c'}(A)F^{a}_{b'}(C_1)F^{a}_{c}(B)} \\
	\&\&\&\& {F^{a'}_{b'}(C_2)F^{b'}_{c'}(A)F^{a}_{c'}(B)F^{a}_{b}(C_1)} \\
	\&\&\& {F^{b'}_{c''}(A)F^{a}_{b'}(C_2)F^{a}_{b'}(C_1)F^{a}_{c}(B)} \\
	\&\&\&\& {F^{b'}_{c''}(A)F^{a}_{b'}(C_2)F^{a}_{c'}(B)F^{a}_{b}(C_1)} \\
	\\
	{F^{a'}_{b'}(C_2C_1)F^{a'}_{c}(B)F^{b}_{c}(A)} \&\& {F^{a'}_{b'}(C_2C_1)F^{b'}_{c}(A)F^{a}_{c}(B)} \&\& {F^{b'}_{c''}(A)F^{a}_{c''}(B)F^{a}_{b}(C_2)F^{a}_{b}(C_1)} \\
	\& {F^{a'}_{c''}(B)F^{a'}_{b}(C_2C_1)F^{b}_{c}(A)} \&\& {F^{b'}_{c''}(A)F^{a}_{b'}(C_2C_1)F^{a}_{c}(B)} \\
	\&\& {F^{a'}_{c''}(B)F^{b}_{c''}(A)F^{a}_{b}(C_2C_1)} \&\& {F^{b'}_{c''}(A)F^{a}_{c''}(B)F^{a}_{b}(C_2C_1)}
	\arrow[""{name=0, anchor=center, inner sep=0}, "{1\circ {A_B}^c}", from=1-1, to=1-3]
	\arrow[""{name=1, anchor=center, inner sep=0}, "{1\circ{A_{C_1}}^{b'}\circ1}", from=1-3, to=2-4]
	\arrow[""{name=2, anchor=center, inner sep=0}, "{1\circ {B_{C_1}}^{a}}", from=2-4, to=3-5]
	\arrow["{{A_{C_2}}^{b'}\circ1}", from=3-5, to=5-5]
	\arrow["{1\circ(F\midscript{b})^a_{C_2,C_1}}", from=7-5, to=9-5]
	\arrow[""{name=3, anchor=center, inner sep=0}, "{{ A_B}^{c''}\circ1}"', from=9-3, to=9-5]
	\arrow["{(F^{a'})^{b}_{C_2,C_1}\circ1=(F\midscript{b})^{a'}_{C_2,C_1}\circ1}"{description}, from=1-1, to=7-1]
	\arrow["{{B_{(C_2C_1)}}^{a'}\circ1}"', from=7-1, to=8-2]
	\arrow["{1\circ{A_{(C_2C_1)}}^{b}}"', from=8-2, to=9-3]
	\arrow["{1\circ{B_{C_2}}^a\circ1}", from=5-5, to=7-5]
	\arrow["{(F\midscript{b})^{a'}_{C_2,C_1}\circ1}"', from=1-3, to=7-3]
	\arrow[""{name=4, anchor=center, inner sep=0}, "{1\circ {A_B}^c}", from=7-1, to=7-3]
	\arrow["{{A_{C_2}}^{b'}\circ1}"', from=2-4, to=4-4]
	\arrow["{1\circ(F^a)^{b'}_{C_2,C_1}}"{description}, from=4-4, to=8-4]
	\arrow[""{name=5, anchor=center, inner sep=0}, "{{A_{(C_2C_1)}}^{b'}\circ1}"{description}, from=7-3, to=8-4]
	\arrow[""{name=6, anchor=center, inner sep=0}, "{1\circ {B_{(C_2C_1)}}^{a}}"{description}, from=8-4, to=9-5]
	\arrow[""{name=7, anchor=center, inner sep=0}, "{1\circ {B_{C_1}}^{a}}"{description}, from=4-4, to=5-5]
	\arrow["{(A\mid B\mid C_2C_1)}", shorten <=85pt, shorten >=85pt, Rightarrow, from=4, to=3]
	\arrow["{(F^a)^B_{C_2,C_1}}"', shorten <=34pt, shorten >=34pt, Rightarrow, from=7, to=6]
	\arrow[shorten <=51pt, shorten >=51pt, Rightarrow, dashed, from=0, to=4]
	\arrow[shorten <=17pt, shorten >=17pt, Rightarrow, dashed, from=2, to=7]
	\arrow["{(F\midscript{b'})^A_{C_2,C_1}}"', shorten <=51pt, shorten >=51pt, Rightarrow, from=1, to=5]
\end{tikzcd}
}
\end{equation*}

\emph{Compatibility of $(A\mid B\mid C)$ with 2-cells $\alpha\colon A\to A'$:}
\begin{equation}\label{ax:tern_alp:B:C}
\tag{$\alpha:B:C$}
-\hspace*{-2cm}\adjustbox{scale=0.65}{
\begin{tikzcd}[ampersand replacement=\&]
	\& {F^{a'}_{b'}(C)F^{a'}_{c}(B)F^{b}_{c}(A')} \\
	{F^{a'}_{b'}(C)F^{a'}_{c}(B)F^{b}_{c}(A)} \&\& {F^{a'}_{b'}(C)F^{b'}_{c}(A')F^{a}_{c}(B)} \\
	\& {F^{a'}_{b'}(C)F^{b'}_{c}(A)F^{a}_{c}(B)} \\
	{F^{a'}_{c'}(B)F^{a'}_{b}(C)F^{b}_{c}(A)} \&\& {F^{b'}_{c'}(A')F^{a}_{b'}(C)F^{a}_{c}(B)} \\
	\& {F^{b'}_{c'}(A)F^{a}_{b'}(C)F^{a}_{c}(B)} \\
	{F^{a'}_{c'}(B)F^{b}_{c'}(A)F^{a}_{b}(C)} \&\& {F^{b'}_{c'}(A')F^{a}_{c'}(B)F^{a}_{b}(C)} \\
	\& {F^{b'}_{c'}(A)F^{a}_{c'}(B)F^{a}_{b}(C)}
	\arrow[""{name=0, anchor=center, inner sep=0}, "{1\circ{A_B}^c}"{description}, from=2-1, to=3-2]
	\arrow["{{A_C}^{b'}\circ1}"{description}, from=3-2, to=5-2]
	\arrow["{1\circ {B_C}^{a}}"{description}, from=5-2, to=7-2]
	\arrow["{1\circ{A_C}^b}"{description}, from=4-1, to=6-1]
	\arrow[""{name=1, anchor=center, inner sep=0}, "{{A_B}^{c'}\circ1}"', from=6-1, to=7-2]
	\arrow["{{B_C}^{a'}\circ1}"{description}, from=2-1, to=4-1]
	\arrow[""{name=2, anchor=center, inner sep=0}, "{1\circ F^{b'}_{c}(\alpha)\circ1}"{description}, from=3-2, to=2-3]
	\arrow["{{{A'}_C}^{b'}\circ1}"{description}, from=2-3, to=4-3]
	\arrow["{1\circ {B_C}^{a}}"{description}, from=4-3, to=6-3]
	\arrow[""{name=3, anchor=center, inner sep=0}, "{F^{b'}_{c}(\alpha)\circ1}"', from=7-2, to=6-3]
	\arrow[""{name=4, anchor=center, inner sep=0}, "{F^{b'}_{c}(\alpha)\circ1}"{description}, from=5-2, to=4-3]
	\arrow["{1\circ F^{b'}_{c}(\alpha)}", from=2-1, to=1-2]
	\arrow["{1\circ{{A'}_B}^c}", from=1-2, to=2-3]
	\arrow["{1\circ{\alpha_B}^c}", shorten <=10pt, shorten >=10pt, Rightarrow, from=1-2, to=3-2]
	\arrow["{{\alpha_C}^{b'}\circ1}", shorten <=17pt, shorten >=17pt, Rightarrow, from=2, to=4]
	\arrow[shorten <=17pt, shorten >=17pt, Rightarrow, dashed, from=4, to=3]
	\arrow["{(A\mid B\mid C)}"{description}, shorten <=34pt, shorten >=34pt, Rightarrow, from=0, to=1]
\end{tikzcd}
=
\begin{tikzcd}[ampersand replacement=\&]
	\& {F^{a'}_{b'}(C)F^{a'}_{c}(B)F^{b}_{c}(A')} \\
	{F^{a'}_{b'}(C)F^{a'}_{c}(B)F^{b}_{c}(A)} \&\& {F^{a'}_{b'}(C)F^{b'}_{c}(A')F^{a}_{c}(B)} \\
	\& {F^{a'}_{c'}(B)F^{a'}_{b}(C)F^{b}_{c}(A')} \\
	{F^{a'}_{c'}(B)F^{a'}_{b}(C)F^{b}_{c}(A)} \&\& {F^{b'}_{c'}(A')F^{a}_{b'}(C)F^{a}_{c}(B)} \\
	\& {F^{a'}_{c'}(B)F^{b}_{c'}(A')F^{a}_{b}(C)} \\
	{F^{a'}_{c'}(B)F^{b}_{c'}(A)F^{a}_{b}(C)} \&\& {F^{b'}_{c'}(A')F^{a}_{c'}(B)F^{a}_{b}(C)} \\
	\& {F^{b'}_{c'}(A)F^{a}_{c'}(B)F^{a}_{b}(C)}
	\arrow["{1\circ{A_C}^b}"{description}, from=4-1, to=6-1]
	\arrow["{{A_B}^{c'}\circ1}"', from=6-1, to=7-2]
	\arrow["{{B_C}^{a'}\circ1}"{description}, from=2-1, to=4-1]
	\arrow["{{{A'}_C}^{b'}\circ1}"{description}, from=2-3, to=4-3]
	\arrow["{1\circ {B_C}^{a}}"{description}, from=4-3, to=6-3]
	\arrow["{F^{b'}_{c}(\alpha)\circ1}"', from=7-2, to=6-3]
	\arrow[""{name=0, anchor=center, inner sep=0}, "{1\circ F^{b'}_{c}(\alpha)}", from=2-1, to=1-2]
	\arrow[""{name=1, anchor=center, inner sep=0}, "{1\circ{{A'}_B}^c}", from=1-2, to=2-3]
	\arrow["{{B_C}^{a'}\circ1}"{description}, from=1-2, to=3-2]
	\arrow["{1\circ{{A'}_C}^b}"{description}, from=3-2, to=5-2]
	\arrow[""{name=2, anchor=center, inner sep=0}, "{{{A'}_B}^c\circ1}"{description}, from=5-2, to=6-3]
	\arrow[""{name=3, anchor=center, inner sep=0}, "{1\circ F^{b'}_{c}(\alpha)\circ1}"{description}, from=6-1, to=5-2]
	\arrow[""{name=4, anchor=center, inner sep=0}, "{1\circ F^{b'}_{c}(\alpha)}"{description}, from=4-1, to=3-2]
	\arrow["{{\alpha_B}^{c'}\circ1}", shorten <=13pt, shorten >=13pt, Rightarrow, from=5-2, to=7-2]
	\arrow[shorten <=17pt, shorten >=17pt, Rightarrow, dashed, from=0, to=4]
	\arrow["{1\circ{\alpha_C}^b}", shorten <=17pt, shorten >=17pt, Rightarrow, from=4, to=3]
	\arrow["{(A'\mid B\mid C)}"{description}, shorten <=34pt, shorten >=34pt, Rightarrow, from=1, to=2]
\end{tikzcd}
}
\end{equation}

\emph{Compatibility of $(A\mid B\mid C)$ with 2-cells $\beta\colon B\to B'$:}
\begin{equation}\label{ax:tern_A:bet:C}
\tag{$A:\beta:C$}
-\hspace*{-2cm}\adjustbox{scale=0.65}{
\begin{tikzcd}[ampersand replacement=\&]
	\& {F^{a'}_{b'}(C)F^{b'}_{c}(A)F^{a}_{c}(B)} \\
	{F^{a'}_{b'}(C)F^{a'}_{c}(B)F^{b}_{c}(A)} \&\& {F^{a'}_{b'}(C)F^{b'}_{c}(A)F^{a}_{c}(B')} \\
	\& {F^{a'}_{b'}(C)F^{a'}_{c}(B')F^{b}_{c}(A)} \\
	{F^{a'}_{c'}(B)F^{a'}_{b}(C)F^{b}_{c}(A)} \&\& {F^{b'}_{c'}(A)F^{a}_{b'}(C)F^{a}_{c}(B')} \\
	\& {F^{a'}_{c'}(B')F^{a'}_{b}(C)F^{b}_{c}(A)} \\
	{F^{a'}_{c'}(B)F^{b}_{c'}(A)F^{a}_{b}(C)} \&\& {F^{b'}_{c'}(A)F^{a}_{c'}(B')F^{a}_{b}(C)} \\
	\& {F^{a'}_{c'}(B')F^{b}_{c'}(A)F^{a}_{b}(C)}
	\arrow["{1\circ{A_B}^c}"{pos=0.3}, from=2-1, to=1-2]
	\arrow["{1\circ{A_C}^b}"', from=4-1, to=6-1]
	\arrow["{{B_C}^{a'}\circ1}"', from=2-1, to=4-1]
	\arrow[""{name=0, anchor=center, inner sep=0}, "{F^{a}_{c'}(\beta)\circ1}"'{pos=0.4}, from=6-1, to=7-2]
	\arrow[""{name=1, anchor=center, inner sep=0}, "{{A_{B'}}^{c'}\circ1}"', from=7-2, to=6-3]
	\arrow["{1\circ F^{a}_{c'}(\beta)}"{pos=0.7}, from=1-2, to=2-3]
	\arrow["{1\circ{B'_C}^{a}}", from=4-3, to=6-3]
	\arrow[""{name=2, anchor=center, inner sep=0}, "{1\circ F^{a'}_{c}(\beta)\circ1}"{description}, from=2-1, to=3-2]
	\arrow[""{name=3, anchor=center, inner sep=0}, "{1\circ{A_{B'}}^c}"{description}, from=3-2, to=2-3]
	\arrow["{{B'_C}^{a'}\circ1}"{description}, from=3-2, to=5-2]
	\arrow["{1\circ{A_C}^b}"{description}, from=5-2, to=7-2]
	\arrow[""{name=4, anchor=center, inner sep=0}, "{F^{a'}_{c'}(\beta)\circ1}"{description}, from=4-1, to=5-2]
	\arrow["{1\circ{A_\beta}^c}", shorten <=13pt, shorten >=13pt, Rightarrow, from=1-2, to=3-2]
	\arrow["{{A_C}^{b'}\circ1}", from=2-3, to=4-3]
	\arrow[shorten <=17pt, shorten >=17pt, Rightarrow, dashed, from=4, to=0]
	\arrow["{1\circ{\beta_C}^{a'}}", shorten <=17pt, shorten >=17pt, Rightarrow, from=2, to=4]
	\arrow["{(A\mid B'\mid C)}"{description}, shorten <=34pt, shorten >=34pt, Rightarrow, from=3, to=1]
\end{tikzcd}
=
\begin{tikzcd}[ampersand replacement=\&]
	\& {F^{a'}_{b'}(C)F^{b'}_{c}(A)F^{a}_{c}(B)} \\
	{F^{a'}_{b'}(C)F^{a'}_{c}(B)F^{b}_{c}(A)} \&\& {F^{a'}_{b'}(C)F^{b'}_{c}(A)F^{a}_{c}(B')} \\
	\& {F^{b'}_{c'}(A)F^{a}_{b'}(C)F^{a}_{c}(B)} \\
	{F^{a'}_{c'}(B)F^{a'}_{b}(C)F^{b}_{c}(A)} \&\& {F^{b'}_{c'}(A)F^{a}_{b'}(C)F^{a}_{c}(B')} \\
	\& {F^{b'}_{c'}(A)F^{a}_{c'}(B)F^{a}_{b}(C)} \\
	{F^{a'}_{c'}(B)F^{b}_{c'}(A)F^{a}_{b}(C)} \&\& {F^{b'}_{c'}(A)F^{a}_{c'}(B')F^{a}_{b}(C)} \\
	\& {F^{a'}_{c'}(B')F^{b}_{c'}(A)F^{a}_{b}(C)}
	\arrow[""{name=0, anchor=center, inner sep=0}, "{1\circ{A_B}^c}"{pos=0.3}, from=2-1, to=1-2]
	\arrow["{{A_C}^{b'}\circ1}"{description}, from=1-2, to=3-2]
	\arrow["{1\circ {B_C}^{a}}"{description}, from=3-2, to=5-2]
	\arrow["{1\circ{A_C}^b}"', from=4-1, to=6-1]
	\arrow[""{name=1, anchor=center, inner sep=0}, "{{A_B}^{c'}\circ1}"{description}, from=6-1, to=5-2]
	\arrow["{{B_C}^{a'}\circ1}"', from=2-1, to=4-1]
	\arrow["{F^{a'}_{c'}(\beta)\circ1}"'{pos=0.4}, from=6-1, to=7-2]
	\arrow[""{name=2, anchor=center, inner sep=0}, "{1\circ F^{a}_{c'}(\beta)\circ1}"{description}, from=5-2, to=6-3]
	\arrow["{{A_{B'}}^{c'}\circ1}"', from=7-2, to=6-3]
	\arrow["{{A_\beta}^{c'}\circ1}", shorten <=13pt, shorten >=13pt, Rightarrow, from=5-2, to=7-2]
	\arrow[""{name=3, anchor=center, inner sep=0}, "{1\circ F^{a}_{c'}(\beta)}"{pos=0.7}, from=1-2, to=2-3]
	\arrow["{{A_C}^{b'}\circ1}", from=2-3, to=4-3]
	\arrow[""{name=4, anchor=center, inner sep=0}, "{1\circ F^{a}_{c}(\beta)}"{description}, from=3-2, to=4-3]
	\arrow["{1\circ{B'_C}^{a}}", from=4-3, to=6-3]
	\arrow["{(A\mid B\mid C)}"{description}, shorten <=34pt, shorten >=34pt, Rightarrow, from=0, to=1]
	\arrow[shorten <=17pt, shorten >=17pt, Rightarrow, dashed, from=3, to=4]
	\arrow["{1\circ{\beta_C}^{a}}", shorten <=17pt, shorten >=17pt, Rightarrow, from=4, to=2]
\end{tikzcd}
}
\end{equation}

\emph{Compatibility of $(A\mid B\mid C)$ with 2-cells $\gamma\colon C\to C'$:}
\begin{equation}\label{ax:tern_A:B:gam}
\tag{$A:B:\gamma$}
\begin{split}
\hspace*{-2cm}\adjustbox{scale=0.65}{
\begin{tikzcd}[ampersand replacement=\&]
	{F^{a'}_{b'}(C)F^{a'}_{c}(B)F^{b}_{c}(A)} \&\& {F^{a'}_{b'}(C)F^{b'}_{c}(A)F^{a}_{c}(B)} \\
	\& {F^{a'}_{c'}(B)F^{a'}_{b}(C)F^{b}_{c}(A)} \&\& {F^{b'}_{c'}(A)F^{a}_{b'}(C)F^{a}_{c}(B)} \\
	{F^{a'}_{b'}( C')F^{a'}_{c}(B)F^{b}_{c}(A)} \&\& {F^{a'}_{c'}(B)F^{b}_{c'}(A)F^{a}_{b}(C)} \&\& {F^{b'}_{c'}(A)F^{a}_{c'}(B)F^{a}_{b}(C)} \\
	\& {F^{a'}_{c'}(B)F^{a'}_{b}(C')F^{b}_{c}(A)} \\
	\&\& {F^{a'}_{c'}(B)F^{b}_{c'}(A)F^{a}_{b}(C')} \&\& {F^{b'}_{c'}(A)F^{a}_{c'}(B)F^{a}_{b}(C')}
	\arrow[""{name=0, anchor=center, inner sep=0}, "{1\circ{A_B}^c}", from=1-1, to=1-3]
	\arrow["{{A_C}^{b'}\circ1}", from=1-3, to=2-4]
	\arrow["{1\circ {B_C}^{a}}", from=2-4, to=3-5]
	\arrow[""{name=1, anchor=center, inner sep=0}, "{1\circ{A_C}^b}"{description}, from=2-2, to=3-3]
	\arrow[""{name=2, anchor=center, inner sep=0}, "{{A_B}^{c'}\circ1}"{description}, from=3-3, to=3-5]
	\arrow[""{name=3, anchor=center, inner sep=0}, "{{B_C}^{a'}\circ1}"{description}, from=1-1, to=2-2]
	\arrow["{1\circ F^{a}_{b}(\gamma)}"{description}, from=3-3, to=5-3]
	\arrow["{1\circ F^{a}_{b}(\gamma)}"{description}, from=3-5, to=5-5]
	\arrow[""{name=4, anchor=center, inner sep=0}, "{{A_B}^{c'}\circ1}"', from=5-3, to=5-5]
	\arrow["{F^{a'}_{b'}(\gamma)\circ1 }"', from=1-1, to=3-1]
	\arrow["{1\circ F^{a'}_{b}(\gamma)\circ1 }"{description}, from=2-2, to=4-2]
	\arrow[""{name=5, anchor=center, inner sep=0}, "{{B_{ C'}}^{a'}\circ1}"', from=3-1, to=4-2]
	\arrow[""{name=6, anchor=center, inner sep=0}, "{{A_{C'}}^b\circ1}"', from=4-2, to=5-3]
	\arrow["{(A\mid B\mid C)}", shorten <=67pt, shorten >=67pt, Rightarrow, from=0, to=2]
	\arrow[shorten <=17pt, shorten >=17pt, Rightarrow, dashed, from=2, to=4]
	\arrow["{1\circ {A_{\gamma}}^b}", shorten <=17pt, shorten >=17pt, Rightarrow, from=1, to=6]
	\arrow["{{B_\gamma}^{a'}\circ1}"', shorten <=17pt, shorten >=17pt, Rightarrow, from=3, to=5]
\end{tikzcd}
}\\
\hspace*{-1.5cm}\adjustbox{scale=0.65}{
\begin{tikzcd}[ampersand replacement=\&]
	\& {F^{a'}_{b'}(C)F^{a'}_{c}(B)F^{b}_{c}(A)} \&\& {F^{a'}_{b'}(C)F^{b'}_{c}(A)F^{a}_{c}(B)} \\
	\&\&\&\& {F^{b'}_{c'}(A)F^{a}_{b'}(C)F^{a}_{c}(B)} \\
	{=} \& {F^{a'}_{b'}( C')F^{a'}_{c}(B)F^{b}_{c}(A)} \&\& {F^{a'}_{b'}(C')F^{b'}_{c}(A)F^{a}_{c}(B)} \&\& {F^{b'}_{c'}(A)F^{a}_{c'}(B)F^{a}_{b}(C)} \\
	\&\& {F^{a'}_{c'}(B)F^{a'}_{b}(C')F^{b}_{c}(A)} \&\& {F^{b'}_{c'}(A)F^{a}_{b'}(C')F^{a}_{c}(B)} \\
	\&\&\& {F^{a'}_{c'}(B)F^{b}_{c'}(A)F^{a}_{b}(C')} \&\& {F^{b'}_{c'}(A)F^{a}_{c'}(B)F^{a}_{b}(C')}
	\arrow[""{name=0, anchor=center, inner sep=0}, "{1\circ{A_B}^c}", from=1-2, to=1-4]
	\arrow[""{name=1, anchor=center, inner sep=0}, "{{A_C}^{b'}\circ1}", from=1-4, to=2-5]
	\arrow[""{name=2, anchor=center, inner sep=0}, "{1\circ {B_C}^{a}}", from=2-5, to=3-6]
	\arrow["{1\circ F^{a}_{b}(\gamma)}", from=3-6, to=5-6]
	\arrow[""{name=3, anchor=center, inner sep=0}, "{{A_B}^{c'}\circ1}"', from=5-4, to=5-6]
	\arrow["{F^{a'}_{b'}(\gamma)\circ1 }"', from=1-2, to=3-2]
	\arrow["{{B_{ C'}}^{a'}\circ1}"', from=3-2, to=4-3]
	\arrow["{{A_{C'}}^b\circ1}"', from=4-3, to=5-4]
	\arrow["{F^{a'}_{b'}(\gamma)\circ1 }"', from=1-4, to=3-4]
	\arrow[""{name=4, anchor=center, inner sep=0}, "{1\circ{A_B}^c}", from=3-2, to=3-4]
	\arrow[""{name=5, anchor=center, inner sep=0}, "{{A_{C'}}^{b'}\circ1}"{description}, from=3-4, to=4-5]
	\arrow[""{name=6, anchor=center, inner sep=0}, "{1\circ {B_{C'}}^{a}}"{description}, from=4-5, to=5-6]
	\arrow["{1\circ F^{a}_{b'}(\gamma)\circ1}"{description}, from=2-5, to=4-5]
	\arrow[shorten <=17pt, shorten >=17pt, Rightarrow, dashed, from=0, to=4]
	\arrow["{1\circ{B_\gamma}^{a}}", shorten <=17pt, shorten >=17pt, Rightarrow, from=2, to=6]
	\arrow["{(A\mid B\mid C')}", shorten <=68pt, shorten >=68pt, Rightarrow, from=4, to=3]
	\arrow["{{A_\gamma}^{b'}\circ1}"', shorten <=17pt, shorten >=17pt, Rightarrow, from=1, to=5]
\end{tikzcd}
}
\end{split}
\end{equation}

\emph{Degeneracy equations:}

\begin{equation}\label{ax:tern_D_1a:B:C}
\tag{$1_a:B:C$}
(1_a\mid B\mid C)=1_{{B_C}^a}
\end{equation}
\begin{equation}\label{ax:tern_D_A:1b:C}
\tag{$A:1_b:C$}
(A\mid 1_b\mid C)=1_{{A_C}^b}
\end{equation} 
\begin{equation}\label{ax:tern_D_A:B:1c}
\tag{$A:B:1_c$}
(A\mid B\mid 1_c)=1_{{A_B}^c}
\end{equation}  

\subsection{Axiom for 4-ary maps}
\label{app:ax_4-ary}

There is a single axiom for $4$-ary maps, which we call the \emph{mecon axiom}, since its shape is a truncated octahedron, also known as a mecon.  Each of the nodes of the diagram consists of a composite of four $1$-cells, and we encode these using sequences of $A,B,C$ and $D$, which uniquely encode the composite as follows.  Firstly, let us write $F(a,-,-,-)$ for the ternary map which we usually call $F^a$ and so on; then, $F(a,b,-,d)$ represents the associated pseudomap.  Each appearance of $C$ on a node represents a $1$-cell of the form $F(x,y,C,z)$ where $x\in\lbrace a,a'\rbrace$, $y\in\lbrace b,b'\rbrace$, $z\in\lbrace d,d'\rbrace$. and similarly for $A,B$ and $D$.  The only question is which of the two values we choose for each of $x,y$ and $z$.  This is determined by the ordering.  For instance, consider $DBAC$.  Since $C$ appears at the end, before $D,B,A$, none of $d,b$ nor $a$ are primed --- thus $C$ is replaced by $F(a,b,C,d)$; as $A$ appears after $C$, but before $D$ and $B$, we prime $c$ but neither $d$ nor $b$ --- thus $A$ is replaced by $F(A,b,c',d)$.  Continuing in this way, priming lowercase letters that appear before the uppercase letter in question, we see that $DBAC$ encodes the composite $1$-cell $F(a',b',c',D)F(a',B,c',d)F(A,b,c',D)F(a,b,C,d)$.

\begin{equation}\label{ax:4-ary_mec}
\tag{Mcn}
\hspace*{-2cm}\adjustbox{scale=0.65}{
\begin{tikzcd}[ampersand replacement=\&]
	\&\& DCAB \&\& DACB \\
	DCBA \&\&\&\&\&\& DABC \\
	\&\& DBCA \&\& DBAC \\
	CDBA \&\&\&\&\&\& ADBC \\
	\&\& BDCA \&\& BDAC \\
	CBDA \&\&\&\&\&\& ABDC \\
	\&\& BCDA \&\& BADC \\
	CBAD \&\&\&\&\&\& ABCD \\
	\&\& BCAD \&\& BACD
	\arrow[""{name=0, anchor=center, inner sep=0}, "{1\circ {A_C}^{b,d}}"{description}, from=3-3, to=3-5]
	\arrow["{1\circ{B_C}^{a',c}\circ1}"', from=4-1, to=6-1]
	\arrow["{{B_D}^{a',c'}\circ1}"{description}, from=3-3, to=5-3]
	\arrow["{1\circ {C_D}^{a',b}\circ1}"{description}, from=5-3, to=7-3]
	\arrow[""{name=1, anchor=center, inner sep=0}, "{{B_C}^{a',d'}\circ1}"{description}, from=6-1, to=7-3]
	\arrow["{1\circ{A_D}^{b,c}}"', from=6-1, to=8-1]
	\arrow[""{name=2, anchor=center, inner sep=0}, "{{B_C}^{a',d'}\circ1}"', from=8-1, to=9-3]
	\arrow["{1\circ{A_D}^{b,c}}"{description}, from=7-3, to=9-3]
	\arrow[""{name=3, anchor=center, inner sep=0}, "{1\circ {A_C}^{b,d}}"{description}, from=5-3, to=5-5]
	\arrow["{{B_D}^{a',c'}\circ1}"{description}, from=3-5, to=5-5]
	\arrow["{1\circ{B_D}^{a,c'}\circ1}", from=4-7, to=6-7]
	\arrow["{1\circ{A_D}^{b,c'}\circ1}"{description}, from=5-5, to=7-5]
	\arrow["{1\circ {C_D}^{a,b}}"{description}, from=7-5, to=9-5]
	\arrow[""{name=4, anchor=center, inner sep=0}, "{1\circ {A_C}^{b,d'}\circ1}"', from=9-3, to=9-5]
	\arrow[""{name=5, anchor=center, inner sep=0}, "{{A_B}^{c',d'}\circ1}"', from=9-5, to=8-7]
	\arrow["{1\circ {C_D}^{a,b}}", from=6-7, to=8-7]
	\arrow[""{name=6, anchor=center, inner sep=0}, "{{A_B}^{c',d'}\circ1}"{description}, from=7-5, to=6-7]
	\arrow[""{name=7, anchor=center, inner sep=0}, "{1\circ{B_C}^{a',d}\circ1}"{description}, from=2-1, to=3-3]
	\arrow["{{C_D}^{a',b'}\circ1}"', from=2-1, to=4-1]
	\arrow["{1\circ {A_B}^{c,d}}", from=2-1, to=1-3]
	\arrow["{1\circ{B_C}^{a,d}}", from=1-5, to=2-7]
	\arrow[""{name=8, anchor=center, inner sep=0}, "{1\circ {A_B}^{c',d}\circ1}"{description}, from=3-5, to=2-7]
	\arrow["{{A_D}^{b',c'}\circ1}", from=2-7, to=4-7]
	\arrow[""{name=9, anchor=center, inner sep=0}, "{1\circ {A_C}^{b',d}\circ1}", from=1-3, to=1-5]
	\arrow[shorten <=17pt, shorten >=17pt, Rightarrow, dashed, from=0, to=3]
	\arrow[shorten <=17pt, shorten >=17pt, Rightarrow, dashed, from=1, to=2]
	\arrow[shorten <=17pt, shorten >=17pt, Rightarrow, dashed, from=6, to=5]
	\arrow["{(B\mid C\mid D)^{a'}\circ1}"{description}, shorten <=26pt, shorten >=26pt, Rightarrow, from=7, to=1]
	\arrow["{(A\mid B\mid D)^{c'}\circ1}"{description}, shorten <=26pt, shorten >=26pt, Rightarrow, from=8, to=6]
	\arrow["{1\circ(A\mid B\mid C)^{d}}"{description}, shorten <=13pt, shorten >=13pt, Rightarrow, from=9, to=0]
	\arrow["{1\circ(A\mid C\mid D)^{b}}"{description, pos=0.6}, shorten <=34pt, shorten >=17pt, Rightarrow, from=3, to=4]
\end{tikzcd}
=
\begin{tikzcd}[ampersand replacement=\&]
	\&\& DCAB \&\& DACB \\
	DCBA \&\&\&\&\&\& DABC \\
	\&\& CDAB \&\& ADCB \\
	CDBA \&\&\&\&\&\& ADBC \\
	\&\& CADB \&\& ACDB \\
	CBDA \&\&\&\&\&\& ABDC \\
	\&\& CABD \&\& ACBD \\
	CBAD \&\&\&\&\&\& ABCD \\
	\&\& BCAD \&\& BACD
	\arrow["{1\circ{B_C}^{a',c}\circ1}"', from=4-1, to=6-1]
	\arrow["{1\circ{A_D}^{b,c}}"', from=6-1, to=8-1]
	\arrow["{{B_C}^{a',d'}\circ1}"', from=8-1, to=9-3]
	\arrow["{1\circ{B_D}^{a,c'}\circ1}", from=4-7, to=6-7]
	\arrow[""{name=0, anchor=center, inner sep=0}, "{1\circ {A_C}^{b,d'}\circ1}"', from=9-3, to=9-5]
	\arrow["{{A_B}^{c',d'}\circ1}"', from=9-5, to=8-7]
	\arrow["{1\circ {C_D}^{a,b}}", from=6-7, to=8-7]
	\arrow["{{C_D}^{a',b'}\circ1}"', from=2-1, to=4-1]
	\arrow[""{name=1, anchor=center, inner sep=0}, "{1\circ {A_B}^{c,d}}", from=2-1, to=1-3]
	\arrow[""{name=2, anchor=center, inner sep=0}, "{1\circ{B_C}^{a,d}}", from=1-5, to=2-7]
	\arrow["{{A_D}^{b',c'}\circ1}", from=2-7, to=4-7]
	\arrow[""{name=3, anchor=center, inner sep=0}, "{1\circ{A_B}^{c,d'}\circ1}"{description}, from=8-1, to=7-3]
	\arrow[""{name=4, anchor=center, inner sep=0}, "{{A_C}^{b',d'}\circ1}", from=7-3, to=7-5]
	\arrow[""{name=5, anchor=center, inner sep=0}, "{1\circ{B_C}^{a,d'}\circ1}"{description}, from=7-5, to=8-7]
	\arrow["{1\circ{B_D}^{a,c}}"{description}, from=5-3, to=7-3]
	\arrow["{1\circ{A_D}^{b',c}\circ1}"{description}, from=3-3, to=5-3]
	\arrow[""{name=6, anchor=center, inner sep=0}, "{1\circ {A_B}^{c,d}}"{description}, from=4-1, to=3-3]
	\arrow["{{C_D}^{a',b'}\circ1}"{description}, from=1-3, to=3-3]
	\arrow[""{name=7, anchor=center, inner sep=0}, "{{A_C}^{b',d'}\circ1}"{description}, from=5-3, to=5-5]
	\arrow["{1\circ{C_D}^{a,b'}\circ1}"{description}, from=3-5, to=5-5]
	\arrow["{{A_D}^{b',c'}\circ1}"{description}, from=1-5, to=3-5]
	\arrow[""{name=8, anchor=center, inner sep=0}, "{1\circ{B_C}^{a,d}}"{description}, from=3-5, to=4-7]
	\arrow["{1\circ{B_D}^{a,c}}"{description}, from=5-5, to=7-5]
	\arrow[""{name=9, anchor=center, inner sep=0}, "{1\circ {A_C}^{b',d}\circ1}", from=1-3, to=1-5]
	\arrow[shorten <=17pt, shorten >=17pt, Rightarrow, dashed, from=2, to=8]
	\arrow[shorten <=17pt, shorten >=17pt, Rightarrow, dashed, from=7, to=4]
	\arrow["{(A\mid B \mid C)^{d'}\circ1}"{description}, shorten <=9pt, shorten >=9pt, Rightarrow, from=4, to=0]
	\arrow["{1\circ(A\mid B\mid D)^{c}}"{description}, shorten <=26pt, shorten >=26pt, Rightarrow, from=6, to=3]
	\arrow[shorten <=17pt, shorten >=17pt, Rightarrow, dashed, from=1, to=6]
	\arrow["{(A\mid C\mid D)^{b'}\circ1}"{description, pos=0.6}, shorten <=34pt, shorten >=17pt, Rightarrow, from=9, to=7]
	\arrow["{1\circ (B\mid C\mid D)^{a}}"{description}, shorten <=26pt, shorten >=26pt, Rightarrow, from=8, to=5]
\end{tikzcd}
}
\end{equation}

\section{Incubator for substitution of binary into binary}

\subsection{First variable}
\label{app:inc-sub-b-in-b-1st}
\begin{center}
\begin{tikzpicture}[triangle/.style = {fill=yellow!50, regular polygon, regular polygon sides=3,rounded corners}]

\path
	(6.5,1.5) node [triangle,draw,shape border rotate=-90,label=135:$\ca$,label=230:$\cb$,inner sep=2pt] (b') {$F$}
	(8.5,1) node [triangle,draw,shape border rotate=-90,label=135:$\cx$,label=230:$\cc$,inner sep=2pt] (c') {$G$};

\draw [-] (7.5,0.65) .. controls +(right:0.25cm) and +(left:0.25cm).. (c'.223);
\draw [] (b') .. controls (7.5,1.5) and +(left:0.8cm).. (c'.137);
\draw [] (c') to node[fill=white] {$\ce$} (10,1);
\draw [-] (5.5,1.15) .. controls +(right:0.25cm) and +(left:0.25cm).. (b'.223);
\draw [-] (5.5,1.85) .. controls +(right:0.25cm) and +(left:0.25cm).. (b'.137);
\end{tikzpicture}
\end{center}

\[\hspace{-1.5cm}\begin{tikzcd}[ampersand replacement=\&]
	{G^{F^{a'}_{b'}}_{C}G_{c}^{F^{a'}_B}G^{F^A_{b}}_{c}} \&\&\& {G^{F^{a'}_{b'}}_{C}G^{F^{a'}_BF^A_{b}}_{c}} \&\&\& {G^{F^{a'}_{b'}}_{C}G^{F^A_{b'}F^{a'}_B}_{c}} \&\&\& {G^{F^{a'}_{b'}}_{C}G_{c}^{F^A_{b'}}G^{F^{a'}_B}_{c}} \\
	{G_{c'}^{F^{a'}_B}G^{F^{a'}_{b}}_{C}G^{F^A_{b}}_{c}} \&\&\&\&\&\&\&\&\& {G^{F^A_{b'}}_{c'}G^{F^{a}_{b'}}_{C}G^{F^{a'}_B}_{c}} \\
	{G_{c'}^{F^{a'}_B}G^{F^A_{b}}_{c'}G^{F^{a}_{b}}_{C}} \&\&\& {G^{F^{a'}_BF^A_{b}}_{c'}G^{F^{a}_{b}}_{C}} \&\&\& {G^{F^A_{b'}F^{a'}_B}_{c'}G^{F^{a}_{b}}_{C}} \&\&\& {G^{F^A_{b'}}_{c'}G^{F^{a'}_B}_{c'}G^{F^{a}_{b}}_{C}}
	\arrow[""{name=0, anchor=center, inner sep=0}, "{1\circ G^{F^{a'}_B,F^A_{b}}_{c}}", from=1-1, to=1-4]
	\arrow[""{name=1, anchor=center, inner sep=0}, "{1\circ G_c(A_B)}", from=1-4, to=1-7]
	\arrow["{{(F^{a'}_B)}_C\circ1}"', from=1-1, to=2-1]
	\arrow["{1\circ {(F^A_{b})}_C}"', from=2-1, to=3-1]
	\arrow[""{name=2, anchor=center, inner sep=0}, "{G^{F^{a'}_B,F^A_{b}}_{c'}\circ1}"', from=3-1, to=3-4]
	\arrow[""{name=3, anchor=center, inner sep=0}, "{(G^{F^A_{b'},F^{a}_B}_{c'})^{-1}\circ1}"', from=3-7, to=3-10]
	\arrow["{{(F^{A}_{b'})}_C\circ1}", from=1-10, to=2-10]
	\arrow["{1\circ {(F^a_{B})}_C}", from=2-10, to=3-10]
	\arrow["{(F^{a'}_BF^A_{b})_C}"{description}, from=1-4, to=3-4]
	\arrow[""{name=4, anchor=center, inner sep=0}, "{G_{c'}(A_B)\circ1}"', from=3-4, to=3-7]
	\arrow[""{name=5, anchor=center, inner sep=0}, "{1\circ (G^{F^A_{b'},F^{a}_B}_{c})^{-1}}", from=1-7, to=1-10]
	\arrow["{(F^A_{b'}F^{a'}_B)_C}"{description}, from=1-7, to=3-7]
	\arrow["{{(A_B)}_C}"{description}, shorten <=20pt, shorten >=20pt, Rightarrow, from=1, to=4]
	\arrow["{G_C^{F^{a'}_B,F^A_{b}}}"{description}, shorten <=20pt, shorten >=20pt, Rightarrow, from=0, to=2]
	\arrow["{(G_C^{F^A_{b'},F^{a}_B})\sharp}"{description}, shorten <=20pt, shorten >=20pt, Rightarrow, from=5, to=3]
\end{tikzcd}\]

From the diagram above, we can see that if $G_c$ is a Gray-functor, then the incubator is just $(A_B)_C$.  

\subsection{Second variable}
\label{app:inc-sub-b-in-b-2nd}
\begin{center}
\begin{tikzpicture}[triangle/.style = {fill=yellow!50, regular polygon, regular polygon sides=3,rounded corners}]

\path
        (6.5,0.5) node [triangle,draw,shape border rotate=-90,label=135:$\cb$,label=230:$\cc$,inner sep=2pt] (q) {$F$} 
	(8.5,1) node [triangle,draw,shape border rotate=-90,label=135:$\ca$,label=230:$\cx$,inner sep=2pt] (c') {$G$};

\draw [-] (q) .. controls (7.5,0.5) and +(left:0.8cm).. (c'.223);
\draw [] (7.5,1.35) .. controls +(right:0.25cm) and +(left:0.25cm).. (c'.137);
\draw [] (c') to node[fill=white] {$\ce$} (10,1);
\draw [-] (5.5,0.15) .. controls +(right:0.25cm) and +(left:0.25cm).. (q.223);
\draw [-] (5.5,0.85) .. controls +(right:0.25cm) and +(left:0.25cm).. (q.137);
\end{tikzpicture}
\end{center}
\[\begin{tikzcd}[ampersand replacement=\&]
	{G_{F^{b'}_{C}}^{a'}G_{F^{B}_{c}}^{a'}G_{F^{b}_{c}}^{A}} \&\& {G_{F^{b'}_{C}}^{a'}G_{F^{b'}_{c}}^{A}G_{F^{B}_{c}}^{a}} \&\& {G_{F^{b'}_{c'}}^{A}G_{F^{b'}_{C}}^{a}G_{F^{B}_{c}}^{a}} \\
	{G_{F^{b'}_{C}F^{B}_{c}}^{a'}G_{F^{b}_{c}}^{A}} \&\&\&\& {G_{F^{b'}_{c'}}^{A}G_{F^{b'}_{C}F^{B}_{c}}^{a}} \\
	{G_{F^{B}_{c'}F^{b}_{C}}^{a'}G_{F^{b}_{c}}^{A}} \&\&\&\& {G_{F^{b'}_{c'}}^{A}G_{F^{B}_{c'}F^{b}_{C}}^{a}} \\
	{G_{F^{B}_{c'}}^{a'}G_{F^{b}_{C}}^{a'}G_{F^{b}_{c}}^{A}} \&\& {G_{F^{B}_{c'}}^{a'}G_{F^{b}_{c'}}^{A}G_{F^{b}_{C}}^{a}} \&\& {G_{F^{b'}_{c'}}^{A}G^{a}_{F^{B}_{c'}}G_{F^{b}_{C}}^{a}}
	\arrow["{G_{F^{b'}_{C},F^{B}_c}^{a'}\circ1}"', from=1-1, to=2-1]
	\arrow["{G^{a'}(B_C)\circ1}"', from=2-1, to=3-1]
	\arrow["{1\circ A_{(F^{B}_c)}}", from=1-1, to=1-3]
	\arrow["{A_{(F^{b'}_{C})}\circ1}", from=1-3, to=1-5]
	\arrow["{1\circ G_{F^{b'}_{C},F^{B}_c}^{a}}", from=1-5, to=2-5]
	\arrow["{1\circ A_{(F^{b}_C)}}"', from=4-1, to=4-3]
	\arrow["{A_{(F^{B}_{c'})}\circ1}"', from=4-3, to=4-5]
	\arrow["{(G_{F^B_{c'},F^{b}_C}^{c})^{-1}\circ1}"', from=3-1, to=4-1]
	\arrow[""{name=0, anchor=center, inner sep=0}, "{A_{(F^{B}_{c'}F^{b}_{C})}}"{description}, from=3-1, to=3-5]
	\arrow[""{name=1, anchor=center, inner sep=0}, "{A_{(F^{b'}_{C}F^{B}_c)}}"{description}, from=2-1, to=2-5]
	\arrow["{1\circ (G_{F^B_{c'},F^{b}_C}^{c})^{-1}}", from=3-5, to=4-5]
	\arrow["{1\circ G^{a}(B_C)}", from=2-5, to=3-5]
	\arrow["{G_{F^{b'}_{C},F^{B}_c}^{A}}"{xshift=0.1cm,pos=0.3}, shorten <=4pt, shorten >=12pt, Rightarrow, from=1-3, to=1]
	\arrow["{A_{(B_C)}}"{xshift=0.1cm}, shorten <=12pt, shorten >=12pt, Rightarrow, from=1, to=0]
	\arrow["{(G_{F^{B}_{c'},F^{b}_{C}}^{A})\flat}"{xshift=0.1cm,pos=0.6}, shorten <=12pt, shorten >=4pt, Rightarrow, from=0, to=4-3]
\end{tikzcd}\]

\bibliographystyle{plain}

\end{document}